\documentclass[10pt,reqno]{amsart} 

\usepackage{appendix}
\usepackage{float} 

\usepackage{soul}

\usepackage{hhline}

\usepackage{yhmath}\usepackage{epsfig,amssymb,amsmath,version}
\usepackage{amssymb,version,graphicx,fancybox,mathrsfs}

\usepackage{enumitem}
\usepackage{bbm}
\usepackage{url,hyperref}
\usepackage{subfigure}
\usepackage{color}
\usepackage{stmaryrd}
\usepackage{multirow}
\usepackage{booktabs,siunitx}
\usepackage{booktabs}
\usepackage{multicol,epstopdf}
\usepackage{xypic}
\usepackage[]{graphicx}
\usepackage[all]{xy}
\usepackage{tikz,ifthen}
\usetikzlibrary{positioning, math, decorations.markings, arrows.meta}
\usetikzlibrary{calc,shapes,cd}
\usetikzlibrary{knots} 
\usepgfmodule{decorations}
\allowdisplaybreaks
\tikzset{knotarrow/.pic={ \draw[edge, <-] (0,0) -- +(-.001,0);}}

\tikzset{edge/.style={line width=0.8}}
\tikzset{wall/.style={very thick}}

\tikzset{->-/.style n args={2}{decoration={markings, mark=at position #1 with {\arrow{#2}}}, postaction={decorate}}} 

\ExplSyntaxOn
\cs_new_eq:NN \ifstreqF \str_if_eq:nnF
\cs_new_eq:NN \ifstreqTF \str_if_eq:nnTF
\ExplSyntaxOff

\tikzset{-o-/.code 2 args={\ifstreqF{#2}{} 
{\ifstreqTF{#2}{>}
   {\pgfkeysalso{decoration={markings,mark=at position #1 with {\arrow[scale=0.8]{#2}}}
                    ,postaction={decorate}}
    }
   {\ifstreqTF{#2}{<}
       {\pgfkeysalso{decoration={markings,mark=at position #1 with {\arrow[scale=0.8]{#2}}}
                    ,postaction={decorate}}
        }
       {\pgfkeysalso{decoration={markings,
                    mark=at position #1 with
                    {\draw[black, fill={#2}] circle[radius=2pt];}}
                    ,postaction={decorate}}
        }
     }
  }}}

\allowdisplaybreaks[4]

\newtheorem{theorem}{Theorem}[section]
\newtheorem{lemma}[theorem]{Lemma}
\newtheorem{definition}[theorem]{Definition}
\newtheorem{corollary}[theorem]{Corollary}
\newtheorem{proposition}[theorem]{Proposition}
\newtheorem{remark}[theorem]{Remark}
\newtheorem{conjecture}[theorem]{Conjecture}

\newcommand{\Rlabel}[1]{\hyperref[R#1]{(R#1)}}

\newcommand{\wideoverunder}[2]{ 
\!\begin{array}{c} 
\scriptstyle{#1}\\[-.1in] 
-\!\!\!-\!\!\!-\!\!\!-\!\!\!-\!\!\!-\\[-.1in] 
\scriptstyle{#2} 
\end{array} 
\! 
} 

\newcommand{\gaa}{\mathsf{g}}

\newcommand{\bp}{\begin{proposition}}
\newcommand{\ep}{\end{proposition}}
\newcommand{\bpr}{\begin{proof}}
\newcommand{\epr}{\end{proof}}

\newcommand{\bt}{\begin{theorem}}
\newcommand{\et}{\end{theorem}}

\newcommand{\bl}{\begin{lemma}}
\newcommand{\el}{\end{lemma}}

\newcommand{\bcr}{\begin{corollary}}
\newcommand{\ecr}{\end{corollary}}

\newcommand{\be}{\begin{equation}}
\newcommand{\ee}{\end{equation}}

\newcommand{\bes}{\begin{equation*}}
\newcommand{\ees}{\end{equation*}}

\newcommand{\ba}{\begin{align}}
\newcommand{\ea}{\end{align}}

\newcommand{\bas}{\begin{align*}}
\newcommand{\eas}{\end{align*}}

\DeclareMathOperator{\skeleton}{\textbf{sk}}
\DeclareMathOperator{\im}{\mathrm{Im}}

\DeclareMathOperator{\Int}{\mathrm{Int}}

\DeclareMathOperator{\sgn}{\mathrm{sgn}}

\graphicspath{{./Figures/} }

\newcommand{\blue}[1]{{\color{blue}#1}}

\begin{document}
\bibliographystyle{alpha}

\title{Quantum cluster realization for projected stated ${\rm SL}_n$-skein algebras}

\author[Min Huang]{Min Huang}
\address{Min Huang, School of Mathematics (Zhuhai), Sun Yat-sen University, Zhuhai, China.}
\email{huangm97@mail.sysu.edu.cn}

\author[Zhihao Wang]{Zhihao Wang}
\address{Zhihao Wang, School of Mathematics, Korea Institute for Advanced Study (KIAS), 85 Hoegi-ro, Dongdaemun-gu, Seoul 02455, Republic of Korea}
\email{zhihaowang@kias.re.kr}


 \maketitle

\begin{abstract}
We introduce a quantum cluster algebra structure $\mathscr A_\omega(\mathfrak{S})$ inside the skew-field fractions ${\rm Frac}\bigl(\widetilde{\mathscr{S}}_\omega(\mathfrak{S})\bigr)$ of the projected stated ${\rm SL}_n$–skein algebra $\widetilde{\mathscr{S}}_\omega(\mathfrak{S})$ (the quotient of the reduced stated ${\rm SL}_n$–skein algebra by the kernel of the quantum trace map) for any triangulable pb surface $\mathfrak{S}$ without interior punctures. To study the relationships among the projected ${\rm SL}_n$-skein algebra $\widetilde{\mathscr{S}}_\omega(\mathfrak{S})$, the quantum cluster algebra $\mathscr A_\omega(\mathfrak{S})$, and its quantum upper cluster algebra $\mathscr U_\omega(\mathfrak{S})$, we construct a splitting homomorphism for $\mathscr U_\omega(\mathfrak{S})$ and show that it is compatible with the splitting homomorphism for $\widetilde{\mathscr{S}}_\omega(\mathfrak{S})$. When every connected component of $\mathfrak{S}$ contains at least two punctures, this compatibility allows us to prove that $\widetilde{\mathscr{S}}_\omega(\mathfrak{S})$ embeds into $\mathscr A_\omega(\mathfrak{S})$ by showing that the stated arcs joining two distinct boundary components of $\mathfrak{S}$ (which generate $\widetilde{\mathscr{S}}_\omega(\mathfrak{S})$) are, up to multiplication by a Laurent monomial in the frozen variables, exchangeable cluster variables. We further conjecture that these exchangeable cluster variables generate the quantum upper cluster algebra $\mathscr U_\omega(\mathfrak{S})$, which, if true, would imply the equality $\widetilde{\mathscr{S}}_\omega(\mathfrak{S})=\mathscr A_\omega(\mathfrak{S})=\mathscr U_\omega(\mathfrak{S})$.

\end{abstract}

\tableofcontents

\newcommand{\ca}{{\cev{a}  }}

\def\fS{\mathfrak{S}}
\def\dS{\widetilde{\cS}_\omega(\fS)}

\def\BZ{\mathbb Z}
\def\Id{\mathrm{Id}}
\def\Mat{\mathrm{Mat}}
\def\BN{\mathbb N}

\def \cb {\color{blue}}
\def \cred {\color{red}}
\def \cbf {\color{blue}\bf}
\def \credf {\color{red}\bf}
\definecolor{ligreen}{rgb}{0.0, 0.3, 0.0}
\def \cg {\color{ligreen}}
\def \cgf {\color{ligreen}\bf}
\definecolor{darkblue}{rgb}{0.0, 0.0, 0.55}
\def \dbf {\color{darkblue}\bf}
\definecolor{anti-flashwhite}{rgb}{0.55, 0.57, 0.68}
\def \afw {\color{anti-flashwhite}}
\def\cF{\mathbb F}
\def\cP{\mathcal P}
\def\embed{\hookrightarrow}
\def\pr{\mathrm{pr}}
\def\cV{\mathcal V}
\def\ot{\otimes}
\def\buu{{\mathbf u}}


\def \ri {{\rm i}}
\newcommand{\bs}[1]{\boldsymbol{#1}}
\newcommand{\cev}[1]{\reflectbox{\ensuremath{\vec{\reflectbox{\ensuremath{#1}}}}}}
\def\bS{\bar \fS}
\def\cE{\mathcal E}
\def\fB{\mathfrak B}
\def\cR{\mathcal R}
\def\cY{\mathcal Y}
\def\cS{\mathscr S}
\def\rS{\overline{\cS}_\omega}

\def\fS{\mathfrak{S}}

\def\MN {(M)}
\def\cN {\mathcal{N}}
\def\SL{{\rm SL}_n}

\def\bP{\mathbb P}
\def\bR{\mathbb R}

\def\SS{\cS_{\omega}(\fS)}
\def\rdS{\overline \cS_{\omega}(\fS)}
\def\rdP{\overline \cS_{\omega}(\mathbb{P}_4)}

\def\Vm{\mathcal V_{\text{mut}}}
\def\Vc{\mathcal V}

\newcommand{\beq}{\begin{equation}}
	\newcommand{\eeq}{\end{equation}}

\section{Introduction}\label{sec-intro}
In this paper, we work with the ground ring $R$, which is a commutative domain with an invertible element
$\omega^{\frac{1}{2}}$. Set $\xi = \omega^{n}$ and $q=\omega^{n^2}$ with
$\xi^{\frac{1}{2n}} = \omega^{\frac{1}{2}}$
and $q^{\frac{1}{2n^2}} = \omega^{\frac{1}{2}}$.  
Define the following constants:
\begin{align}\label{intro-constants}
\mathbbm{c}_{i}= (-q)^{i-n} q^{\frac{1-n}{2n}},\quad
\mathbbm{t}= (-1)^{n-1} q^{\frac{1-n^2}{n}},\quad 
\mathbbm{a} =   q^{-\frac{n+1-2n^2}{4}}.
\end{align}
Unless otherwise specified, all algebras in this paper are assumed to be $R$-algebras.

\subsection{(Quantum) cluster algebras}
Cluster algebras, introduced by Fomin and Zelevinsky \cite{FZ}, are a class of commutative algebras equipped with a distinguished set of generators known as cluster variables. These variables are grouped into collections called clusters, which are related by certain birational transformations known as mutations (\S\ref{sec-mutation-classical}). In \cite{BFZ}, Berenstein, Fomin, and Zelevinsky introduced the notion of upper cluster algebras, defined as the intersection of the Laurent polynomial rings associated with each cluster. From a geometric perspective, upper cluster algebras are often more natural objects to consider. Cluster variables are rational functions by construction. In \cite{FZ}, Fomin and Zelevinsky proved that they are Laurent polynomials of initial cluster variables, known as the Laurent phenomenon. Laurent polynomials were proven to have non-negative coefficients, known as positivity, see \cite{LS,GHKK,D}.

The quantum analogues of cluster algebras and upper cluster algebras (Definition \ref{def-quan-cluster-algebra}) were introduced by Berenstein and Zelevinsky in \cite{BZ}, where they also showed that the Laurent phenomenon admits a quantum analogue in this setting.

For a (quantum) cluster algebra $\mathscr A$ and its upper cluster algebra $\mathscr U$, the Laurent phenomenon ensures that $\mathscr A\subseteq \mathscr U$. However, in general, $\mathscr A\neq \mathscr U$. The problem of whether a cluster algebra coincides with its upper cluster algebra was posed by Berenstein, Fomin, and Zelevinsky and has been studied in \cite{BFZ}. This question has attracted considerable attention since its inception; see, for example, \cite{M,M1,CLS,GY,SW,L,MW,IOS,CGGLS} and references therein. If $\mathscr A=\mathscr U$, then the cluster algebra $\mathscr A$ enjoys several desirable properties, including the existence of a generic basis and a theta basis \cite{CKQ,Q,GLS,GHKK,GLS1}. 

The search for (quantum) cluster algebra and upper (quantum) cluster algebra structures on important algebraic and geometric objects has drawn significant interest; see, for example, \cite{BFZ,FZ,PT,GLS2,S,SSW,SW,CK,G,L1} and the references therein. Cluster algebras and cluster varieties have also been constructed in the context of Teichmüller theory \cite{FG06,FST,GSV,GS19}, and are closely related to skein theory \cite{muller2016skein,ishibashi2023skein,LY22,LY23,KimWang,Kim21,IY1}.

\subsection{Reduced stated ${\rm SL}_n$-skein algebras and quantum trace maps}\label{intro-sec-reduced}

A  \emph{pb surface} $\fS$ is obtained from a compact oriented surface $\overline{\fS}$ by removing finitely many points, which are called \emph{punctures}, such that every boundary component of $\fS$ is diffeomorphic to an open interval. 
An embedding $c:(0,1)\rightarrow \fS$ is called an \emph{ideal arc} if both
$\bar c(0)$ and $\bar c(1)$ are punctures, where $\bar c\colon [0,1] \to \overline{\fS}$ is the `closure' of $c$.
For any positive integer $k$, we use $\mathbb P_k$ to denote the pb surface obtained from the closed disk by removing $k$ punctures from the boundary. 

The \emph{stated ${\rm SL}_n$–skein algebra} $\cS_\omega(\fS)$ of a pb surface $\fS$  
is the quotient of the $R$–module freely generated by the set of isotopy classes of stated $n$–webs  
(Definition~\ref{def-n-web}) in $\fS \times (-1,1)$, subject to the relations \eqref{w.cross}–\eqref{wzh.eight}.  
For two stated $n$-webs $\alpha,\beta$, their product $\alpha\beta$ is defined by stacking $\alpha$ over $\beta$.
The stated ${\rm SL}_n$–skein algebra was introduced by L{\^e} and Sikora \cite{LS21} as a generalization of the ${\rm SL}_n$–skein algebra  
\cite{Sik05} to the stated setting, extending the stated ${\rm SL}_2$– and ${\rm SL}_3$–skein algebras  \cite{le2018triangular,higgins2020triangular} to the ${\rm SL}_n$ case.  
The \emph{reduced stated ${\rm SL}_n$–skein algebra} $\overline{\cS}_\omega(\fS)$ \cite{LY23},  
which is the main focus of this paper, is defined as the quotient of $\cS_\omega(\fS)$  
by the two–sided ideal generated by all bad arcs (see Figure~\ref{Fig;badarc}).

\vspace{0.2cm}

Let $\fS$ be a triangulable pb surface (see \S\ref{sec-traceX}),  
and let $\lambda$ be a triangulation of $\fS$—that is, a maximal collection of pairwise disjoint, non-isotopic ideal arcs in $\fS$  
(note that self-folded triangles are not allowed; see \S\ref{sec-traceX}).  
Cutting $\fS$ along all ideal arcs that are not isotopic to any component of $\partial\fS$ yields a collection of triangles (copies of $\mathbb P_3$),  
denoted by $\mathbb F_\lambda$.  
For each $\tau=\mathbb P_3\in \mathbb F_\lambda$, there is a weighted quiver $\Gamma_\tau$ inside $\tau$ (see Figure~\ref{Fig;coord_ijk})  
whose arrows have weight $1$, except those lying on the boundary, which have weight $\tfrac{1}{2}$.  
The vertices of $\Gamma_\tau$ are called \emph{small vertices}.  
By gluing the triangles in $\bigsqcup_{\tau\in \mathbb F_\lambda}\tau$ back together to recover $\fS$,  
we obtain a quiver $\Gamma_\lambda$ on $\fS$: whenever two edges are identified, the small vertices on these edges are identified as well,  
and any pair of arrows of equal weight but opposite direction between two vertices cancel each other.
We denote by $V_\lambda$ the vertex set of $\Gamma_\lambda$; each element of $V_\lambda$ is called a \emph{small vertex}.

Let $Q_\lambda\colon V_\lambda\times V_\lambda\rightarrow \tfrac{1}{2}\mathbb Z$ be the signed adjacency matrix of $\Gamma_\lambda$ (see \eqref{eq-def-Q-lambda-re}).  
The \emph{$n$-th root Fock–Goncharov algebra} is defined by
\[
\mathcal{Z}_{\omega}(\fS,\lambda)
=  R \langle
Z_v^{\pm 1},\, v \in V_\lambda \rangle \big/ \bigl(
Z_v Z_{v'}= \omega^{\, 2 Q_\lambda(v,v')} Z_{v'} Z_v \;\; \text{for } v,v'\in V_\lambda \bigr).
\]
The $\mathcal X$-version quantum trace, established in \cite{LY23}, is the algebra homomorphism
\begin{align*}
    {\rm tr}_\lambda\colon \overline{\cS}_\omega(\fS)
    \longrightarrow \mathcal{Z}_{\hat\omega}(\fS,\lambda)
    \qquad \text{(see \cite{BW11,LY22,Kim20} for $n=2,3$).}
\end{align*}
L{\^e} and Yu introduced a subalgebra \cite{LY23}, called the \emph{balanced Fock–Goncharov algebra},
\begin{align}
\mathcal{Z}_{\omega}^{\rm bl}(\fS,\lambda)
= \operatorname{span}_R\{Z^{\bf k}\mid {\bf k}\in\mathcal B_\lambda\}
\subset \mathcal{Z}_\omega(\fS,\lambda),
\end{align}
where $\mathcal B_\lambda$ is the subgroup of $\mathbb Z^{V_\lambda}$ defined in \eqref{B_lambda}.
It was proved in \cite{LY23} that $\operatorname{im} {\rm tr}_\lambda \subset \mathcal{Z}_{\omega}^{\rm bl}(\fS,\lambda)$.

\vspace{0.2cm}

It is well known that $\mathcal{Z}_{\omega}^{\rm bl}(\fS,\lambda)$ is an Ore domain \cite{Cohn}.  
We denote its skew-field of fractions by ${\rm Frac}\bigl(\mathcal{Z}_{\omega}^{\rm bl}(\fS,\lambda)\bigr)$.
Let $\lambda'$ be another triangulation of $\fS$.  
Kim and the second author established an isomorphism \cite{KimWang}
\begin{align*}
    \Theta_{\lambda\lambda'}^\omega
    \colon {\rm Frac}\bigl(\mathcal{Z}_{\omega}^{\rm bl}(\fS,\lambda')\bigr)
    \longrightarrow
    {\rm Frac}\bigl(\mathcal{Z}_{\omega}^{\rm bl}(\fS,\lambda)\bigr)
\end{align*}
using a sequence of generalized quantum $\mathcal X$-mutations (see \eqref{eq-Theta2}, \eqref{eq-Theta-change2}, and Proposition~\ref{Prop-rest-bal}).  
See \cite{LY23} for a different construction of $\Theta_{\lambda\lambda'}$, and \cite{Kim21} for the case ${\rm SL}_3$.
It was shown in \cite{KimWang} that the following diagram commutes (Theorem~\ref{thm-main-compatibility}):
\begin{align}
        \label{intro-eq-compability-tr-mutation}
        \xymatrix{
        & \overline{\cS}_\omega(\fS) \ar[dl]_{{\rm tr}_{\lambda'}} \ar[dr]^{{\rm tr}_\lambda} & \\
        {\rm Frac}\bigl(\mathcal{Z}^{\rm bl}_\omega(\fS,\lambda')\bigr)
        \ar[rr]_-{\Theta^\omega_{\lambda\lambda'}} & &
        {\rm Frac}\bigl(\mathcal{Z}^{\rm bl}_\omega(\fS,\lambda)\bigr)
        }.
\end{align}

Assume that $\fS$ has no interior punctures.  
There is another antisymmetric matrix \eqref{eq-anti-matric-P-def}
\begin{align*}
    P_\lambda\colon V_\lambda\times V_\lambda\rightarrow n\mathbb Z
\end{align*}
associated with the triangulation $\lambda$ of $\fS$, which equals $-\overline{\mathsf P}_\lambda$ defined in \cite[Equations~(163) and (205)]{LY23}.
Put $\Pi_\lambda=\frac{1}{n}P_\lambda$.
The \emph{$\mathcal A$-version quantum torus} of $(\fS,\lambda)$ is defined by
\begin{equation*}
\mathcal{A}_{\omega}(\fS,\lambda)
= R \langle 
A_v^{\pm 1},\, v \in V_\lambda \rangle \big/ \bigl(
A_v A_{v'}= \xi^{\Pi_\lambda(v,v')} A_{v'} A_v
\text{ for } v,v'\in V_\lambda \bigr),
\end{equation*}
where $\xi=\omega^n$.
There exists an algebra isomorphism \cite{LY23} (see \eqref{def-A-bal-isomor})
\begin{align}
 \psi_\lambda \colon \mathcal{A}_{\omega}(\fS,\lambda)\longrightarrow
\mathcal{Z}_{\omega}^{\rm bl}(\fS,\lambda).
\end{align}
Moreover, there is an algebra homomorphism \cite{LY23}
\[
{\rm tr}_\lambda^A\colon
\overline{\cS}_\omega(\fS)\longrightarrow \mathcal{A}_{\omega}(\fS,\lambda)
\]
with the following properties:
\begin{enumerate}[label={\rm (\alph*)}]\itemsep0.3em
    \item\label{intro-thm-trace-A-a}
    $\mathcal{A}_{\omega}^{+}(\fS,\lambda)\subset\operatorname{im} {\rm tr}_\lambda^A\subset \mathcal{A}_{\omega}(\fS,\lambda)$,  
    where $\mathcal{A}_{\omega}^{+}(\fS,\lambda)$ is the $R$-subalgebra of $\mathcal{A}_{\omega}(\fS,\lambda)$ generated by $A^{\bf k}$ for ${\bf k}\in\mathbb N^{V_\lambda}$.  
    Here $A^{\bf k}$ is the Laurent monomial defined using the Weyl-ordered product (see \eqref{eq-A-power}).

    \item \label{intro-thm-trace-A-b}
    If $n=2,3$, or $n>3$ and $\fS$ is a polygon (i.e., $\fS=\mathbb P_k$ for some $k>2$), then ${\rm tr}_\lambda^A$ is injective.

    \item\label{intro-thm-trace-A-c}
    The following diagram commutes:
    \begin{align}
        \label{intro-eq-compability-tr-A-X-diag}
        \xymatrix{
        &  \overline{\cS}_\omega(\fS) \ar[dl]_{{\rm tr}_{\lambda}^A} \ar[dr]^{{\rm tr}_\lambda} & \\
        \mathcal{A}_{\omega}(\fS,\lambda) \ar[rr]_-{\psi_\lambda} & &
        \mathcal{Z}_\omega^{\rm bl}(\fS,\lambda)
        }
    \end{align}
\end{enumerate}


\subsection{The ${\rm SL}_n$ quantum cluster structure}
Let $\fS$ be a triangulable pb surface without interior punctures, and let $\lambda$ be a triangulation of $\fS$.  
Properties \ref{intro-thm-trace-A-a} and \ref{intro-thm-trace-A-b} of ${\rm tr}_\lambda^A$ imply that $\overline{\cS}_\omega(\fS)$ is an Ore domain when $n=2,3$.  We denote its skew-field of fractions by ${\rm Frac}\bigl(\overline{\cS}_\omega(\fS)\bigr)$.  
It was established in \cite{muller2016skein,ishibashi2023skein,LY22} that there is a natural quantum seed (see Definition~\ref{def-quantum-seed}) inside ${\rm Frac}\bigl(\overline{\cS}_\omega(\fS)\bigr)$ for $n=2,3$.  
In this paper, we generalize this construction to arbitrary ${\rm SL}_n$.

To achieve this, we require the injectivity of ${\rm tr}_\lambda^A$ (or ${\rm tr}_\lambda$), which has not been proved for general $n$ but is expected to hold.  
We therefore define the \emph{projected stated ${\rm SL}_n$–skein algebra} (Definition~\ref{def-key-algebra}) by
\[
\widetilde{\cS}_\omega(\fS):=
\overline{\cS}_\omega(\fS)\big/ \ker {\rm tr}_\lambda^A=
\overline{\cS}_\omega(\fS)\big/ \ker {\rm tr}_\lambda.
\]
With this definition, we can show that 
${\rm tr}_\lambda^A\colon \widetilde{\cS}_\omega(\fS)\to \mathcal{A}_{\omega}(\fS,\lambda)$ is injective and that $\widetilde{\cS}_\omega(\fS)$ is an Ore domain (Lemma~\ref{lem-basic-lem}).  
We then regard $\widetilde{\cS}_\omega(\fS)$ as a subalgebra of $\mathcal{A}_{\omega}(\fS,\lambda)$ via ${\rm tr}_\lambda^A$.  
Property \ref{intro-thm-trace-A-a} of ${\rm tr}_\lambda^A$ further implies
\begin{align}\label{intro-identity-Frac}
    {\rm Frac}\bigl(\widetilde{\cS}_\omega(\fS)\bigr)
= {\rm Frac}\bigl(\mathcal{A}_{\omega}(\fS,\lambda)\bigr).
\end{align}
Property \ref{intro-thm-trace-A-b} of ${\rm tr}_\lambda^A$ shows that 
$\widetilde{\cS}_\omega(\fS)=\overline{\cS}_\omega(\fS)$
when $n=2,3$, or $n>3$ and $\fS$ is a polygon.

\vspace{0.2cm}

To construct the quantum seed inside ${\rm Frac}\bigl(\widetilde{\cS}_\omega(\fS)\bigr)$, we first specify the vertex set $\mathcal V$ and the mutable vertex set $\mathcal V_{\rm mut}$ (see \S\ref{sec-mutation-classical}).  
Set $\mathcal V = V_\lambda$, and let $\mathcal V_{\rm mut} = \mathring{V}_\lambda$ be the subset of vertices of $V_\lambda$ lying in the interior of $\fS$.  
With the identification in \eqref{intro-identity-Frac}, the triple $(Q_\lambda,\Pi_\lambda,M_\lambda)$ forms a quantum seed (Definition~\ref{def-quantum-seed}) inside the skew-field ${\rm Frac}\bigl(\widetilde{\cS}_\omega(\fS)\bigr)$ (Lemma~\ref{lem-seed-skein}), where  
\begin{align*}
    M_\lambda \colon \mathbb Z^{V_\lambda} \longrightarrow {\rm Frac}\bigl(\widetilde{\cS}_\omega(\fS)\bigr), 
    \qquad {\bf t} \longmapsto A^{\bf t}.
\end{align*}

For each $i\in \mathcal V$, define ${\bf e}_i\in \mathbb Z^{\mathcal V}$ by  
\begin{align*}
    {\bf e}_i(v)=\delta_{i,v}\quad \text{for } v\in \mathcal V.
\end{align*}
The map $M_\lambda$ is then uniquely determined by the elements $A_i = M_\lambda({\bf e}_i)$, which we call the \emph{(quantum) cluster variables}.  
A cluster variable $A_i$ is \emph{frozen} if $i\in \mathcal V \setminus \mathcal V_{\rm mut}$, and \emph{exchangeable} otherwise.  
These notions apply to any quantum seed (see \S\ref{sec-mutation-quantum}), although we state them here only for $(Q_\lambda,\Pi_\lambda,M_\lambda)$.

For each $k\in \mathcal V_{\rm mut}$, the quantum $\mathcal A$-mutation $\mu_{k,A}$ produces a new quantum seed
\[
(Q',\Pi',M') = \mu_{k,A}\bigl(Q_\lambda,\Pi_\lambda,M_\lambda\bigr)
\]
(see \S\ref{sec-mutation-quantum}).  
Let $\mathsf S_{\fS,\lambda}$ denote the collection of quantum seeds in ${\rm Frac}\bigl(\widetilde{\cS}_\omega(\fS)\bigr)$ obtained from $(Q_\lambda,\Pi_\lambda,M_\lambda)$ by finitely many quantum $\mathcal A$-mutations.  
Define
\[
\mathscr{A}_{\fS,\lambda} := \mathscr{A}_{\mathsf S_{\fS,\lambda}},
\qquad
\mathscr{U}_{\fS,\lambda} := \mathscr{U}_{\mathsf S_{\fS,\lambda}}
\]
as in Definition~\ref{def-quan-cluster-algebra}.

\vspace{0.2cm}

The following theorem asserts that this quantum cluster structure inside ${\rm Frac}\bigl(\widetilde{\cS}_\omega(\fS)\bigr)$ is independent of the choice of triangulation $\lambda$.

\begin{theorem}[Theorem~\ref{thm-main-1}]\label{intro-naturality}
    Let $\fS$ be a triangulable pb surface without interior punctures, and let $\lambda$, $\lambda'$ be two triangulations of $\fS$. Then we have   $$\mathsf{S}_{\fS,\lambda}=\mathsf{S}_{\fS,\lambda'},\;
    \mathscr{A}_{\fS,\lambda}=
    \mathscr{A}_{\fS,\lambda'},\text{ and }
    \mathscr{U}_{\fS,\lambda}=
    \mathscr{U}_{\fS,\lambda'}\quad \text{(see Definition~\ref{def-tri-quantum}).}$$
\end{theorem}

It is well known that any two triangulations $\lambda$ and $\lambda'$ are related by a sequence of flips.  
Thus we may assume that $\lambda'$ is obtained from $\lambda$ by a single flip.  
To prove Theorem~\ref{intro-naturality}, it suffices to show that the quantum seeds
\[
(Q_\lambda,\Pi_\lambda,M_\lambda)
\quad\text{and}\quad
(Q_{\lambda'},\Pi_{\lambda'},M_{\lambda'})
\]
are related by a sequence of quantum $\mathcal A$–mutations.

Indeed, there exists a sequence of vertices  
$v_1,v_2,\ldots,v_r \in \mathcal V_{\rm mut}$ \cite{FG06,GS19} (see Figure~\ref{Fig;mutation_sequence_for_flip}) such that
\begin{align*}
   Q_{\lambda'}
   = \mu_{v_r}\,\cdots\,\mu_{v_2}\,\mu_{v_1}(Q_\lambda),
\end{align*}
where each $\mu_{v_i}$ denotes the matrix mutation defined in \eqref{eq-mutation-Q}, and $r = \tfrac{1}{6}(n^3-n)$.  
Hence it remains to show that  
\begin{align}\label{intro-mutation-sequence-flip}
   (Q_{\lambda'},\Pi_{\lambda'},M_{\lambda'})
   = \mu_{v_r,A}\,\cdots\,\mu_{v_2,A}\,\mu_{v_1,A}
     \bigl(Q_\lambda,\Pi_\lambda,M_\lambda\bigr).
\end{align}
For small $n$, this identity can be verified by direct computation, as carried out for $n=2,3$ in \cite{muller2016skein,ishibashi2023skein}.  
However, for general ${\rm SL}_n$ the number of required mutations makes such a calculation infeasible.

To overcome this, we decompose each quantum $\mathcal A$–mutation into two steps (Lemma~\ref{lem-decom-A}) and establish the compatibilities  
(Diagrams~\eqref{eq-com-step-one} and \eqref{eq-com-step-one}) between these steps and the corresponding two steps of the generalized $\mathcal X$–mutation  
(see \eqref{eq-quantum-mutation_Z} and \eqref{lem-def-nu-sharp}).  
This yields the compatibility (Proposition~\ref{prop-comp}) between the quantum $\mathcal A$–mutation and the generalized $\mathcal X$–mutation $\nu_{k}^\omega$ (see \eqref{eq-def-quantum-mutation-X}).  
Combining this compatibility with \eqref{intro-eq-compability-tr-mutation} and \eqref{intro-eq-compability-tr-A-X-diag} proves \eqref{intro-mutation-sequence-flip}.

An immediate consequence of these arguments is that the composition of quantum $\mathcal A$–mutations
\[
\mu_{v_1,A}\circ\cdots\circ\mu_{v_r,A}\colon
{\rm Frac}\bigl(\mathcal A_{\omega}(\fS,\lambda')\bigr)\longrightarrow
{\rm Frac}\bigl(\mathcal A_{\omega}(\fS,\lambda)\bigr)
\]
establishes the naturality of the $\mathcal A$–quantum trace maps  
(see Corollary~\ref{thm-naturality-sln-cluster}), coinciding with $\overline{\Psi}_{\lambda\lambda'}^A$ in \cite[Theorem~14.1]{LY23}.


\subsection{The inclusion of the ${\rm SL}_n$-skein algebra into the ${\rm SL}_n$ quantum (upper) cluster algebra}
Under the assumption of Theorem~\ref{intro-naturality}, we write  
\begin{align*}
    \mathscr{A}_{\omega}(\fS)&:=\mathscr{A}_{\fS,\lambda},\qquad
    \text{called the \emph{${\rm SL}_n$ quantum  cluster algebra} of $\fS$,}\\
    \mathscr{U}_{\omega}(\fS)&:=\mathscr{U}_{\fS,\lambda},\qquad \text{called the \emph{${\rm SL}_n$ quantum upper cluster algebra} of $\fS$,}
\end{align*}
since these algebras are independent of the choice of triangulation $\lambda$ by Theorem~\ref{intro-naturality}.

Assume that each component of $\fS$ has at least two punctures.  
A properly embedded oriented arc in $\fS$ is called an \emph{essential arc}  
if its endpoints lie on two distinct components of $\partial \fS$.  
Lemma~\ref{lem-essential-arc} shows that the algebra $\widetilde{\cS}_\omega(\fS)$ is generated by finitely many stated essential arcs.

The following theorem is the main result of this paper.  
It expresses every stated essential arc as a Weyl-ordered product of an exchangeable cluster variable and a Laurent monomial in frozen cluster variables.  
Consequently,
\[
\widetilde{\cS}_\omega(\fS)\subset \mathscr{A}_{\omega}(\fS).
\]

\begin{figure}[h]
    \centering
    \includegraphics[width=350pt]{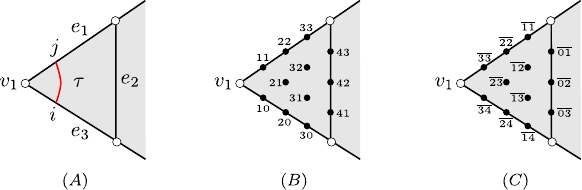}
    \caption{(A) The triangle $\tau$ with edges labeled $e_1,e_2,e_3$, and a distinguished vertex labeled $v_1$ (note that $e_1 \neq e_3$).  
(B) The labeling of the small vertices in $V_\lambda \cap \tau$ for $n=4$.  
(C) An alternative labeling of the small vertices in $V_\lambda \cap \tau$ for $n=4$.
}\label{Fig;tau-v1}
\end{figure}

\begin{theorem}[Theorem \ref{thm-skein-inclusion-A}]\label{intro-thm-skein-inclusion-A}
Let $\fS$ be a pb surface without interior punctures. We require that every component of $\fS$ contains at least two punctures.
    Then we have $$\widetilde{\cS}_\omega(\fS)\subset \mathscr A_{\omega}(\fS).$$
 Moreover, each stated essential arc is the Weyl-ordered product (see \eqref{Weyl-A}) of an exchangeable cluster variable (i.e., a non-frozen cluster variable) and a Laurent monomial in frozen cluster variables (Definition~\ref{def-quan-cluster-algebra}). To be precise, we have the following:
\begin{enumerate}[label={\rm (\alph*)}]\itemsep0,3em
     \item Let $1 \le j \le i \le n$, and let $C_{ij}$ be the stated corner arc represented by the red arc in Figure~\ref{Fig;tau-v1}(A), oriented counterclockwise around $v_1$.  
Let $\lambda$ be a triangulation of $\fS$ that contains the ideal arcs $e_1,e_2,e_3$ shown in Figure~\ref{Fig;tau-v1}(A).  
For any $j,k$ with $1 \le j < k$, define  
\begin{align}\label{intro-eq-mukj}
\mu_{(k;j)}=\mu_{k j}\cdots\mu_{k1},
\end{align}
where the small vertices $i j \in V_\lambda$ are labeled as in Figure~\ref{Fig;tau-v1}(B).
Then we have 
     \begin{align}\label{intro-eq-Cij}
         C_{ij}=\begin{cases}
[A_{i1}\cdot A_{i0}^{-1}\cdot A_{11}^{-1}] & \mbox{ if $j=1$},\\
[A_{i0}^{-1}\cdot A_{i-1,0}] & \mbox{ if $j=i$},\\
[\mu_{(j;j-1)}\cdots \mu_{(i-2;j-1)} \mu_{(i-1;j-1)} (A_{j,j-1})\cdot  A_{i0}^{-1}\cdot A_{jj}^{-1}] & \mbox{ if $1<j<i$},
\end{cases}
     \end{align}
   \item Let $1 \le j \le i \le n$, and let $\overline C_{ij}$ be the stated corner arc represented by the red arc in Figure~\ref{Fig;tau-v1}(A), oriented clockwise around $v_1$.  
Let $\lambda$ be a triangulation of $\fS$ that contains the ideal arcs $e_1,e_2,e_3$ shown in Figure~\ref{Fig;tau-v1}(A).  
For any $k,t$ such that $k+t<n$, we denote 
\begin{align}\label{intro-eq-bar-mukj}
\overline\mu_{(k;t)}=\mu_{\overline{k,k+1}} \cdots \mu_{\overline{k,k+t}},
\end{align}
where the small vertices $\overline{ij} \in V_\lambda$ are labeled as in Figure~\ref{Fig;tau-v1}(C).
Then we have
     \begin{align}\label{intro-eq-bar-Cij}
         \overline C_{ij}=\begin{cases}
[\overline A_{jj}^{-1}\cdot \overline A_{j-1,j}] & \mbox{ if $i=n$},\\
 [\overline A_{in}^{-1}\cdot \overline A_{i-1,n}] & \mbox{ if $i=j$},\\
[\overline \mu_{(j;n-i)}\cdots \overline \mu_{(i-2;n-i)} \overline \mu_{(i-1;n-i)} (\overline A_{j,j+1})\cdot \overline A_{jj}^{-1}\cdot \overline A_{in}^{-1}] & \mbox{ if $1\leq j<i<n$},
\end{cases}
     \end{align}
where $\overline A_{ij} = A_{\overline{ij}}\in\mathcal A_\omega(\fS,\lambda)$.

\item For any $1\leq i,j\leq n$, let $D_{ij}$ be the stated essential arc represented by the red arc in Figure~\ref{Fig;essential-P4}(A).
Let $\lambda$ be a triangulation of $\fS$ that contains the ideal arcs $c_1,c_2,c_3,c_4,c_5$ shown in Figure~\ref{Fig;essential-P4}(A) (we allow $c_1=c_3$).
We label the small vertices in $V_\lambda$
contained in the quadrilateral bounded by $c_1\cup c_2\cup c_3\cup c_4$ as Figure~\ref{Fig;essential-P4}(B).

For any $j>1$, denote 
$$\overline \mu^{\diamondsuit}_j=\bar \mu_{(1;n-j)}\cdots \bar \mu_{(j-2;n-j)}\bar \mu_{(j-1;n-j)},$$
where $\overline{\mu}_{(k;t)}$ is defined as in \eqref{intro-eq-bar-mukj}.
For any $i,j$ with $i\geq j>1$, denote $$\mu^{\diamondsuit}_{(i;j-1)}=\left(\mu_{(2;1)}\mu_{(3;2)}\cdots\mu_{(j-1;j-2)}\right)\circ \left(\mu_{(j;j-1)}\mu_{(j+1;j-1)}\cdots \mu_{(i-1;j-1)}\right),$$
where $\mu_{(k;j)}$ is defined as in
\eqref{intro-eq-mukj}.
Then we have
   \begin{align}\label{into-eq-Dij}
       D_{ij}=
    \begin{cases}
        [A_{i1}\cdot A_{i0}^{-1}\cdot \overline A_{1n}^{-1}]
        & \mbox{ if $i\geq j=1$,}\\
        [\overline \mu^{\diamondsuit}_{j}(\overline A_{12})\cdot A_{10}^{-1}\cdot \overline A_{jn}^{-1} ] & \mbox{ if $j>i=1$,}\vspace{1.5mm}\\
        [\left(\mu_{j-1,j-1}\cdots\mu_{22}\mu_{11}\right)\circ \overline \mu^{\diamondsuit}_j\circ  \mu^{\diamondsuit}_{(i;j-1)}(A_{j-1,j-1}) \cdot A_{i0}^{-1}\cdot\overline A_{jn}^{-1}] & \mbox{ if $i\geq j>1$,}\vspace{1.5mm}\\
        [\left(\mu_{i-1,i-1}\cdots\mu_{22}\mu_{11}\right)\circ \overline \mu^{\diamondsuit}_j\circ \mu^{\diamondsuit}_{(i;i-1)}(A_{i-1,i-1}) \cdot A_{i0}^{-1}\cdot\overline A_{jn}^{-1}] & \mbox{ if $j>i>1$}.
    \end{cases}
   \end{align}
 \end{enumerate}

\end{theorem}

\begin{figure}[h]
    \centering
    \includegraphics[width=220pt]{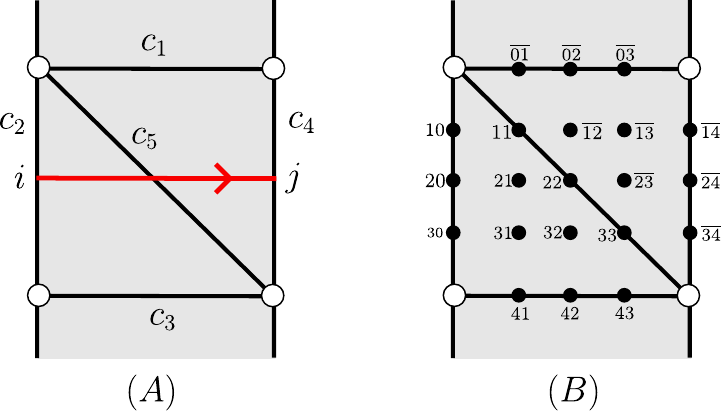}
    \caption{(A) The picture for an essential arc, in red. (B) The labeling for small vertices contained in the quadrilateral bounded by $c_1\cup c_2\cup c_3\cup c_4$ for $n=4$.}\label{Fig;essential-P4}
\end{figure}

Then we outline the proof of
Theorem~\ref{intro-thm-skein-inclusion-A}.

\vspace{0.2cm}

First, we establish the following Theorem, which states a weaker inclusion $\widetilde{\cS}_\omega(\fS)\subset \mathscr U_{\omega}(\fS).$

\begin{theorem}[Theorem~\ref{thm-main-3}]\label{intro-upper-inclusion}
    Let $\fS$ be a triangulable pb surface without interior punctures. Then
    $$
    \dS \subset \mathscr{U}_{\omega}(\fS).
    $$
\end{theorem}

When $n=2$ or $3$, the algebra $\dS$ admits a generating set of so-called \emph{elementary webs} (this generating set is not known for $n>3$).  
Each elementary web is readily verified—by straightforward calculations—to be a cluster variable.  
Using this generating set, one can establish Theorems~\ref{intro-thm-skein-inclusion-A} and \ref{intro-upper-inclusion} following the arguments of \cite{muller2016skein,ishibashi2023skein}.
For general ${\rm SL}_n$ with $n>4$, this approach is no longer available.  
Instead, we develop new, more constructive methods to prove Theorems~\ref{intro-thm-skein-inclusion-A} and \ref{intro-upper-inclusion},  
which not only yield the desired results but also reveal several interesting phenomena.

Let $e$ be an ideal arc of $\fS$ lying entirely in the interior of $\fS$, and let $\fS'$ be the pb surface obtained by cutting $\fS$ along $e$.  
Following \cite{LS21} (see \S\ref{sub-splitting}), there is an algebra homomorphism, called the \emph{splitting homomorphism},
\[
\mathbb{S}_e : \overline{\mathscr{S}}_\omega(\fS) \longrightarrow \overline{\mathscr{S}}_\omega(\fS'),
\]
which, when $\fS'$ is triangulable, further induces an injective algebra homomorphism (Lemma~\ref{lem-basic-lem}(c)):
\begin{align}\label{intro-splitting-eq}
    \mathbb{S}_e : \widetilde{\mathscr{S}}_\omega(\fS) \longrightarrow \widetilde{\mathscr{S}}_\omega(\fS').
\end{align}

Let $\lambda$ be a triangulation of $\fS$.  
To show $\dS \subset \mathscr{U}_{\omega}(\fS)$,
by Theorem~\ref{intro-upper-inclusion} (\cite[Theorem~5.1]{BZ}), it is enough to prove
\[
    \dS \subset \mathbb{T}_k \qquad \text{for every } k\in \mathcal V_{\rm mut},
\]
where $\mathbb{T}_k$ denotes the quantum torus associated with the quantum seed
\(
\mu_{k,A}\bigl(Q_\lambda,\Pi_\lambda,M_\lambda\bigr)
\)
(see \eqref{eq-T-w}).  
For $n=2$ this inclusion is immediate because
\(
\mathbb{T}_k = \mathcal{A}_\omega(\fS,\lambda'),
\)
where $\lambda'$ is obtained from $\lambda$ by flipping the ideal arc corresponding to $k$.
Combining with established results in \cite{schrader2017continuous,LY23,KimWang},  
we prove the inclusion
\(
\dS \subset \mathbb{T}_k
\)
by cutting out a subsurface of type $\mathbb{P}_3$ or $\mathbb{P}_4$ whose interior contains the vertex $k$.

\vspace{0.2cm}

Then, we construct the splitting homomorphism for $\mathscr{U}_\omega(\fS)$, which is compatible with the one in \eqref{intro-splitting-eq}.

\begin{theorem}[Theorem~\ref{thm:splitU}]\label{intro-thm:splitU}
Let a triangulable pb surface $\fS$ and edge $e$ be an ideal arc in $\lambda$, and let $\fS'$ be the pb surface obtained by cutting $\fS$ along $e'$.
Assume that $\fS$ contains no interior punctures.
Then there exists an injective algebra homomorphism $$\mathbb S_e^U:\mathscr{U}_{\omega}(\fS)\to \mathscr{U}_{\omega}(\fS'),$$ which is compatible with $\mathbb S_e$ in the sense that the following diagram commutes:
    $$\centerline{\xymatrix{
  &\widetilde{\mathscr S}_\omega(\fS) \ar@{^{(}->}[rrr]^{\mathbb S_e}\ar@{^{(}->}[d]  &&& \widetilde{\mathscr S}_\omega(\fS') \ar@{^{(}->}[d]          \\
  &  \mathscr U_{\omega}(\fS) \ar@{^{(}->}[rrr]^{\mathbb S_e^U}  &&& \mathscr U_{\omega}(\fS').}}$$ 
\end{theorem}

We first establish the splitting homomorphism for quantum cluster upper algebras in a general setting (Proposition~\ref{prop:split}).  
We then verify that the ${\rm SL}_n$ quantum upper cluster algebra $\mathscr{U}_{\omega}(\fS)$ satisfies the hypotheses of this general result. Together with Proposition~\ref{prop:image}.  
this proves Theorem~\ref{intro-thm:splitU}.

\vspace{0.2cm}

We are close to prove Theorem~\ref{intro-thm-skein-inclusion-A}. Before that we consider two special cases when $\fS=\mathbb P_3,\mathbb P_4$.

\vspace{0.2cm}

There is a dual quiver $\mathcal N(\mathbb P_3,v_1)$ (see Figure~\ref{dualquiver}) associated with the quiver $\Gamma_{\mathbb P_3}$.  
As illustrated in Figure~\ref{dualquiver}, the endpoints of $\mathcal N(\mathbb P_3,v_1)$ along $e_1$ and $e_3$ are labeled by $1,2,\ldots,n$.  
For any $i,j\in\{1,2,\ldots,n\}$, define the path set $\mathsf P(\mathbb P_3,v_1,i,j)$ to consist of all paths in $\mathcal N(\mathbb P_3,v_1)$ starting at the $i$ lying in $e_3$ and ending at the $j$ lying in $e_1$ (see \eqref{def-path-P3-v1}).  
Assume that $\fS=\mathbb P_3$.  
Then \cite[Theorem~10.5]{LY23} (Theorem~\ref{lem-trace-image-cornerarc-P3}) together with \eqref{intro-eq-compability-tr-A-X-diag} implies
\begin{align}\label{intro-path-sum}
    C_{ij}=\sum_{p\in\mathsf P(\mathbb P_3,v_1,i,j)} A_p \in \mathcal A_{\omega}(\mathbb P_3),
\end{align}
where $A_p$ is the Laurent monomial defined in \eqref{def-path-A-monomial}.  

For $j=1$ or $j=i$, the set $\mathsf P(\mathbb P_3,v_1,i,j)$ contains a single path, and \eqref{intro-eq-Cij} follows directly from \eqref{intro-path-sum}.  
When $1<j<i$, there exist a ``maximum'' path $p_{\max}^{(i,j)}$, containing the largest number of small vertices on its left, and a ``minimum'' path $p_{\min}^{(i,j)}$, containing the fewest such vertices, such that every $p\in \mathsf P(\mathbb P_3,v_1,i,j)$ is between $p_{\max}^{(i,j)}$ and $p_{\min}^{(i,j)}$, denoted as
\[
   p_{\min}^{(i,j)} \le p \le p_{\max}^{(i,j)}.
\]
The paths $p_{\max}^{(i,j)}$ and $p_{\min}^{(i,j)}$ bound a quadrilateral whose two “tails’’ connect to $e_1$ and $e_3$, respectively.  
The small vertices contained in this quadrilateral are
\begin{align}
   v_{j,j-1},\ldots,v_{j,1},\ldots,v_{i-1,1},\ldots,v_{i-1,j-1}.
\end{align}
A promising observation is that
\begin{align}\label{intro-mutation-Ap}
   \sum_{p_{\min}^{(i,j)} \le p \le p_{\max}^{(i,j)}} A_p
   = \bigl[\mu_{j,j-1}\cdots\mu_{j,1}\cdots\mu_{i-1,1}\cdots\mu_{i-1,j-1}(A_{j,j-1}) \cdot A_{i0}^{-1} A_{jj}^{-1}\bigr],
\end{align}
which is precisely the formula in \eqref{intro-eq-Cij}.  

To prove \eqref{intro-mutation-Ap}, note that the set $\mathsf P(\mathbb P_3,v_1,i,j)$ is in one-to-one correspondence with the disjoint union of the sets of paths from $i-1$ to $j-1$ and from $i-1$ to $j$ for ${\rm SL}_{n-1}$ (Lemma~\ref{lem:divide2}).  The result then follows by induction combined with a detailed analysis of the mutations
\(\mu_{(j;j-1)}\cdots \mu_{(i-2;j-1)} \mu_{(i-1;j-1)}\) (Lemmas~\ref{lem:quiver0}--\ref{lem:mut2}).

The equality \eqref{intro-eq-bar-Cij} (Proposition~\ref{prop-bar-Cij}) for $\fS=\mathbb P_3$ can be established in the same way.

Using \eqref{intro-eq-Cij} and \eqref{intro-eq-bar-Cij}, we establish Equation~\eqref{into-eq-Dij} for $\fS=\mathbb P_4$ (Theorem~\ref{thm-P4-Dij}) by cutting $\mathbb P_4$ along $c_5$ into two triangles. More precisely, in light of the injectivity and compatibility of the splitting homomorphisms $\mathbb S_{c_5}$ and $\mathbb S^U_{c_5}$ for the projected ${\rm SL}_n$ skein algebra and the corresponding quantum upper cluster algebra, respectively (Theorem \ref{intro-thm:splitU}), it suffices to verify \eqref{into-eq-Dij} after applying $\mathbb S_{c_5}$ and $\mathbb S^U_{c_5}$ to the left and right of \eqref{into-eq-Dij}, respectively. The cases $i\geq j=1$ and $j>i=1$ follow by direct computation. For the case $i\geq j>1$, a key observation is that the subquiver of $\overline \mu^{\diamondsuit}_j\circ  \mu^{\diamondsuit}_{(i;j-1)}(Q_{\mathbb P_4})$ formed by the vertices $v_{j-1,j-1},\cdots, v_{22},v_{11}$ is of type $A$ of linear orientation (Lemma \ref{lem:quiver2}). Relying on this, we apply the expansion formula for type $A$ quantum cluster algebras \cite{R,CL,H1} to compute the exchangeable cluster variable $\left(\mu_{j-1,j-1}\cdots\mu_{22}\mu_{11}\right)\circ \overline \mu^{\diamondsuit}_j\circ  \mu^{\diamondsuit}_{(i;j-1)}(A_{j-1,j-1})$ (Corollary \ref{cor:cv}). A direct check then confirms that \eqref{into-eq-Dij} holds after applying $\mathbb S_{c_5}$ and $\mathbb S^U_{c_5}$ to both sides. The case $j>i>1$ can be proved in a same manner.

Finally, combining the above results for $\mathbb{P}_3$ and $\mathbb{P}_4$ with detailed calculations for  
$\mathbb{S}^{U}_{e_2}$ and $\mathbb{S}^{U}_{c_1}\circ \mathbb{S}^{U}_{c_3}$  
(Lemmas~\ref{lem:split_e2}--\ref{lem-splitting-essential-arc})  
and invoking Theorem~\ref{intro-thm:splitU},  
we prove Theorem~\ref{intro-thm-skein-inclusion-A} by either cutting out a $\mathbb{P}_3$ along $e_2$  
or cutting out a $\mathbb{P}_4$ along $c_1$ and $c_3$.

\vspace{0.2cm}

We conclude the introduction with the following two equivalent conjectures.

\begin{conjecture}\label{con-equality-Skein-A-U}
      Under the same assumption as Theorem~\ref{intro-thm-skein-inclusion-A},
     we have 
     $$\widetilde{\cS}_\omega(\fS)= \mathscr{A}_{\omega}(\fS)= \mathscr{U}_{\omega}(\fS).$$
\end{conjecture}

\begin{conjecture}\label{con-intro}
      Under the same assumption as Theorem~\ref{intro-thm-skein-inclusion-A}, the 
$R\langle A^{\pm1}_v\mid v\in \mathcal V\setminus \mathcal V_{\mathrm{mut}}\rangle$-algebra
 $\mathscr{U}_{\omega}(\fS)$
 is generated by the exchangeable cluster variables appearing in 
 Equations~\eqref{intro-eq-Cij}, \eqref{intro-eq-bar-Cij}, and \eqref{into-eq-Dij}.
\end{conjecture}

One can check that Conjecture \ref{con-intro} holds for $\fS=\mathbb P_3$ and $n=3$, $4$ or $5$.

It was shown in \cite{LY23} that 
$R\langle A^{\pm1}_v\mid v\in \mathcal V\setminus \mathcal V_{\mathrm{mut}}\rangle\subset \dS$ (the case when $j=i$ in \eqref{intro-eq-Cij} also implies this inclusion).
Then Theorem~\ref{intro-thm-skein-inclusion-A} leads to the following result,  
which provides a promising potential approach to proving Theorem~\ref{con-equality-Skein-A-U}.
We will investigate this in a future project. 

\begin{corollary}\label{intro-cor}
Conjectures~\ref{con-equality-Skein-A-U} and \ref{con-intro} are equivalent.
\end{corollary}

Conjecture~\ref{con-equality-Skein-A-U} was confirmed for $n=2$ in \cite{muller2016skein,LS21}. Then, it follows from
Corollary~\ref{intro-cor} that Conjecture~\ref{con-intro} holds when $n=2$.

{\bf Acknowledgements:}
This work was partially supported by the National Natural Science Foundation
of China (No.12471023) (M.H.). Part of this research was carried out while Z.W. was a Dual PhD student at the University of Groningen (The Netherlands) and Nanyang Technological University (Singapore), supported by a PhD scholarship from the University of Groningen and a research scholarship from Nanyang Technological University.  
Z.W. was supported by a KIAS Individual Grant (MG104701) at the Korea Institute for Advanced Study.

\section{Quantum trace maps for reduced stated $\SL$-skein algebras}

The quantum trace maps for reduced stated $\SL$-skein algebras were constructed in \cite{LY23}. 
There are two versions of these maps, called the $\mathcal X$-version and the $\mathcal A$-version, which are compatible to each other. 
Both are algebra homomorphisms from the reduced stated $\SL$-skein algebra to a quantum torus. 
The $X$-quantum torus coincides with the $n$-root Fock--Goncharov algebra, which is related to the $X$-cluster variety \cite{KimWang}. 
In \S\ref{sec-seed-structure}, we will use 
the $A$-quantum torus to construct quantum seeds (Definition~\ref{def-quantum-seed}).

\subsection{Reduced stated $\SL$-skein algebras}\label{sub-def-reduced-skein}

\def\Si{\fS}

Let $\fS$ be a pb surface (see \S\ref{intro-sec-reduced}).
Consider the 3-manifold $\Si \times (-1,1)$, the thickened surface. For a point $(x,t) \in \Si \times (-1,1)$, the value $t$ is called the {\bf height} of this point. 
\begin{definition}[\cite{LS21}]\label{def-n-web}
    An {\bf $n$-web} $\alpha$ in $\Si\times(-1,1)$ is a disjoint union of oriented closed curves and a directed finite graph properly embedded into $\Si\times(-1,1)$, satisfying the following requirements:
\begin{enumerate}
    \item $\alpha$ only contains $1$-valent or $n$-valent vertices. Each $n$-valent vertex is a source or a  sink. The set of $1$-valent vertices is denoted as $\partial \alpha$, which are called \textbf{endpoints} of $\alpha$. For any boundary component $c$ of $\Si$, we require that the points of  $\partial\alpha\cap (c\times(-1,1))$ have mutually distinct heights.
    \item Every edge of the graph is an embedded oriented  closed interval  in $\Si\times(-1,1)$.
    \item $\alpha$ is equipped with a transversal \textbf{framing}. 
    \item The set of half-edges at each $n$-valent vertex is equipped with a  cyclic order. 
    \item $\partial \alpha$ is contained in $\partial\Si\times (-1,1)$ and the framing at these endpoints is given by the positive direction of $(-1,1)$.
\end{enumerate}
We will consider $n$-webs up to (ambient) \textbf{isotopy} which are continuous deformations of $n$-webs in their class. 
The empty $n$-web, denoted by $\emptyset$, is also considered as an $n$-web, with the convention that $\emptyset$ is only isotopic to itself. 

A {\bf state} for $\alpha$ is a map $s\colon\partial\alpha\rightarrow \{1,2,\cdots,n\}$. A {\bf stated $n$-web} in $\Si\times(-1,1)$ is an $n$-web equipped with a state.
\end{definition}

By regarding $\fS$ as $\fS\times\{0\}$, there is a projection $\text{pr}\colon\fS\times(-1,1)\rightarrow \fS.$
We say the (stated) $n$-web $\alpha$ is in {\bf vertical position} if 
\begin{enumerate}
    \item the framing at everywhere is given by the positive direction of $(-1,1)$,
    \item $\alpha$ is in general position with respect to the projection  $\text{pr}\colon \Si\times(-1,1)\rightarrow \Si\times\{0\}$,
    \item at every $n$-valent vertex, the cyclic order of half-edges as the image of $\text{pr}$ is given by the positive orientation of $\Si$ (drawn counter-clockwise in pictures).
\end{enumerate}

For every (stated) $n$-web $\alpha$, we can isotope $\alpha$ to be in vertical position. For each boundary component $c$ of $\fS$, the heights of $\partial\alpha\cap (c\times(-1,1))$ determine a linear order on  $c\cap \text{pr}(\alpha)$.
Then a {\bf (stated) $n$-web diagram} of $\alpha$ is $\text{pr}(\alpha)$ equipped with the usual over/underpassing information at each double point (called a crossing) and a linear order on $c\cap \text{pr}(\alpha)$ for each boundary component $c$ of $\fS$.

Let $S_n$ denote the permutation group on the set $\{1,2,\cdots,n\}$. 
For an integer $i\in\{1,2,\cdots,n\}$, we use $\bar{i}$ to denote $n+1-i$.

\def\M {M,\cN}

The \textbf{stated $\SL$-skein algebra} $\cS_{\omega}(\fS)$ of $\fS$ is
the quotient module of the $R$-module (see \S\ref{sec-intro}) freely generated by the set 
 of all isotopy classes of stated 
$n$-webs in $\fS\times (-1,1)$ subject to  relations \eqref{w.cross}-\eqref{wzh.eight} (see \eqref{intro-constants} for involved constants in $R$).

\beq\label{w.cross}
q^{-\frac{1}{n}} 
\raisebox{-.20in}{

\begin{tikzpicture}
\tikzset{->-/.style=

{decoration={markings,mark=at position #1 with

{\arrow{latex}}},postaction={decorate}}}
\filldraw[draw=white,fill=gray!20] (-0,-0.2) rectangle (1, 1.2);
\draw [line width =1pt,decoration={markings, mark=at position 0.5 with {\arrow{>}}},postaction={decorate}](0.6,0.6)--(1,1);
\draw [line width =1pt,decoration={markings, mark=at position 0.5 with {\arrow{>}}},postaction={decorate}](0.6,0.4)--(1,0);
\draw[line width =1pt] (0,0)--(0.4,0.4);
\draw[line width =1pt] (0,1)--(0.4,0.6);
\draw[line width =1pt] (0.4,0.6)--(0.6,0.4);
\end{tikzpicture}
}
- q^{\frac{1}{n}}
\raisebox{-.20in}{
\begin{tikzpicture}
\tikzset{->-/.style=

{decoration={markings,mark=at position #1 with

{\arrow{latex}}},postaction={decorate}}}
\filldraw[draw=white,fill=gray!20] (-0,-0.2) rectangle (1, 1.2);
\draw [line width =1pt,decoration={markings, mark=at position 0.5 with {\arrow{>}}},postaction={decorate}](0.6,0.6)--(1,1);
\draw [line width =1pt,decoration={markings, mark=at position 0.5 with {\arrow{>}}},postaction={decorate}](0.6,0.4)--(1,0);
\draw[line width =1pt] (0,0)--(0.4,0.4);
\draw[line width =1pt] (0,1)--(0.4,0.6);
\draw[line width =1pt] (0.6,0.6)--(0.4,0.4);
\end{tikzpicture}
}
= (q^{-1}-q)
\raisebox{-.20in}{

\begin{tikzpicture}
\tikzset{->-/.style=

{decoration={markings,mark=at position #1 with

{\arrow{latex}}},postaction={decorate}}}
\filldraw[draw=white,fill=gray!20] (-0,-0.2) rectangle (1, 1.2);
\draw [line width =1pt,decoration={markings, mark=at position 0.5 with {\arrow{>}}},postaction={decorate}](0,0.8)--(1,0.8);
\draw [line width =1pt,decoration={markings, mark=at position 0.5 with {\arrow{>}}},postaction={decorate}](0,0.2)--(1,0.2);
\end{tikzpicture}
},
\eeq 
\beq\label{w.twist}
\raisebox{-.15in}{
\begin{tikzpicture}
\tikzset{->-/.style=
{decoration={markings,mark=at position #1 with
{\arrow{latex}}},postaction={decorate}}}
\filldraw[draw=white,fill=gray!20] (-1,-0.35) rectangle (0.6, 0.65);
\draw [line width =1pt,decoration={markings, mark=at position 0.5 with {\arrow{>}}},postaction={decorate}](-1,0)--(-0.25,0);
\draw [color = black, line width =1pt](0,0)--(0.6,0);
\draw [color = black, line width =1pt] (0.166 ,0.08) arc (-37:270:0.2);
\end{tikzpicture}}
= \mathbbm{t}
\raisebox{-.15in}{
\begin{tikzpicture}
\tikzset{->-/.style=
{decoration={markings,mark=at position #1 with
{\arrow{latex}}},postaction={decorate}}}
\filldraw[draw=white,fill=gray!20] (-1,-0.5) rectangle (0.6, 0.5);
\draw [line width =1pt,decoration={markings, mark=at position 0.5 with {\arrow{>}}},postaction={decorate}](-1,0)--(-0.25,0);
\draw [color = black, line width =1pt](-0.25,0)--(0.6,0);
\end{tikzpicture}}
,  
\eeq
\beq\label{w.unknot}
\raisebox{-.20in}{
\begin{tikzpicture}
\tikzset{->-/.style=
{decoration={markings,mark=at position #1 with
{\arrow{latex}}},postaction={decorate}}}
\filldraw[draw=white,fill=gray!20] (0,0) rectangle (1,1);
\draw [line width =1pt,decoration={markings, mark=at position 0.5 with {\arrow{>}}},postaction={decorate}](0.45,0.8)--(0.55,0.8);
\draw[line width =1pt] (0.5 ,0.5) circle (0.3);
\end{tikzpicture}}
= (-1)^{n-1} [n]\ 
\raisebox{-.20in}{
\begin{tikzpicture}
\tikzset{->-/.style=
{decoration={markings,mark=at position #1 with
{\arrow{latex}}},postaction={decorate}}}
\filldraw[draw=white,fill=gray!20] (0,0) rectangle (1,1);
\end{tikzpicture}}
,\ \text{where}\ [n]=\frac{q^n-q^{-n}}{q-q^{-1}},
\eeq
\beq\label{wzh.four}
\raisebox{-.30in}{
\begin{tikzpicture}
\tikzset{->-/.style=
{decoration={markings,mark=at position #1 with
{\arrow{latex}}},postaction={decorate}}}
\filldraw[draw=white,fill=gray!20] (-1,-0.7) rectangle (1.2,1.3);
\draw [line width =1pt,decoration={markings, mark=at position 0.5 with {\arrow{>}}},postaction={decorate}](-1,1)--(0,0);
\draw [line width =1pt,decoration={markings, mark=at position 0.5 with {\arrow{>}}},postaction={decorate}](-1,0)--(0,0);
\draw [line width =1pt,decoration={markings, mark=at position 0.5 with {\arrow{>}}},postaction={decorate}](-1,-0.4)--(0,0);
\draw [line width =1pt,decoration={markings, mark=at position 0.5 with {\arrow{<}}},postaction={decorate}](1.2,1)  --(0.2,0);
\draw [line width =1pt,decoration={markings, mark=at position 0.5 with {\arrow{<}}},postaction={decorate}](1.2,0)  --(0.2,0);
\draw [line width =1pt,decoration={markings, mark=at position 0.5 with {\arrow{<}}},postaction={decorate}](1.2,-0.4)--(0.2,0);
\node  at(-0.8,0.5) {$\vdots$};
\node  at(1,0.5) {$\vdots$};
\end{tikzpicture}}=(-q)^{-\frac{n(n-1)}{2}}\cdot \sum_{\sigma\in \fS_n}
(-q^{-\frac{1-n}n})^{\ell(\sigma)} \raisebox{-.30in}{
\begin{tikzpicture}
\tikzset{->-/.style=
{decoration={markings,mark=at position #1 with
{\arrow{latex}}},postaction={decorate}}}
\filldraw[draw=white,fill=gray!20] (-1,-0.7) rectangle (1.2,1.3);
\draw [line width =1pt,decoration={markings, mark=at position 0.5 with {\arrow{>}}},postaction={decorate}](-1,1)--(0,0);
\draw [line width =1pt,decoration={markings, mark=at position 0.5 with {\arrow{>}}},postaction={decorate}](-1,0)--(0,0);
\draw [line width =1pt,decoration={markings, mark=at position 0.5 with {\arrow{>}}},postaction={decorate}](-1,-0.4)--(0,0);
\draw [line width =1pt,decoration={markings, mark=at position 0.5 with {\arrow{<}}},postaction={decorate}](1.2,1)  --(0.2,0);
\draw [line width =1pt,decoration={markings, mark=at position 0.5 with {\arrow{<}}},postaction={decorate}](1.2,0)  --(0.2,0);
\draw [line width =1pt,decoration={markings, mark=at position 0.5 with {\arrow{<}}},postaction={decorate}](1.2,-0.4)--(0.2,0);
\node  at(-0.8,0.5) {$\vdots$};
\node  at(1,0.5) {$\vdots$};
\filldraw[draw=black,fill=gray!20,line width =1pt]  (0.1,0.3) ellipse (0.4 and 0.7);
\node  at(0.1,0.3){$\sigma_{+}$};
\end{tikzpicture}},
\eeq
where the ellipse enclosing $\sigma_+$  is the minimum crossing positive braid representing a permutation $\sigma\in S_n$ and $\ell(\sigma)=\#\{(i,j)\mid 1\leq i<j\leq n,\ \sigma(i)>\sigma(j)\}$ is the length of $\sigma\in S_n$.

\beq\label{wzh.five}
   \raisebox{-.30in}{
\begin{tikzpicture}
\tikzset{->-/.style=
{decoration={markings,mark=at position #1 with
{\arrow{latex}}},postaction={decorate}}}
\filldraw[draw=white,fill=gray!20] (-1,-0.7) rectangle (0.2,1.3);
\draw [line width =1pt](-1,1)--(0,0);
\draw [line width =1pt](-1,0)--(0,0);
\draw [line width =1pt](-1,-0.4)--(0,0);
\draw [line width =1.5pt](0.2,1.3)--(0.2,-0.7);
\node  at(-0.8,0.5) {$\vdots$};
\filldraw[fill=white,line width =0.8pt] (-0.5 ,0.5) circle (0.07);
\filldraw[fill=white,line width =0.8pt] (-0.5 ,0) circle (0.07);
\filldraw[fill=white,line width =0.8pt] (-0.5 ,-0.2) circle (0.07);
\end{tikzpicture}}
   = 
   \mathbbm{a} \sum_{\sigma \in \fS_n} (-q)^{-\ell(\sigma)}\,  \raisebox{-.30in}{
\begin{tikzpicture}
\tikzset{->-/.style=
{decoration={markings,mark=at position #1 with
{\arrow{latex}}},postaction={decorate}}}
\filldraw[draw=white,fill=gray!20] (-1,-0.7) rectangle (0.2,1.3);
\draw [line width =1pt](-1,1)--(0.2,1);
\draw [line width =1pt](-1,0)--(0.2,0);
\draw [line width =1pt](-1,-0.4)--(0.2,-0.4);
\draw [line width =1.5pt,decoration={markings, mark=at position 1 with {\arrow{>}}},postaction={decorate}](0.2,1.3)--(0.2,-0.7);
\node  at(-0.8,0.5) {$\vdots$};
\filldraw[fill=white,line width =0.8pt] (-0.5 ,1) circle (0.07);
\filldraw[fill=white,line width =0.8pt] (-0.5 ,0) circle (0.07);
\filldraw[fill=white,line width =0.8pt] (-0.5 ,-0.4) circle (0.07);
\node [right] at(0.2,1) {$\sigma(n)$};
\node [right] at(0.2,0) {$\sigma(2)$};
\node [right] at(0.2,-0.4){$\sigma(1)$};
\end{tikzpicture}},
\eeq
\beq \label{wzh.six}
\raisebox{-.20in}{
\begin{tikzpicture}
\tikzset{->-/.style=
{decoration={markings,mark=at position #1 with
{\arrow{latex}}},postaction={decorate}}}
\filldraw[draw=white,fill=gray!20] (-0.7,-0.7) rectangle (0,0.7);
\draw [line width =1.5pt,decoration={markings, mark=at position 1 with {\arrow{>}}},postaction={decorate}](0,0.7)--(0,-0.7);
\draw [color = black, line width =1pt] (0 ,0.3) arc (90:270:0.5 and 0.3);
\node [right]  at(0,0.3) {$i$};
\node [right] at(0,-0.3){$j$};
\filldraw[fill=white,line width =0.8pt] (-0.5 ,0) circle (0.07);
\end{tikzpicture}}   = \delta_{\bar j,i }\,  \mathbbm{c}_{i} \raisebox{-.20in}{
\begin{tikzpicture}
\tikzset{->-/.style=
{decoration={markings,mark=at position #1 with
{\arrow{latex}}},postaction={decorate}}}
\filldraw[draw=white,fill=gray!20] (-0.7,-0.7) rectangle (0,0.7);
\draw [line width =1.5pt](0,0.7)--(0,-0.7);
\end{tikzpicture}},
\eeq
\beq \label{wzh.seven}
\raisebox{-.20in}{
\begin{tikzpicture}
\tikzset{->-/.style=
{decoration={markings,mark=at position #1 with
{\arrow{latex}}},postaction={decorate}}}
\filldraw[draw=white,fill=gray!20] (-0.7,-0.7) rectangle (0,0.7);
\draw [line width =1.5pt](0,0.7)--(0,-0.7);
\draw [color = black, line width =1pt] (-0.7 ,-0.3) arc (-90:90:0.5 and 0.3);
\filldraw[fill=white,line width =0.8pt] (-0.55 ,0.26) circle (0.07);
\end{tikzpicture}}
= \sum_{i=1}^n  (\mathbbm{c}_{\bar i})^{-1}\, \raisebox{-.20in}{
\begin{tikzpicture}
\tikzset{->-/.style=
{decoration={markings,mark=at position #1 with
{\arrow{latex}}},postaction={decorate}}}
\filldraw[draw=white,fill=gray!20] (-0.7,-0.7) rectangle (0,0.7);
\draw [line width =1.5pt,decoration={markings, mark=at position 1 with {\arrow{>}}},postaction={decorate}](0,0.7)--(0,-0.7);
\draw [line width =1pt](-0.7,0.3)--(0,0.3);
\draw [line width =1pt](-0.7,-0.3)--(0,-0.3);
\filldraw[fill=white,line width =0.8pt] (-0.3 ,0.3) circle (0.07);
\filldraw[fill=black,line width =0.8pt] (-0.3 ,-0.3) circle (0.07);
\node [right]  at(0,0.3) {$i$};
\node [right]  at(0,-0.3) {$\bar{i}$};
\end{tikzpicture}},
\eeq
\beq\label{wzh.eight}
\raisebox{-.20in}{

\begin{tikzpicture}
\tikzset{->-/.style=

{decoration={markings,mark=at position #1 with

{\arrow{latex}}},postaction={decorate}}}
\filldraw[draw=white,fill=gray!20] (-0,-0.2) rectangle (1, 1.2);
\draw [line width =1.5pt,decoration={markings, mark=at position 1 with {\arrow{>}}},postaction={decorate}](1,1.2)--(1,-0.2);
\draw [line width =1pt](0.6,0.6)--(1,1);
\draw [line width =1pt](0.6,0.4)--(1,0);
\draw[line width =1pt] (0,0)--(0.4,0.4);
\draw[line width =1pt] (0,1)--(0.4,0.6);
\draw[line width =1pt] (0.4,0.6)--(0.6,0.4);
\filldraw[fill=white,line width =0.8pt] (0.2 ,0.2) circle (0.07);
\filldraw[fill=white,line width =0.8pt] (0.2 ,0.8) circle (0.07);
\node [right]  at(1,1) {$i$};
\node [right]  at(1,0) {$j$};
\end{tikzpicture}
} =q^{\frac{1}{n}}\left(\delta_{{j<i} }(q^{-1}-q)\raisebox{-.20in}{

\begin{tikzpicture}
\tikzset{->-/.style=

{decoration={markings,mark=at position #1 with

{\arrow{latex}}},postaction={decorate}}}
\filldraw[draw=white,fill=gray!20] (-0,-0.2) rectangle (1, 1.2);
\draw [line width =1.5pt,decoration={markings, mark=at position 1 with {\arrow{>}}},postaction={decorate}](1,1.2)--(1,-0.2);
\draw [line width =1pt](0,0.8)--(1,0.8);
\draw [line width =1pt](0,0.2)--(1,0.2);
\filldraw[fill=white,line width =0.8pt] (0.2 ,0.8) circle (0.07);
\filldraw[fill=white,line width =0.8pt] (0.2 ,0.2) circle (0.07);
\node [right]  at(1,0.8) {$i$};
\node [right]  at(1,0.2) {$j$};
\end{tikzpicture}
}+q^{\delta_{i,j}}\raisebox{-.20in}{

\begin{tikzpicture}
\tikzset{->-/.style=

{decoration={markings,mark=at position #1 with

{\arrow{latex}}},postaction={decorate}}}
\filldraw[draw=white,fill=gray!20] (-0,-0.2) rectangle (1, 1.2);
\draw [line width =1.5pt,decoration={markings, mark=at position 1 with {\arrow{>}}},postaction={decorate}](1,1.2)--(1,-0.2);
\draw [line width =1pt](0,0.8)--(1,0.8);
\draw [line width =1pt](0,0.2)--(1,0.2);
\filldraw[fill=white,line width =0.8pt] (0.2 ,0.8) circle (0.07);
\filldraw[fill=white,line width =0.8pt] (0.2 ,0.2) circle (0.07);
\node [right]  at(1,0.8) {$j$};
\node [right]  at(1,0.2) {$i$};
\end{tikzpicture}
}\right),
\eeq
where   
$\delta_{j<i}= 
\begin{cases}
1  & j<i\\
0 & \text{otherwise}
\end{cases},\ 
\delta_{i,j}= 
\begin{cases} 
1  & i=j\\
0  & \text{otherwise}
\end{cases}$, and small white dots represent an arbitrary orientation of the edges (left-to-right or right-to-left), consistent for the entire equation. The black dot represents the opposite orientation. When a boundary edge of a shaded area is directed, the direction indicates the height order of the endpoints of the diagrams on that directed line, where going along the direction increases the height, and the involved endpoints are consecutive in the height order. The height order outside the drawn part can be arbitrary.

Our parameter $\omega^{\frac{1}{2}}$ corresponds to $\hat q^{-1}$ in \cite{LY23} (our setting fits well with the quantum cluster theory). 
Note that $R$ carries a natural $\mathbb{Z}[\omega^{\pm \frac{1}{2}}]$-module structure, making it a 
$\mathbb{Z}[\omega^{\pm \frac{1}{2}}]$-algebra. 
For the remainder of this paper, we will assume $R = \mathbb{Z}[\omega^{\pm \frac{1}{2}}]$. 
All results remain valid for a general $R$ by applying the functor 
$-\otimes_{\mathbb{Z}[\omega^{\pm \frac{1}{2}}]} R$.

 \def \Sv{\cS_n(\Si,\mathbbm{v})}

 The algebra structure for $\cS_{\omega}(\fS)$ is given by stacking the stated $n$-webs, i.e. for any two stated $n$-webs $\alpha,\alpha'
 \subset\fS\times(-1,1)$, the product $\alpha\alpha'$ is defined by stacking $\alpha$ above $\alpha'$. That is, if $\alpha \subset \fS\times(0,1)$ and $\alpha' \subset \fS\times (-1,0)$, we have $\alpha \alpha' = \alpha \cup \alpha'$.

For a boundary puncture $p$ of a pb surface $\fS$, 
corner arcs $C(p)_{ij}$ and $\overline{C}(p)_{ij}$ are stated arcs depicted as in Figure \ref{Fig;badarc}.
For a boundary puncture $p$ which is not on a $\mathbb P_1$ component of $\fS$, set 
$$C_p=\{C(p)_{ij}\mid i<j\},\quad\overline{C}_p=\{\overline{C}(p)_{ij}\mid i<j\}.$$  
Each element of $C_p\cup \overline{C}_p$ is called a \emph{bad arc} at $p$. 
\begin{figure}[h]
    \centering
    \includegraphics[width=150pt]{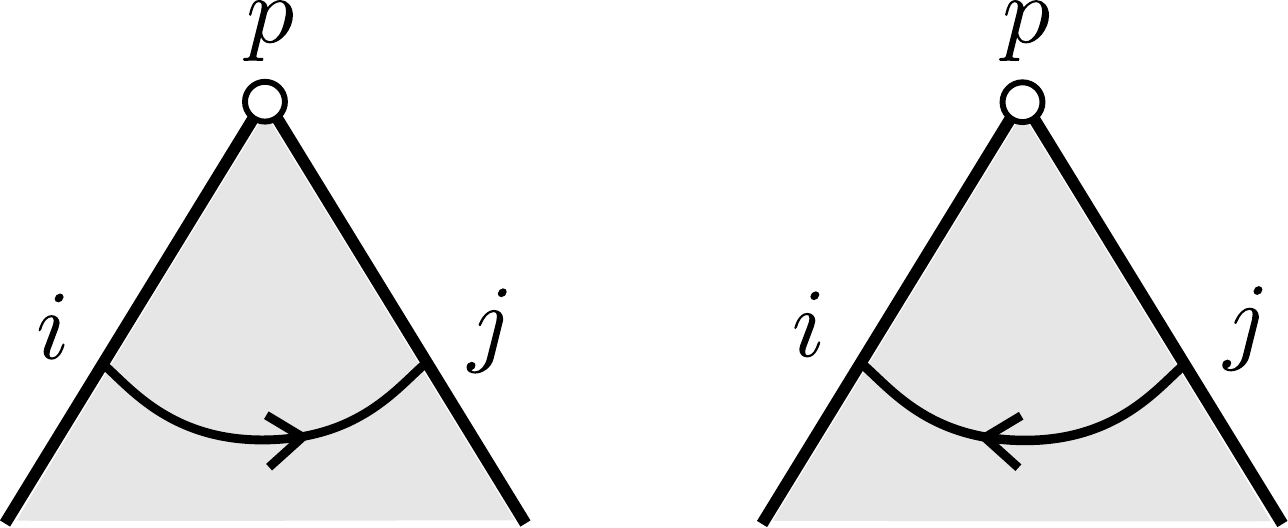}
    \caption{The left is $C(p)_{ij}$ and the right is $\overline{C}(p)_{ij}$.}\label{Fig;badarc}
\end{figure}

For a pb surface $\fS$,  $$\overline \cS_{\omega}(\fS) = \cS_{\omega}(\fS)/I^{\text{bad}}$$ 
is called the \textbf{reduced stated $\SL$-skein algebra}, defined in \cite{LY23}, where $I^{\text{bad}}$ is the two-sided ideal of $\cS_{\omega}(\fS)$ generated by all bad arcs.

\def\al{\alpha}

\subsection{The splitting homomorphism}\label{sub-splitting}
\def\cut{\mathsf{Cut}}
\def\pr{{\bf pr}}

Let $e$ be an interior ideal arc of a pb surface $\fS$ such that it is contained in the interior of $\fS$. After cutting $\fS$ along $e$, we get a new pb surface $\cut_e(\fS)$, which has two copies $e_1,e_2$ for $c$ such that 
${\fS}= \cut_e(\fS)/(e_1=e_2)$. We use $\pr_e$ to denote the projection from $\cut_e(\fS)$ to $\fS$.  Suppose that $\alpha$ is  a stated $n$-web diagram in $\fS$, which is transverse to $e$.
Let $s$ be a map from $e\cap\alpha$ to $\{1,2,\cdots,n\}$, and let $h$ be a linear order on $e\cap\alpha$. Then there is a lift diagram $\alpha(h,s)$ for a stated $n$-web in $\cut_e(\fS)$. 
 For $i=1,2$, the heights of the endpoints of $\alpha(h,s)$ on $e_i$ are induced by $h$ (via $\pr_e$), and the states of the endpoints of $\alpha(h,s)$ on $e_i$ are induced by $s$ (via $\pr_e$).
Then the splitting map 
$$
\mathbb{S}_e : \mathscr{S}_\omega(\fS) \to \mathscr{S}_\omega(\cut_e(\fS))
$$
is defined by 
\begin{align}\label{eq-def-splitting}
    \mathbb S_e(\alpha) =\sum_{s\colon \alpha \cap e \to \{1,\cdots, n\}} \alpha(h, s). 
\end{align}
Furthermore $\mathbb S_e$ is an $R$-algebra homomorphism \cite{LS21}.

It is well-known that \cite{LS21}
\begin{align}\label{com-splitting-skein}
    \mathbb S_{e_1}\circ \mathbb S_{e_2}
    = \mathbb S_{e_2}\circ \mathbb S_{e_1}
\end{align}
for any two disjoint interior ideal arcs $e_1,e_2$ of $\fS$. 

We can observe that $\mathbb S_e$ sends bad arcs to bad arcs.
So it induces an algebra homomorphism from $\rdS$
to $\rS(\cut_e(\fS))$, which is still denoted as $\mathbb S_e$.
When there is no confusion we will omit the subscript $e$ 
from $\mathbb S_e$.

\def\bT{\mathbb T}

\def\bZ{\mathbb Z}

\subsection{Reflection}\label{sub-sec-invariant}
Recall that $R = \bZ[\omega^{\pm\frac{1}{2}}]$. 
An $R$-algebra with reflection is
an $R$-algebra $A$ equipped with a $\mathbb Z$-linear anti-involution ${\bf Rf}$, called the reflection, such that ${\bf Rf}(\omega^{\frac{1}{2}}) = \omega^{-\frac{1}{2}}$. In other words, ${\bf Rf}\colon A\rightarrow A$ is a $\mathbb Z$-linear map such that for all
$x,y\in A$,
$${\bf Rf}(xy) = {\bf Rf}(y){\bf Rf}(x),\quad
{\bf Rf}(\omega^{\frac{1}{2}}x)
=\omega^{-\frac{1}{2}}x,\quad
{\bf Rf}^2={\rm Id}_A.$$
An element $z\in A$ is called reflection invariant if ${\bf Rf}(z)=z$. If $B$ is another $R$-algebra
with reflection ${\bf Rf}'$, then a map $f\colon A\rightarrow B$ is {\bf reflection invariant} if $f\circ {\bf Rf} = {\bf Rf}'\circ f$.

For a pb surface $\fS$, 
it is known that there is a unique reflection $$* \colon \cS_{\omega} (\fS)\to \cS_{\omega} (\fS)$$ such that, for a stated n-web diagram $\alpha$, $*(\alpha)$ is defined from $\alpha$ by switching all the crossings and reversing the height order on each boundary edge \cite[Theorem 4.9]{LS21}. 

A stated $n$-web diagram $\alpha$ in a pb surface $\fS$ is \textbf{reflection-normalizable} if $*(\alpha) = \omega^{k}\alpha$ for $k \in \bZ$. Note that such $k$ is unique if each connected component of $\fS$ contains non-empty boundaries \cite{LS21}. We define the \textbf{reflection-normalization} of $\alpha$ by
$$[\alpha]_{\rm norm} := \omega^{\frac{k}{2}}\alpha.$$
Then we have $*([\alpha]_{\rm norm}) = [\alpha]_{\rm norm}$, i.e., $[\alpha]_{\rm norm}$ is reflection invariant. 

The following lemma holds since the linear order $h$ in \eqref{eq-def-splitting} may be chosen arbitrarily. 

\begin{lemma}\label{lem-reflection}
    Let $\fS$ be a pb surface with an ideal arc $e$. Then the splitting map 
    $\mathbb S_e\colon \cS_{\omega}(\fS)\rightarrow \cS_{\omega} (\mathsf{Cut}_e(\fS))$ is reflection invariant.
\end{lemma}

\subsection{Quantum torus and its $n$-th root of version}
Let $\mathcal V$ be a finite set, and $Q$ be an anti-symmetric matrix $Q\colon \mathcal V\times \mathcal V\rightarrow \frac{1}{2}\mathbb Z$.
Define the quantum torus
\begin{equation*}
\bT_q(Q) = R \langle 
X_v^{\pm 1}, v \in \mathcal V \rangle / (
X_v 
X_{v'}= q^{\, 2 Q(v,v')} 
X_{v'} 
X_v \text{ for } v,v'\in\mathcal  V),
\end{equation*}
and its $n$-th root of version as follows:
\begin{equation*}
\bT_{\omega}(Q) = R \langle 
Z_v^{\pm 1}, v \in \mathcal  V \rangle / (
Z_v 
Z_{v'}= \omega^{\, 2 Q(v,v')} 
Z_{v'} 
Z_v \text{ for } v,v'\in \mathcal  V ).
\end{equation*}
There is an $R$-algebra embedding from $\bT_q(Q)$
to $\bT_\omega(Q)$ defined by $X_v\mapsto Z_v^n$
for $v\in\mathcal  V$.

A useful notion is the Weyl-ordered product, which is a certain normalization of a Laurent monomial defined as follows: for any $v_1,\ldots,v_r \in\mathcal  V$ and $a_1,\ldots,a_r \in \mathbb{Z}$,
\begin{align}
\label{Weyl_ordering-X-1}
\left[ X_{v_1}^{a_1} X_{v_2}^{a_2} \cdots X_{v_r}^{a_r} \right] := q^{-\sum_{i<j} a_i a_j Q(v_i,v_j)} X_{v_1}^{a_1} X_{v_2}^{a_2} \cdots X_{v_r}^{a_r}\\
\label{Weyl_ordering-Z-1}
\left[ Z_{v_1}^{a_1} Z_{v_2}^{a_2} \cdots Z_{v_r}^{a_r} \right] := \omega^{-\sum_{i<j} a_i a_j Q(v_i,v_j)} Z_{v_1}^{a_1} Z_{v_2}^{a_2} \cdots Z_{v_r}^{a_r}.
\end{align}
In particular, for ${\bf t} = (t_v)_{v\in \mathcal{V}} \in \mathbb{Z}^{\mathcal  V}$, the following notation for the corresponding Weyl-ordered Laurent monomial will become handy:
\begin{align}\label{wey-normal-monomial-XZ}
    X^{\bf t} := \left[ \prod_{v\in \mathcal{V}} X_v^{t_v}\right],\quad
Z^{\bf t} := \left[ \prod_{v\in \mathcal{V}} Z_v^{t_v}\right].
\end{align}
Due to the property of Weyl-ordered products, note that these elements are independent of the choice of the order of the product inside the bracket.
The element $X^{\bf t}$ (or $Z^{\bf t}$) is called the {\bf (Laurent) monomial}. It is well-known that $\{X^{\bf t}\mid {\bf t}\in \mathbb Z^{\mathcal  V}\}$ (resp.
$\{Z^{\bf t}\mid {\bf t}\in \mathbb Z^{\mathcal  V}\}$) is an $R$ basis of $\bT_{q}(Q)$ (resp.
$\bT_{\omega}(Q)$), called the {\bf monomial basis}.

For any ${\bf k},{\bf t}\in\mathbb Z^{\mathcal V}$, it is straightforward to verify that
\begin{align*}
    X^{\bf k} X^{\bf t} = q^{2{\bf k}Q{\bf t}^T} X^{\bf t} X^{\bf k} \text{ and }
Z^{\bf k} Z^{\bf t} = \omega^{2{\bf k}Q{\bf t}^T} Z^{\bf t} Z^{\bf k},
\end{align*}
where ${\bf k}$ and ${\bf t}$ are regarded as row vectors. 

There is a reflection
\begin{align}\label{def-eq-ref-torus}
    * \colon \bT_{\omega}(Q)\longrightarrow \bT_{\omega}(Q), \qquad
    Z_v \longmapsto Z_v \quad (v\in \mathcal  V).
\end{align}
It is straightforward to verify that 
$*(Z^{\mathbf k}) = Z^{\mathbf k}$ for all $\mathbf k \in \mathbb Z^{\mathcal  V}$.

\def\bk{{\bf k}}

Let $\mathcal C$ be a subgroup of $\mathbb Z^{\mathcal  V}$. Define 
\begin{align}
    \label{submonoid}
    \bT_\omega(Q;\mathcal C) = \text{span}_R\{Z^\bk\mid \bk\in \mathcal C\}\subset
\bT_\omega(Q).
\end{align}
Then $\bT_\omega(Q;\mathcal C)$ is a subalgebra of $\bT_\omega(Q)$, called the {\bf group subalgebra of $\bT_\omega(Q)$}.

The following lemma will be used in \S\ref{sec-skein-cluster}.

\begin{lemma}\label{lem-frac-subalgebra}
Let $\bT_\omega(Q;\mathcal C)$ be a group subalgebra of $\bT_\omega(Q)$.
Suppose that $X \in \bT_\omega(Q)$ and $0 \neq P \in \bT_{\omega}(Q;\mathcal C)$ satisfy 
$XP \in \bT_{\omega}(Q;\mathcal C)$. Then $X \in \bT_{\omega}(Q;\mathcal C)$.
\end{lemma}

\begin{proof}
Choose a linear order on $\mathcal V$, and let $\leq$ denote the induced lexicographic order on $\mathbb Z^{\mathcal  V}$.  
For any nonzero element
$\sum_{1 \leq i \leq m} r_i Z^{{\bf k}_i} \in \bT_\omega(Q),$
with $0 \neq r_i \in R$ and distinct ${\bf k}_i \in \mathbb Z^{\mathcal V}$,  
assume ${\bf k}_i < {\bf k}_1$ for $2 \leq i \leq m$.  
Define 
\[
\deg\!\left(\sum_{1 \leq i \leq m} r_i Z^{{\bf k}_i}\right) = {\bf k}_1.
\]
It is straightforward to verify that for any two nonzero elements 
$X_1, X_2 \in \bT_\omega(Q)$ one has
\[
\deg(X_1 X_2) = \deg(X_1) + \deg(X_2).
\]

Now suppose 
\(
X = X' + \sum_{1 \leq i \leq l} r_i Z^{{\bf k}_i},
\)
where $X' \in \bT_{\omega}(Q;\mathcal C)$,  
$0 \neq r_i \in R$, and ${\bf k}_i \in \mathbb Z^{V} \setminus \mathcal C$ are distinct ($1 \leq i \leq l$).  
Note that $l$ could be $0$.  
By assumption, 
\[
\Bigl(\sum_{1 \leq i \leq l} r_i Z^{{\bf k}_i}\Bigr) P \in \bT_{\omega}(Q;\mathcal C).
\]
To conclude the proof, we need to show that $l=0$, which would imply $X=X' \in \bT_{\omega}(Q;\mathcal C)$.  

Assume, on the contrary, that $l>0$.  
Set 
\(
X'' = \sum_{1 \leq i \leq l} r_i Z^{{\bf k}_i}.
\)
Then 
\[
\deg(X''P) = \deg(X'') + \deg(P).
\]
Since $\deg(X'') \in \mathbb Z^{V} \setminus \mathcal C$ and $\deg(P) \in \mathcal C$, it follows that 
\[
\deg(X''P) \in \mathbb Z^{V} \setminus \mathcal C.
\]
This contradicts the assumption that $X'' P \in \bT_{\omega}(Q;\mathcal C)$.  
Hence $l=0$, and the claim follows.
\end{proof}

\subsection{The $\mathcal X$-version quantum trace map}
\label{sec-traceX}

\begin{figure}[h]
    \centering
    \includegraphics[width=240pt]{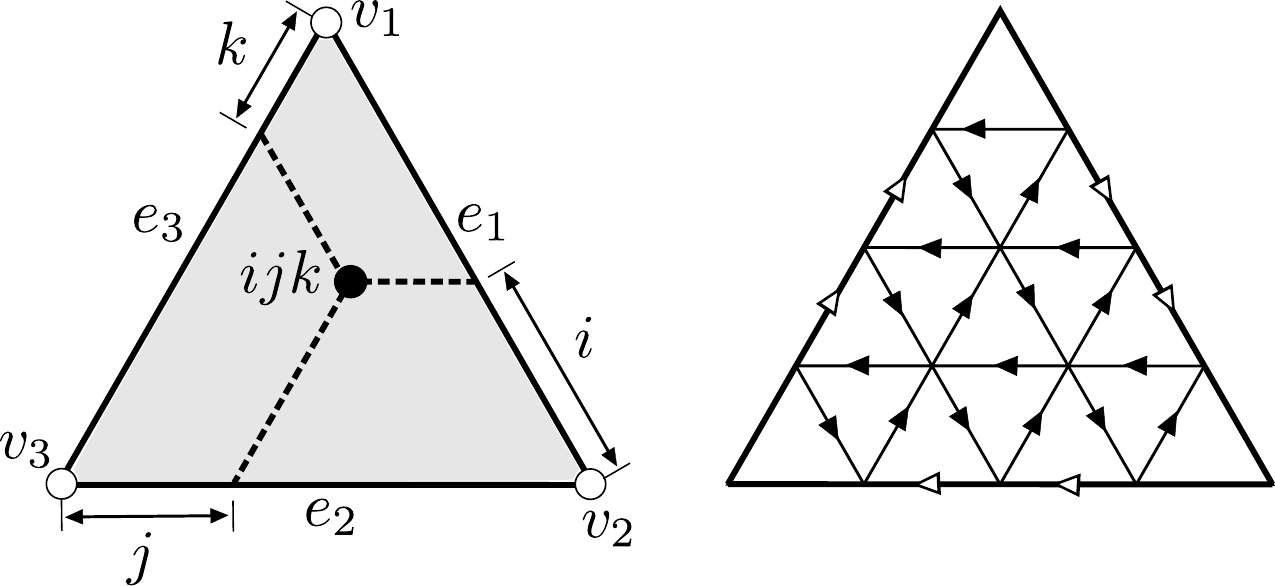}
    \caption{Barycentric coordinates $ijk$ and a $4$-triangulation with its quiver.}\label{Fig;coord_ijk}
\end{figure}

In this subsection we review the $\mathcal X$-version quantum trace map defined in \cite{LY23}, which is an algebra homomorphism from the reduced stated $\SL$-skein algebra of a triangulable pb surface $\fS$ to a quantum torus algebra associated to a special quiver built from any choice of an `(ideal) triangulation' of $\fS$.

To begin with, consider barycentric coordinates for an ideal triangle $\bP_3$ so that
\begin{equation*}
\bP_3=\{(i,j,k)\in\bR^3\mid i,j,k\geq 0,\ i+j+k=n\}\setminus\{(0,0,n),(0,n,0),(n,0,0)\}, 
\end{equation*}
where $(i,j,k)$ (or $ijk$ for simplicity) are the barycentric coordinates. 
Let $v_1=(n,0,0)$, $v_2=(0,n,0)$, $v_3=(0,0,n)$. 
Let $e_i$ denote the edge on $\partial \bP_3$ whose endpoints are $v_i$ and $v_{i+1}$. 
See Figure \ref{Fig;coord_ijk}.

The \textbf{$n$-triangulation} of $\bP_3$ is the subdivision of $\bP_3$ into $n^2$ small triangles using lines $i,j,k=\text{constant integers}$. 
For the $n$-triangulation of $\bP_3$, the vertices and edges of all small triangles, except for $v_1,v_2,v_3$ and the small edges adjacent to them, form a quiver $\Gamma_{\bP_3}$.
An \textbf{arrow} is the direction of a small edge defined as follows. If a small edge $e$ is in the boundary $\partial\bP_3$ then $e$ has the clockwise direction of $\partial \bP_3$. If $e$ is interior then its direction is the same with that of a boundary edge of $\bP_3$ parallel to $e$. Assign weight $\frac{1}{2}$ to any boundary arrow and weight $1$ to any interior arrow.

\def\bZ{\mathbb Z}

Let $V_{\bP_3}$ be the set of all
points with integer barycentric coordinates of $\bP_3$:
\begin{align}\label{def-V-P3}
V_{\bP_3} = \{ijk \in \bP_3 \mid i, j, k \in \bZ\}.
\end{align}
Note that $V_{\bP_3}$ does not contain
$v_1$, $v_2$, or $v_3$.
Elements of $V_{\bP_3}$ are called \textbf{small vertices}, and small vertices on the boundary of $\bP_3$ are called the \textbf{edge vertices}. 

A pb surface $\fS$ is said to be {\bf triangulable} if no connected component of $\fS$ is of the following types:  
\begin{enumerate}[label={\rm (\arabic*)}]
    \item the monogon $\mathbb{P}_1$,  
    \item the bigon $\mathbb{P}_2$,  
    \item a sphere with one or two punctures,  
    \item a closed surface (i.e., a pb surface without punctures).  
\end{enumerate}

A {\bf triangulation} $\lambda$ of $\fS$ is a collection of disjoint ideal arcs in $\fS$ with the following properties: (1) any two arcs in $\lambda$ are not isotopic; (2) $\lambda$ is maximal under condition (1); (3) every puncture is adjacent to at least two ideal arcs.
Our definition of triangulation excludes the so-called self-folded triangles.
We will call each ideal arc in $\lambda$ an {\bf edge} of $\lambda$.
If an edge is isotopic to a boundary component of $\fS$, we call such an edge a {\bf boundary edge}.
We use $\mathbb F_{\lambda}$ to denote the set of faces after we cut $\fS$ along all ideal arcs in $\lambda$.
It is well-known that any triangulable surface admits a triangulation.

Suppose that $\fS$ is a triangulable pb surface with a triangulation $\lambda$.
By cutting $\fS$ along all edges not isotopic to a boundary edge, we have a disjoint union of ideal triangles. Each triangle is called a \textbf{face} of $\lambda$. Then
\begin{equation}\label{eq.glue}
\fS = \Big( \bigsqcup_{\tau\in\mathbb F_\lambda} \tau \Big) /\sim,
\end{equation}
where each face $\tau$ is regarded as a copy of $\bP_3$, and $\sim$ is the identification of edges of the faces to recover $\lambda$. 
Each face $\tau$ is characterized by a \textbf{characteristic map} 
\begin{align}\label{eq-character-map}
    f_\tau\colon \mathbb P_3 \to \fS,
\end{align}
which is a homeomorphism when we restrict $f_\tau$ to $\Int\mathbb P_3$ or the interior of each edge of $\mathbb P_3$.

An \textbf{$n$-triangulation} of $\lambda$ is a collection of $n$-triangulations of the faces $\tau$ which are compatible with the gluing $\sim$,  where compatibility means, for any edges $b$ and $b'$ glued via $\sim$, the vertices on $b$ and $b'$ are identified. Define
$$V_\lambda=\bigcup_{\tau\in\mathbb F_\lambda} V_\tau, \quad V_\tau=f_\tau(V_{\mathbb P_3}).$$
The images of the weighted quivers $\Gamma_{\mathbb P_3}$ by $f_\tau$ form a quiver $\Gamma_\lambda$ on $\fS$.
Note that when edges $b$ and $b'$ are glued, a small edge on $b$ and the corresponding small edge of $b'$ have opposite directions, i.e. the resulting arrows are of weight $0$. 

For any two $v,v'\in V_\lambda$, define
$$
a_\lambda(v,v') = \begin{cases} w \quad & \text{if there is an arrow from $v$ to $v'$ of weight $w$},\\
0 &\text{if there is no arrow between $v$ and $v'$.} 
\end{cases}$$
Let $Q_\lambda\colon V_\lambda\times V_\lambda \to \frac{1}{2}\bZ$ be the signed adjacency matrix of the weighted quiver $\Gamma_\lambda$ defined by 
\begin{equation}\label{eq-def-Q-lambda-re}
Q_\lambda(v,v') = a_\lambda(v,v') - a_\lambda(v',v).
\end{equation}
Especially we use $Q_{\bP_3}$ to denote $Q_\lambda$ when $\fS=\bP_3$.

\begin{remark}\label{rem-Q-LY}
    Note that the arrows in the quiver of Figure~\ref{Fig;coord_ijk} are oriented oppositely to those in \cite[Figure~10]{LY23}, 
and that $Q_\lambda$ in this paper equals 
$-\tfrac{1}{2}\,\overline{\mathsf{Q}}_\lambda$ as defined in \cite{LY23}.  
We adopt this modification because, as will be seen in \S\ref{sec-seed-structure}, our setting aligns more naturally with the framework of quantum cluster theory.
\end{remark}

\def\bT{\mathbb T}

\def\tr{{\rm tr}_{\lambda}}
\def\X{\mathcal X_{\omega}(\fS,\lambda)}
\def\Y{\mathcal Y_{q}(\fS,\lambda)}
\def\bX{\mathcal Z_{\omega}^{\rm bl}(\fS,\lambda)}
\def\bk{{\bf k}}
\def\V{V_\lambda}

The \textbf{Fock-Goncharov algebra} is the quantum torus algebra defined as follows:
\begin{equation}
\mathcal{X}_{q}(\fS,\lambda)= \bT_q(Q_\lambda) = R \langle 
X_v^{\pm 1}, v \in V_\lambda \rangle / (
X_v 
X_{v'}= q^{\, 2 Q_\lambda(v,v')} 
X_{v'} 
X_v \text{ for } v,v'\in V_\lambda ).
\end{equation}
The \textbf{$n$-th root Fock-Goncharov algebra} is the quantum torus algebra defined as follows:
\begin{equation}
\mathcal{Z}_{\omega}(\fS,\lambda)= \bT_{\omega}(Q_\lambda) = R \langle 
Z_v^{\pm 1}, v \in V_\lambda \rangle / (
Z_v 
Z_{v'}= \omega^{\, 2 Q_\lambda(v,v')} 
Z_{v'} 
Z_v \text{ for } v,v'\in V_\lambda ).
\end{equation}

Recall that $w= q^{n^2}$. There is an algebra embedding
from 
$\mathcal{X}_q(\fS,\lambda)$ to 
$\mathcal{Z}_\omega(\fS,\lambda)$ given by 
$$X_v \mapsto Z_v^n$$
for $v\in \V$.
We will regard 
$\mathcal{X}_q(\fS,\lambda)$ as a subalgebra of 
$\mathcal{Z}_\omega(\fS,\lambda)$.

Let ${\bf k}_i\colon V_{\bP_3} \to\bZ\ (i=1,2,3)$ be the functions defined by
\begin{equation}
{\bf k}_1(ijk)=i,\quad {\bf k}_2(ijk)=j,\quad {\bf k}_3(ijk)=k.\label{def:proj}
\end{equation} 
Let $\mathcal B_{\bP_3}\subset\bZ^{V_{\bP_3}}$ be the subgroup generated by ${\bf k}_1,{\bf k}_2,{\bf k}_3$ and $(n\bZ)^{V_{\bP_3}}$. Elements in $\mathcal{B}_{\bP_3}$ are called \textbf{balanced}.

A vector $\bk\in\bZ^{\V}$ is \textbf{balanced} if its pullback ${\bf k}_\tau:=f_\tau^\ast\bk$ to ${\bP_3}$ is balanced for every face of $\lambda$, where for every face $\tau$ and its characteristic map  $f_\tau\colon\mathbb P_3\to\fS$, the pullback $f_\tau^\ast\bk$ is a vector $\V\to\bZ$ given by $f_\tau^\ast\bk(v)=\bk(f_\tau(v))$. 
Let 
\begin{align}
    \label{B_lambda}
    \mathcal{B}_\lambda = \mbox{the subgroup of $\mathbb{Z}^{V_\lambda}$ generated by all the balanced vectors.}
\end{align}

\def\lt{\text{lt}}

The \textbf{balanced Fock-Goncharov algebra} \cite{LY23} is defined as the monomial subalgebra
\begin{align}
    \label{Z_bl_omega}
\bX=\text{span}_R\{Z^\bk\mid \bk\in\mathcal B_\lambda\}\subset
\mathcal{Z}_\omega(\fS,\lambda),
\end{align}
where $Z^{\bf k}$ is defined as in \eqref{wey-normal-monomial-XZ}.
It is easy to verify that $
\mathcal{X}_q(\fS,\lambda) \subset \bX$, and that $\{Z^\bk\mid \bk\in\mathcal B_\lambda\}$ is an $R$-basis of $\bX$.

Suppose $v,v'\in V_\lambda$, If $v$ and $v'$ are not contained in the same boundary edge, define 
\begin{align}\label{def-H-Q-interior}
    H_\lambda(v,v')=Q_\lambda(v,v').
\end{align}
If $v$ and $v'$ are contained in the same boundary edge, define 
\begin{align}\label{eq-def-H-lambda}
    H_\lambda(v,v')=\begin{cases}
    1 & \text{when $v=v'$},\\
    -1 & \text{when there is an arrow from $v’$ to $v$},\\
    0 & \text{otherwise.}
\end{cases}
\end{align}

\begin{lemma}\cite[Proposition~11.10]{LY23}\label{lem-balanced-H}
Suppose that $\fS$ contains no interior punctures. 
    A vector ${\bf k}\in\mathbb Z^{V_\lambda}$ is balanced if and only if ${\bf k} H_\lambda\in (n\mathbb Z)^{V_\lambda}$.
\end{lemma}

\begin{theorem}\cite{BW11,Kim20,LY22,LY23}\label{thm.quantum_trace}
    Suppose that $\fS$ is a triangulable pb surface with a triangulation $\lambda$. Then the following hold.

    \begin{enumerate}[label={\rm (\alph*)}]
    \item There is an algebra homomorphism
    $$\tr \colon \rdS\rightarrow
    \mathcal{Z}_\omega(\fS,\lambda),$$
    called the quantum trace map,
    such that $\im\tr\subset \bX.$

    \item  
    $\tr$ is injective when $n=2$ or when $n=3$.

    \item 
    $\tr$ is injective when $\fS$ is a polygon.
    \end{enumerate}
\end{theorem}

\def\barV{V}
\def\barK{K}
\def\barVl{V_\lambda}
\def\barKt{K_\tau}

\subsection{The $\mathcal A$-version quantum tori} \label{sec;A_tori}
In this subsection, assume $\fS$ contains interior punctures.

For a small vertex $v\in\barV_\lambda$ and an ideal triangle $\tau \in \cF_\lambda$, we define its \textbf{skeleton} $\skeleton_\tau(v)\in \bZ [\barV_\nu]$ as follows.

For a face $\nu\in\cF_\lambda$ containing $v$, suppose $v=(ijk)\in V_\nu$. Draw a weighted directed graph $Y_v$ properly embedded into $\nu$ as in the left of Figure~\ref{Fig;skeleton}, where an edge of $Y_v$ has weight $i$, $j$ or $k$ according as the endpoint lying on the edge $e_1$, $e_2$ or $e_3$ respectively. 
\begin{figure}[h]
    \centering
    \includegraphics[width=350pt]{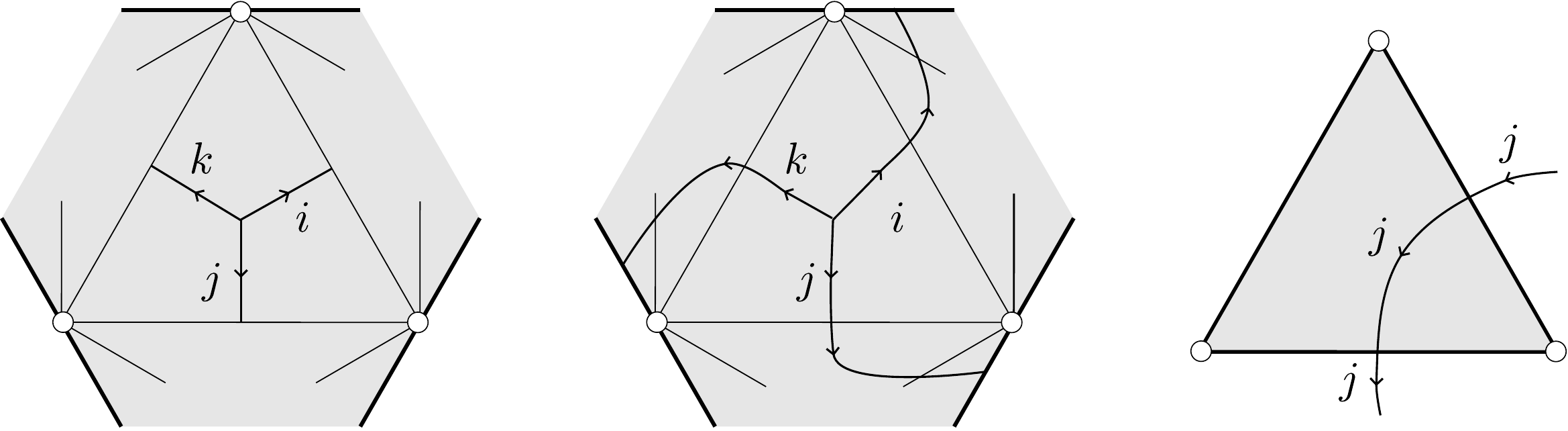}
    \caption{Left: Weighted graph $Y_v$,\quad Middle: Elongation $\widetilde{Y}_v$,\quad Right: Turning left}\label{Fig;skeleton}
\end{figure}

\def\barVt{V_\tau}

Elongate the nonzero-weighted edges of $Y_v$ to have an embedded weighted directed graph $\widetilde{Y}_v$ as drawn in the middle of Figure~\ref{Fig;skeleton}. Here, each edge is elongated by turning left whenever it enters a triangle. 
The part of the elongated edge in a triangle $\tau$ is called a \textbf{(arc) segment} of $\widetilde{Y}_v$ in $\tau$. In addition, we also regard $Y_v$ as a segment of $\widetilde{Y}_v$, called the \textbf{main segment}.

For the main segment $s=Y_v$, define $Y(s) = v \in \barV_\nu$. For an arc segment $s$ in a triangle $\tau$, define $Y(s)\in \barV_\tau$ to be the small vertex of the following weighted graph. 
$$
s=\begin{array}{c}\includegraphics[scale=0.27]{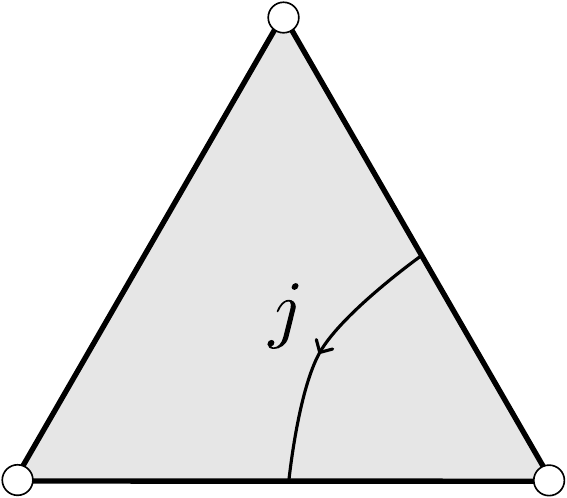}\end{array}
\longrightarrow\quad Y(s):=\begin{array}{c}\includegraphics[scale=0.27]{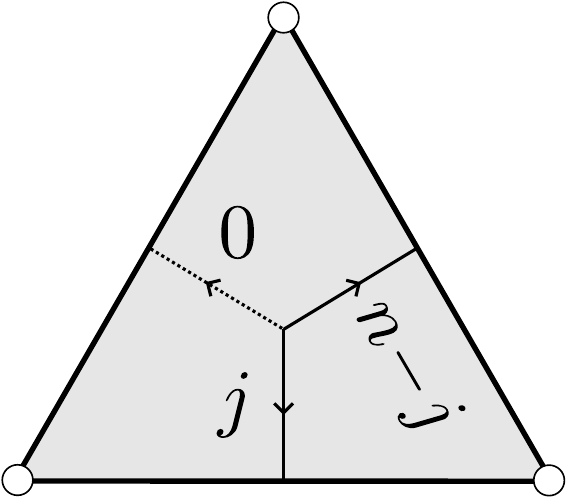}\end{array}
$$
Define $\skeleton_\tau(v)$ by
\begin{equation}
\skeleton_\tau(v) = \sum_{s \subset \tau\cap\tilde{Y}_v} Y(s) \in \bZ [\barVt],
\end{equation}
where the sum is over all the segments of $\widetilde{Y}_v$ in $\tau$.
It is known that $\skeleton_\tau(v)$ is well-defined \cite[Lemma~11.4]{LY23}. 

Define a $\bZ_3$-invariant map
$$\barK_{\bP_3}\colon \barV_{\mathbb P_3}\times\barV_{\mathbb P_3}\to\bZ$$
such that if $v=ijk$ and $v'=i'j'k'$ satisfy $i'\leq i$ and $j'\geq j$ then 
\begin{align}
\barK_{\bP_3}(v,v')=jk'+ki'+i'j.
\end{align}
Under identifying $\bP_3$ and a face $\tau$ of $\lambda$, we also use $\barK_{\tau}$ as $\barK_{\bP_3}$. 

Recall that $\cF_\lambda$ denotes the set of all the faces of the triangulation $\lambda$. 
For $u, v \in \barVl$ and a face $\tau \in \cF_\lambda$ containing $v$,  define 
\begin{equation}\label{eq-surgen-exp}
\barK_{\lambda}(u,v)=\barKt(\skeleton_\tau(u),v)=\sum_{s\subset\tau\cap\tilde{Y}_u}\barKt(Y(s),v).
\end{equation}
It is known that $\barK_{\lambda}$ is well-defined \cite[Lemma~11.5]{LY23}. 
Note that our $K_\lambda$ equals $\overline{\mathsf K}_\lambda$ in \cite{LY23}.

\def\barP{P}

Define the anti-symmetric integer matrices $\barP_{\lambda}$ by 
\begin{align}\label{eq-anti-matric-P-def}
\barP_\lambda:=2\barK_\lambda Q_\lambda\barK^t_\lambda.
\end{align}

\begin{remark}\label{rem-P-P}
    Our $P_\lambda$ is the matrix $-\overline{\mathsf P}_\lambda$ defined in \cite[Equation~(163) and (205)]{LY23}.
    From the definition of $\overline{\mathsf P}_\lambda$ in \cite{LY23}, we know that 
    each entry of $P_\lambda$ is in $n\mathbb Z.$
\end{remark}

\cite[Equation~(214)]{LY23} implies that 
\begin{align}\label{eq-prod-PQ}
    P_\lambda Q_\lambda
    =\begin{pmatrix}
-2n^2(\text{Id}_{\mathring{V}_\lambda\times \mathring{V}_\lambda}) & * \\
O    &  * \\
\end{pmatrix},
\end{align}
where $\mathring{V}_\lambda\subset V_\lambda$ consists of all small vertices contained in the interior of $\fS$.
Equation \eqref{eq-prod-PQ} demonstrates that $(P_\lambda, Q_\lambda)$ forms a compatible pair as defined in \cite[Definition~3.1]{BZ}, which plays a crucial role in \S\ref{sec-skein-cluster} for defining the quantum seed (Definition~\ref{def-quantum-seed}) within the skew-field of the reduced stated ${\rm SL}_n$-skein algebra.

\def\Y{A}

The following is the \textbf{$\mathcal A$-version quantum torus} of $(\fS,\lambda)$ \cite{LY23}:
\begin{equation}
\mathcal{A}_{\omega}(\fS,\lambda) = R \langle 
\Y_v^{\pm 1}, v \in V_\lambda \rangle / (
\Y_v 
\Y_{v'}= \omega^{P_\lambda(v,v')} 
\Y_{v'} 
\Y_v \text{ for } v,v'\in V_\lambda ).
\end{equation}
For any $v_1,\ldots,v_r \in V_\lambda$ and $a_1,\ldots,a_r \in \mathbb{Z}$,
\begin{align}
\label{Weyl_ordering-Y}
\left[ \Y_{v_1}^{a_1} \Y_{v_2}^{a_2} \cdots \Y_{v_r}^{a_r} \right] := \omega^{-\frac{1}{2}\sum_{i<j}a_i a_j P_\lambda(v_i,v_j)} \Y_{v_1}^{a_1} \Y_{v_2}^{a_2} \cdots \Y_{v_r}^{a_r}
\end{align}
In particular, for ${\bf t} = (t_v)_{v\in V_\lambda} \in \mathbb{Z}^{V_\lambda}$, define
\begin{align}\label{def-monomial-for-A}
    \Y^{\bf t} := \left[ \prod_{v\in V_\lambda} \Y_v^{t_v}\right].
\end{align}

\begin{theorem}[{\cite[Theorem~11.7]{LY23}}]\label{thm-transition-LY}
The $R$-linear maps 
\begin{align}\label{def-A-bal-isomor}
 \psi_\lambda &\colon \mathcal{A}_{\omega}(\fS,\lambda)\to
\mathcal{Z}_{\omega}(\fS,\lambda),
\quad  \Y^\mathbf{k}\mapsto Z^{\mathbf{k} K_{\lambda}}, \ ({\mathbf{k}} \in \bZ^{V_{\lambda}})
\end{align}
are $R$-algebra embeddings with $\im \psi_\lambda= \mathcal{Z}^{\rm bl}_{\omega}(\fS,\lambda)$. 
\end{theorem}

\def\A{\mathcal{A}_{\omega}(\fS,\lambda)}

\def\Ap{\mathcal{A}_{\omega}^{+}(\fS,\lambda)}

It is shown in \cite[Lemma~11.9(c)]{LY23} that
$H_\lambda K_\lambda= n I$, where $I$ is the identity matrix. Then the inverse of 
\begin{align}\label{eq-iso-A-balanced-psi}
    \psi_\lambda \colon \mathcal{A}_{\omega}(\fS,\lambda)\to
\mathcal{Z}_{\omega}^{\rm bl}(\fS,\lambda)
\end{align}
is the following
\begin{align}\label{eq-iso-A-balanced-psi-inv}
    \mathcal{Z}_{\omega}^{\rm bl}(\fS,\lambda)\to
\mathcal{A}_{\omega}(\fS,\lambda),\quad
Z^{\bf k}\mapsto A^{\frac{1}{n} {\bf k}H_\lambda}
\end{align}
for ${\bf k}\in \mathcal B_\lambda$. It follows from Lemma \ref{lem-balanced-H} that
$\frac{1}{n} {\bf k}H_\lambda\in\mathbb Z^{V_\lambda}$.

\def\barV{V}

\def\bZ{\mathbb Z}

\subsection{The $\mathcal A$-version quantum trace map}
\label{sec-traceA}


For $v=(ijk) \in  \barV_\nu\subset \barV_\lambda$ with a triangle $\nu$ of $\lambda$,  consider the graph $\widetilde{Y}_v$ defined in \S\ref{sec;A_tori}. 
By replacing a $k$-labeled edge of $\widetilde{Y}_v$ with $k$-parallel edges, we obtain a stated $n$-web $g''_v$, adjusted by a sign: 
$$\widetilde{Y}_v=
\begin{array}{c}\includegraphics[scale=0.40]{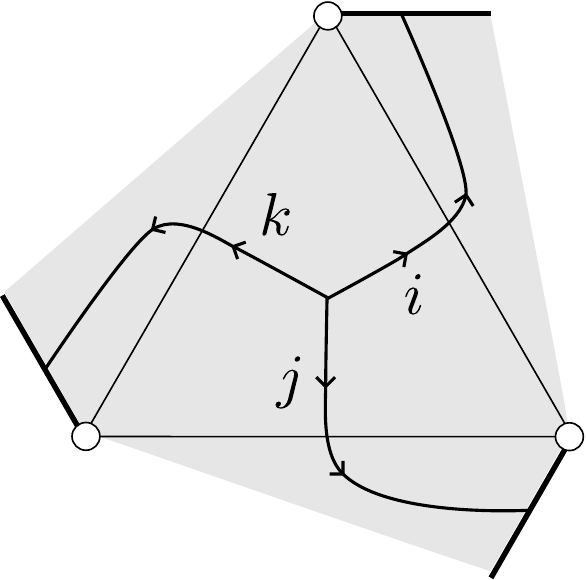}\end{array}\longrightarrow g''_v:=(-1)^{\binom{n}{2}}\begin{array}{c}\includegraphics[scale=0.40]{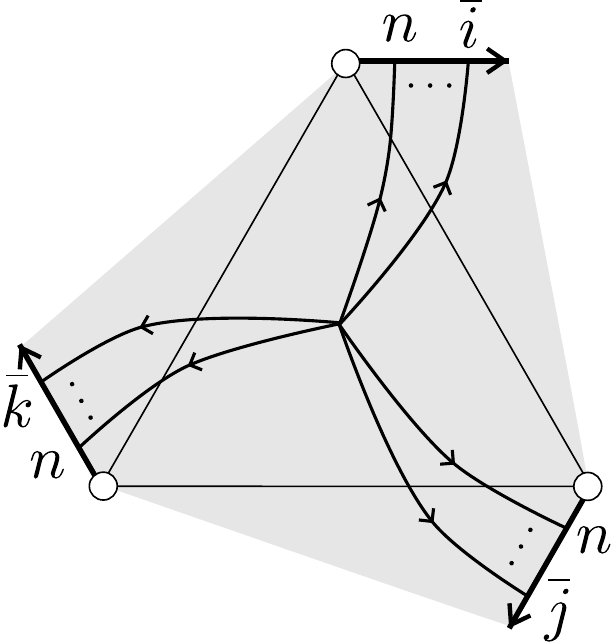}\end{array}. $$
It is known that $g''_v$ is reflection-normalizable \cite[Lemma~4.12]{LY23}.

Define $\gaa_v\in \cS_{\omega}(\fS)$ as the reflection invariant of $\gaa_v''$. 
We still use $\gaa_v$ to denote the image of $\gaa_v$ in $\overline{\cS}_{\omega}(\fS)$ by the projection $\cS_{\omega}(\fS)\to \overline{\cS}_{\omega}(\fS)$. 
The following lemma is shown in \cite{LY23}.

\begin{lemma}\label{gaa-com}
    For any $v,v'\in V_\lambda$, we have
    \begin{align}
    \gaa_v\gaa_{v'} = \omega^{P_\lambda(v,v')
    } \gaa_{v'}\gaa_v.
\end{align}
\end{lemma}

\def\trA{{\rm tr}_{\lambda}^A}
\def\SK{\overline{\cS}_\omega(\fS)}

We use $\Ap$ to denote the $R$-subalgebra of $\A$ generated by $A^{\bf k}$ for ${\bf k}\in\mathbb N^{V_\lambda}$.

\begin{definition}\label{def-M-vector}
    Suppose $M$ is a complex square matrix with rows and columns labeled by vertices in $V_\lambda$.
For each $v\in V_\lambda$, we use $M(v,*)$ to denote the row vector in $\mathbb C^{V_\lambda}$ such that its entry labeled by $u$ is $M(v,u)$ for $u\in V_\lambda$.
\end{definition}

\begin{theorem}\cite{LY23}\label{thm-trace-A}
    Let $\fS$ be a triangulable pb surface without interior punctures, and let $\lambda$ be an ideal triangulation of $\fS$. 
    There exists an algebra homomorphism 
    $$\trA\colon
    \overline{\cS}_\omega(\fS)\rightarrow\A$$
    with the following properties
    \begin{enumerate}[label={\rm (\alph*)}]\itemsep0,3em

    \item\label{thm-trace-A-a} For each $v\in V_\lambda$, we have $\trA(\gaa_v) =  A_v$.

    \item\label{thm-trace-A-b} We have $\Ap\subset\im\trA\subset\A$.

    \item If $n=2,3$, or $n>3$ and $\fS$ is a polygon, then $\trA$ is injective.

    \item\label{thm-trace-A-d} The following diagram commutes:
    \begin{align}
        \label{eq-compability-tr-A-X-diag}
        \xymatrix{
        & \rdS \ar[dl]_{{\rm tr}_{\lambda}^A} \ar[dr]^{{\rm tr}_\lambda} & \\
        \A \ar[rr]_-{\psi_\lambda} & & \mathcal{Z}_\omega(\fS,\lambda)}.
    \end{align}
    In particular, we have $\tr(\gaa_v)
    =Z^{K_\lambda(v,*)}$ for $v\in V_\lambda$.
    \end{enumerate}
\end{theorem}

\def\X{\mathcal{Z}_\omega(\fS,\lambda)}

\subsection{Projected stated ${\rm SL}_n$-skein algebra}\label{sub-sec-Kernel-trace}
Let $\fS$ be a triangulable pb surface, and let $\lambda$ be an ideal triangulation of $\fS$.  
As introduced in \S\ref{sec-traceX}, there is an algebra homomorphism  
\[
\tr \colon \rdS \to \mathcal{Z}_\omega(\fS,\lambda).
\]  
According to Theorem~\ref{thm.quantum_trace}, this homomorphism is known to be injective only in certain special cases, and it remains an open question whether $\tr$ is injective in general.  

For the purpose of introducing a quantum cluster algebra structure within the skew-field of $\rdS$, we require the injectivity of the quantum trace map. To address this, we will work with the quotient algebra $\rdS/(\ker \tr)$ in \S\ref{sec-skein-cluster} instead of $\rdS$. To justify this replacement, we will show that $\ker \tr$ is independent of the choice of triangulation $\lambda$.

\begin{lemma}\label{lem-ker-naturality}
    Let $\fS$ be a triangulable pb surface, and let $\lambda,\lambda'$ be two ideal triangulations of $\fS$. We have 
    $\ker \tr = \ker{\rm tr}_{\lambda'}$.
\end{lemma}

We postpone the proof of Lemma~\ref{lem-ker-naturality} to \S\ref{sec-proof-lem}.

\def\dS{\widetilde{\cS}_\omega(\fS)}

\begin{definition}\label{def-key-algebra}
    Let $\fS$ be a triangulable pb surface, and let $\lambda$ be an ideal triangulations of $\fS$. Define 
    $\dS:=\rdS/(\ker \tr)$, called {\bf projected stated ${\rm SL}_n$-skein algebra}.
    Lemma~\ref{lem-ker-naturality} shows that 
    $\dS$ is independent of the choice of $\lambda$.
\end{definition}

\begin{remark}
    It is conjectured in \cite{LY23} that $\dS=\rdS$.
\end{remark}

Let $\fS$ be a triangulable pb surface with a triangulation $\lambda$. 
In principle, one can cut the surface $\fS$ along any ideal arc, but here let us assume that $e$ is a non-boundary edge of $\lambda$. As in \S\ref{sub-splitting}, we denote by  
$\cut_e(\fS)$ 
the pb surface obtained from $\fS$ by cutting along edge $e$, and denote the projection map by
$$
\pr_e : \cut_e(\fS) \to \fS,
$$
so that $\pr_e^{-1}({e})$ consists of two ideal arcs $e',e''$ of $\cut_e(\fS)$. In \S\ref{sub-splitting} we studied the induced splitting homomorphism $\mathbb{S}_e$ between the skein algebras and that between the reduced skein algebras; here we will need the latter
$$
\mathbb{S}_e : \rS(\fS) \to \rS(\cut_e(\fS)).
$$
We further need to investigate a `splitting map' between the ($n$-th root)  Fock-Goncharov algebras. Note that  
$\lambda_e = (\lambda\setminus\{e\})\cup\{e',e''\}$ is a triangulation of
$\cut_e(\fS)$. We say $\lambda_e$ is induced from $\lambda$.
For a small
vertex $v$ in $e$, i.e. $v\in V_\lambda$ lying in $e$, we have $\pr^{-1}_{e}(v) = \{v',v''\}$, where $v'$ and $v''$ are small vertices of $\lambda'$. 
There is an algebra embedding
\begin{align}\label{eq-splitting-torus}
    \mathcal S_e
    \colon
    \mathcal{Z}_\omega(\fS,\lambda)\rightarrow
    \mathcal{Z}_\omega(\cut_e(\fS),\lambda_e)
\end{align}
defined on the generators $Z_v$, $v\in V_\lambda$, by
\begin{align}
\label{cutting_homomorphism_for_Z_omega}
    \mathcal S_e
    (
    Z_v)=
    \begin{cases}
        Z_v & \text{ if $v$ is not contained in $e$},\\
        [
        Z_{v'} Z_{v''}] & \text{ if $v$ is contained in $e$ and 
        $\pr^{-1}_{e}(v) = \{v',v''\}$},
    \end{cases}
\end{align}
where $[\sim]$ is the Weyl-ordered product \eqref{Weyl_ordering-Z-1}; since we are not allowing self-folded triangles in this paper, in fact $Z_{v'}$ commutes with $Z_{v''}$ in the second case, and hence $\mathcal{S}_e(Z_v) = Z_{v'}Z_{v''}$.
It is easy to show that $\mathcal{S}_e$ is injective since it sends the monomial basis of $\mathcal{Z}_\omega(\fS,\lambda)$ to a subset of the monomial basis of $\mathcal{Z}_\omega(\cut_e(\fS),\lambda_e)$.

It is well-known that 
\begin{align}\label{com-splitting-torus}
    \mathcal S_{e_1}\circ \mathcal S_{e_2}
    = \mathcal S_{e_2}\circ \mathcal S_{e_1}
\end{align}
for any $e_1,e_2\in\lambda$.

The following statement is about the compatibility of the quantum trace map ${\rm tr}_\lambda 
   $ (Theorem \ref{thm.quantum_trace}) with the splitting homomorphisms.

\begin{theorem}[\cite{LY23}]\label{thm-trace-cut}
    The following diagram commutes
    \begin{equation}\label{eq-com-tr-splitting}
\begin{tikzcd}
\rdS \arrow[r, "\mathbb S_e"]
\arrow[d, "\tr"]  
&  \overline{\cS}_{\omega}(\cut_e(\fS)) \arrow[d, "{\rm tr}_{\lambda_e}"] \\
 \mathcal{Z}_\omega(\fS,\lambda)
 \arrow[r, "\mathcal S_e"] 
&  
\mathcal{Z}_\omega(\cut_e(\fS),\lambda_e)
\end{tikzcd}
\end{equation}
    where $\mathbb S_e$ is the splitting homomorphism defined in \S\ref{sub-splitting}.
\end{theorem}

\def\bN{\mathbb N}

Let $A$ be an $R$-domain and let $S \subset A$ be a subset of non-zero elements. Let $\mathsf{Pol}(S)$ be the $R$-subalgebra of $A$ generated by $S$ and $\mathsf{LPol}(S)$ be the set of all $a \in A$ such that there is an $S$-monomial $m$ satisfying $am \in \mathsf{Pol}(S)$. If $A = \mathsf{LPol}(S)$ we say $S$ \textbf{weakly generates} $A$.

\begin{definition}
    For an $R$-domain $A$, consider a finite set $S = \{a_1,\cdots,a_r\} \subset A$. We call $S$ a \textbf{quantum torus frame} for $A$ if $S$ satisfies the following conditions; 
\begin{enumerate}
    \item $S$ is $q$-commuting and each element of $S$ is non-zero, 
    \item $S$ weakly generates $A$, 
    \item the set $\{a_1^{n_1}\cdots a_r^{n_r}\mid  n_i \in \bN\}$ is $R$-linearly independent.
\end{enumerate}
\end{definition}

The following result shows that the quantum trace map and the splitting map for $\rdS$ descend to the corresponding maps on $\rS$.

\begin{lemma}\label{lem-basic-lem}
Let $\fS$ be a triangulable pb surface.

    (a) Let $\lambda$ be a triangulation of 
    $\fS$.
    The quantum trace map 
    $\tr\colon \rdS\to \X$
    induces an injective algebra homomorphism
     $$\tr\colon \dS\to \X.$$
    When $\fS$ contains no interior punctures, the quantum trace map 
    $\trA\colon\rdS\to\A$
     induces an injective algebra homomorphism
     $$
    \trA\colon\dS\to\A.$$

    (b) We have $\dS$ is a domain.
    When $\fS$ contains no interior punctures, then $\dS$ is an Ore domain.

    (c) Let $e$ be an ideal arc of $\fS$ such that $\cut_e(\fS)$ is also triangulable (see \S\ref{sub-splitting}).
    Then the splitting map 
$$
\mathbb{S}_e \colon\overline{\mathscr{S}}_\omega(\fS) \to \overline{\mathscr{S}}_\omega(\cut_e(\fS))
$$ induces an injective algebra homomorphism
$$
\mathbb{S}_e \colon\widetilde{\mathscr{S}}_\omega(\fS) \to \widetilde{\mathscr{S}}_\omega(\cut_e(\fS))
$$

(d) The set \(\{\gaa_v\mid v\in V_\lambda\}\) is a quantum torus frame of $\dS$.
\end{lemma}
\begin{proof}
    (a) This is trivial from the definition of $\dS$ and Theorem~\ref{thm-trace-A}(d).

    (b) 
    It follows from (a) that $\dS$ is a subalgebra of $\X$, which is a domain \cite{Cohn}. Then $\dS$ is a domain.
    
    From (a) and Theorem~\ref{thm-trace-A}(b), we can regard $\dS$ as an $R$-subalgebra of $\A$ with
    \begin{align}\label{sandwitch-projected}
        \Ap\subset\dS\subset\A
    \end{align}
    when $\fS$ contains no interior punctures. 
    Then \cite[Proposition~2.2]{LY22} shows that
    $\dS$ is an Ore domain.

    (c) Diagram~\eqref{eq-com-tr-splitting} implies that $\mathbb S_e(\ker\tr)
    \subset \ker{\rm tr}_{\lambda_e}$. Then the splitting map 
$
\mathbb{S}_e \colon\overline{\mathscr{S}}_\omega(\fS) \to \overline{\mathscr{S}}_\omega(\cut_e(\fS))
$ induces an algebra homomorphism
$$
\mathbb{S}_e \colon\widetilde{\mathscr{S}}_\omega(\fS) \to \widetilde{\mathscr{S}}_\omega(\cut_e(\fS)).
$$
Diagram \eqref{eq-com-tr-splitting} shows the following diagram commutes 
 \begin{equation*}
\begin{tikzcd}
\dS \arrow[r, "\mathbb S_e"]
\arrow[d, "\tr"]  
&  \widetilde{\cS}_{\omega}(\cut_e(\fS)) \arrow[d, "{\rm tr}_{\lambda_e}"] \\
 \mathcal{Z}_\omega(\fS,\lambda)
 \arrow[r, "\mathcal S_e"] 
&  
\mathcal{Z}_\omega(\cut_e(\fS),\lambda_e)
\end{tikzcd}
\end{equation*}
Since all $\tr$, ${\rm tr}_{\lambda_e}$, and 
$\mathcal S_e$  are injective, then so is 
$
\mathbb{S}_e \colon\widetilde{\mathscr{S}}_\omega(\fS) \to \widetilde{\mathscr{S}}_\omega(\cut_e(\fS)).
$

(d) In the proof of (b), we showed 
$\Ap\subset\dS\subset\A$ with $\gaa_v= A_v$ for $v\in V_\lambda$. This completes the proof. 
\end{proof}

\section{Quantum cluster algebras}\label{sec-quantum-clu-al}

We recall basics of quantum cluster algebras, obtained by gluing quantum torus algebras associated to cluster seeds along quantum mutations. We refer the readers to \cite{BZ,FG09a}. Then we extend the known quantum mutations to balanced subalgebras for $n$-th root variables. Such an extension was studied in \cite{Hiatt,BW11} for $n=2$, and in \cite{Kim21} for $n=3$.

\subsection{Classical cluster $\mathcal A$- and $\mathcal X$-mutations}\label{sec-mutation-classical}

Fix a non-empty set $\mathcal{V}$, and a non-empty subset $\mathcal{V}_{\rm mut}$ of $\mathcal{V}$. The elements of $\mathcal{V}$, $\mathcal{V}_{\rm mut}$ and $\mathcal{V}\setminus \mathcal{V}_{\rm mut}$ are called {\bf vertices}, {\bf mutable vertices} and {\bf frozen vertices}, respectively. By a quiver $\Gamma$ we mean a directed graph whose set of vertices is $\mathcal{V}$ 
and equipped with weight on the edges so that the weight on each edge is $1$ unless the edge connects two frozen vertices, in which case the weight is $\frac{1}{2}$. An edge of weight $1$ will be called an {\bf arrow}, and an edge of weight $\frac{1}{2}$ a {\bf half-arrow}. As before, we denote by $Q = (Q(u,v))_{u,v\in \mathcal{V}}$ to denote the signed adjacency matrix of $\Gamma$. A quiver is considered up to equivalence, where two quivers are equivalent if they yield the same signed adjacency matrix. In particular, a quiver can be represented by a representative quiver without an oriented cycle of length 1 or 2. Let $\mathcal{F}$ be the field of rational functions over $\mathbb{Q}$ with the set of algebraically independent generators enumerated by $\mathcal{V}$, say $\mathcal{F} = \mathbb{Q}(\{y_v\}_{v\in \mathcal{V}})$. 

A {\bf cluster $\mathcal{X}$-seed} (resp. {\bf cluster $\mathcal{A}$-seed}) is a pair $\mathcal{D}_X = (\Gamma,(X_v)_{v\in \mathcal{V}})$ (resp. $\mathcal{D}_A= (\Gamma,(A_v)_{v\in \mathcal{V}})$), where $\Gamma$ is a quiver and $\{X_v\}_{v\in \mathcal{V}}$ (resp. $\{A_v\}_{v\in \mathcal{V}}$) forms an algebraically independent generating set of $\mathcal{F}$ over $\mathbb{Q}$. The signed adjacency matrix $Q$ of $\Gamma$ is called the {\bf exchange matrix} of the seed $\mathcal{D}_X$ (resp. $\mathcal{D}_A$).

\def\sgn{\text{sgn}}

Suppose that $k\in\mathcal V_{
\rm mut}$. The {\bf mutation}
$\mu_{k,X}$ at the 
mutable vertex $k\in\mathcal{V}$ is defined to be the process that transforms an $\mathcal X$-seed $\mathcal{D}_X= (\Gamma,(X_v)_{v\in \mathcal{V}})$ into a new seed
$\mu_{k,X}(\mathcal D_X) = \mathcal D_X' = 
(\Gamma', (
X_v')_{v\in 
\mathcal{V}})
$.
Here 
$$
X_v'= \begin{cases}
    X_v^{-1} & v=k,\\
    X_v (1+ 
    X_k^{-\text{sgn}(Q(v,k)) })^{-Q(v,k)} & v\neq k,
\end{cases}
$$
where 
$$
\sgn(a)=
\begin{cases}
    1 & a>0,\\
    0 & a=0,\\
    -1 & a<0,
\end{cases}
$$
for $a\in\mathbb R$, and $\Gamma'$ is obtained from 
$\Gamma$ by the following procedures:
\begin{enumerate}
    \item for each pair of arrows $i\rightarrow k$ and $k\rightarrow j$ draw an arrow $i\rightarrow j$,
    
    \item reverse all the arrows incident to the vertex $k$,

    \item delete a maximal pairs of arrows $i\rightarrow j$ and $i\rightarrow j$ going in the opposite directions.
\end{enumerate}
We use $Q' = (Q'(u,v))_{u,v\in\mathcal V}$ to denote the adjacency matrix of $\Gamma'$. Then we have 
\begin{align}\label{eq-mutation-Q}
Q'(u,v) = \begin{cases}
    - Q(u,v) & k\in\{u,v\},\\
    Q(u,v) +\frac{1}{2}\Big(Q(u,k)|Q(k,v)|+
    |Q(u,k)|Q(k,v)\Big) & k\notin\{u,v\}.
\end{cases}
\end{align}
We use $\mu_k(Q)$ (resp. $\mu_k(\Gamma)$)
to denote $Q'$ (resp. $\Gamma'$).

The {\bf mutation}
$\mu_{k,A}$ at the 
mutable vertex $k\in\mathcal{V}$ is defined to be the process that transforms an $\mathcal A$-seed $\mathcal{D}_A= (\Gamma,(A_v)_{v\in \mathcal{V}})$ into a new seed
$\mu_{k,A}(\mathcal D_A) = \mathcal D_A' = 
(\Gamma', (
A_v')_{v\in 
\mathcal{V}})
$.
Here 
$$
A_v'= \begin{cases}
    A^{-1}_k(\prod_{j\in\mathcal V} A_j^{[Q(j,k)]_+} + \prod_{j\in\mathcal V} A_j^{[-Q(j,k)]_+}) & v=k,\\
    A_v & v\neq k,
\end{cases}
$$
where $[a]_+$ stands for the positive part of a real number $a$
      $$ [a]_+ := \max\{a,0\} = {\textstyle \frac{1}{2}(a+|a|)},$$
and $\Gamma'$ is obtained from $\Gamma$ by the same procedures introduced for $\mu_{k,X}$.

It is well-known that $\mu_{k,X}(\mu_{k,X}(\mathcal D_X)) = \mathcal D_X$ and 
$\mu_{k,A}(\mu_{k,A}(\mathcal D_A)) = \mathcal D_A$.
\def\Fr{{\rm Frac}}

Two $*$-seeds in $\mathcal F$ are said to be
mutation-equivalent if they are transformed to each other by a finite sequence of seed
mutations. An equivalence class of $*$-seeds is called a $*$-mutation class.
Here $*$ is $\mathcal X$ or $\mathcal A$.

\subsection{Quantum $\mathcal A$-mutations and quantum cluster algebras}\label{sec-mutation-quantum}
Let $\mathcal F$ be a skew-field over $R=\mathbb Z[\omega^{\pm\frac{1}{2}}]$. Recall that $\xi=\omega^n$.
\begin{definition}[\cite{BZ}]\label{def-quantum-seed}
    A {\bf quantum seed} is a triple $(Q,\Pi,M)$, where 
\begin{enumerate}
    \item $Q=(Q(u,v))_{u,v\in\mathcal V}$ is an exchange matrix;

    \item $\Pi=(\Pi(u,v))_{u,v\in\mathcal V}$ is a anti-symmetric matrix with integral entries satisfying the compatibility relation
    $$\sum_{k\in\mathcal V}Q(k,u) \Pi(k,v)=\delta_{u,v} d_u$$
    for all $u\in\mathcal V_{\text{mut}}$
    and $v\in\mathcal V$, where $d_u$ is a positive integer for $u\in\mathcal V_{\text{mut}}$;

    \item $M\colon \mathbb Z^{\mathcal V}\rightarrow\mathcal F\setminus\{0\}$ is a function such that 
    \begin{align}\label{eq-M-relation}
        M({\bf k}) M({\bf t}) = \xi^{\frac{1}{2} {\bf k} \Pi {\bf t}^T} M({\bf k} + {\bf t})
    \end{align}
    for row vectors ${\bf k},{\bf t}\in \mathbb Z^{\mathcal V}$, and $M(\mathbb Z^{\mathcal V})$ is a basis of a quantum torus over $R$ whose skew-field is $\mathcal F$.
\end{enumerate}
\end{definition}

For each $i\in\mathcal V$, define ${\bf e}_i\in\mathbb Z^{\mathcal V}$ with
\begin{align}\label{eq-def-vector-ei-ele}
    {\bf e}_i(v)=\delta_{i,v} \text{ for $v\in\mathcal V$}.
\end{align}
Then the function $M$ is uniquely determined by the values $A_i:=M({\bf e}_i)$, which we call {\bf (quantum) cluster variables}.
A (quantum) cluster variable $A_i$ is called a {\bf frozen (quantum) cluster variable} if $i\in\mathcal V\setminus\mathcal V_{\text{mut}}$, and an {\bf exchangeable (quantum) cluster variable} otherwise.

We label the vertex set $\mathcal V$ as 
$\{v_1,\cdots,v_r\}$.
For any ${\bf k}=(k_v)_{v\in\mathcal V}\in \mathbb Z^{\mathcal V}$, we have 
\begin{align}\label{eq-A-power}
    M({\bf k}) = A^{\bf k}:= \left[\prod_{v\in \mathcal V}A_v^{k_v}\right],
\end{align}
where $\left[\prod_{v\in \mathcal V}A_v^{k_v}\right]$ is the Weyl-ordered product 
\begin{align}\label{Weyl-A}
    \left[\prod_{v\in \mathcal V}A_v^{k_v}\right]= \xi^{-\frac{1}{2}\sum_{i<j} k_{v_i} k_{v_j}\Pi(i,j)} A_{v_1}^{k_{v_1}} A_{v_2}^{k_{v_2}} \cdots A_{v_r}^{k_{v_r}}.
\end{align}
Note that the Weyl-ordered product in \eqref{Weyl-A} is independent of how we label the vertex set $\mathcal V$.

Note that the anti-symmetric matrix $\Pi$ is uniquely determined by the function $M$ because of \eqref{eq-M-relation}.

Given a quantum seed $(Q,\Pi,M)$ in $\mathcal F$ and 
$k\in\mathcal V_{\text{mut}}$, the quantum $\mathcal A$-mutation produces a new quantum seed $(Q',\Pi',M')=\mu_{k,A}(Q,\Pi,M)$ according to the following rules:
\begin{itemize}
    \item $Q'=\mu_k(Q)$;

    \item the mutation on the quantum cluster variables is
    \begin{equation*}
        A_i':=M'({\bf e}_i)=
        \begin{cases}
            M({\bf e}_i) & i\neq k\\
            M(-{\bf e}_k + \sum_{j\in\mathcal V}[Q(j,k)]_{+}{\bf e}_j)
            + M(-{\bf e}_k + \sum_{j\in\mathcal V}[-Q(j,k)]_{+}{\bf e}_j)
            & i=k
        \end{cases}
    \end{equation*}

    \item $\Pi'=(\Pi'(i,j))_{i,j\in\mathcal V}$, where, for each pair $i,j\in\mathcal V$, the entry $\Pi'(i,j)$ is the unique integer such that 
    $$M'({\bf e}_i)M'({\bf e}_j)
    =\xi^{\Pi'(i,j)} M'({\bf e}_j)M'({\bf e}_i).$$
\end{itemize}
There is a formula of $\Pi'(i,j)$ expressed in entries of $\Pi$ \cite[Equation~(3.4)]{BZ}
\begin{align}\label{eq-for-P'}
    \Pi'(i,j)=
    \begin{cases}
        \Pi(i,j) & k\notin\{i,j\},\\
        -\Pi(k,j) + \sum_{v\in\mathcal V} [Q(v,k)]_{+} \Pi(v,j) & j\neq i=k,\\
        -\Pi(i,k) + \sum_{v\in\mathcal V} [Q(v,k)]_{+} \Pi(i,v) & i\neq j=k,\\
        \Pi(k,k)=0 & i=j=k.
    \end{cases}
\end{align}
It is well-known that $\mu_{k,A}(\mu_{k,A}(Q,\Pi,M)) = (Q,\Pi,M)$ \cite[Proposition~4.10]{BZ}.

\begin{definition}\label{def-quantum-class}
    Two quantum seeds in $\mathcal F$ are said to be
mutation-equivalent if they are transformed to each other by a finite sequence of quantum 
$\mathcal A$-mutations. An equivalence class of quantum seeds is called a quantum 
$\mathcal A$-mutation class.
\end{definition}


The relations among the quantum seeds in a given quantum $\mathcal A$-mutation class $\mathsf{S}$ can be encoded in the 
exchange graphs.

\def\Exc{\mathsf{Exch}_\mathsf{S}}

\begin{definition}\cite[page 37]{ishibashi2023skein}\label{def-exch-graph}
    The exchange graph of the quantum $\mathcal A$-mutation class $\mathsf{S}$ is a graph $\mathsf{Exch}_\mathsf{S}$ with vertices $w$ corresponding
to the quantum seeds $s^{w}$
in $\mathsf{S}$, together with edges of the forms
$s^{w} \wideoverunder{\mu_{k,A}}{} s^{w'}$, where $\mu_{k,A}(s^{w})= s^{w'}$ and 
$k\in\mathcal V_{\text{mut}}$.
\end{definition}

Suppose two exchange graphs $\mathsf{Exch}_{\mathsf{S}}$ and 
$\mathsf{Exch}_{\mathsf{S}'}$ have one vertex in common, then it is straightforward to verify that 
$\mathsf{Exch}_{\mathsf{S}}=\mathsf{Exch}_{\mathsf{S}'}$ (or $\mathsf{S}= \mathsf{S}'$).

For any vertex $w$ in $\Exc$ and the quantum seed $s^w = (Q^{w}, \Pi^w, M^w)$ corresponding to $w$, define 
\begin{align}
\textbf{A}^{+}(w)&:=\{M^w({\bf k})\mid {\bf k}\in\mathbb Z_{\geq 0}^{\mathcal V}\}\subset
\textbf{A}(w):=\{M^w({\bf k})\mid {\bf k}\in\mathbb Z^{\mathcal V}\}\subset \mathcal F
\label{eq-A-w}
\\
    \mathbb T(w)&:= \text{span}_R(\textbf{A}(w))\subset \mathcal F.
    \label{eq-T-w}
\end{align}
We call $\mathbb T(w)$ the {\bf based quantum torus} associated to the quantum seed $s^w$.
We use ${\rm Frac}(\mathbb T(w))$ to denote 
 the skew-field of fractions of $\mathbb T(w)$; for the existence of ${\rm Frac}(\mathbb T(w))$, see \cite{Cohn}.

Let $s^w = (Q^{w}, \Pi^w, M^w)$ and $s^{w'} = (Q^{w'}, \Pi^{w'}, M^{w'})$ be two quantum seeds in $\mathsf S$ such that 
$\mu_{k,A}(s^w) = s^{w'}$, where $k\in\mathcal V_{\text{mut}}$. 
For each $i\in\mathcal V$, set 
$$A_i=M^w({\bf e}_i)\text{ and }
A_i'=M^{w'}({\bf e}_i).$$
Then the mutation $\mu_{k,A}$ can be regarded as a skew-field isomorphism 
\begin{align}\label{eq-mut-iso}
    \mu_{k,A}\colon {\rm Frac}(\mathbb T(w'))\rightarrow {\rm Frac}(\mathbb T(w))
\end{align}
defined by 
\begin{equation*}
        \mu_{k,A}(A_i')=
        \begin{cases}
            A_i & i\neq k,\\
            A^{-{\bf e}_k + \sum_{j\in\mathcal V}[Q(j,k)]_{+}{\bf e}_j}
            + A^{-{\bf e}_k + \sum_{j\in\mathcal V}[-Q(j,k)]_{+}{\bf e}_j}
            & i=k,
        \end{cases}
    \end{equation*}
    for $i\in\mathcal V$ (see \eqref{eq-A-power} and \eqref{Weyl-A} for  $A^{*}$).
Note that 
\begin{align}\label{eq-iso-mutation-eqal}
    \mu_{k,A}(A_i') = M^{w'}({\bf e}_i)
\end{align}
for $i\in\mathcal V$.

\begin{definition}\label{def-quan-cluster-algebra}
    The {\bf quantum cluster algebra} associated with a quantum $\mathcal A$-mutation class $\mathsf S$ of
quantum seeds is the $R$-subalgebra $\mathscr{A}_{\mathsf S}\subset \mathcal F$ generated by $\bigcup_{w\in\Exc}{\bf A}^{+}(w)$  (see \eqref{eq-A-w}) and the inverses of frozen quantum
cluster variables. Each exchangeable quantum cluster variable of any $s^w\in \mathsf S$ is called an {\bf exchangeable quantum cluster variable} of $\mathscr{A}_{\mathsf S}$.

The {\bf quantum upper
cluster algebra} is defined to be
$\mathscr{U}_{\mathsf S}:=\bigcap _{w\in\Exc}
\mathbb T(w)\subset \mathcal F$ (see \eqref{eq-T-w}).

For each vertex $w\in\Exc$, the {\bf upper bound} at $w$ is defined to be 
$$\mathscr{U}_{\mathsf S}(w):= \mathbb T(w)\cap \bigcap _{w'}
\mathbb T(w')\subset \mathcal F,$$ where 
$w'\in \Exc$ runs over the vertices adjacent to $w$. 
\end{definition}

The following two theorems describe the relations among the three definitions above.

\begin{theorem}\cite[Corollary~5.2]{BZ}\label{thm-inclusion-quantum}
    We have the inclusion
    $$\mathscr{A}_{\mathsf S}\subset \mathscr{U}_{\mathsf S}.$$
\end{theorem}

\begin{theorem}\cite[Theorem~5.1]{BZ}\label{thm-upper-w-upper}
    For any $w\in\Exc$, we have 
    $$\mathscr{U}_{\mathsf S}(w)=\mathscr{U}_{\mathsf S}.$$
\end{theorem}

\begin{remark}
    Let $R'$ be another commutative domain with an invertible element $(\omega')^{\frac{1}{2}}$. Then there is an algebra homomorphism from $R=\mathbb Z[\omega^{\pm\frac{1}{2}}]$ to $R'$ sending $w^{\frac{1}{2}}$ to  $(w')^{\frac{1}{2}}$. This algebra homomorphism induces an $R$-algebra structure for $R'$. 
    Then $\mathcal F\otimes_R R'$ is a skew-field over $R'$. 
    Actually, all arguments in this paper about the quantum cluster work for general ground ring $R'$ due to the functor ``$-\otimes_R R'$".
    In particular, when $R'=\mathbb Q$ and 
    $(\omega')^{\frac{1}{2}}=1$, 
    the quantum $\mathcal A$-mutation introduced in this subsection is equivalent to the classical cluster $\mathcal A$-mutation introduced in \S\ref{sec-mutation-classical}.
\end{remark}

\subsection{A decomposition of the quantum $\mathcal A$-mutation}\label{sec-A-mutation-decom}
In \cite{FG09b}, the authors introduced a decomposition for the classical cluster $\mathcal A$-mutation. They divided the classical cluster $\mathcal A$-mutation into two steps. In this subsection, we will divide the quantum $\mathcal A$-mutation into two steps, which is equivalent to decompose the isomorphism in \eqref{eq-mut-iso} into the combination of two isomorphisms. 
In \S\ref{sec-compatibility}, we will use this decomposition to prove the compatibility between the quantum $\mathcal A$-mutation and the quantum $\mathcal X$-mutation introduced in \S\ref{subsec-mutation-Fock}.

Let $s^w = (Q, \Pi, M)$ and $s^{w'} = (Q', \Pi', M')$ be two quantum seeds in a quantum 
$\mathcal A$-mutation class $\mathsf S$ such that 
$\mu_{k,A}(s^w) = s^{w'}$, where $k\in\mathcal V_{\text{mut}}$. 
For each $i\in\mathcal V$, set 
$A_i=M({\bf e}_i)\text{ and }
A_i'=M'({\bf e}_i).$
Define
\begin{align}\label{eq-def-mu-step-1}
    \mu_{k,A}'\colon {\rm Frac}(\mathbb T(w'))\rightarrow {\rm Frac}(\mathbb T(w))
\end{align}
such that
\begin{equation}\label{eq-def-mu-step1}
        \mu_{k,A}'(A_i')=
        \begin{cases}
            A_i & i\neq k,\\
            A^{-{\bf e}_k + \sum_{j\in\mathcal V}[Q(j,k)]_{+}{\bf e}_j}
            & i=k,
        \end{cases}
    \end{equation}
    for $i\in\mathcal V$ (see \eqref{eq-A-power} and \eqref{Weyl-A} for  $A^{*}$).

\begin{lemma}
    $\mu_{k,A}'$ is a well-defined isomorphism between skew-fields.
\end{lemma}
\begin{proof}
    For any $i,j\in\mathcal V$, we have 
    $A_i'A_j' = \xi^{\Pi'(i,j)} A_j' A_i'\in {\rm Frac}(\mathbb T(w'))$.
    To show the well-definedness of $\mu_{k,A}'$, it suffices to show that 
    \begin{align*}
        \mu_{k,A}'(A_i')\mu_{k,A}'(A_j') = \xi^{\Pi'(i,j)}\mu_{k,A}'( A_j') \mu_{k,A}'(A_i')\in {\rm Frac}(\mathbb T(w)).
    \end{align*}

    {\bf Case 1}: When $i=j=k$ or $k\notin\{i,j\}$, this is trivial from Equations~\eqref{eq-for-P'} and \eqref{eq-def-mu-step1}.

    {\bf Case 2}: When $j\neq i=k$, we have 
    \begin{align*}
        \mu_{k,A}'(A_i')\mu_{k,A}'(A_j')
        =& A^{-{\bf e}_k + \sum_{v\in\mathcal V}[Q(v,k)]_{+}{\bf e}_v} A_j\\
        =& \xi^{(-{\bf e}_k + \sum_{v\in\mathcal V}[Q(v,k)]_{+}{\bf e}_v) \Pi {\bf e}_j^T} A_j A^{-{\bf e}_k + \sum_{v\in\mathcal V}[Q(v,k)]_{+}{\bf e}_v}\\
        =& \xi^{-\Pi(k,j) + \sum_{v\in\mathcal V} [Q(v,k)]_{+} \Pi(v,j)}
        A_j A^{-{\bf e}_k + \sum_{v\in\mathcal V}[Q(v,k)]_{+}{\bf e}_v}\\
        =& \xi^{\Pi'(i,j)}\mu_{k,A}'( A_j') \mu_{k,A}'(A_i').
    \end{align*}

    {\bf Case 3}: When $i\neq j=k$, the proof is similar with the proof of case 2.

    It is a trivial check that the follow map 
    \begin{align}
    {\rm Frac}(\mathbb T(w))&\rightarrow {\rm Frac}(\mathbb T_{w'}) \nonumber
    \\
    A_i&\mapsto
        \begin{cases}
            A_i' & i\neq k, \\
            (A')^{-{\bf e}_k + \sum_{j\in\mathcal V}[-Q'(j,k)]_{+}{\bf e}_j}
            & i=k,
        \end{cases}
        \label{eq-inverse}
\end{align}
is the inverse of $\mu_{k,A}'$.
\end{proof}

\begin{remark}
    Note that we  use $[Q'(k,j)]_{+}$ instead of $[Q'(j,k)]_{+}$ in \eqref{eq-inverse}.
    We can directly check that the map defined in \eqref{eq-inverse} is a well-defined homomorphism between skew-fields using the following equations:
     $$
    \begin{cases}
       \sum_{v\in\mathcal V} [Q'(v,k)]_{+} \Pi'(v,j) =\sum_{v\in\mathcal V} [Q'(k,v)]_{+} \Pi'(v,j) & j\neq i=k,\\
       \sum_{v\in\mathcal V} [Q'(v,k)]_{+} \Pi'(i,v) = \sum_{v\in\mathcal V} [Q'(k,v)]_{+} \Pi'(i,v) & i\neq j=k.
    \end{cases}
    $$    
\end{remark}

\def\ZV{\mathbb Z^{\mathcal V}}

 Define $Q(*,k)\in\ZV$ to be the $k$-th column of $Q$, i.e., the $j$-th entry is $Q(j,k)$ for $j\in\mathcal V$.
Note that 
\begin{align}\label{eq-Q-vector-sum}
    Q(*,k) =\sum_{j\in\mathcal V}[Q(j,k)]_{+}{\bf e}_j - \sum_{j\in\mathcal V}[-Q(j,k)]_{+}{\bf e}_j.
\end{align}
Define
\begin{align}\label{def-mu-sharp}
    \mu_{k,A}^{\sharp}\colon {\rm Frac}(\mathbb T(w))\rightarrow {\rm Frac}(\mathbb T(w))
\end{align}
such that
\begin{equation}\label{eq-def-mu-step2}
        \mu_{k,A}^{\sharp}(A_i)=
        \begin{cases}
            A_i & i\neq k,
            \\
            A_k\left(1 + \xi^{-d_k/2} A^{-Q(*,k)}\right)^{-1}
            & i=k.
        \end{cases}
    \end{equation}

\begin{lemma}\label{lem-decom-A}
    \begin{enumerate}[label={\rm (\alph*)}]\itemsep0,3em

    \item The map $\mu_{k,A}^\sharp$ in \eqref{def-mu-sharp} is a well-defined isomorphism between skew-fields.

    \item\label{lem-decom-A-b} We have $\mu_{k,A} = \mu_{k,A}^{\sharp}\circ \mu_{k,A}'.$

    \end{enumerate}
\end{lemma}
\begin{proof}
    First we show that
    \begin{align}\label{eq-com-i-neq-k}
        \text{$A_i A^{Q(k,*)} = A^{Q(k,*)} A_i\in {\rm Frac}(\mathbb T(w))$ when $i\neq k.$}
    \end{align}
     It is equivalent to show 
    $Q(k,*)\Pi{\bf e}_i^T=0$. We have 
    $$Q(k,*)\Pi{\bf e}_i^T=\sum_{v\in\mathcal V} Q(k,v) \Pi(v,i)=-\sum_{v\in\mathcal V} Q(v,k) \Pi(v,i)=-\delta_{i,k}d_k=0.$$

        To show the well-definedness of $\mu_{k,A}^{\sharp}$, it suffices to show that 
    \begin{align*}
        \mu_{k,A}^{\sharp}(A_i)\mu_{k,A}^{\sharp}(A_j) = \xi^{\Pi(i,j)}\mu_{k,A}^{\sharp}( A_j) \mu_{k,A}^{\sharp}(A_i)\in {\rm Frac}(\mathbb T(w)).
    \end{align*}

    {\bf Case 1}: When $i=j=k$ or $k\notin\{i,j\}$, this is trivial from equation 
    \eqref{eq-def-mu-step2}.

    {\bf Case 2}: When $j\neq i=k$, we have 
    \begin{align*}
        \mu_{k,A}^{\sharp}(A_i)\mu_{k,A}^{\sharp}(A_j)
        =& A_k\left(1 + \xi^{-d_k/2} A^{Q(k,*)}\right)^{-1} A_j\\
        =& A_k A_j\left(1 + \xi^{-d_k/2} A^{Q(k,*)}\right)^{-1}\\
        =& \xi^{\Pi(k,j)}
        A_j A_k\left(1 + \xi^{-d_k/2} A^{Q(k,*)}\right)^{-1}\\
        =&\xi^{\Pi(i,j)} \mu_{k,A}^{\sharp}(A_j)\mu_{k,A}^{\sharp}(A_i).
    \end{align*}

    {\bf Case 3}: When $i\neq j=k$, the proof is similar with the proof of case 2. 

    Then we will show $\mu_{k,A} = \mu_{k,A}'\circ \mu_{k,A}^\sharp.$ It suffices to show that 
    $\mu_{k,A} (A_i')= \mu_{k,A}^{\sharp}(\mu_{k,A}'(A_i'))$ for $i\in\mathcal V.$ This is trivial when $i\neq k$.
    We know that there exists an integer $m$ such that 
    $$A_k A^{\sum_{j\in\mathcal V}[Q(j,k)]_{+}{\bf e}_j} = \xi^m A^{\sum_{j\in\mathcal V}[Q(j,k)]_{+}{\bf e}_j} A_k.$$
    Then we have 
    \begin{align}\label{eq-mu-k-A-k}
       \mu_{k,A}'(A_k') =A^{-{\bf e}_k+\sum_{j\in\mathcal V}[Q(j,k)]_{+}{\bf e}_j}= \xi^{-\frac{m}{2}} A^{\sum_{j\in\mathcal V}[Q(j,k)]_{+}{\bf e}_j} A_k^{-1}. 
    \end{align}
    It follows that 
    \begin{equation}\label{eq-mu-sharp-mu'}
    \begin{split}
        \mu_{k,A}^{\sharp}(\mu_{k,A}'(A_k'))
        =& \xi^{-\frac{m}{2}} A^{\sum_{j\in\mathcal V}[Q(j,k)]_{+}{\bf e}_j}\left(1 + \xi^{-d_k/2} A^{Q(k,*)}\right) A_k^{-1}\\
        =& \xi^{-\frac{m}{2}} A^{\sum_{j\in\mathcal V}[Q(j,k)]_{+}{\bf e}_j} A_k^{-1} +
        \xi^{-\frac{m}{2}} \xi^{-d_k/2} A^{\sum_{j\in\mathcal V}[Q(j,k)]_{+}{\bf e}_j} A^{Q(k,*)} A_k^{-1}.
        \end{split}
    \end{equation}
    Equation \eqref{eq-com-i-neq-k} shows that $A^{\sum_{j\in\mathcal V}[Q(j,k)]_{+}{\bf e}_j} A^{Q(k,*)}= A^{Q(k,*)} A^{\sum_{j\in\mathcal V}[Q(j,k)]_{+}{\bf e}_j}$.
    We have 
    $$A^{Q(k,*)} A_k^{-1} = \xi^{-\sum_{i\in\mathcal V}Q(k,i)\Pi(i,k)}
    A_k^{-1} A^{Q(k,*)}
    =\xi^{d_k}
    A_k^{-1} A^{Q(k,*)}.$$
    This implies that 
    \begin{equation}\label{eq-mu-sharp-mu'1}
    \begin{split}
        \xi^{-\frac{m}{2}} \xi^{-d_k/2} A^{\sum_{j\in\mathcal V}[Q(j,k)]_{+}{\bf e}_j} A^{Q(k,*)} A_k^{-1}
        &= A^{-{\bf e}_k + Q(k,*) + \sum_{j\in\mathcal V}[Q(j,k)]_{+}{\bf e}_j}\\& = A^{-{\bf e}_k+\sum_{j\in\mathcal V}[Q(k,j)]_{+}{\bf e}_j},
        \end{split}
    \end{equation}
    where the last equality comes from \eqref{eq-Q-vector-sum}.
    Equations \eqref{eq-mu-k-A-k}, \eqref{eq-mu-sharp-mu'}, and \eqref{eq-mu-sharp-mu'1} imply that 
    $$\mu_{k,A}^{\sharp}(\mu_{k,A}'(A_k'))
    =A^{-{\bf e}_k+\sum_{j\in\mathcal V}[Q(j,k)]_{+}{\bf e}_j}+ A^{-{\bf e}_k+\sum_{j\in\mathcal V}[Q(k,j)]_{+}{\bf e}_j}=\mu_{k,A}(A_k').$$

\end{proof}

\subsection{Quantum $\mathcal X$-mutations for Fock-Goncharov algebras}

Suppose that 
$\mathcal{D}_X= (\Gamma,(X_v)_{v\in \mathcal{V}})$ is a cluster $\mathcal X$-seed 
whose 
exchange matrix 
is $Q$.
Define the {\bf Fock-Goncharov algebra} associated to $\mathcal D_X$ to be the quantum torus algebra
\begin{align*}
\mathcal{X}_q(\mathcal{D}_X)  = \bT_q(Q)  = R\langle X_v^{\pm 1}, v\in \mathcal{V} \rangle / (X_v X_{v'} = q^{2Q(v,v')} X_{v'} X_v \mbox{ for } v,v' \in \mathcal{V}).
\end{align*}
We understand $\mathcal{X}_q(\mathcal{D}_X)$ as the quantum algebra associated to the seed $\mathcal{D}_X$, which yields the classical algebra associated to $\mathcal{D}_X$ as $q\to 1$. In this subsection we recall the quantum version of the mutation
formula.

The quantum mutation shall be a non-commutative rational map. So we need to use the skew-field of fractions of $\mathcal{X}_q(\mathcal{D})$, which we denote by ${\rm Frac}(\mathcal{X}_q(\mathcal{D}_X))$; for the existence of ${\rm Frac}(\mathcal{X}_q(\mathcal{D}_X))$, see \cite{Cohn}.

The following special function is a crucial ingredient.
\begin{definition}[compact quantum dilogarithm \cite{F95,FK94}]
    The quantum dilogarithm for a quantum parameter
$q$ is the function
$$\Psi^q(x) = \prod_{r=1}^{+\infty} (1+ q^{2r-1} x)^{-1}.$$
\end{definition}
For our purposes, the infinite product can be understood formally, as what matters for us is only the following:

\def\sgn{{\rm sgn}}

\begin{lemma}\label{lem-F-P}
    Suppose that $xy = q^{2m} yx$ for some $m\in \mathbb{Z}$. Then we have 
    $${\rm Ad}_{\Psi^q(x)}(y) = y F^q(x,m),$$
    where
    \begin{align}\label{eq-def-Fq}
        F^q(x,m)=\prod_{r=1}^{|m|} (1+ q^{(2r-1){\rm sgn}(m)}x)^{\sgn(m)}.
    \end{align}
\end{lemma}

It is easy to check that ${\rm Ad}_{\Psi^q(X_k)}$ extends to a unique skew-field isomorphism from ${\rm Frac}(\mathcal{X}_q(\mathcal{D}_X))$ to itself, which we denote by ${\rm Ad}_{\Psi^q(X_k)}$.

\begin{definition}[quantum $\mathcal X$-mutation for Fock-Goncharov algebras \cite{BZ,FG09a,FG09b}]\label{def.quantum_X-mutation}
     Suppose 
     that $\mathcal{D}_X = (\Gamma,(X_v)_{v\in \mathcal{V}})$ is an $\mathcal X$-seed, 
     $k$$\in \mathcal{V}_{\rm mut}$ is a 
     mutable vertex of $\Gamma$, and $\mathcal D_X'=\mu_k(\mathcal D_X)$.
      The quantum mutation map 
      is the 
      isomorphism between the 
      skew-fields
      $$\mu_{
      \mathcal{D}_X,\mathcal{D}_X'}^q=\mu_k^q \, \colon \, \Fr(
      \mathcal{X}_q(\mathcal D_X'))\rightarrow
      \Fr(
      \mathcal{X}_q(\mathcal D_X))$$
      given by the composition
      $$\mu_k^q= \mu_k^{\sharp q}\circ \mu_k',$$
      where $\mu'_k$ is the isomorphism of skew-fields
      $$\mu_k'\colon \Fr(
      \mathcal{X}_q(\mathcal D_X'))\rightarrow
      \Fr(
      \mathcal{X}_q(\mathcal D_X))$$ 
      given by 
      \begin{align}\label{eq-quantum-mutation}
      \mu_k'(
      X_v') = 
      \begin{cases}
          X_k^{-1}& \mbox{if } v=k,\\
          \left[ X_v X_k^{[Q(v,k)]_+}\right]
          & \mbox{if } v\neq k.
      \end{cases}
      \end{align}
      where $[a]_+$ stands for the positive part of a real number $a$
      $$
      [a]_+ := \max\{a,0\} = {\textstyle \frac{1}{2}(a+|a|)},
$$
and $[\sim]$ is the Weyl-ordered product as in \eqref{Weyl_ordering-X-1},
      while
       $\mu^{\sharp q}_k$ is the following automorphism of skew-field
       $$\mu_k^{\sharp q} = {\rm Ad}_{\Psi^q(
       X_k)}\colon \Fr(
       \mathcal{X}_q(\mathcal D))\rightarrow
      \Fr(
      \mathcal{X}_q(\mathcal D)).$$
\end{definition}

The following lemma will be used in \S\ref{sec-skein-cluster}.

\begin{lemma}\cite[Lemma~3.4]{KimWang}\label{lem-commute}
Let $\mathbb{F} = \mathbb{Z}$ or $\mathbb{Z}/m\mathbb{Z}$ for some positive integer $m$.
    Let $\mathcal V_0\subset\mathcal V$ contain $k$, and let ${\bf t}'=(t_v')_{v\in\mathcal V}\in\bZ ^{\mathcal V}$ 
    satisfy $\sum_{v\in\mathcal V} Q'(u,v) t_v' = 0\in 
    \mathbb{F}$ for all $u\in\mathcal V_0$.
    Suppose that  $\mu_k'((
    X')^{{\bf t}'}) = 
    X^{\bf t}$, where
    ${\bf t}=(
    t_v)_{v\in\mathcal V}\in\bZ ^{\mathcal V}$.
    Then we have 
    $\sum_{v\in\mathcal V} Q(u,v) t_v = 0\in 
    \mathbb{F}$ for all $u\in\mathcal V_0$.
\end{lemma}

\subsection{Quantum mutations for balanced $n$-th root Fock-Goncharov algebras}\label{subsec-mutation-Fock}
In this subsection, we review the quantum $\mathcal{X}$-mutation for the balanced Fock--Goncharov algebra, as constructed in \cite{KimWang}. In \S\ref{sec-compatibility}, we establish the compatibility between quantum $\mathcal{A}$- and $\mathcal{X}$-mutations via the algebra isomorphism \eqref{eq-iso-A-balanced-psi}, which will be used to prove the naturality of the quantum cluster structure inside ${\rm Frac}(\rdS)$, defined in \S\ref{sec-seed-inside}.

Suppose 
that $\mathcal{D}_X= (\Gamma,(X_v)_{v\in \mathcal{V}})$ is an $\mathcal X$-seed 
whose exchange matrix is $Q$.
Define the $n$-th root Fock-Goncharov algebra associated to $\mathcal D_X$ to be
\begin{align}\label{eq-Fock-Goncharov-n}
\mathcal{Z}_\omega(\mathcal{D}_X)  = \mathbb{T}_\omega(Q)  = R\langle Z_v^{\pm 1}, v\in \mathcal{V} \rangle / (Z_v Z_{v'} = \omega^{2Q(v,v')} Z_{v'} Z_v \mbox{ for } v,v' \in \mathcal{V}),
\end{align}
into which the Fock-Goncharov algebra $\mathcal{X}_q(\mathcal{D}_X)$ embeds as
$$
X_v \mapsto Z_v^n, \quad \forall v\in \mathcal{V}.
$$

We use $\Fr(
\mathcal{Z}_\omega(\mathcal{D}_X))$ to denote the 
skew-field of fractions of $
\mathcal{Z}_\omega(\mathcal D_X)$. 
Naturally, ${\rm Frac}(\mathcal{X}_q(\mathcal{D}_X))$ embeds into ${\rm Frac}(\mathcal{Z}_\omega(\mathcal{D}_X))$.

Let $k$ be a 
mutable vertex of $\Gamma$ and let 
$$
\mathcal D_X' =
\mu_k(\mathcal D_X).
$$
Define the 
     skew-field isomorphism $$\nu_k'\colon \Fr(
     \mathcal{Z}_\omega(\mathcal D_X'))\rightarrow
      \Fr(
      \mathcal{Z}_\omega(\mathcal D_X))$$
      such that the action of $\nu_k'$ on the generators is given by 
      the formula \eqref{eq-quantum-mutation} for $\mu'_k$ with $
      X_{v}$ replaced by $
      Z_{v}$:
     \begin{align}\label{eq-quantum-mutation_Z}
      \nu_k'(Z_v') = 
      \begin{cases}
          Z_k^{-1}& \mbox{if } v=k,\\
          \left[ Z_v Z_k^{[Q(v,k)]_+}\right]
          & \mbox{if } v\neq k.
      \end{cases}
      \end{align}
      It is easy to check that $\nu'_k$ extends $\mu'_k$.

      For an $n$-root extension of the automorphism part $\mu^{\sharp q}_k$, we consider the following subalgebra of $\mathcal{Z}_\omega(\mathcal{D}_X)$.
   \begin{definition}[\cite{KimWang}]\label{def-mut-Z}
   Define  
       $$\mathcal B_{\mathcal D_X}=\{{\bf t}=(
       t_v)_{v\in\mathcal V}\in\mathbb Z^{
       \mathcal{V}}\mid \sum_{v\in 
       \mathcal{V}} Q(u,v)
       t_v=0\in
       \mathbb{Z}/n \mathbb{Z}\text{ for all }u\in\mathcal V_{
       {\rm mut}}\},$$
    and define 
    $\mathcal{Z}^{\rm mbl}_\omega(\mathcal{D}_X)$
    to be the $R$-submodule of $
    \mathcal{Z}_\omega(\mathcal D_X)$ spanned by $
    Z^{{\bf t}}$ for ${\bf t}\in \mathcal B_{\mathcal D_X}$. We call $\mathcal{Z}^{\rm mbl}_\omega(\mathcal{D}_X)$ the {\bf mutable-balanced subalgebra} of $\mathcal{Z}_\omega(\mathcal{D}_X)$.
   \end{definition}

\def\ZM{\mathcal{Z}^{\rm mbl}_\omega}

Note that $\mathcal B_{\mathcal D_X}$ is a subgroup of $\mathbb Z^{
\mathcal{V}}$. So 
$\mathcal{Z}^{\rm mbl}_\omega(\mathcal{D}_X)$
    is 
    indeed an $R$-subalgebra of $
    \mathcal{Z}_\omega(\mathcal D_X)$.
    Since $n \mathbb Z^{
    \mathcal{V}}\subset\mathcal B_{\mathcal D_X}\subset \mathbb Z^{
    \mathcal{V}}$, we have 
    $$\mathcal{X}_q(\mathcal{D}_X) \subset \mathcal{Z}^{\rm mbl}_\omega(\mathcal{D}_X) \subset \mathcal{Z}_\omega(\mathcal{D}_X).$$
    The mutable-balanced subalgebra $\mathcal{Z}^{\rm mbl}_\omega(\mathcal{D}_X)$ is precisely the subalgebra of the $n$-th root Fock-Goncharov algebra $\mathcal{Z}_\omega(\mathcal{D}_X)$ to which we can extend the automorphism-part map $\mu^{\sharp q}_k$ for $\mathcal{X}_q(\mathcal{D}_X)$.
    

\begin{lemma}\cite[Lemma~3.7]{KimWang}\label{lem.nu_sharp_well-defined}
There exists a unique skew-field isomorphism
    \begin{align}\label{lem-def-nu-sharp}
        \nu_k^{\sharp \omega}:={\rm Ad}_{\Psi^q(
    X_k)}\colon \Fr(
    \mathcal{Z}^{\rm mbl}_\omega(\mathcal{D}_X))\rightarrow
      \Fr(
      \mathcal{Z}^{\rm mbl}_\omega(\mathcal{D}_X))
    \end{align}
      such that for each ${\bf t} = (t_v)_{v\in \mathcal{V}} \in \mathcal{B}_{\mathcal{D}_X}$, we have
      $$
      \nu^{\sharp \omega}_k(Z^{\bf t}) = Z^{\bf t} \, F^q(X_k, m),
$$
where $m = \frac{1}{n} \sum_{v\in \mathcal{V}} Q(k,v) t_v$.
\end{lemma}

Lemma~\ref{lem-F-P} implies that $\nu^{\sharp \omega}_k$ extends $\mu^{\sharp q}_k$ (Definition~\ref{def.quantum_X-mutation}).

The following lemma says that the skew-fields of fractions of the mutable-balanced subalgebras are compatible with the monomial-transformation isomorphism $\nu'_k$.

\begin{lemma}\cite[Lemma~3.8]{KimWang}\label{lem.nu_prime_restricts}
    The map $\Fr(    \mathcal{Z}_\omega(\mathcal{D}_X'))\xrightarrow{\nu_k'}
      \Fr(
      \mathcal{Z}_\omega(\mathcal{D}_X))$ defined in \eqref{eq-quantum-mutation_Z} restricts to a  
      skew-field isomorphism
      $\Fr(
      \mathcal{Z}^{\rm mbl}_\omega(\mathcal{D}_X'))\xrightarrow{\nu_k'}
      \Fr(
      \mathcal{Z}^{\rm mbl}_\omega(\mathcal{D}_X)).$ \qed
\end{lemma}

Define the quantum mutation $\nu_k^{\omega}$ for the $n$-th root Fock-Goncharov algebras, or more precisely for the skew-fields of fractions of their mutable-balanced subalgebras,  
to be the composition 
\begin{align}\label{eq-def-quantum-mutation-X}
\nu^\omega_k := \nu^{\sharp\omega}_k \circ \nu'_k ~:~ 
\Fr(
\mathcal{Z}^{\rm mbl}_\omega(\mathcal{D}_X'))\xrightarrow{\nu_k'}
      \Fr(
      \mathcal{Z}^{\rm mbl}_\omega(\mathcal{D}_X))\xrightarrow{\nu_k^{\sharp \omega}}
      \Fr(
      \mathcal{Z}^{\rm mbl}_\omega(\mathcal{D}_X)).
\end{align}

\begin{lemma}\cite[Lem~3.9]{KimWang}\label{lem:nu_extends_mu}
    The map $\nu^\omega_k$ extends $\mu^q_k$.
\end{lemma}

\begin{lemma}\label{lem:nu_involutation}
    The map $\nu^\omega_k\circ \nu^\omega_k$ equals identity.
\end{lemma}
\begin{proof}
For any ${\bf k}\in \mathbb Z^{\mathcal V}$, it is well-known that
$$\mu^q_k(\mu^q_k(X^{\bf k}))=X^{\bf k}\in  \Fr(\mathcal{X}_q(\mathcal{D}_X)).$$
 A similar argument shows that
 $$\nu^\omega_k(\nu^\omega_k(Z^{\bf t}))=Z^{\bf t}\in  \Fr(\mathcal{Z}^{\rm mbl}_\omega
 (\mathcal{D}_X))$$
  for any ${\bf k}\in \mathcal B_{\mathcal D_X}$.
\end{proof}

    In order to apply to our situation, consider a triangulable pb surface $\frak{S}$ with a triangulation $\lambda$.
    Letting $\mathcal{V} = V_\lambda$ and $\mathcal{V}_{\rm mut}$ be the subset of $V_\lambda$ consisting of the vertices contained in the interior of $\frak{S}$, one obtains a cluster $\mathcal{X}$-seed $\mathcal{D}_{\lambda} := (\Gamma_\lambda,(X_v)_{v\in V_\lambda})$.
    We denote the $n$-root Fock-Goncharov algebra and its mutable-balanced subalgebra as
    $$
    \mathcal{Z}_\omega(\fS,\lambda) = \mathcal{Z}_\omega(\mathcal{D}_{\lambda}), \qquad
    \mathcal{Z}_\omega^{\rm mbl}(\fS,\lambda) = \mathcal{Z}^{\rm mbl}_\omega(\mathcal{D}_{\lambda}).
$$
On the other hand, recall the balanced subalgebra $\mathcal{Z}^{\rm bl}_\omega(\fS,\lambda)$ of $\mathcal{Z}_\omega(\frak{S}, \lambda)$ defined in \eqref{Z_bl_omega}.

\begin{lemma}\cite[Lemma~3.11]{KimWang}
\label{lem:bl_in_mbl}
  We have $\mathcal{Z}^{\rm bl}_\omega(\fS,\lambda) \subset \mathcal{Z}^{\rm mbl}_\omega(\fS,\lambda)$.
\end{lemma}


\subsection{The naturality of the quantum trace map}\label{sec-naturality-trace}

Let $\fS$ be a triangulable pb surface with a triangulation $\lambda$. Suppose $e\in\lambda$ is not a boundary edge. 
There exists a unique idea arc $e'$ such that $e'\neq e$ and $\lambda':=(\lambda\setminus\{e\})\cup\{e'\}$ is a triangulation of $\fS$.
We call that $\lambda$ and $\lambda'$ are obtained from each other by a {\bf flip}.
In \cite{FG06,GS19}, the authors showed that $\Gamma_{\lambda'}$ and $\Gamma_\lambda$ are related to each
other by a special sequence of mutations.

Let $V_{\lambda;e}$ be the set of all vertices of $V_\lambda$ that either lie in the interior of one of these two triangles or lie on $e$; so $V_{\lambda;e}$ consists of $(n-1)^2$ vertices. For each $i=0,1,2,\ldots,n-2$, we will define a subset $V_{\lambda;e}^{(i)}$ of $V_{\lambda;e}$. Draw this ideal quadrilateral for $\lambda$ as in Figure \ref{Fig;mutation_sequence_for_flip}, so that $e$ is the `middle vertical line', and after flipping at $e$, the flipped $e$ would be the `middle horizontal line'. As in Figure \ref{Fig;mutation_sequence_for_flip}, for each vertex $v$ in $V_{\lambda,e}$ one can define the vertical distance $d_{\rm vert}(v) \in \mathbb{N}$ from the middle horizontal line, and the horizontal distance $d_{\rm hori}(v) \in \mathbb{N}$ from the middle vertical line. Define
$$
V_{\lambda;e}^{(i)} := \left\{v\in V_{\lambda;e} \, \left| \, \begin{array}{ll} d_{\rm vert}(v) \le i, &  d_{\rm vert}(v) \equiv i \, (\mbox{mod } 2), \\ d_{\rm hori}(v) \le n-2-i, & d_{\rm hori}(v) \equiv n-2-i (\mbox{mod } 2) \end{array} \right. \right\}.
$$
See Figure \ref{Fig;mutation_sequence_for_flip}; so each $V^{(i)}_{\lambda;e}$ can be viewed as forming the grid points of a quadrilateral, consisting of $(i+1)(n-i-1)$ points. Notice that these sets for different $i$ are not necessarily disjoint with each other. For example, if $n\ge 4$, then $V_{\lambda;e}^{(0)}$ and $V_{\lambda;e}^{(2)}$ have $n-3$ vertices in common.

\begin{figure}[h]
    \centering
    \scalebox{1.0}{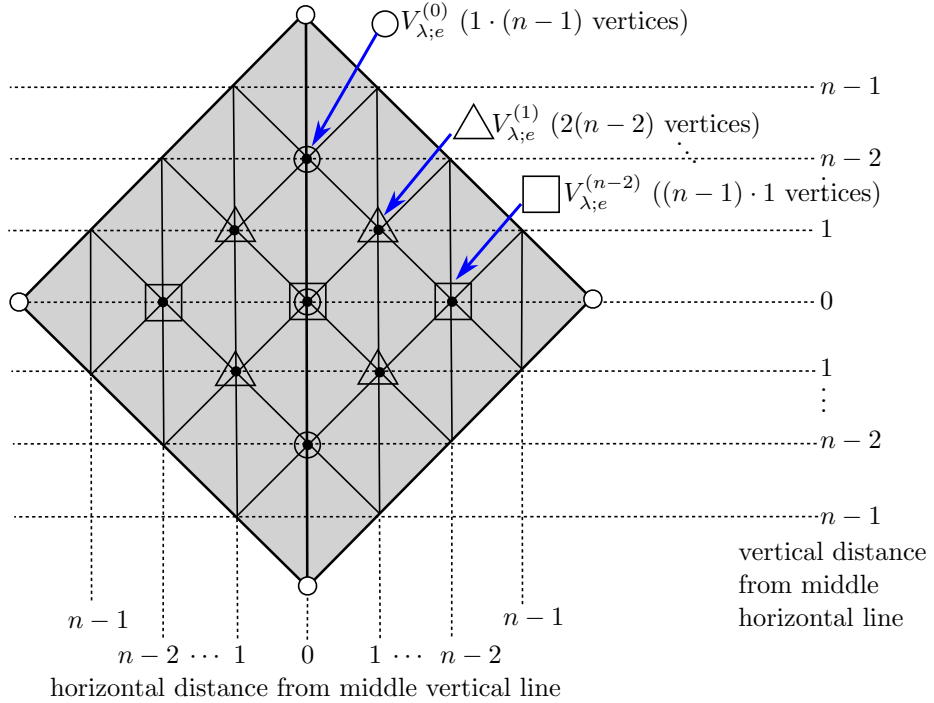}
    \caption{Vertices of the $n$-triangulation quiver involved in the flip of triangulations; the above example picture is for the case $n=4$}\label{Fig;mutation_sequence_for_flip}
\end{figure}

The sought-for mutation sequence of Fock, Goncharov and Shen consists of first mutations at all vertices of $V_{\lambda;e}^{(0)}$ in any order, then mutations at all vertices of $V_{\lambda;e}^{(1)}$ in any order, etc., then lastly mutations at all vertices of $V_{\lambda;e}^{(n-2)}$ in any order. The length of this mutation sequence is $\sum_{i=0}^{n-2} (i+1)(n-1-i) = \sum_{j=1}^{n-1} j(n-j) = \frac{1}{6}(n^3-n) =:r$. Let us denote the corresponding sequence of vertices as
$$
v_1,v_2,\ldots,v_r.
$$
So, the first $n-1$ of them would be the elements of $V^{(0)}_{\lambda;e}$, and we have
\begin{align}
\label{two_seeds_connected_by_sequence_of_mutations}
\mathcal{D}_{{\lambda'}} = \mu_{v_r} \cdots \mu_{v_2} \mu_{v_1} (\mathcal{D}_{\lambda})   \end{align}
and in particular,
\begin{align}\label{eq-mutation-flip-quiver}
    \Gamma_{\lambda'} = \mu_{v_r} \cdots \mu_{v_2} \mu_{v_1} (\Gamma_\lambda),  \quad
Q_{\lambda'} = \mu_{v_r} \cdots \mu_{v_2} \mu_{v_1} (Q_\lambda).
\end{align}

Note that, at each mutation step $\mu_{v_i}$ for $1\leq i\leq r$, the number of arrows between any two vertices on the same boundary component of $\fS$ never changes (see \cite[Figure~10.3]{FG06}). 
For any $0\leq i\leq r$, define 
$Q_\lambda^{(i)}:=
\mu_{v_i} \cdots \mu_{v_2} \mu_{v_1} (Q_\lambda).$
Then 
\begin{align}\label{def-equal-QiQ}
    Q_\lambda^{(i)}(u,v) = Q_\lambda(u,v)
\end{align}
for any two vertices $u,v$ on the same boundary component of $\fS$.

 The $n$-root balanced version of the quantum coordinate change isomorphism is defined to be
\begin{equation}\label{eq-Theta2}
\Theta_{
\lambda \lambda'}^{\omega}:=\nu_{v_1}^{\omega}\circ\cdots\circ\nu_{v_r}^{\omega}\colon\Fr(\mathcal Z_\omega^{\rm mbl}(\fS,\lambda'))\rightarrow 
\Fr(\mathcal Z_\omega^{\rm mbl}(\fS,\lambda)).
\end{equation}

Suppose now that $\lambda$ and $\lambda'$ are any two triangulations of $\fS$. A {\bf triangulation sweep} connecting 
$\lambda$ and $\lambda'$ is a sequence of triangulations
$\Lambda=(\lambda_1,\cdots,\lambda_m)$ such that 
$\lambda_1=\lambda$, $\lambda_m=\lambda'$, and $\lambda_{i+1}$ is obtained from $\lambda_i$ by a flip for each $1\leq i\leq m-1$. It is well known that for any $\lambda$ and $\lambda'$, there exists a triangulation sweep $\Lambda$ connecting $\lambda$ and $\lambda'$ (\cite{Lab09}). For any such triangulation sweep $\Lambda$, we define the corresponding $\mathcal{X}$-version quantum coordinate change isomorphism
\begin{equation}\label{eq-Theta-change2}
\Theta_{\Lambda}^{\omega}  :=\Theta_{
\lambda_1\lambda_2}^{\omega}\circ\cdots\circ\Theta_{
\lambda_{m-1}\lambda_m}^{\omega}\colon\Fr(
\mathcal{Z}^{\rm mbl}_\omega(\fS,\lambda'))\rightarrow 
\Fr(
\mathcal{Z}^{\rm mbl}_\omega(\fS,\lambda)).
\end{equation}

\begin{proposition}\cite[Proposition~3.14]{KimWang}
\label{prop:Theta_omega_consistency}
    $\Theta^\omega_\Lambda$ depends only on $\lambda$ and $\lambda'$. 
\end{proposition}

Therefore, when we do not have to keep track of a specific triangulation sweep, we can write
$$
\Theta^\omega_{\lambda\lambda'} = \Theta^\omega_\Lambda.
$$

\begin{proposition}\cite[Proposition~4.9]{KimWang}\label{Prop-rest-bal}
    The isomorphism $$\Theta_{\lambda\lambda'}^{\omega}\colon\Fr(
\mathcal{Z}^{\rm mbl}_\omega(\fS,\lambda'))\rightarrow 
\Fr(
\mathcal{Z}^{\rm mbl}_\omega(\fS,\lambda))$$ restricts to an isomorphism 
$$\Theta_{\lambda\lambda'}^{\omega}\colon\Fr(
\mathcal{Z}^{\rm bl}_\omega(\fS,\lambda'))\rightarrow 
\Fr(
\mathcal{Z}^{\rm bl}_\omega(\fS,\lambda)).$$
\end{proposition}

\def\wl{\widetilde\lambda}

\begin{theorem}\cite[Theorem~4.1]{KimWang}\label{thm-main-compatibility}
    Let $\fS$ be a triangulable pb surface. For any two triangulations $\lambda$ and $\lambda'$ of $\fS$, the following diagram commutes:
    \begin{align}
        \label{eq-compability-tr-mutation}
        \xymatrix{
        & \rdS \ar[dl]_{{\rm tr}_{\lambda'}} \ar[dr]^{{\rm tr}_\lambda} & \\
        {\rm Frac}(\mathcal{Z}^{\rm bl}_\omega(\fS,\lambda')) \ar[rr]_-{\Theta^\omega_{\lambda\lambda'}} & & {\rm Frac}(\mathcal{Z}^{\rm bl}_\omega(\fS,\lambda))}.
    \end{align}
    That is, we have
    \begin{align}
    \label{compatibility_equation}
        {\rm tr}_\lambda = \Theta^\omega_{\lambda\lambda'} \circ {\rm tr}_{\lambda'}.
    \end{align}
\end{theorem}

\subsection{Proof of Lemma \ref{lem-ker-naturality}}\label{sec-proof-lem} 
\begin{proof}[Proof of Lemma \ref{lem-ker-naturality}]
    The commutativity of diagram~\eqref{eq-compability-tr-mutation}, together with the fact that $\Theta^\omega_{\lambda\lambda'}$ is an isomorphism, proves Lemma~\ref{lem-ker-naturality}.
\end{proof}

\section{Quantum cluster structure inside $\text{Frac}(\dS)$}\label{sec-seed-structure}
In this section, we assume throughout that the pb surface $\fS$ is triangulable and has no interior punctures.  

Recall that the algebra $\dS$ (Definition~\ref{def-key-algebra}) was introduced as a quotient of the reduced stated ${\rm SL}_n$-skein algebra $\rdS$, chosen so that the quantum trace maps become injective (Lemma~\ref{lem-basic-lem}).  
Moreover, Lemma~\ref{lem-basic-lem}(b) guarantees the existence of the skew-field $\mathrm{Frac}(\dS)$ associated with $\dS$.  
Our aim in this section is to construct a quantum cluster structure in $\mathrm{Frac}(\dS)$ (Lemma~\ref{lem-seed-skein} and Theorem~\ref{thm-main-1}) and to define the ${\rm SL}_n$ quantum (upper) cluster algebra inside $\mathrm{Frac}(\dS)$ (Definition~\ref{def-sln-quantum-cluster-al}).  
In \S\ref{sec-skein-cluster} and \S\ref{sec-skein-inclusion-cluster}, we will further show that $\dS$ is contained in this ${\rm SL}_n$ quantum (upper) cluster algebra when every connected component of $\fS$ contains at least two boundary components.  

For the cases $n=2,3$, related results are available in \cite{ishibashi2023skein,LY22,muller2016skein}.

\subsection{The quantum seeds inside $\mathrm{Frac}(\dS)$}\label{sec-seed-inside}
Let $\lambda$ be a triangulation of $\fS$.  
As in \S\ref{subsec-mutation-Fock}, set 
$\mathcal V = V_\lambda$ and $\mathcal V_{\text{mut}} = \mathring{V}_\lambda$, 
where $\mathring{V}_\lambda$ denotes the set of small vertices lying in the interior of $\fS$.  

Recall the anti-symmetric matrices $Q_\lambda$ (see \S\ref{sec-traceX}) and $P_\lambda$ (see \S\ref{sec;A_tori}).  
Define  
\begin{align}\label{eq-prod-Qpi}
    \Pi_\lambda := \tfrac{1}{n} P_\lambda.
\end{align}
Equation~\eqref{eq-prod-PQ} implies that  
\begin{align}\label{eq-prod-PiQ}
    \sum_{k\in V_\lambda}
    Q_\lambda(k,u)\,\Pi_\lambda(k,v)
    = 2n \,\delta_{u,v},
\end{align}
for all $u\in \mathcal V_{\text{mut}}$ and $v\in \mathcal V$.  

For each $v\in \mathcal V$, we defined an element $\gaa_v \in \rdS$ (see \S\ref{sec-traceA}), and we use the same notation $\gaa_v$ for its image under the projection $\rdS \twoheadrightarrow \dS$.  
Recall that $\xi = \omega^n$.  
Then Lemma~\ref{gaa-com} shows that  
\begin{align}\label{eq-gaa-com}
    \gaa_v \gaa_{v'} = \xi^{\Pi_\lambda(v,v')} \gaa_{v'} \gaa_v
    \quad\in \dS.
\end{align}

For any $v_1,\ldots,v_r \in \mathcal V$ and $a_1,\ldots,a_r \in \mathbb{Z}$, define the Weyl-ordered product by  
\begin{align}\label{Weyl_ordering-gaa}
    \left[ \gaa_{v_1}^{a_1} \gaa_{v_2}^{a_2} \cdots \gaa_{v_r}^{a_r} \right] 
    := \xi^{-\frac{1}{2}\sum_{i<j} a_i a_j \Pi_\lambda(v_i,v_j)} 
       \gaa_{v_1}^{a_1} \gaa_{v_2}^{a_2} \cdots \gaa_{v_r}^{a_r}.
\end{align}
For ${\bf t} = (t_v)_{v\in V_\lambda} \in \mathbb{Z}^{\mathcal V}$, set  
\begin{align}\label{def-gaa-monomial}
    \gaa^{\bf t} := \left[ \prod_{v\in V_\lambda} \gaa_v^{t_v}\right].
\end{align}

Finally, define the map  
\begin{align}\label{def-M-lambda}
    M_\lambda \colon \mathbb Z^{\mathcal V} \longrightarrow \Fr(\dS), 
    \quad {\bf t} \longmapsto \gaa^{\bf t}.
\end{align}
Then we have the following.

\begin{lemma}\label{lem-seed-skein}
Let $\fS$ be a triangulable pb surface and has no interior punctures, and let $\lambda$ be a triangulation of $\fS$.
    Then $(Q_\lambda,\Pi_\lambda,M_\lambda)$ is a quantum seed (Definition~\ref{def-quantum-seed}) inside the skew-field $\Fr(\dS)$.
\end{lemma}
\begin{proof}
By definition, $Q_\lambda$ is an exchange matrix.  
Since $P_\lambda$ is anti-symmetric with entries in $n\mathbb{Z}$ (Remark~\ref{rem-P-P}), it follows that $\Pi_\lambda$ is an integral anti-symmetric matrix, and equation~\eqref{eq-prod-PiQ} holds.  

Equation~\eqref{eq-gaa-com} implies
\[
    M_\lambda({\bf k})\, M_\lambda({\bf t}) 
    = \xi^{\frac{1}{2}\,{\bf k} P {\bf t}^T} 
      M_\lambda({\bf k} + {\bf t}).
\]

By Lemma~\ref{lem-basic-lem}(a), the map $\trA\colon\dS \to \A$ is injective, so we may regard $\dS$ as a subalgebra of $\A$.  
Lemma~\ref{thm-trace-A} then shows that $\gaa_v = A_v$ for all $v \in \mathcal V$, and moreover
\[
    \Ap \subset \dS \subset \A.
\]
Hence $M_\lambda(\mathbb Z^{\mathcal V})$ coincides with the monomials of the quantum torus $\A$, and we obtain 
\[
    \Fr(\A) = \Fr(\dS).
\]

\end{proof}

Given a triangulation $\lambda$ of $\fS$, we identify $\Fr(\A)$ (resp. $A_v \in \Fr(\A)$ for $v \in V_\lambda$) with $\Fr(\dS)$ (resp. $\gaa_v \in \Fr(\dS)$ for $v \in V_\lambda$).  
With this identification, the triple $(Q_\lambda,\Pi_\lambda,M_\lambda)$ forms a quantum seed in the skew-field $\Fr(\A)$, where  
\begin{align}\label{def-M-lambda-A}
    M_\lambda \colon \mathbb Z^{V_\lambda} \to \Fr(\A), 
    \quad {\bf t} \mapsto A^{\bf t}.
\end{align}

\begin{definition}\label{def-tri-quantum}
    Let $\mathsf{S}_{\fS,\lambda}$ denote the quantum $\mathcal A$-mutation class (Definition~\ref{def-quantum-class}) containing the quantum seed $(Q_\lambda,\Pi_\lambda,M_\lambda)$.  

    Define $\mathscr{A}_{\fS,\lambda}$ (resp.~$\mathscr{U}_{\fS,\lambda}$) as $\mathscr{A}_{\mathsf S_{\fS,\lambda}}$ (resp.~$\mathscr{U}_{\mathsf S_{\fS,\lambda}}$) in the sense of Definition~\ref{def-quan-cluster-algebra}.
\end{definition}

By Theorem~\ref{thm-inclusion-quantum}, we obtain the following:

\begin{lemma}
    There are inclusions
    \[
        \mathscr{A}_{\mathsf S_{\fS,\lambda}}
        \subset 
        \mathscr{U}_{\mathsf S_{\fS,\lambda}}
        \subset 
        \Fr(\dS).
    \]
\end{lemma}

The constructions of $\mathsf{S}_{\fS,\lambda}$, $\mathscr{A}_{\mathsf S_{\fS,\lambda}}$, and $\mathscr{U}_{\mathsf S_{\fS,\lambda}}$ all depend a priori on the triangulation $\lambda$.  
In the remainder of this section, we will prove that the definitions in Definition~\ref{def-tri-quantum} are, in fact, independent of the choice of $\lambda$.

\subsection{The compatibility between quantum $\mathcal A$ and $\mathcal X$-mutations}\label{sec-compatibility}

In this subsection, we will show that the 
quantum $\mathcal A$ and $\mathcal X$-mutations introduced in \S\ref{sec-quantum-clu-al} are compatible to each other, which will be used in \S\ref{subsec-naturality} to show the naturality of Definition~\ref{def-tri-quantum}.

\def\Xb{\mathcal Z_{\omega}^{\rm bl}(\fS,\lambda)}

For every cluster $\mathcal X$-seed $\mathcal D_X$,
 we defined the $n$-th root Fock-Goncharov algebra $\mathcal{Z}_\omega(\mathcal{D}_X)$ (Equation~\eqref{eq-Fock-Goncharov-n}) and a subalgebra $\mathcal{Z}^{\rm mbl}_\omega(\mathcal{D}_X)$ of $\mathcal{Z}_\omega(\mathcal{D}_X)$ (Definition~\ref{def-mut-Z}).
Let $k\in\Vm$, and let $\mu_{k,X}(\mathcal D_X)=\mathcal D_X'$. We introduced the following mutation (Equation~\eqref{eq-def-quantum-mutation-X})
$$
\nu^\omega_k \colon
\Fr(
\mathcal{Z}^{\rm mbl}_\omega(\mathcal{D}_X'))\rightarrow
      \Fr(
      \mathcal{Z}^{\rm mbl}_\omega(\mathcal{D}_X)).$$

 As discussed in \S\ref{subsec-mutation-Fock}, we have $\mathcal D_{\lambda}=(\Gamma_\lambda,(X_v)_{v\in V_\lambda})$ is a cluster $\mathcal X$-seed and $\mathcal{Z}^{\rm bl}_\omega(\fS,\lambda) \subset \mathcal{Z}^{\rm mbl}_\omega(\mathcal D_{\lambda})$ (Lemma~\ref{lem:bl_in_mbl}).

\def\bv{\mathbbm{v}}

As introduced in \S\ref{sec-seed-inside}, the triple
\begin{align}\label{def-qum-seed-w}
    w_\lambda:=(Q_\lambda,\Pi_\lambda,M_\lambda)
\end{align}
 is a quantum seed inside the skew-field $\Fr(\A)$.
Theorem \ref{thm-transition-LY} shows that there is an algebra isomorphism from $\A$ to $\Xb$.
For a sequence $\mathbbm{v}=(v_1,\cdots,v_m)$ of mutable vertices in $V_\lambda$, define
\begin{align}\label{eq-def-wD}
    \text{$w_\lambda^{\mathbbm{v}}=\mu_{v_m,A}(\cdots \mu_{v_1,A}(w_\lambda))\text{ and }
\mathcal D_\lambda^{\mathbbm{v}}=\mu_{v_m,X}(\cdots \mu_{v_1,X}(\mathcal D_{\lambda}))
$}
\end{align}
Note that we allow $\bv$ to be an empty sequence and define 
$$\text{$w_\lambda^{\emptyset}=w_\lambda$, and $\mathcal D_\lambda^{\emptyset}=\mathcal D_\lambda$.}$$
Suppose ${w_\lambda^{\mathbbm{v}}}=
(Q_\lambda',\Pi_\lambda',M_\lambda')$, then define
\begin{align}\label{eq-def-piv}
    \Pi_\lambda^{\bv}=\Pi_\lambda'.
\end{align}

It is natural to ask whether there exists a subalgebra 
$\mathcal Z_\omega^{\mathrm{bl}}(\mathcal D_\lambda^{\mathbbm{v}}) \subset \ZM(\mathcal D_\lambda^{\mathbbm{v}})$ 
that is isomorphic to $\mathbb T(w_\lambda^{\mathbbm{v}})$, where $\mathbb T(w_\lambda^{\mathbbm{v}})$ 
is defined in~\eqref{eq-T-w}. In the remainder of this subsection, we shall construct 
$\mathcal Z_\omega^{\mathrm{bl}}(\mathcal D_\lambda^{\mathbbm{v}})$ and show that the isomorphism 
$\nu_k^\omega$ in~\eqref{eq-def-quantum-mutation-X} descends to 
$\mathcal Z_\omega^{\mathrm{bl}}$. This construction will be used in 
\S\ref{subsec-naturality} to establish the main result of this section 
(Theorem~\ref{thm-main-1}).

Recall that $\Xb$ is the subalgebra of $\X$ generated by elements $Z^{\bf k}$ with 
${\bf k}\in \mathbb Z^{V_\lambda}$ satisfying ${\bf k}H_\lambda \in n\mathbb Z^{V_\lambda}$, 
where $H_\lambda$ is defined in~\eqref{eq-def-H-lambda}. 
The key step in defining 
$\mathcal Z_\omega^{\mathrm{bl}}(\mathcal D_\lambda^{\mathbbm{v}})$ 
is therefore to introduce an appropriate matrix $H_\lambda^\bv$ 
associated with $\mathcal Z(\mathcal D_{\lambda}^\bv)$.

\def\bv{\mathbbm{v}}

Note that the matrix mutation defined in \eqref{eq-mutation-Q} can be generalized to any square matrix $Q=(Q(u,v))_{u,v\in\mathcal V}$ with entries in $\mathbb C$, which is still denoted as $\mu_k$.
For a sequence $\mathbbm{v}=(v_1,\cdots,v_m)$ of mutable vertices in $V_\lambda$, define 
\begin{align}\label{eq-def-DH-v}
    Q_\lambda^\bv=\mu_{v_m}(\cdots \mu_{v_1}(Q_\lambda))\text{ and }
 H_\lambda^\bv=\mu_{v_m}(\cdots \mu_{v_1}(H_\lambda)).
\end{align}
When $\bv=\emptyset$, define 
$$\text{$Q_\lambda^{\emptyset}=Q_\lambda$, and $H_\lambda^{\emptyset}=H_\lambda$.}$$

The following lemma shows that $H_\lambda^\bv$ coincides with $Q_\lambda^\bv$ away from the boundary, thereby establishing the naturality of the definition in \eqref{eq-def-DH-v} by comparison with \eqref{def-H-Q-interior}.

\begin{lemma}\label{lem-matrix-HQ}
  Let $\bv=(v_1,\ldots,v_m)$ be a sequence of mutable vertices in $V_\lambda$.  
  For $u,v\in V_\lambda$, if $u$ and $v$ are not contained in the same boundary edge, then
  \[
    H_\lambda^\bv(u,v) = Q_\lambda^\bv(u,v).
  \]
\end{lemma}

\begin{proof}
We proceed by induction on $m$.  

When $m=0$, the claim follows directly from \eqref{def-H-Q-interior}.  

For $1\leq i\leq m$, define
\[
  Q_\lambda^{(i)} := \mu_{v_i}\!\bigl(\cdots \mu_{v_1}(Q_\lambda)\bigr),
  \qquad
  H_\lambda^{(i)} := \mu_{v_i}\!\bigl(\cdots \mu_{v_1}(H_\lambda)\bigr),
\]
with initial terms $Q_\lambda^{(0)} = Q_\lambda$ and $H_\lambda^{(0)} = H_\lambda$.  

Assume that for some $s\geq 1$ we have  
\[
   H_\lambda^{(s-1)}(u,v) = Q_\lambda^{(s-1)}(u,v),
\]
whenever $u$ and $v$ do not lie on the same boundary edge.  

- **Case 1: $v_s \in \{u,v\}$.**  
  In this situation,
  \[
    H_\lambda^{(s)}(u,v) 
      = - H_\lambda^{(s-1)}(u,v)
      = - Q_\lambda^{(s-1)}(u,v)
      = Q_\lambda^{(s)}(u,v).
  \]

- **Case 2: $v_s \notin \{u,v\}$.**  
  Then by the mutation rule,
  \begin{align*}
    H_\lambda^{(s)}(u,v) 
      &= H_\lambda^{(s-1)}(u,v) 
        + \tfrac{1}{2}\Bigl(
           H_\lambda^{(s-1)}(u,v_s)\,\bigl|H_\lambda^{(s-1)}(v_s,v)\bigr|
          +\bigl|H_\lambda^{(s-1)}(u,v_s)\bigr|\,H_\lambda^{(s-1)}(v_s,v)
        \Bigr) \\
      &= Q_\lambda^{(s-1)}(u,v) 
        + \tfrac{1}{2}\Bigl(
           Q_\lambda^{(s-1)}(u,v_s)\,\bigl|Q_\lambda^{(s-1)}(v_s,v)\bigr|
          +\bigl|Q_\lambda^{(s-1)}(u,v_s)\bigr|\,Q_\lambda^{(s-1)}(v_s,v)
        \Bigr) \\
      &= Q_\lambda^{(s)}(u,v).
  \end{align*}

Thus the claim holds for $s$, completing the induction.
\end{proof}

For a sequence $\mathbbm{v}=(v_1,\cdots,v_m)$ of mutable vertices in $V_\lambda$, $k\in\mathring{V}_\lambda$, and $\epsilon\in\{-,+\}$, define 
\begin{align*}
    E_{k,\epsilon,\bv} (i,j)&=
    \begin{cases}
        \delta_{i,j} & j\neq k,\\
        -1 & i=j=k,\\
        [-\epsilon Q_\lambda^{\bv}(i,k)]_+
        & i\neq k=j,
    \end{cases}\\
     F_{k,\epsilon,\bv} (i,j)&=
    \begin{cases}
        \delta_{i,j} & i\neq k,\\
        -1 & i=j=k,\\
        [-\epsilon Q_\lambda^{\bv}(j,k)]_+
        & j\neq k=i,
    \end{cases}
\end{align*}
for $i,j\in V_\lambda.$
Note that 
\begin{align}\label{eq-tran-EF}
    E_{k,\epsilon,\bv}^T= F_{k,\epsilon,\bv}.
\end{align}

\begin{lemma}\cite{BZ}\label{lem-matrix-mul}
    For a sequence $\mathbbm{v}=(v_1,\cdots,v_m)$ of mutable vertices in $V_\lambda$ and $k\in\mathring{V}_\lambda$, set $\mathbbm{u}=(v_1,\cdots,v_m,k)$.

\begin{enumerate}[label={\rm (\alph*)}]\itemsep0,3em

\item\label{lem-matrix-mul-a} $Q_\lambda^{\mathbbm{u}}
=E_{k,\epsilon,\bv}Q_\lambda^{\mathbbm{v}} F_{k,\epsilon,\bv},\quad
\Pi_\lambda^{\mathbbm{u}}
=E_{k,\epsilon,\bv}^T \Pi_\lambda^{\mathbbm{v}} E_{k,\epsilon,\bv}.$

\item\label{lem-matrix-mul-b} $E_{k,\epsilon,\bv}^2= F_{k,\epsilon,\bv}^2=I.$

\item\label{lem-matrix-mul-c} $(E_{k,-,\bv} E_{k,+,\bv})^T \Pi_\lambda^{\mathbbm{v}} (E_{k,-,\bv} E_{k,+,\bv})= \Pi_\lambda^{\mathbbm{v}}.$

\item\label{lem-matrix-mul-d} $Q_\lambda^{\mathbbm{v}} \Pi_\lambda^{\mathbbm{v}} = Q_\lambda \Pi_\lambda.$

\end{enumerate}
    
\end{lemma}

\def\bu{\mathbbm{u}}

\begin{lemma}\label{lem-mutation-H}
 Under the same assumptions as Lemma~\ref{lem-matrix-mul}, we have 
  \begin{enumerate}[label={\rm (\alph*)}]\itemsep0.3em
    \item\label{lem-mutation-H-a} 
    $H_\lambda^{\mathbbm{u}}
    =E_{k,\epsilon,\bv}\,H_\lambda^{\mathbbm{v}}\,F_{k,\epsilon,\bv}
    =E_{k,\epsilon,\bv}\,H_\lambda^{\mathbbm{v}}\,E_{k,\epsilon,\bv}^T$.
    
    \item\label{lem-mutation-H-b} 
    $\mu_k\!\bigl(\mu_k(H_\lambda^{\mathbbm{v}})\bigr)
    = H_\lambda^{\mathbbm{v}}$.
  \end{enumerate}
\end{lemma}
\begin{proof}
    (a) It follows from Lemma~\ref{lem-matrix-HQ} and Lemma~\ref{lem-matrix-mul}(a) (with $Q$ replaced by $H$).

    (b) From part (a), we have
  \begin{align*}   \mu_k\!\bigl(\mu_k(H_\lambda^{\mathbbm{v}})\bigr)
    &= \mu_k(H_\lambda^{\mathbbm{u}}) \\
    &= E_{k,+,\bu}\,H_\lambda^{\mathbbm{u}}\,F_{k,+,\bu} \\
    &= E_{k,+,\bu}\,E_{k,-,\bv}\,H_\lambda^{\mathbbm{v}}\,F_{k,-,\bv}\,F_{k,+,\bu}.
  \end{align*}
  It is straightforward to check that 
  $E_{k,+,\bu}=E_{k,-,\bv}$ and 
  $F_{k,+,\bu}=F_{k,-,\bv}$. 
  Then Lemma~\ref{lem-matrix-mul}\,(b) gives
  \begin{align*}
    \mu_k\!\bigl(\mu_k(H_\lambda^{\mathbbm{v}})\bigr)
    &= E_{k,-,\bv}\,E_{k,-,\bv}\,H_\lambda^{\mathbbm{v}}\,F_{k,-,\bv}\,F_{k,-,\bv} \\
    &= H_\lambda^{\mathbbm{v}}.
  \end{align*}
\end{proof}

Recall that we defined the matrices $K_\lambda$ (Equation~\eqref{eq-surgen-exp}) and $P_\lambda$ (Equation~\ref{eq-anti-matric-P-def}). Remarks~\ref{rem-Q-LY}, \ref{rem-P-P}, and \cite[Lemma~11.9]{LY23} imply the following.

\begin{lemma}\label{lem-matrix-id-LY}
We have the following identities
\begin{enumerate}[label={\rm (\alph*)}]\itemsep0,3em

\item\label{lem-matrix-id-LY-a} $H_\lambda K_\lambda= n I$.

\item\label{lem-matrix-id-LY-b} $K_\lambda Q_\lambda K_\lambda^T =\frac{1}{2} P_\lambda.$

\end{enumerate}
    
\end{lemma}

\begin{lemma}\label{lem-key-com-mutation}
  Let $\bv=(v_1,\cdots,v_m)$ be a sequence of mutable vertices in $V_\lambda$.
\begin{enumerate}[label={\rm (\alph*)}]\itemsep0,3em

\item\label{lem-key-com-mutation-a} There exists a unique matrix $K_\lambda^{\bv}$ such that 
$H_\lambda^{\mathbbm{v}} K_\lambda^{\mathbbm{v}}=nI$.

\item\label{lem-key-com-mutation-b} $H_\lambda^{\mathbbm{v}} \Pi_\lambda^{\mathbbm{v}} (H_\lambda^{\mathbbm{v}})^T=2n Q_\lambda^{\mathbbm{v}} $.

\end{enumerate}
\end{lemma}
\begin{proof}
    We prove the lemma using induction on $m$. When $m=0$, Lemma~\ref{lem-matrix-id-LY}(a) implies (a). 
    Lemma~\ref{lem-matrix-id-LY}(a) and (b) show that 
    $H_\lambda P_\lambda H_\lambda^T=2n^2 Q_\lambda$. Since $\Pi_\lambda=\frac{1}{n} P_\lambda$, then 
    $H_\lambda \Pi_\lambda H_\lambda^T=2n Q_\lambda$.
    Assume that the lemma holds for $m$ ($m\geq 0$). Let $v_{m+1}$ be a mutable vertex, let $k=v_{m+1}$, and let $\mathbbm{u}=(v_1,\cdots,v_m,k)$.
    From the inductive hypothesis, there exists a unique matrix $K_\lambda^{\bv}$ such that 
$H_\lambda^{\mathbbm{v}} K_\lambda^{\mathbbm{v}}=nI$.
    Define 
    $$K_\lambda^{\mathbbm{u}}
=F_{k,\epsilon,\bv}K_\lambda^{\mathbbm{v}} E_{k,\epsilon,\bv}.$$

We have 
\begin{align*}
H_\lambda^{\mathbbm{u}} K_\lambda^{\mathbbm{u}}&=
E_{k,\epsilon,\bv}H_\lambda^{\mathbbm{v}} F_{k,\epsilon,\bv} F_{k,\epsilon,\bv}K_\lambda^{\mathbbm{v}} E_{k,\epsilon,\bv}
     \quad (\mbox{by Lemma~\ref{lem-mutation-H}})\\
     &= E_{k,\epsilon,\bv}H_\lambda^{\mathbbm{v}} K_\lambda^{\mathbbm{v}} E_{k,\epsilon,\bv}
     \quad ( \mbox{by Lemma~\ref{lem-matrix-mul}(b)})\\
     & = nE_{k,\epsilon,\bv}E_{k,\epsilon,\bv}
     \quad ( \mbox{by inductive hypothesis})\\
     & = nI \quad ( \mbox{by Lemma~\ref{lem-matrix-mul}(b)}).
\end{align*}
This proves part (a) for $\bu$.  

  For part (b), we compute
\begin{align*}
H_\lambda^{\mathbbm{u}} \Pi_\lambda^{\mathbbm{u}} (H_\lambda^{\mathbbm{u}})^T&=H_\lambda^{\mathbbm{u}} 
E_{k,-,\bv}^T \Pi_\lambda^{\mathbbm{v}} E_{k,-,\bv} (H_\lambda^{\mathbbm{u}})^T 
     \quad ( \mbox{by Lemma~\ref{lem-matrix-mul}(a)})\\
     &= E_{k,+,\bv}H_\lambda^{\mathbbm{v}} E_{k,+,\bv}^T 
E_{k,-,\bv}^T \Pi_\lambda^{\mathbbm{v}} E_{k,-,\bv} E_{k,+,\bv}(H_\lambda^{\mathbbm{v}})^T E_{k,+,\bv}^T
     \quad ( \mbox{by Lemma~\ref{lem-mutation-H}})\\
     &= E_{k,+,\bv}H_\lambda^{\mathbbm{v}} \Pi_\lambda^{\mathbbm{v}}(H_\lambda^{\mathbbm{v}})^T E_{k,+,\bv}^T
     \quad ( \mbox{by Lemma~\ref{lem-matrix-mul}(c)})\\
     &=2n E_{k,+,\bv}Q_\lambda^{\mathbbm{v}}  E_{k,+,\bv}^T
     \quad ( \mbox{by inductive hypothesis})\\
     &=2n Q_\lambda^{\mathbbm{u}}  
     \quad ( \mbox{by Lemma~\ref{lem-matrix-mul}(a) and \eqref{eq-tran-EF}}).
\end{align*}
Thus (b) also holds for $\bu$, completing the induction.
\end{proof}

For a sequence $\mathbbm{v}=(v_1,\cdots,v_m)$ of mutable vertices in $V_\lambda$, define 
\begin{align}
&\mathcal B_\lambda^{\bv}:=\{Z^{\bf k}\mid {\bf k} H_\lambda^{\bv}\in (n\mathbb Z)^{V_\lambda}, {\bf k}\in \mathbb Z^{V_\lambda}\}\\
\label{def-Z-D-v-lambda}
    &\mathcal Z_\omega^{\rm bl}(\mathcal D_\lambda^{\bv}):=
    \text{span}_R\{Z^{\bf k}\mid {\bf k} \in \mathcal B_\lambda^{\bv}\}\subset \mathcal Z_\omega(\mathcal D_\lambda^{\bv}),
\end{align}
where $\mathcal D_\lambda^{\bv}$ is defined in \eqref{eq-def-wD}.
It is easy to show that $\mathcal B_\lambda^{\bv}$ is a subgroup of $\mathbb Z^{V_\lambda}$ and $\mathcal Z_\omega^{\rm bl}(\mathcal D_\lambda^{\bv})$ is a subalgebra of $\mathcal Z_\omega(\mathcal D_\lambda^{\bv})$ with 
$\mathcal X_\omega(\mathcal D_\lambda^{\bv})\subset \mathcal Z_\omega^{\rm bl}(\mathcal D_\lambda^{\bv})$.

\def\Zbv{\mathcal Z_\omega^{\rm bl}(\mathcal D_\lambda^{\bv})}
\def\Zmv{\mathcal Z_\omega^{\rm mbl}(\mathcal D_\lambda^{\bv})}
\def\Tv{\mathbb T({w_\lambda^{\mathbbm{v}}})}
\def\Hv{H_\lambda^{\bv}}
\def\Kv{K_\lambda^{\bv}}
\def\Qv{Q_\lambda^{\bv}}
\def\Pv{\Pi_\lambda^{\bv}}
\def\Bv{\mathcal B_\lambda^{\bv}}

\begin{lemma}\label{lem-com-bal-T}
  Let $\bv=(v_1,\ldots,v_m)$ be a sequence of mutable vertices in $V_\lambda$.  
  The map
  \[
    \varphi_\lambda^{\bv}\colon \Zbv \longrightarrow \mathbb T(w_\lambda^{\mathbbm{v}}), 
    \qquad 
    Z^{\bf k}\longmapsto A^{\tfrac{1}{n}{\bf k} H_\lambda^{\bv}}
    \quad \text{for } {\bf k}\in\Bv,
  \]
  where $\mathbb T(w_\lambda^{\mathbbm{v}})$ is defined in~\eqref{eq-T-w}, is an algebra isomorphism.
\end{lemma}

\begin{proof}
  For any ${\bf k}, {\bf t}\in \Bv$, we have 
  \[
    Z^{\bf k} Z^{\bf t}= \omega^{2 {\bf k}\Qv {\bf t}^T}\, Z^{\bf t} Z^{\bf k}\in\Zbv.
  \]
  To show that $\varphi_\lambda^{\bv}$ is a well-defined algebra homomorphism, it suffices to check that
  \[
    \varphi_\lambda^{\bv}(Z^{\bf k})\,\varphi_\lambda^{\bv}(Z^{\bf t})
    = \omega^{2 {\bf k}\Qv {\bf t}^T}\,
      \varphi_\lambda^{\bv}(Z^{\bf t})\,\varphi_\lambda^{\bv}(Z^{\bf k})
    \in\Tv.
  \]

  Indeed,
  \begin{align*}
    \varphi_\lambda^{\bv}(Z^{\bf k})\,
    \varphi_\lambda^{\bv}(Z^{\bf t})
    &= A^{\tfrac{1}{n}{\bf k} H_\lambda^{\bv}}
       A^{\tfrac{1}{n}{\bf t} H_\lambda^{\bv}} \\
    &= \xi^{\tfrac{1}{n^2}\,{\bf k}\Hv\Pv(\Hv)^T {\bf t}^T}\,
       A^{\tfrac{1}{n}{\bf t} H_\lambda^{\bv}}
       A^{\tfrac{1}{n}{\bf k} H_\lambda^{\bv}} \\
    &= \xi^{\tfrac{2}{n}\,{\bf k}\Qv {\bf t}^T}\,
       A^{\tfrac{1}{n}{\bf t} H_\lambda^{\bv}}
       A^{\tfrac{1}{n}{\bf k} H_\lambda^{\bv}}
       \qquad \text{(by Lemma~\ref{lem-key-com-mutation}(b))} \\
    &= \omega^{2 {\bf k}\Qv {\bf t}^T}\,
       \varphi_\lambda^{\bv}(Z^{\bf t})\,\varphi_\lambda^{\bv}(Z^{\bf k})
       \qquad \text{(by\;$\omega=\xi^{1/n}$)}.
  \end{align*}
  Hence $\varphi_\lambda^{\bv}$ is an algebra homomorphism.  

  \smallskip
  Finally, Lemma~\ref{lem-key-com-mutation}(a) provides a matrix 
  $K_\lambda^{\bv}$ such that $\Hv \Kv = nI$.  
  Then the map
  \[
    \Tv \longrightarrow \Zbv, 
    \qquad 
    A^{\bf k}\longmapsto Z^{{\bf k} K_\lambda^{\bv}}
    \quad \text{for } {\bf k}\in\mathbb Z^{V_\lambda},
  \]
  is the inverse of $\varphi_\lambda^{\bv}$.  
  Thus $\varphi_\lambda^{\bv}$ is an algebra isomorphism.
\end{proof}

\begin{lemma}\label{eq-bl-inclusion-mut-vu}
  For any sequence $\mathbbm{v}=(v_1,\ldots,v_m)$ of mutable vertices in $V_\lambda$, we have 
  \[
    \Zbv \subset \Zmv.
  \]
\end{lemma}

\begin{proof}
  Let ${\bf t}=(t_v)_{v\in V_\lambda}\in \mathcal B_\lambda^{\bv}$.  
  For any mutable vertex $u$, Lemma~\ref{lem-matrix-HQ} implies
  \[
    \sum_{v\in V_\lambda} Q_\lambda^{\bv}(u,v)\,t_v
    = \sum_{v\in V_\lambda} H_\lambda^{\bv}(u,v)\,t_v
    = 0 \in \mathbb{Z}/n\mathbb{Z}.
  \]
  By the definition of $\Zmv$ (Definition~\ref{def-mut-Z}), this condition ensures that 
  $Z^{\bf t}\in \Zmv$.  
  Hence every generator of $\Zbv$ lies in $\Zmv$, proving the inclusion.
\end{proof}

\def\bvi{\bv(i)}
\def\bu{\mathbbm{u}}

Let $\bv=(v_1,\cdots,v_m)$ be a sequence of mutable vertices in $V_\lambda$, and let $k$ be a mutable vertex in $V_\lambda$.
Set $\mathbbm u= (v_1,\cdots,v_m,k)$.
In \S\ref{subsec-mutation-Fock}, we introduced the following
 skew-field isomorphisms
\begin{align*}
   \nu_{k}'&\colon \Fr(
      \mathcal{Z}^{\rm mbl}_\omega(\mathcal{D}_\lambda^{\bu}))\rightarrow
      \Fr(
      \mathcal{Z}^{\rm mbl}_\omega(\mathcal{D}_\lambda^{\bv})),\\
      \nu_{k}^{\sharp\omega}&\colon \Fr(
      \mathcal{Z}^{\rm mbl}_\omega(\mathcal{D}_\lambda^{\bv}))\rightarrow
      \Fr(
      \mathcal{Z}^{\rm mbl}_\omega(\mathcal{D}_\lambda^{\bv})),\\      \nu_{k}^\omega&=\nu_{k}^{\sharp\omega}\circ\nu_{k}' \colon \Fr(
      \mathcal{Z}^{\rm mbl}_\omega(\mathcal{D}_\lambda^{\bu}))\rightarrow
      \Fr(
      \mathcal{Z}^{\rm mbl}_\omega(\mathcal{D}_\lambda^{\bv})).
\end{align*}

We have the following.
\begin{lemma}\label{lem-rest-iso-to-balanced}
    The skew-field isomorphism $\nu_{k}^\omega \colon \Fr(
      \mathcal{Z}^{\rm mbl}_\omega(\mathcal{D}_\lambda^{\bu}))\rightarrow
      \Fr(
      \mathcal{Z}^{\rm mbl}_\omega(\mathcal{D}_\lambda^{\bv}))$ restricts to a skew-field isomorphism 
      $$\nu_{k}^\omega \colon \Fr(
      \mathcal{Z}^{\rm bl}_\omega(\mathcal{D}_\lambda^{\bu}))\rightarrow
      \Fr(
      \mathcal{Z}^{\rm bl}_\omega(\mathcal{D}_\lambda^{\bv})).$$
\end{lemma}
\begin{proof}
  We first show that 
  $\nu_{k}'$ (resp. $\nu_{k}^{\sharp\omega}$) restricts to a skew-field homomorphism
  \[
    \nu_{k}' \colon \Fr(
      \mathcal{Z}^{\rm bl}_\omega(\mathcal{D}_\lambda^{\bu}))\longrightarrow
      \Fr(
      \mathcal{Z}^{\rm bl}_\omega(\mathcal{D}_\lambda^{\bv})),
      \quad 
      \text{resp. }\;
      \nu_{k}^{\sharp\omega} \colon \Fr(
      \mathcal{Z}^{\rm bl}_\omega(\mathcal{D}_\lambda^{\bv}))\longrightarrow
      \Fr(
      \mathcal{Z}^{\rm bl}_\omega(\mathcal{D}_\lambda^{\bv})).
  \]
  The claim for $\nu_{k}^{\sharp\omega}$ is immediate because of Lemma~\ref{lem.nu_sharp_well-defined} and $\mathcal{X}_\omega(\mathcal{D}_\lambda^{\bv})\subset 
  \mathcal{Z}^{\rm bl}_\omega(\mathcal{D}_\lambda^{\bv})$.

  To simplify notation, write
  \[
    Q = Q_\lambda^{\bv}, \quad H = H_\lambda^{\bv}, \qquad 
    Q' = Q_\lambda^{\bu}, \quad H' = H_\lambda^{\bu}.
  \]
  Then $Q'=\mu_k(Q)$ and $H'=\mu_k(H)$.

  Let ${\bf t}' = (t_v')_{v\in V_\lambda}$ be such that ${\bf t}' H'\in (n\mathbb Z)^{V_\lambda}$. 
  By definition of $\nu_k'$ (see Equation~\eqref{eq-quantum-mutation_Z}), we have
  \[
    \nu_k'\bigl((Z')^{{\bf t}'}\bigr) = Z^{\bf t},
  \]
  where ${\bf t} = (t_v)_{v\in V_\lambda}$ is given by
  \[
    t_v =
    \begin{cases}
      t_v' & v\neq k,\\[4pt]
      -t_k' + \sum_{i\in V_\lambda} [Q(i,k)]_+\,t_i' & v=k.
    \end{cases}
  \]

  To show $Z^{\bf t}\in
      \mathcal{Z}^{\rm bl}_\omega(\mathcal{D}_\lambda^{\bv})$, it suffices to show that
  \[
    \sum_{v\in V_\lambda} H(u,v)\,t_v = 0 \in \mathbb Z/n\mathbb Z
    \qquad \text{for all } u\in V_\lambda.
  \]

  \medskip
  \noindent
  **Case 1: $u=k$.**  
  Lemma~\ref{lem-matrix-HQ} gives $H(k,k)=Q(k,k)=0$ and $H'(k,k)=Q'(k,k)=0$.  
  Hence
  \[
    \sum_{v\in V_\lambda} H(u,v)\,t_v 
    = \sum_{\substack{v\in V_\lambda\\ v\neq k}} H(u,v)\,t_v
    = \sum_{\substack{v\in V_\lambda\\ v\neq k}} -H'(u,v)\,t_v'
    = -\sum_{v\in V_\lambda} H'(u,v)\,t_v' = 0.
  \]

  \medskip
  \noindent
  **Case 2: $u\neq k$.**  
  We compute:
  \begin{align*}
    \sum_{v\in V_\lambda} H(u,v)\,t_v 
    &= \sum_{\substack{v\in V_\lambda\\ v\neq k}} H(u,v)\,t_v + H(u,k)\,t_k \\[4pt]
    &= \sum_{\substack{v\in V_\lambda\\ v\neq k}} H'(u,v)\,t_v'
      + \tfrac{1}{2}\sum_{\substack{v\in V_\lambda\\ v\neq k}}
        \bigl(H'(u,k)|H'(k,v)| + |H'(u,k)|H'(k,v)\bigr)t_v' \\
    &\quad - H'(u,k)\Bigl(-t_k' + \sum_{v\in V_\lambda}[Q'(k,v)]_+\,t_v'\Bigr).
  \end{align*}
  Using Lemma~\ref{lem-matrix-HQ}, this simplifies to
  \[
    \sum_{v\in V_\lambda} H(u,v)\,t_v =
    \begin{cases}
      \sum_{v\in V_\lambda} H'(u,v)\,t_v' = 0\in \mathbb Z/n\mathbb Z & \text{if } H'(u,k)\geq 0,\\[6pt]
      \sum_{v\in V_\lambda} H'(u,v)\,t_v' - H'(u,k)\!\!\sum_{v\in V_\lambda} H'(k,v)\,t_v' = 0\in \mathbb Z/n\mathbb Z & \text{if } H'(u,k)<0.
    \end{cases}
  \]

This shows that $\nu_{k}^\omega \colon \Fr(
      \mathcal{Z}^{\rm mbl}_\omega(\mathcal{D}_\lambda^{\bu}))\rightarrow
      \Fr(
      \mathcal{Z}^{\rm mbl}_\omega(\mathcal{D}_\lambda^{\bv}))$ restricts to a skew-field homomorphism 
      $$\nu_{k}^\omega \colon \Fr(
      \mathcal{Z}^{\rm bl}_\omega(\mathcal{D}_\lambda^{\bu}))\rightarrow
      \Fr(
      \mathcal{Z}^{\rm bl}_\omega(\mathcal{D}_\lambda^{\bv})).$$
Similarly, we have 
$\nu_{k}^\omega \colon \Fr(
      \mathcal{Z}^{\rm mbl}_\omega(\mathcal{D}_\lambda^{\bv}))\rightarrow
      \Fr(
      \mathcal{Z}^{\rm mbl}_\omega(\mathcal{D}_\lambda^{\bu}))$ restricts to a skew-field homomorphism 
      $$\nu_{k}^\omega \colon \Fr(
      \mathcal{Z}^{\rm bl}_\omega(\mathcal{D}_\lambda^{\bv}))\rightarrow
      \Fr(
      \mathcal{Z}^{\rm bl}_\omega(\mathcal{D}_\lambda^{\bu})).$$
This completes the proof because $\nu_{k}^\omega\circ \nu_{k}^\omega$ is the identity (Lemma~\ref{lem:nu_involutation}).
\end{proof}

In Equations~\eqref{eq-mut-iso}, \eqref{eq-def-mu-step1}, and \eqref{def-mu-sharp}, we introduced the following skew-field isomorphisms
\begin{align*}
\mu_{k,A}&\colon {\rm Frac}(\mathbb T(w_\lambda^\bu))\rightarrow {\rm Frac}(\mathbb T(w_\lambda^{\bv})),\\
    \mu_{k,A}'&\colon {\rm Frac}(\mathbb T(w_\lambda^{\bu}))\rightarrow {\rm Frac}(\mathbb T(w_\lambda^{\bv})),\\
\mu_{k,A}^{\sharp} &\colon {\rm Frac}(\mathbb T(w^{\bv}_\lambda))\rightarrow {\rm Frac}(\mathbb T(w_\lambda^{\bv})).
\end{align*}
Lemma \ref{lem-decom-A}(b)
shows $\mu_{k,A}=\mu_{k,A}^{\sharp}\circ \mu_{k,A}'$.
The algebra isomorphisms in Lemma \ref{lem-com-bal-T}
$$\varphi_\lambda^{\bv}\colon\Zbv\rightarrow \mathbb T({w_\lambda^{\mathbbm{v}}}), \quad \varphi_\lambda^{\bu}\colon 
      \mathcal{Z}^{\rm bl}_\omega(\mathcal{D}_\lambda^{\bu})\rightarrow \mathbb T({w_\lambda^{\mathbbm{u}}})$$
induce the following skew-field isomorphisms 
$$\varphi_\lambda^{\bv}\colon\Fr(\Zbv)\rightarrow \mathbb \Fr(\mathbb T({w_\lambda^{\mathbbm{v}}})), \quad \varphi_\lambda^{\bu}\colon 
     \Fr( \mathcal{Z}^{\rm bl}_\omega(\mathcal{D}_\lambda^{\bu}))\rightarrow \Fr(\mathbb T({w_\lambda^{\mathbbm{u}}})).$$

The following proposition shows the compatibility between mutations $\mu_{k,A}$
and $\nu_{k}^\omega$.

\begin{proposition}\label{prop-comp}
Let $\bv=(v_1,\ldots,v_m)$ be a sequence of mutable vertices in $V_\lambda$, and let $k$ be another mutable vertex in $V_\lambda$.
Set $\mathbbm{u}=(v_1,\ldots,v_m,k)$.
Then the following diagram commutes:
\begin{equation*}
\begin{tikzcd}
\Fr\bigl(\mathcal{Z}^{\rm bl}_\omega(\mathcal{D}_\lambda^{\bu})\bigr)
  \arrow[r, "\nu_k^\omega"]
  \arrow[d, "\varphi_\lambda^{\bu}"']  
& \Fr\bigl(\mathcal{Z}^{\rm bl}_\omega(\mathcal{D}_\lambda^{\bv})\bigr) 
  \arrow[d, "\varphi_\lambda^{\bv}"]  \\
\Fr\bigl(\mathbb T({w_\lambda^{\mathbbm{u}}})\bigr)
  \arrow[r, "\mu_{k,A}"] 
& \Fr\bigl(\mathbb T({w_\lambda^{\mathbbm{v}}})\bigr)
\end{tikzcd}.
\end{equation*}
\end{proposition}

\begin{proof}
It suffices to show the commutativity of the following two diagrams
\begin{equation}\label{eq-com-step-one}
\begin{tikzcd}
\Fr\bigl(\mathcal{Z}^{\rm bl}_\omega(\mathcal{D}_\lambda^{\bu})\bigr)
  \arrow[r, "\nu_k'"]
  \arrow[d, "\varphi_\lambda^{\bu}"']  
& \Fr\bigl(\mathcal{Z}^{\rm bl}_\omega(\mathcal{D}_\lambda^{\bv})\bigr) 
  \arrow[d, "\varphi_\lambda^{\bv}"]  \\
\Fr\bigl(\mathbb T({w_\lambda^{\mathbbm{u}}})\bigr)
  \arrow[r, "\mu_{k,A}'"] 
& \Fr\bigl(\mathbb T({w_\lambda^{\mathbbm{v}}})\bigr)
\end{tikzcd}
\end{equation}
and
\begin{equation}\label{eq-com-step-two}
\begin{tikzcd}
\Fr\bigl(\mathcal{Z}^{\rm bl}_\omega(\mathcal{D}_\lambda^{\bv})\bigr)
  \arrow[r, "\nu_k^{\sharp\omega}"]
  \arrow[d, "\varphi_\lambda^{\bv}"']  
& \Fr\bigl(\mathcal{Z}^{\rm bl}_\omega(\mathcal{D}_\lambda^{\bv})\bigr) 
  \arrow[d, "\varphi_\lambda^{\bv}"]  \\
\Fr\bigl(\mathbb T({w_\lambda^{\mathbbm{v}}})\bigr)
  \arrow[r, "\mu_{k,A}^{\sharp}"] 
& \Fr\bigl(\mathbb T({w_\lambda^{\mathbbm{v}}})\bigr)
\end{tikzcd}.
\end{equation}

\smallskip
\noindent
**Step 1: Commutativity of \eqref{eq-com-step-one}.**  
For notational simplicity, write
\[
Q=Q_\lambda^{\bv}, \quad H=H_\lambda^{\bv}, 
\qquad 
Q'=Q_\lambda^{\bu}, \quad H'=H_\lambda^{\bu}.
\]

Let ${\bf t}'=(t_v')_{v\in V_\lambda}$ with ${\bf t}'H'\in (n\mathbb Z)^{V_\lambda}$.  
By definition of $\nu_k'$ (Equation~\eqref{eq-quantum-mutation_Z}), 
\[
\nu_k'\bigl((Z')^{{\bf t}'}\bigr) = Z^{\bf t},
\]
where ${\bf t}=(t_v)_{v\in V_\lambda}$ is given by
\[
t_v=
\begin{cases}
t_v' & v\neq k,\\[4pt]
-t_k' + \sum_{i\in V_\lambda}[Q(i,k)]_+\,t_i' & v=k.
\end{cases}
\]

Set ${\bf b}'=\tfrac{1}{n}{\bf t}'H'$. Then
\[
\mu_{k,A}'\!\bigl(\varphi_\lambda^{\bu}((Z')^{{\bf t}'})\bigr)
= \mu_{k,A}'(A^{{\bf b}'})
= A^{\bf b},
\]
where ${\bf b}=(b_v)_{v\in V_\lambda}$ with
\[
b_v=
\begin{cases}
b_v' + b_k'[Q(v,k)]_+ & v\neq k,\\[4pt]
-b_k' & v=k.
\end{cases}
\]
On the other hand, set ${\bf c}=\tfrac{1}{n}{\bf t}H$. Then
\[
\varphi_\lambda^{\bv}\!\bigl(\nu_k'((Z')^{{\bf t}'})\bigr) = A^{\bf c}.
\]
Thus it suffices to show $n{\bf b}=n{\bf c}$.

   When $v\neq k$, we have 
    \begin{align*}
        nb_v =&  nb_v' + nb_k'[Q(v,k)]_{+}\\
        =& \left(\sum_{u\in V_\lambda} t_u' H'(u,v)\right)
        + ([H(v,k)]_{+}) \sum_{u\in V_\lambda} t_u' H'(u,k)\\
        =& \sum_{k\neq u\in V_\lambda} t_u'  H(u,v) +\frac{1}{2}\sum_{k\neq u\in V_\lambda} t_u'\left(H(u,k)|H(k,v)|+
    |H(u,k)|H(k,v)\right)\\
    & -t_k' H(k,v) - ([H(v,k)]_{+}) \sum_{u\in V_\lambda} t_u' H(u,k)\\
    =&-t_k' H(k,v)+ \sum_{k\neq u\in V_\lambda} t_u'  H(u,v) \\
    &  +\frac{1}{2}\sum_{k\neq u\in V_\lambda} t_u'\left(H(u,k)(|H(k,v)|-2[H(v,k)]_{+})+
    |H(u,k)|H(k,v)\right) \\
     =& -t_k' H(k,v)+\sum_{k\neq u\in V_\lambda} t_u'  H(u,v)  +\frac{1}{2}\sum_{k\neq u\in V_\lambda} t_u'\left(H(u,k) H(k,v)+
    |H(u,k)|H(k,v)\right) \\
    =& -t_k' H(k,v)+\sum_{k\neq u\in V_\lambda} t_u'  H(u,v)  + \sum_{k\neq u\in V_\lambda} t_u'[H(u,k)]_{+} H(k,v),
    \end{align*}
    and 
    \begin{align*}
        nc_v =& \sum_{u\in V_\lambda} t_u H(u,v)\\
        =&(-t_k' +\sum_{u\in \mathcal V} [Q(u,k)]_+ t_u') H(k,v) + \sum_{k\neq u\in V_\lambda} t_u' H(u,v)\\
        =& -t_k' H(k,v)+\sum_{k\neq u\in V_\lambda} t_u'  H(u,v)  + \sum_{k\neq u\in V_\lambda} t_u'[H(u,k)]_{+} H(k,v)\\
        =&nb_v.
    \end{align*}

     When $v= k$, we have 
    \begin{align*}
        nb_k = -nb_k'
        = -\sum_{k\neq u\in V_\lambda} t_u'  H(u,k)'
        =\sum_{k\neq u\in V_\lambda} t_u'  H(u,k)
        =\sum_{k\neq u\in V_\lambda} t_u  H(u,k)
        =nc_k.
    \end{align*}
   This shows the commutativity of \eqref{eq-com-step-one}.

\smallskip
\noindent
**Step 2: Commutativity of \eqref{eq-com-step-two}.**  
Let ${\bf f}=(f_v)_{v\in V_\lambda}$ with ${\bf f}H\in(n\mathbb Z)^{V_\lambda}$. 
Write $H(k,*)$ for the $k$-th row of $H$.  
By Lemma~\ref{lem.nu_sharp_well-defined},
\begin{align}\label{varphi-nu-Z-f}
\varphi_\lambda^{\bv}(\nu_k^{\sharp}(Z^{\bf f}))
= \varphi_\lambda^{\bv}(Z^{\bf f} F^q(X_k,m))
= A^{\tfrac{1}{n}{\bf f}H} F^q(A^{H(k,*)},m),
\end{align}
where $m=\tfrac{1}{n}\sum_{v\in V_\lambda} Q(k,v)f_v$
and $X_k=Z_k^n$.

Let ${\bf s}=\tfrac{1}{n}{\bf f}H$
and decompose ${\bf s}={\bf s}'+s_k{\bf e}_k$ with ${\bf s}'(k)=0$, where ${\bf e}_k$ is defined as in \eqref{eq-def-vector-ei-ele}.  
Suppose
\begin{align}\label{eq-vector-s-A-j}
   A^{\bf s} = \xi^{\frac{j}{2}} A^{{\bf s}'} A_k^{s_k}
\end{align}
 for some integer $j$.
 Equation \eqref{eq-prod-PiQ} implies that the
    $d_k$ in \eqref{eq-def-mu-step2} is $2n$. 
    Then 
    \begin{align}\label{eq-mu-varphi-Z}
        \mu_{k,A}^{\sharp}(\varphi_\lambda^{\bv}(Z^{\bf f}))
        =\xi^{\frac{j}{2}}\mu_{k,A}^{\sharp}(A^{{\bf s}'})
        \mu_{k,A}^{\sharp}(A_k)^{s_k}
        = \xi^{\frac{j}{2}} A^{{\bf s}'} \left(A_k\left(1 + q^{-1}A^{Q(k,*)}\right)^{-1}\right)^{s_k}.
    \end{align}
    
    We use $\Pi$ to denote $\Pi_\lambda^{\bv}$.
    Then
    $$ Q(k,*) \Pi {\bf e}_k^T
    =\sum_{u\in \V} Q(k,u) \Pi(u,k) = -2n,$$
    where the last equality comes from \eqref{eq-prod-PiQ} and Lemma \ref{lem-matrix-mul}(d).
    This implies that 
    \begin{align}\label{eq-commu-AAk}
        A^{Q(k,*)} A_k =\xi^{-2n} A_k A^{Q(k,*)}
    =q^{-2} A_k A^{Q(k,*)}.
    \end{align}
    Note that $$s_k=\frac{1}{n}\sum_{v\in\V} H(v,k) f_v=-\frac{1}{n}\sum_{v\in\V} Q(k,v) f_v =-m.$$
    When $m\geq 0$, we have 
    \begin{equation}\label{eq-m-big-zero}
    \begin{split}
        &\left(A_k\left(1 + q^{-1} A^{Q(k,*)}\right)^{-1}\right)^{s_k}\\
        =& \left(1 + q^{-1} A^{Q(k,*)}\right) A_k^{-1}
        \cdots \left(1 + q^{-1} A^{Q(k,*)}\right) A_k^{-1}\\
        =& A_k^{-m} \left(1 + q^{2m-1} A^{Q(k,*)}\right)\cdots  \left(1 + q^{3} A^{Q(k,*)}\right) \left(1 + q A^{Q(k,*)}\right)
        \quad(\mbox{by \eqref{eq-commu-AAk}})\\
        =& A_k^{s_k} F^q(A^{H(k,*)}, m)
        \quad (\mbox{by \eqref{eq-def-Fq}}).
    \end{split}
    \end{equation}
    Similarly, when $m\leq 0$, we have 
    \begin{equation}\label{eq-m-small-zero}
    \begin{split}
        &\left(A_k\left(1 + q^{-1} A^{Q(k,*)}\right)^{-1}\right)^{s_k}\\
        =& A_k\left(1 + q^{-1}A^{Q(k,*)}\right)^{-1}
        \cdots A_k\left(1 + q^{-1} A^{Q(k,*)}\right)^{-1}\\
        =& A_k^{-m} \left(1 + q^{-(2m-1)} A^{Q(k,*)}\right)^{-1}\cdots  \left(1 + q^{-3}A^{Q(k,*)}\right)^{-1} \left(1 + q^{-1}A^{Q(k,*)}\right)^{-1}\\
        =& A_k^{s_k} F^q(A^{H(k,*)}, m).
    \end{split}
    \end{equation}
    Therefore we have 
    \begin{align*}
        \mu_{k,A}^{\sharp}(\varphi_\lambda^{\bv}(Z^{\bf d}))
        =& \xi^{\frac{j}{2}} A^{{\bf s}'} \left(A_k\left(1 + q^{-1}A^{Q(k,*)}\right)^{-1}\right)^{s_k}
        \quad(\mbox{by \eqref{eq-mu-varphi-Z}})\\
        = & \xi^{\frac{j}{2}} A^{{\bf s}'}
         A_k^{s_k} F^q(A^{H(k,*)}, m)
         \quad(\mbox{by \eqref{eq-m-big-zero} and \eqref{eq-m-small-zero}})\\
        =& A^{{\bf s}}
          F^q(A^{H(k,*)}, m)\quad(\mbox{by \eqref{eq-vector-s-A-j}})\\
         =&\varphi_\lambda^{\bv} (\nu_k^{\sharp} (Z^{\bf d}))\quad(\mbox{by \eqref{varphi-nu-Z-f}}).
    \end{align*}
      This shows the commutativity of \eqref{eq-com-step-two}. 

\smallskip
Together, Steps~1 and~2 show that the full diagram commutes.
\end{proof}

\subsection{The naturality of the quantum cluster algebra structure inside the ${\rm SL}_n$-skein theory}
\label{subsec-naturality}

The main result of this section is the following theorem, which establishes that the quantum $\mathcal A$-mutation class $\mathsf{S}_{\fS,\lambda}$, defined in Definition~\ref{def-tri-quantum}, is independent of the choice of triangulation $\lambda$.  
Consequently, there exists a natural quantum cluster structure inside $\text{Frac}(\dS)$.  
The proof relies on the compatibility between quantum $\mathcal A$- and $\mathcal X$-mutations, shown in \S\ref{sec-compatibility}, together with the compatibility between $\mathcal X$-mutations and the $\mathcal X$-version of the quantum trace maps (Theorem~\ref{thm-main-compatibility}), established in \cite{KimWang}.

\begin{theorem}[Naturality of the quantum cluster structure inside $\text{Frac}(\dS)$]\label{thm-main-1}
    Let $\fS$ be a triangulable pb surface without interior punctures, and let $\lambda$, $\lambda'$ be two triangulations of $\fS$. Then we have   $$\mathsf{S}_{\fS,\lambda}=\mathsf{S}_{\fS,\lambda'},\;
    \mathscr{A}_{\fS,\lambda}=
    \mathscr{A}_{\fS,\lambda'},\text{ and }
    \mathscr{U}_{\fS,\lambda}=
    \mathscr{U}_{\fS,\lambda'}\quad \text{(see Definition~\ref{def-tri-quantum}).}$$
\end{theorem}
\begin{proof}
   We only need to show that $\mathsf{S}_{\fS,\lambda}=\mathsf{S}_{\fS,\lambda'}$.  
It suffices to prove that $\mathsf{S}_{\fS,\lambda}$ contains the quantum seed $w_{\lambda'}=(Q_{\lambda'},\Pi_{\lambda'},M_{\lambda'})$.  
Since any two triangulations can be connected by a sequence of flips, we may assume that $\lambda'$ is obtained from $\lambda$ by a single flip.  
As discussed in \S\ref{sec-naturality-trace}, there exist mutable vertices
$v_1,v_2,\ldots,v_r$ such that 
\begin{align*}
\mathcal{D}_{{\lambda'}} = \mu_{v_r} \cdots \mu_{v_2} \mu_{v_1} (\mathcal{D}_{\lambda}). 
\end{align*}

For each $1\leq i\leq r$, define 
\[
Q_\lambda^{(i)} :=
\mu_{v_i} \cdots \mu_{v_2} \mu_{v_1}(Q_\lambda), \qquad
H_\lambda^{(i)} :=
\mu_{v_i} \cdots \mu_{v_2} \mu_{v_1}(H_\lambda),
\]
with the convention that 
\[
Q_\lambda^{(0)} := Q_\lambda, 
\qquad 
H_\lambda^{(0)} := H_\lambda.
\]

Before proceeding with the proof of Theorem~\ref{thm-main-1}, we state the following useful lemma, which will be proved later.

\begin{lemma}\label{lem-Hlambda-H}
    We have $H_\lambda^{(r)}=H_{\lambda'}$.
\end{lemma}


Set $\bv=(v_1,\ldots,v_r)$. 
By Lemma~\ref{lem-Hlambda-H} we have
\[
\mathcal{Z}^{\mathrm{bl}}_\omega(\mathcal{D}_\lambda^{\bv})
=\mathcal{Z}^{\mathrm{bl}}_\omega(\mathcal{D}_{\lambda'})
=\mathcal{Z}^{\mathrm{bl}}_\omega(\fS,\lambda').
\]
Proposition~\ref{prop-comp} therefore yields the commutative diagram
\begin{equation}\label{proof-thm-diag-1}
\begin{tikzcd}[row sep=3em, column sep=5em]
\Fr(\mathcal{Z}^{\mathrm{bl}}_\omega(\fS,\lambda'))\arrow[r, "\nu_{v_1}^\omega\circ\cdots\circ \nu_{v_r}^\omega"]
\arrow[d, "\varphi_{\lambda'}"]  
& \Fr(\mathcal{Z}^{\mathrm{bl}}_\omega(\fS,\lambda)) \arrow[d, "\varphi_\lambda"]  \\
 \Fr(\mathbb T(w_\lambda^{\mathbbm{v}}))
 \arrow[r, "\mu_{v_1,A}\circ\cdots\circ \mu_{v_r,A}"] 
&  \Fr(\mathbb T(w_\lambda))
\end{tikzcd},
\end{equation}
where $\varphi_\lambda(Z^{\mathbf k})= A^{\frac{1}{n}\mathbf k H_\lambda}$ and 
$\varphi_{\lambda'}((Z')^{\mathbf k'})= (A')^{\frac{1}{n}\mathbf k' H_{\lambda'}}$
for $\mathbf k,\mathbf k'\in\mathbb Z^{V_\lambda}$ with 
$\mathbf k H_\lambda,\mathbf k' H_{\lambda'}\in (n\mathbb Z)^{V_\lambda}$.  
To simplify notation we write $\nu$ (resp.\ $\mu$) for 
$\nu_{v_1}^\omega\circ\cdots\circ \nu_{v_r}^\omega$ (resp.\ $\mu_{v_1,A}\circ\cdots\circ \mu_{v_r,A}$).
Note that $\nu$ is the restriction of $\Theta_{\lambda\lambda'}$ (see \eqref{eq-Theta2}).
Diagram \eqref{eq-compability-tr-mutation} gives the commutative triangle
\begin{align}\label{proof-thm-diag-2}
        \xymatrix{
        & \dS \ar[dl]_{{\rm tr}_{\lambda'}} \ar[dr]^{{\rm tr}_\lambda} & \\
        {\rm Frac}(\mathcal{Z}^{\mathrm{bl}}_\omega(\fS,\lambda')) \ar[rr]_-{\nu} & & {\rm Frac}(\mathcal{Z}^{\mathrm{bl}}_\omega(\fS,\lambda)) }.
\end{align}

Recall from \S\ref{sec-traceA} that for each $v\in\V$ we defined an element $\gaa_v\in\dS$, the definition depending on the triangulation.  In the sequel we write $\gaa_{v,\lambda}$ (resp.\ $\gaa_{v,\lambda'}$) for the element defined with respect to $\lambda$ (resp.\ $\lambda'$).  
Identify $\Fr(\mathbb T(w_\lambda))$ with $\Fr(\dS)$ so that $\gaa_{v,\lambda}=A_v$ for $v\in V_\lambda$.  
Then, for each $v\in\V$, we have
\[
\varphi_\lambda\bigl(\tr(\gaa_{v,\lambda})\bigr)
=\varphi_\lambda\bigl(Z^{K_\lambda(v,*)}\bigr)=A_v=\gaa_{v,\lambda},
\]
where $K_\lambda(v,*)$ is as in Definition~\ref{def-M-vector}; the first equality is from Lemma~\ref{thm-trace-A}(d) and the second is from Lemma~\ref{lem-matrix-id-LY}(a).
Lemma~\ref{lem-basic-lem}(d) implies the identity
\begin{equation}\label{eq-iden-key-step}
    \varphi_\lambda\bigl(\tr(x)\bigr)=x\qquad\text{for all }x\in\dS.
\end{equation}

Write $w_\lambda^{\bv}=(Q',\Pi',M')$. For $v\in\V$ we compute
\begin{align*}
    M'({\bf e}_v)&=\mu(A_v') \quad ( \mbox{by \eqref{eq-iso-mutation-eqal}})\\
    &=\mu(\varphi_{\lambda'} (Z^{K_{\lambda'}(v,*)})) \quad ( \mbox{by Lemma~\ref{lem-matrix-id-LY}(a)}) \\
    &=\mu(\varphi_{\lambda'} (\text{tr}_{\lambda'}(\gaa_{v,\lambda'}))) \quad ( \mbox{by Lemma~\ref{thm-trace-A}(d)})\\
    &=\varphi_{\lambda}(\nu (\text{tr}_{\lambda'}(\gaa_{v,\lambda'}))) \quad ( \mbox{by \eqref{proof-thm-diag-1}})\\
     &=\varphi_{\lambda}(\tr(\gaa_{v,\lambda'})) \quad (\mbox{by \eqref{proof-thm-diag-2}})\\
     &=\gaa_{v,\lambda'} \quad (\mbox{by \eqref{eq-iden-key-step}})\\
     &=M_{\lambda'}({\bf e}_v)
     \quad (\mbox{by \eqref{def-M-lambda}})
\end{align*}
Hence $M'=M_{\lambda'}$.

Equation~\eqref{eq-mutation-flip-quiver} yields $Q'=\mu(Q_\lambda)=Q_{\lambda'}$.  Since $\Pi_{\lambda'}$ (resp.\ $\Pi'$) is uniquely determined by $M_{\lambda'}$ (resp.\ $M'$), we conclude $w_\lambda^{\bv}=\mu(w_\lambda)=w_{\lambda'}$.  
Therefore $\mathsf{S}_{\fS,\lambda}$ contains the quantum seed $w_{\lambda'}=(Q_{\lambda'},\Pi_{\lambda'},M_{\lambda'})$, as required.

\end{proof}

Before we prove Lemma~\ref{lem-Hlambda-H}, we state the following.

\begin{lemma}\label{lem-Hi-H}
For any two vertices $u,v$ lying on the same boundary component of $\fS$, we have 
\[
    H_\lambda^{(i)}(u,v) = H_\lambda(u,v),
\]
where $0\leq i\leq r$.
\end{lemma}
\begin{proof}
We argue by induction on $i$.  
For $i=0$, the claim is immediate, since 
$H_\lambda^{(0)} = H_\lambda.$

Assume the statement holds for $i-1\geq 0$, i.e.,
\begin{align}\label{eq-H-i-1-H-l}
    H_\lambda^{(i-1)}(u_1,v_1) = H_\lambda(u_1,v_1)\text{ whenever $u_1,v_1$ lie on the same boundary component of $\fS$}.
\end{align}
Then, by the mutation formula, we obtain
\begin{equation}\label{eq-Hi-i-1}
\begin{split}
H_\lambda^{(i)}(u,v) 
&= H_\lambda^{(i-1)}(u,v) 
 + \tfrac{1}{2}\bigl( H_\lambda^{(i-1)}(u,k)\,|H_\lambda^{(i-1)}(k,v)|
 + |H_\lambda^{(i-1)}(u,k)|\,H_\lambda^{(i-1)}(k,v) \bigr) \\
&= H_\lambda^{(i-1)}(u,v) 
 + \tfrac{1}{2}\bigl( Q_\lambda^{(i-1)}(u,k)\,|Q_\lambda^{(i-1)}(k,v)|
 + |Q_\lambda^{(i-1)}(u,k)|\,Q_\lambda^{(i-1)}(k,v) \bigr),
 \end{split}
\end{equation}
where the second equality comes from Lemma~\ref{lem-matrix-HQ}.

On the other hand, we have
\[
Q_\lambda^{(i)}(u,v) 
= Q_\lambda^{(i-1)}(u,v) 
+ \tfrac{1}{2}\bigl( Q_\lambda^{(i-1)}(u,k)\,|Q_\lambda^{(i-1)}(k,v)|
+ |Q_\lambda^{(i-1)}(u,k)|\,Q_\lambda^{(i-1)}(k,v) \bigr).
\]

By Equation~\eqref{def-equal-QiQ}, we know that 
$Q_\lambda^{(i)}(u,v)=Q_\lambda^{(i-1)}(u,v)$.  
Therefore,
\[
Q_\lambda^{(i-1)}(u,v) 
+ \tfrac{1}{2}\bigl( Q_\lambda^{(i-1)}(u,k)\,|Q_\lambda^{(i-1)}(k,v)|
+ |Q_\lambda^{(i-1)}(u,k)|\,Q_\lambda^{(i-1)}(k,v) \bigr) = 0.
\]
Together with \eqref{eq-H-i-1-H-l} and \eqref{eq-Hi-i-1}, this shows that
\[
H_\lambda^{(i)}(u,v) = H_\lambda^{(i-1)}(u,v) = H_\lambda(u,v).
\]
\end{proof}

\begin{proof}[Proof of Lemma~\ref{lem-Hlambda-H}]
    Lemma \ref{lem-matrix-HQ} implies that
    \begin{align*}
        H_\lambda^{(r)}(u,v)=
        Q_\lambda^{(r)}(u,v)=
        Q_{\lambda'}(u,v)=
        H_{\lambda'}(u,v)
    \end{align*}
for any two vertices $u,v$ not on the same boundary component of $\fS$.

Suppose $u,v$ lie on the same boundary component of $\fS$.
Since $Q_\lambda(u_1,v_1)=Q_{\lambda'}(u_1,v_1)$ for any vertices $u_1,v_1$ on a common boundary component of $\fS$, Equation~\eqref{eq-def-H-lambda} implies
$H_\lambda(u,v)=H_{\lambda'}(u,v)$.
By Lemma~\ref{lem-Hi-H},
\[
H_\lambda^{(r)}(u,v)=H_\lambda(u,v)=H_{\lambda'}(u,v).
\]
This completes the proof.
 
\end{proof}

\begin{definition}\label{def-sln-quantum-cluster-al}
    Let $\fS$ be a triangulable pb surface without interior punctures.
    Define   $$\mathsf{S}(\fS):=\mathsf{S}_{\fS,\lambda},\;
    \mathscr{A}_{\omega}(\fS):=
    \mathscr{A}_{\mathsf{S}(\fS)},\text{ and }
    \mathscr{U}_{\omega}({\fS}):=
    \mathscr{U}_{\mathsf{S}(\fS)}\quad \text{(Definition~\ref{def-tri-quantum}),}$$
    where $\lambda$ is a triangulation of $\fS$.
    We call $\mathscr{A}_{\omega}(\fS)$
    (resp. $\mathscr{U}_{\omega}({\fS})$)
    the {\bf ${\rm SL}_n$ quantum cluster algebra} (resp. {\bf ${\rm SL}_n$ quantum upper cluster algebra}) of $\fS$ 
\end{definition}

Theorem \ref{thm-main-1} implies that  
$\mathsf{S}(\fS)$, $\mathscr{A}_{\omega}(\fS)$, and $\mathscr{U}_{\omega}(\fS)$
are independent of the choice of $\lambda$.

\begin{remark}
  Theorem~\ref{thm-main-1} was proved in \cite{muller2016skein} and \cite{ishibashi2023skein} for $n=2$ and $n=3$, respectively, using direct computations.  
For general $n$, such a direct approach is impractical: a single flip involves too many mutation steps, and the resulting products of the elements $\{\gaa_{v,\lambda},\gaa_{v,\lambda'}\mid v\in V_\lambda\}$  are extremely difficult to compute in $\dS$ (we need to check some equalities in $\dS$, each side of such an equality is a product of $\{\gaa_{v,\lambda},\gaa_{v,\lambda'}\mid v\in V_\lambda\}$).  
Our proof of Theorem~\ref{thm-main-1}, by contrast, avoids these heavy computations and provides a more constructive argument.

\end{remark}

We conclude this section by introducing the $\mathcal A$-version quantum coordinate change isomorphism and demonstrating its compatibility with the $\mathcal A$-version quantum trace, a result of independent interest.

The proof of Theorem \ref{thm-main-1} shows that the isomorphism
$$\Theta_{
\lambda \lambda'}^{\omega}\colon\Fr(\mathcal Z_\omega^{\rm mbl}(\fS,\lambda'))\rightarrow 
\Fr(\mathcal Z_\omega^{\rm mbl}(\fS,\lambda))$$
restricts to an isomorphism 
$$\Theta_{
\lambda \lambda'}^{\omega}\colon\Fr(\mathcal Z_\omega^{\rm bl}(\fS,\lambda'))\rightarrow 
\Fr(\mathcal Z_\omega^{\rm bl}(\fS,\lambda))$$
for any two triangulations $\lambda,\lambda'$ of $\fS$,
which has been proved in \cite{KimWang} using a different technique.
Then commutative diagram \eqref{eq-compability-tr-mutation} implies that
the following diagram commutes (see also \cite{KimWang}):
    \begin{align}
        \label{eq-compability-tr-mutation-new}
        \xymatrix{
        & \rdS \ar[dl]_{{\rm tr}_{\lambda'}} \ar[dr]^{{\rm tr}_\lambda} & \\
        {\rm Frac}(\mathcal{Z}^{\rm bl}_\omega(\fS,\lambda')) \ar[rr]_-{\Theta^\omega_{\lambda\lambda'}} & & {\rm Frac}(\mathcal{Z}^{\rm bl}_\omega(\fS,\lambda))}.
    \end{align}
The proof of Theorem \ref{thm-main-1} also shows that the following diagram commutes
\begin{equation}\label{proof-thm-diag-1-new}
\begin{tikzcd}[row sep=3em, column sep=5em]
\Fr(\mathcal{Z}^{\rm bl}_\omega(\fS,\lambda'))\arrow[r, "\Theta_{\lambda\lambda'}^\omega"]
\arrow[d, "\varphi_{\lambda'}"]  
& \Fr(\mathcal{Z}^{\rm bl}_\omega(\fS,\lambda)) \arrow[d, "\varphi_\lambda"]  \\
 \Fr(\mathcal A_{\omega} (\fS,\lambda'))
 \arrow[r, "\mu_{v_1,A}\circ\cdots\circ \mu_{v_r,A}"] 
&  \Fr(\A)
\end{tikzcd},
\end{equation}
for any two triangulations $\lambda,\lambda'$ of $\fS$ related to each other by a flip.
Note that the isomorphisms $\varphi_\lambda$
and $\varphi_{\lambda'}$ are inverse to isomorphisms 
$\psi_\lambda$
and $\psi_{\lambda'}$ respectively (Lemma~\ref{thm-transition-LY}).

Recall that there is more than one way to order \( v_1, \ldots, v_r \) to obtain \eqref{eq-mutation-flip-quiver} when \( n > 2 \). Diagram \eqref{proof-thm-diag-1-new} demonstrates that \( \mu_{v_1,A} \circ \cdots \circ \mu_{v_r,A} \) is independent of the ordering of \( v_1, \ldots, v_r \) because \( \varphi_\lambda \) and \( \varphi_{\lambda'} \) are isomorphisms, and \( \Theta^\omega_{\lambda\lambda'} \) is independent of the ordering of \( v_1, \ldots, v_r \).
Then we define 
$$\Phi_{\lambda\lambda'}^\omega:=
\mu_{v_1,A}\circ\cdots\circ \mu_{v_r,A}\colon
\Fr(\mathcal A_{\omega} (\fS,\lambda'))\rightarrow
\Fr(\mathcal A_{\omega} (\fS,\lambda)).$$

Suppose that $\lambda$ and $\lambda'$ are any two triangulations of $\fS$ (see \S\ref{sec-naturality-trace}). 
Let $\Lambda=(\lambda_1,\cdots,\lambda_m)$ be a triangulation sweep 
connecting 
$\lambda$ and $\lambda'$.
Define the corresponding $\mathcal A$-version quantum coordinate change isomorphism 
\begin{equation}\label{eq-Phi-change}
\Phi_{\Lambda}^{\omega}  :=\Phi_{
\lambda_1\lambda_2}^{\omega}\circ\cdots\circ\Phi_{
\lambda_{m-1}\lambda_m}^{\omega}\colon\Fr(\mathcal A_{\omega} (\fS,\lambda'))\rightarrow 
\Fr(\mathcal A_{\omega} (\fS,\lambda)).
\end{equation}

The following shows that the definition of $\Phi_{\Lambda}^{\omega}$ is independent of the choice of $\Lambda$ and $\Phi_{\Lambda}^{\omega}$ relates the two quantum trace maps $\text{tr}_\lambda^A$
and $\text{tr}_{\lambda'}^A$. It is the $\mathcal A$-version for Proposition \ref{prop:Theta_omega_consistency} and Theorem \ref{thm-main-compatibility}.

\begin{corollary}\label{thm-naturality-sln-cluster}
    Let $\fS$ be a triangulable pb surface without interior punctures,  let $\lambda$, $\lambda'$ be two triangulations of $\fS$, and let $\Lambda$
    a triangulation sweep 
connecting 
$\lambda$ and $\lambda'$.
We have the following:

\begin{enumerate}[label={\rm (\alph*)}]\itemsep0,3em

\item $\Phi_\Lambda^\omega$ only depends on $\lambda$ and $\lambda'$.

\item We denote $\Phi_\Lambda^\omega$ as $\Phi_{\lambda\lambda'}^\omega$ because of (a). Then the following diagram commutes
\begin{equation}\label{eq-com-coord-change-A-X}
\begin{tikzcd}
 \Fr(\mathcal A_{\omega} (\fS,\lambda'))\arrow[r, "\Phi_{\lambda\lambda'}^\omega"]
\arrow[d, "\psi_{\lambda'}"]  
&  \Fr(\A) \arrow[d, "\psi_\lambda"]  \\
\Fr(\mathcal{Z}^{\rm bl}_\omega(\fS,\lambda'))
 \arrow[r, "\Theta_{\lambda\lambda'}^\omega"] 
& \Fr(\mathcal{Z}^{\rm bl}_\omega(\fS,\lambda))
\end{tikzcd},
\end{equation}
where $\psi_\lambda$ and $\psi_{\lambda'}$ are defined in \eqref{eq-iso-A-balanced-psi}.

\item The following diagram commutes
\begin{align}
       \label{eq-compability-tr-mutation-A}
        \xymatrix{
        & \rdS \ar[dl]_{{\rm tr}_{\lambda'}^A} \ar[dr]^{{\rm tr}_\lambda^A} & \\
        \Fr(\mathcal A_{\omega} (\fS,\lambda')) \ar[rr]_-{\Phi^\omega_{\lambda\lambda'}} & & \Fr(\mathcal A_{\omega} (\fS,\lambda))}.
\end{align}

\end{enumerate}
\end{corollary}
\begin{proof}
(a)    The diagram \eqref{proof-thm-diag-1-new} shows the following diagram commutes
    \begin{equation}\label{proof-thm-2-diag-1}
\begin{tikzcd}
\Fr(\mathcal{Z}^{\rm bl}_\omega(\fS,\lambda'))\arrow[r, "\Theta_{\Lambda}^\omega"]
\arrow[d, "\varphi_{\lambda'}"]  
& \Fr(\mathcal{Z}^{\rm bl}_\omega(\fS,\lambda)) \arrow[d, "\varphi_\lambda"]  \\
 \Fr(\mathcal A_{\omega} (\fS,\lambda'))
 \arrow[r, "\Phi_{\Lambda}^\omega"] 
&  \Fr(\A)
\end{tikzcd}.
\end{equation}
Since $\Theta_{\Lambda}^\omega$ depends only on $\lambda$ and $\lambda'$ (Proposition~\ref{prop:Theta_omega_consistency}), then 
so does $\Phi_{\Lambda}^\omega$.

(b) The commutative diagram \eqref{proof-thm-2-diag-1}, together with the identities \( (\psi_\lambda)^{-1} = \varphi_\lambda \) and \( (\psi_{\lambda'})^{-1} = \varphi_{\lambda'} \), implies the commutativity of diagram 
\eqref{eq-com-coord-change-A-X}.

(c) Lemma \ref{thm-trace-A}(d) demonstrates that
$$\psi_\lambda\circ \text{tr}_\lambda^A=\tr,\quad
\psi_{\lambda'}\circ \text{tr}_{\lambda'}^A= \text{tr}_{\lambda'}.$$
Then commutative diagrams \eqref{eq-compability-tr-mutation-new} and \eqref{eq-com-coord-change-A-X} imply the commutativity of diagram
\eqref{eq-compability-tr-mutation-A}.

\end{proof}

\begin{remark}
    In \cite{LY23}, the authors defined an isomorphism
    $$
    \overline{\Psi}_{\lambda\lambda'}^A \colon \Fr(\mathcal A_{\omega} (\fS,\lambda')) \longrightarrow
    \Fr(\mathcal A_{\omega} (\fS,\lambda))
    $$
    such that the diagram \eqref{eq-compability-tr-mutation-A} commutes when $\Phi_{\lambda\lambda'}^\omega$ is replaced by $\overline{\Psi}_{\lambda\lambda'}^A$.
    L{\^e} and Yu also proved the uniqueness of $\overline{\Psi}_{\lambda\lambda'}^A$ that makes this diagram commute \cite{LY23}. 
    It follows that
    $$
    \overline{\Psi}_{\lambda\lambda'}^A = \Phi_{\lambda\lambda'}^\omega.
    $$
\end{remark}

\section{The inclusion of skein algebras into quantum upper cluster algebras}\label{sec-skein-cluster}

The following theorem is the main result of this section. It asserts that the projected stated ${\rm SL}_n$-skein algebra $\dS$ (Definition~\ref{def-key-algebra}) is contained in the ${\rm SL}_n$ quantum cluster upper algebra $\mathscr{U}_{\omega}(\fS)$ (Definition~\ref{def-sln-quantum-cluster-al}) whenever $\fS$ is a triangulable pb surface without interior punctures.

\begin{theorem}\label{thm-main-3}
    Let $\fS$ be a triangulable pb surface without interior punctures. Then
    $$
    \dS \subset \mathscr{U}_{\omega}(\fS).
    $$
\end{theorem}

In \S\ref{sec-splitting-upper}, we will construct the splitting homomorphism for $\mathscr{U}_{\omega}(\fS)$ and prove that it is compatible with the splitting homomorphism of $\dS$ introduced in \S\ref{sub-splitting}.  
Finally, in \S\ref{sec-skein-inclusion-cluster}, we establish the stronger inclusion 
$\dS \subset \mathscr{A}_{\omega}(\fS)$ when every connected component of $\fS$ contains at least two punctures, where $\mathscr{A}_{\omega}(\fS)$ denotes the ${\rm SL}_n$ quantum cluster algebra defined in Definition~\ref{def-sln-quantum-cluster-al}.  

Although it is well known that $\mathscr{A}_{\omega}(\fS) \subset \mathscr{U}_{\omega}(\fS)$, it is necessary to first show $\dS \subset \mathscr{U}_{\omega}(\fS)$. The reason is that our proof of $\dS \subset \mathscr{A}_{\omega}(\fS)$ relies on cutting $\fS$ into a new pb surface and a triangle or quadrilateral, which in turn requires $\dS \subset \mathscr{U}_{\omega}(\fS)$ and the compatibility between the splitting homomorphisms of $\dS$ and $\mathscr{U}_{\omega}(\fS)$.

\subsection{Quantum trace maps for triangles and quadrilaterals}

Recall that, for each positive integer $k$, we use $\mathbb P_k$ to denote the pb surface obtained from the closed disk by removing $k$ punctures on its boundary.

We label the three vertices and three edges of $\mathbb P_3$ as $v_1,v_2,v_3$ and $e_1,e_2,e_3$, respectively, as shown in Figure~\ref{Fig;coord_ijk}.  
Note that once $v_1$ is fixed, the labeling of all other vertices and edges is uniquely determined.
Recall that the `$n$-triangulation' of $\mathbb P_3$ is the subdivision of $\bP_3$ into $n^2$ small triangles (Figure~\ref{Fig;coord_ijk}).
 Consider a graph dual to this $n$-triangulation, and give orientations on the edges as in Figure \ref{dualquiver}, presented as blue arrows there. The resulting oriented graph, consisting of $3n$ vertices on the boundary of $\mathbb{P}_3$, $n^2$ vertices in the interior, and $\frac{3n(n+1)}{2}$ oriented edges, is referred to as a `network' in \cite{schrader2017continuous}; let's denote it as $\mathcal N({\mathbb P_3},v_1)$.
 The vertices of $\mathcal N({\mathbb P_3},v_1)$  contained in $e_1$ (and $e_3$)
 are labeled by $1,2,\cdots,n$, as shown in Figure \ref{dualquiver}.

\begin{figure}[h]
	\centering
	\includegraphics[width=4.2cm]{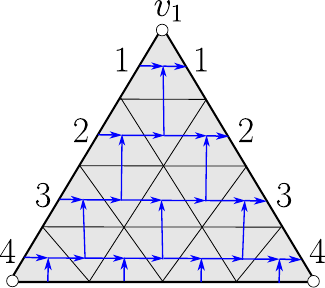}
	\caption{The dual quiver for the one in Figure \ref{Fig;coord_ijk} when $n=4$.}\label{dualquiver}
\end{figure}

 For any $i,j\in\{1,2,\cdots,n\}$, 
 define
 \begin{align}\label{def-path-P3-v1}
     \mathsf P(\mathbb P_3,v_1,i,j):=
     \{\text{paths in $\mathcal N({\mathbb P_3},v_1)$ going from $i$ contained in $e_3$ to $j$ contained in $e_1$}\},
 \end{align}
  where a path is a concatenation of edges of $\mathcal N({\mathbb P_3},v_1)$ respecting the orientations.

 Note that $\mathsf P(\mathbb P_3,v_1,i,j)=\emptyset$ when $i<j$, and 
 $\mathsf P(\mathbb P_3,v_1,i,j)$ consists of one element when $i=j$, which we may denote by $\gamma_{ii}$.
 Recall that $V_{\mathbb P_3}$ is the vertex set of the quiver in Figure \ref{Fig;coord_ijk}.
 For each path $\gamma\in \mathsf P(\mathbb P_3,v_1,i,j)$, we define an integer-valued vector
 ${\bf k}^\gamma  =(k^\gamma_v)_{v\in V_{\mathbb{P}_3}} \in\mathbb Z^{V_{\mathbb P_3}}$ 
 as
 \begin{align}\label{eq-def-k-gamma}
 k^\gamma_v =
 \begin{cases}
     1 & \text{if $v$ is on your left when you walk along $\gamma$,}\\
     0 & \text{otherwise}.
 \end{cases}
 \end{align}
 For example, for the unique path $\gamma$ from $2$ to $1$, there is exactly one $v \in V_{\mathbb{P}_3}$ with $k^\gamma_v=1$, for the vertex $v$ between 1 and 2 contained in $e_3$. For the unique path $\gamma_{22}$ from $2$ to $2$, there are two $v' \in V_{\mathbb{P}_3}$ with $k^{\gamma_{22}}_{v'}=1$.

Define 
\begin{align}\label{eq-k-sum-k1}
    {\bf k}=
    \sum_{i=\{1,2,\cdots,n\}} \, \sum_{\gamma\in \mathsf P(\mathbb P_3,v_1,i,i)} 
    {\bf k}^\gamma \, \in \mathbb{Z}^{V_{\mathbb{P}_3}}.
\end{align}
Note that ${\bf k}^{\gamma_{11}}=0$ and
$${\bf k} = \sum_{i=1}^n {\bf k}^{\gamma_{ii}} = {\bf k}^{\gamma_{22}} +\cdots + {\bf k}^{\gamma_{nn}}={\bf k}_1,$$
where ${\bf k}_1$ is defined in \eqref{def:proj}.

Let $v\in\{v_1,v_2,v_3\}$, and let $\alpha(v)$ be a counterclockwise corner arc in $\mathbb{P}_3$, surrounding the puncture $p$.
For $i,j\in\{1,2,3\}$, let $\alpha(v)_{ij} \in \overline{\cS}_{\omega}(\mathbb{P}_3)$ be the stated corner arc whose state values at the initial and terminal endpoints are $i$ and $j$, respectively.

The following lemma gives a neat description of the value of the quantum trace of a stated corner arc in a triangle.
\begin{theorem}[{\cite[Theorem~10.5]{LY23}}]\label{lem-trace-image-cornerarc-P3}
    We have
    \begin{align}\label{eq-trace-P3-arc}
{\rm tr}_{\mathbb P_3}(\alpha(v_1)_{ij})
=\sum_{\gamma\in\mathsf{P}(\mathbb P_3,v_1,i,j)}
Z^{n{\bf k}^\gamma -{\bf k}},
\end{align}
where ${\bf k}^\gamma$ and ${\bf k}$ are defined in \eqref{eq-def-k-gamma} and \eqref{eq-k-sum-k1} respectively. 
\end{theorem}

\def\Pt{\mathbb P_3}

We use $\mathcal D_{\Pt}$ to denote the $\mathcal X$-seed $(\Gamma_{\Pt},(X_v)_{v\in V_{\Pt}})$. 
Recall that, for any mutable vertex $k$, we defined $\mathcal D_{\Pt}^{(k)}$ (see \eqref{eq-def-wD}) and $\mathcal Z_{\omega}^{\rm bl}(\mathcal D_{\Pt}^{(k)})$ (see \eqref{def-Z-D-v-lambda}). 
Note that $\mathcal D_{\Pt}=\mathcal D_{\Pt}^{(k,k)}$ and $\mathcal Z_{\omega}^{\rm bl}(\mathcal D_{\Pt})=\mathcal Z_{\omega}^{\rm bl}(\mathcal D_{\Pt}^{(k,k)})$, since 
$\mu_k(\mu_k(H_{\Pt}))=H_{\Pt}$ (Lemma~\ref{lem-mutation-H}(b)).
Therefore, by Lemma~\ref{lem-rest-iso-to-balanced}, there exists a skew-field isomorphism
\[
\nu_{k}^{\omega}\colon \Fr(
      \mathcal{Z}^{\rm bl}_\omega(\mathcal{D}_{\Pt}))\;\longrightarrow\;
      \Fr(
      \mathcal{Z}^{\rm bl}_\omega(\mathcal{D}_{\mathbb P_3}^{(k)})),
\]
which extends the map
\[
\mu_k^q\colon \Fr(\mathcal{X}_q(\mathcal{D}_{\Pt}))\;\longrightarrow\;    \Fr(\mathcal{X}_q(\mathcal{D}_{\lambda}^{(k)})),
\]
as established in Lemma~\ref{lem:nu_extends_mu}.

\begin{lemma}\label{lem-key-mutation-polynomial-P3}
    For any element $\beta\in \overline{\cS}_\omega(\mathbb P_3)$, we have 
    $$\nu_k^{\omega}({\rm tr}_{\mathbb P_3}(\beta))\in 
      \mathcal{Z}^{\rm bl}_\omega(\mathcal{D}_{\mathbb P_3}^{(k)}).$$
\end{lemma}
\begin{proof}
 \cite[Theorem~6.1(c)]{LY23} implies that 
$\overline{\cS}_\omega(\mathbb P_3)$ is generated by 
\[
\{\alpha(v)_{ij}\mid 1\leq j\leq i\leq n,\; v=v_1,v_2,v_3\}.
\]
Since any vertex of $\mathbb P_3$ may be chosen as $v_1$, it suffices to show that 
\[
\nu_k^{\omega}\bigl({\rm tr}_{\mathbb P_3}(\alpha(v_1)_{ij})\bigr)\in 
      \mathcal{Z}^{\rm bl}_\omega(\mathcal{D}_{\mathbb P_3}^{(k)})
      \quad \text{for } 1\leq j\leq i\leq n.
\]

By Theorem~\ref{lem-trace-image-cornerarc-P3}, we have
\[
{\rm tr}_{\mathbb P_3}(\alpha(v_1)_{ij})
=\sum_{\gamma\in\mathsf{P}(\mathbb P_3,v_1,i,j)}
Z^{n{\bf k}^\gamma -{\bf k}}.
\]
It is straightforward to check that
\begin{align}\label{eq-commute-Zk-interior}
\sum_{u\in V_{\mathbb P_3}}{\bf k}(u)\, Q_{\mathbb P_3}(u,u')=0,
\quad\text{and}\quad
Z^{{\bf k}} Z_{u'} = Z_{u'} Z^{{\bf k}},
\end{align}
whenever $u'\in V_{\mathbb P_3}$ lies in the interior of $\mathbb P_3$.
Moreover, for any two paths $\gamma,\gamma'\in \mathsf{P}(\mathbb P_3,v_1,i,j)$ and any $u\in V_{\mathbb P_3}$ contained in the boundary, we have 
${\bf k}^\gamma(u) = {\bf k}^{\gamma'}(u)$. 
Thus
\[
Z^{n{\bf k}^\gamma -{\bf k}}
=\omega^{m_{ij}}Z^{-{\bf k}} Z^{n{\bf k}^\gamma},
\]
for some $m_{ij}\in\mathbb Z$ independent of $\gamma$. Hence,
\[
{\rm tr}_{\mathbb P_3}(\alpha(v_1)_{ij})
=\omega^{m_{ij}} Z^{-{\bf k}}
\sum_{\gamma\in\mathsf{P}(\mathbb P_3,v_1,i,j)}
Z^{n{\bf k}^\gamma}.
\]

It follows that
\begin{align*}
    \nu_k^{\omega}\bigl({\rm tr}_{\mathbb P_3}(\alpha(v_1)_{ij})\bigr) 
    &= \omega^{m_{ij}}  \nu_k^{\omega}(Z^{-{\bf k}})\,
    \nu_k^{\omega}\!\left(\sum_{\gamma\in\mathsf{P}(\mathbb P_3,v_1,i,j)}
Z^{n{\bf k}^\gamma}\right)\\
&= \omega^{m_{ij}}  \nu_k^{\omega}(Z^{-{\bf k}})\,
    \mu_k^{q}\!\left(\sum_{\gamma\in\mathsf{P}(\mathbb P_3,v_1,i,j)}
X^{{\bf k}^\gamma}\right).
\end{align*}

By \cite[Proposition~4.2]{schrader2017continuous}, we know that 
\[
\mu_k^{q}\!\left(\sum_{\gamma\in\mathsf{P}(\mathbb P_3,v_1,i,j)}
X^{{\bf k}^\gamma}\right)\in \mathcal{Z}^{\rm bl}_\omega(\mathcal{D}_{\mathbb P_3}^{(k)}).
\]
On the other hand, by definition of $\nu_k'$, we have
\[
\nu_k'(Z^{-{\bf k}})=(Z')^{{\bf k}'}
\quad\text{for some }{\bf k}'\in \mathbb Z^{V_\lambda}\qquad \text{(see \eqref{eq-quantum-mutation_Z})}.
\]
Then Lemma~\ref{lem-commute} and \eqref{eq-commute-Zk-interior} imply
\begin{align}\label{eq-proof-P3-zero-Q}
\sum_{u\in V_{\mathbb P_3}}{\bf k}'(u) Q_{\mathbb P_3}'(u,k)=0,
\end{align}
where $Q_{\mathbb P_3}'=\nu_k(Q_{\mathbb P_3})$.
Recall that $\nu_k^{\omega}= \nu_k^{\omega\sharp}\circ \nu_k'$ (see \eqref{eq-def-quantum-mutation-X}).
Therefore,
\[
\nu_k^{\omega}(Z^{-{\bf k}})
=\nu_k^{\omega\sharp}(\nu_k'(Z^{-{\bf k}}))
=\nu_k^{\omega\sharp}\bigl((Z')^{{\bf k}'}\bigr)
= (Z')^{{\bf k}'}
\;\in\; \mathcal{Z}^{\rm bl}_\omega(\mathcal{D}_{\mathbb P_3}^{(k)}),
\]
where the last equality comes from \eqref{eq-proof-P3-zero-Q} and Lemma~\ref{lem.nu_sharp_well-defined}.

This completes the proof.

\end{proof}

We now state a result for $\mathbb P_4$ analogous to Lemma~\ref{lem-key-mutation-polynomial-P3}.  
Note that $\mathbb P_4$ admits two triangulations $\lambda_4,\lambda_4'$; see Figure~\ref{P4_two_triangulations}.

\begin{figure}[h]
    \centering \includegraphics{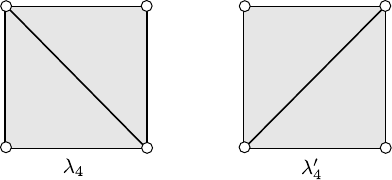}
    \caption{Two triangulations $\lambda_4$ and $\lambda_4'$ of $\mathbb{P}_4$.}\label{P4_two_triangulations}
\end{figure}

Consider the three corner arcs $a,b,c$ in $\mathbb P_4$, shown in Figure~\ref{P4_corner_arcs}. 
For $i,j\in\{1,2,\ldots,n\}$, let $a_{ij}$ denote the stated web diagram whose starting point (resp.\ endpoint) of $a$ is labeled by $i$ (resp.\ $j$). 
Similarly, we define $b_{ij}$ and $c_{ij}$. 
We then have the following.

\begin{figure}[h]
    \centering
\begingroup%
  \makeatletter%
  \providecommand\color[2][]{%
    \errmessage{(Inkscape) Color is used for the text in Inkscape, but the package 'color.sty' is not loaded}%
    \renewcommand\color[2][]{}%
  }%
  \providecommand\transparent[1]{%
    \errmessage{(Inkscape) Transparency is used (non-zero) for the text in Inkscape, but the package 'transparent.sty' is not loaded}%
    \renewcommand\transparent[1]{}%
  }%
  \providecommand\rotatebox[2]{#2}%
  \newcommand*\fsize{\dimexpr\f@size pt\relax}%
  \newcommand*\lineheight[1]{\fontsize{\fsize}{#1\fsize}\selectfont}%
  \ifx\svgwidth\undefined%
    \setlength{\unitlength}{96.37795276bp}%
    \ifx\svgscale\undefined%
      \relax%
    \else%
      \setlength{\unitlength}{\unitlength * \real{\svgscale}}%
    \fi%
  \else%
    \setlength{\unitlength}{\svgwidth}%
  \fi%
  \global\let\svgwidth\undefined%
  \global\let\svgscale\undefined%
  \makeatother%
  \begin{picture}(1,1)%
    \lineheight{1}%
    \setlength\tabcolsep{0pt}%
    \put(0,0){\includegraphics[width=\unitlength,page=1]{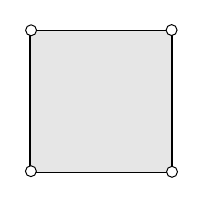}}%
    \put(0.03084271,0.49939851){\color[rgb]{0,0,0}\makebox(0,0)[lt]{\lineheight{1.25}\smash{\begin{tabular}[t]{l}$e_1$\end{tabular}}}}%
    \put(0.40414908,0.27322287){\color[rgb]{0,0,0}\makebox(0,0)[lt]{\lineheight{1.25}\smash{\begin{tabular}[t]{l}$a$\end{tabular}}}}%
    \put(0.33709571,0.56961649){\color[rgb]{0,0,0}\makebox(0,0)[lt]{\lineheight{1.25}\smash{\begin{tabular}[t]{l}$b$\end{tabular}}}}%
    \put(0.57763257,0.60308036){\color[rgb]{0,0,0}\makebox(0,0)[lt]{\lineheight{1.25}\smash{\begin{tabular}[t]{l}$c$\end{tabular}}}}%
    \put(0.87128392,0.49939851){\color[rgb]{0,0,0}\makebox(0,0)[lt]{\lineheight{1.25}\smash{\begin{tabular}[t]{l}$e_4$\end{tabular}}}}%
    \put(0.46050854,0.89131787){\color[rgb]{0,0,0}\makebox(0,0)[lt]{\lineheight{1.25}\smash{\begin{tabular}[t]{l}$e_2$\end{tabular}}}}%
    \put(0.46050854,0.05087684){\color[rgb]{0,0,0}\makebox(0,0)[lt]{\lineheight{1.25}\smash{\begin{tabular}[t]{l}$e_5$\end{tabular}}}}%
    \put(0,0){\includegraphics[width=\unitlength,page=2]{P4_corner_arcs.pdf}}%
  \end{picture}%
\endgroup%

    \caption{Three corner arcs $a,b,c$ in $\mathbb P_4$. The four edges of $\mathbb P_4$ are labeled by $e_1,e_2,e_4,e_5$.}\label{P4_corner_arcs}
\end{figure}

\begin{lemma}\cite{LY23,KimWang}\label{lem-satuared-rd}
    The reduced stated skein algebra $\rdP$ is generated by
    $a_{ij},b_{ij},c_{ij}$ for $i,j \in \{1,\ldots,n\}$ with $i\geq j$.
\end{lemma}

Recall that the $n$-triangulation of $\lambda_4$ in $\mathbb P_4$ is the subdivision of $\mathbb P_4$ into $2n^2$ small triangles.
Consider three graphs $\mathcal N_a$, $\mathcal N_b$, and $\mathcal N_c$ dual to this $n$-triangulation, and orient their edges as in Figure~\ref{network1}, where the orientations are drawn in blue. 
Each oriented graph consists of $4n$ vertices on the boundary of $\mathbb{P}_4$ ($2n$ of them are labeled as in Figure~\ref{network1}), $2n^2-n$ interior vertices, and $3n^2+2n$ oriented edges.  
Note that, disregarding edge orientations, these three graphs are identical: each is the dual graph of the $n$-triangulation of $\lambda_4$.


\begin{figure}
	\centering
    \includegraphics[width=15cm]{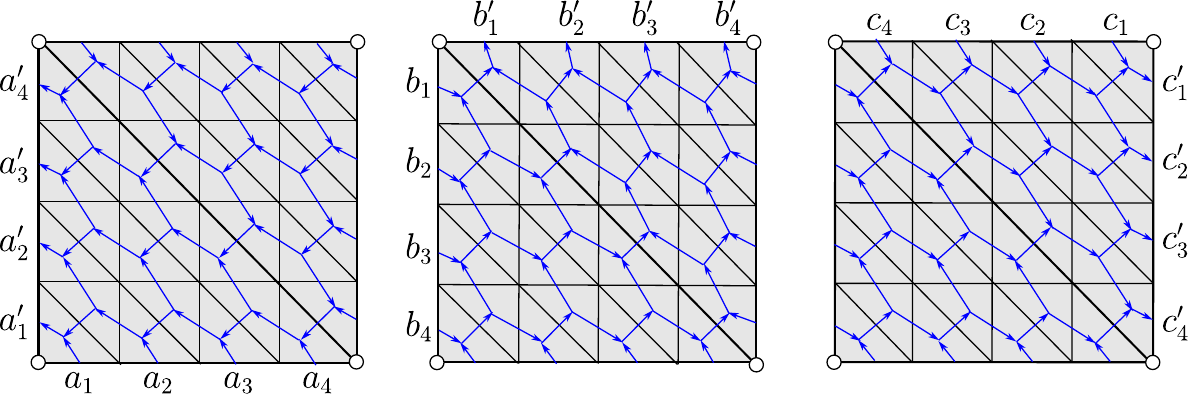}
	\caption{For $n=4$, the left, middle, and right dual graphs correspond to $\mathcal N_a$, $\mathcal N_b$, and $\mathcal N_c$, respectively.}\label{network1}
\end{figure}

For $i,j\in\{1,2,\ldots,n\}$ and $x\in\{a,b,c\}$, let 
$\mathsf P(x,\lambda_4,i,j)$ denote the set of oriented paths in $\mathcal N_x$ from $x_i$ to $x_j'$, where a path is a concatenation of edges in $\mathcal N_x$ respecting their orientations.  
Note that $\mathsf P(x,\lambda_4,i,j)=\emptyset$ if $i<j$, and $\mathsf P(x,\lambda_4,i,i)$ consists of a unique path $\gamma_{ii}^x$.

For each $\gamma\in \mathsf P(x,\lambda_4,i,j)$, define a vector ${\bf k}^\gamma_x=(k^\gamma_v)_{v\in V_{\lambda_{4}}}\in\mathbb Z^{V_{\lambda_{4}}}$ by
\begin{align}
 k^\gamma_v =
 \begin{cases}
     1 & \text{if $v$ lies to the left of $\gamma$,}\\
     0 & \text{otherwise}.
 \end{cases}
\end{align}

For each $x\in\{a,b,c\}$, set
\begin{align}
    {\bf k}_x=
    \sum_{i=1}^n \;\sum_{\gamma\in \mathsf P(x,\lambda_{4},i,i)} 
    {\bf k}^\gamma_x \;\in \mathbb{Z}^{V_{\lambda_{4}}}.
\end{align}
Note that 
${\bf k}_x = \sum_{i=1}^n {\bf k}^{\gamma_{ii}^x}_x = {\bf k}^{\gamma_{22}^x}_x + \cdots + {\bf k}^{\gamma_{nn}^x}_x$, since ${\bf k}^{\gamma_{11}^x}_x=0$.

\begin{lemma}\cite{KimWang}\label{lem-quantum-trace-P4}
     For any $i,j\in\{1,2,\ldots,n\}$ and $x\in\{a,b,c\}$, we have 
     \[
     {\rm tr}_{\lambda_4}(x_{ij}) = \sum_{\gamma\in\mathsf{P}(x,\lambda_{4},i,j)}
     Z^{\,n{\bf k}^\gamma_x -{\bf k}_x}.
     \]
\end{lemma}

We use $\mathcal D_{\lambda_{4}}$ to denote the $\mathcal X$-seed $(\Gamma_{\lambda_{4}},(X_v)_{v\in V_{\lambda_{4}}})$ associated to a triangulation $\lambda_{4}$ of $\mathbb P_4$. 
Recall from \S\ref{sec-compatibility} that, for any mutable vertex $k$, we defined $\mathcal D_{\lambda_{4}}^{(k)}$ and $\mathcal Z_{\omega}^{\rm bl}(\mathcal D_{\lambda_{4}}^{(k)})$. 
Note that $\mathcal D_{\lambda_{4}}=\mathcal D_{\lambda_{4}}^{(k,k)}$ and $\mathcal Z_{\omega}^{\rm bl}(\mathcal D_{\lambda_{4}})=\mathcal Z_{\omega}^{\rm bl}(\mathcal D_{\lambda_{4}}^{(k,k)})$.
By Lemma~\ref{lem-rest-iso-to-balanced}, there exists a skew-field isomorphism
\[
\nu_{k}^{\omega}\colon \Fr (
      \mathcal{Z}^{\rm bl}_\omega(\mathcal{D}_{\lambda_{4}}))\longrightarrow\;
      \Fr(
      \mathcal{Z}^{\rm bl}_\omega(\mathcal{D}_{\lambda_{4}}^{(k)})).
\]

\begin{lemma}\label{lem-key-mutation-polynomial-P4}
    For any element $\beta\in \overline{\cS}_\omega(\mathbb P_4)$, we have 
    \[
    \nu_k^{\omega}({\rm tr}_{\lambda_{4}}(\beta))\in 
      \mathcal{Z}^{\rm bl}_\omega(\mathcal{D}_{\lambda_{4}}^{(k)}).
    \]
\end{lemma}

\begin{proof}
Using Lemma~\ref{lem-quantum-trace-P4}, the proof proceeds analogously to Lemma~\ref{lem-key-mutation-polynomial-P3}.
\end{proof}

The following proposition extends Lemmas~\ref{lem-key-mutation-polynomial-P3} and \ref{lem-key-mutation-polynomial-P4} to arbitrary triangulable pb surfaces without interior punctures, serving as a key step in the proof of Theorem~\ref{thm-main-3}.

\begin{proposition}\label{prop-key-bal-polynomial}
    Let $\fS$ be a triangulable pb surface without interior punctures, let $\lambda$ be a triangulation of $\fS$, and let $k\in V_\lambda$ be a mutable vertex. 
    For any element $\beta\in \overline{\cS}_\omega(\fS)$, we have 
    $$\nu_k^{\omega}({\rm tr}_{\lambda}(\beta))\in 
      \mathcal{Z}^{\rm bl}_\omega(\mathcal{D}_{\lambda}^{(k)}),$$
    where the isomorphism $\nu_k^{\omega} \colon \Fr (
      \mathcal{Z}^{\rm bl}_\omega(\mathcal{D}_{\lambda}))\longrightarrow\;
      \Fr(
      \mathcal{Z}^{\rm bl}_\omega(\mathcal{D}_{\lambda}^{(k)}))$ is 
      defined using the mutation $\mu_k(\mathcal{D}_{\lambda}^{(k)})=\mathcal{D}_{\lambda}^{(k)}$.
    
\end{proposition}
\begin{proof}
    Define a subset $E$ of $\lambda$ such that $$E=\begin{cases}
        \lambda\setminus\{e\} & \text{$k$ is contained in $e$ for some $e\in\lambda$},\\
        \lambda & \text{otherwise}.
    \end{cases}$$
The surface obtained from $\fS$ by cutting along ideal arcs in $E$ is the disjoint union of a pb surface $\fS'$ consisting of $\mathbb P_3$ and a polygon $\mathbb P$ containing $k$. Note that $\mathbb P=\mathbb P_3$ or $\mathbb P_4.$
The triangulation $\lambda$ of $\fS$ induces a triangulation $\lambda_E$ of $\fS'\sqcup \mathbb P$, which induces a triangulation $\lambda'$ for $\fS'$ and a triangulation $\lambda_{\mathbb P}$ for $\mathbb P$. 
Then $$\mathcal{Z}^{\rm bl}_\omega(\mathcal{D}_{\lambda_E})=
\mathcal{Z}^{\rm bl}_\omega(\mathcal{D}_{\lambda'})\otimes_R \mathcal{Z}^{\rm bl}_\omega(\mathcal{D}_{\lambda_{\mathbb P}})\text{ and }
\overline{\cS}_{\omega}(\fS'\cup\mathbb P)=\overline{\cS}_{\omega}(\fS')\otimes_R \overline{\cS}_{\omega}(\mathbb P).$$
Define 
\begin{align*}
    \mathcal S_E&:=\prod_{e\in E}\mathcal S_e\colon
\mathcal{Z}^{\rm bl}_\omega(\mathcal{D}_{\lambda})\rightarrow
\mathcal{Z}^{\rm bl}_\omega(\mathcal{D}_{\lambda_E}).\\
 \mathbb S_E&:=\prod_{e\in E}\mathbb S_e\colon
\overline{\cS}_{\omega}(\fS) \rightarrow
\overline{\cS}_{\omega}(\fS'\cup\mathbb P). 
\end{align*}
 Note that $\mathcal S_E$ and $\mathbb S_E$ are well-defined because of the commutativity of the splitting homomorphisms (see \eqref{com-splitting-skein} and \eqref{com-splitting-torus}).

Let $e\in E$, we use $\lambda_e$ to denote the triangulation of $\mathsf{Cut}_e(\fS)$ induced by $\lambda$. 
Since $k$ is contained in the interior of $\mathbb P$, we can define an algebra embedding $\mathcal S_e^{(k)} \colon
\mathcal{Z}_\omega(\mathcal{D}_{\lambda}^{(k)})\rightarrow
\mathcal{Z}_\omega(\mathcal{D}_{\lambda_e}^{(k)})$ using \eqref{com-splitting-torus}.
Similarly as $\mathcal S_E$, define
$$\mathcal S_E^{(k)}:=\prod_{e\in E}\mathcal S_e^{(k)}\colon
\mathcal{Z}_\omega(\mathcal{D}_{\lambda}^{(k)})\rightarrow
\mathcal{Z}_\omega(\mathcal{D}_{\lambda_E}^{(k)}).$$
Lemma~\ref{eq-bl-inclusion-mut-vu} and \cite[Equation~(39)]{KimWang} imply that the above algebra embedding restricts to the following one:
$$\mathcal S_E^{(k)}\colon
\mathcal{Z}^{\rm bl}_\omega(\mathcal{D}_{\lambda}^{(k)})\rightarrow
\mathcal{Z}^{\rm mbl}_\omega(\mathcal{D}_{\lambda_E}^{(k)}),$$
\cite[Equation~(43)]{KimWang} implies that the following diagram commutes
\begin{align}\label{eq-com-S-vi-new}
\xymatrix{
{\rm Frac}(\mathcal{Z}^{\rm bl}_\omega(\mathcal{D}_{\lambda})) \ar[r]^{\mathcal{S}_E} \ar[d]_{\nu^\omega_{k}} & {\rm Frac}(\mathcal{Z}^{\rm bl}_\omega(\mathcal{D}_{\lambda_E})) \ar[d]^{\nu^\omega_{k}} \\
{\rm Frac}(\mathcal{Z}^{\rm bl}_\omega(\mathcal{D}_{\lambda}^{(k)})) \ar[r]^{\mathcal{S}_E^{(k)}} & {\rm Frac}(\mathcal{Z}^{\rm mbl}_\omega(\mathcal{D}_{\lambda_E}^{(k)}))
}.
\end{align}

Since $\nu^\omega_k\big({\rm Frac}(\mathcal{Z}^{\rm bl}_\omega(\mathcal{D}_{\lambda_E}))\big)\subset {\rm Frac}(\mathcal{Z}^{\rm bl}_\omega(\mathcal{D}_{\lambda_E}^{(k)}))$ and $\nu^\omega_k\colon {\rm Frac}(\mathcal{Z}^{\rm bl}_\omega(\mathcal{D}_{\lambda}))\rightarrow {\rm Frac}(\mathcal{Z}^{\rm bl}_\omega(\mathcal{D}_{\lambda}^{(k)}))$ is an isomorphism, we have 
$$\mathcal S_E^{(k)}\big(
{\rm Frac}(\mathcal{Z}^{\rm bl}_\omega(\mathcal{D}_{\lambda}^{(k)}))\big)
\subset 
{\rm Frac}(\mathcal{Z}^{\rm bl}_\omega(\mathcal{D}_{\lambda_E}^{(k)})).$$
Then the commutative diagram in \eqref{eq-com-S-vi-new} implies the following commutative diagram:
\begin{align}\label{eq-com-S-vi}
\xymatrix{
{\rm Frac}(\mathcal{Z}^{\rm bl}_\omega(\mathcal{D}_{\lambda})) \ar[r]^{\mathcal{S}_E} \ar[d]_{\nu^\omega_{k}} & {\rm Frac}(\mathcal{Z}^{\rm bl}_\omega(\mathcal{D}_{\lambda_E})) \ar[d]^{\nu^\omega_{k}} \\
{\rm Frac}(\mathcal{Z}^{\rm bl}_\omega(\mathcal{D}_{\lambda}^{(k)})) \ar[r]^{\mathcal{S}_E^{(k)}} & {\rm Frac}(\mathcal{Z}^{\rm bl}_\omega(\mathcal{D}_{\lambda_E}^{(k)}))
}.
\end{align}

Theorem \ref{thm-trace-cut} implies that the following diagram commutes
\begin{equation}\label{eq-com-tr-splitting-E}
\begin{tikzcd}
\rdS \arrow[r, "\mathbb S_E"]
\arrow[d, "\tr"]  
&  \overline{\cS}_{\omega}(\fS'\cup\mathbb P) \arrow[d, "{\rm tr}_{\lambda_E}"] \\
 \mathcal{Z}_\omega(\mathcal D_{\lambda})
 \arrow[r, "\mathcal S_E"] 
&  
\mathcal{Z}_\omega(\mathcal D_{\lambda_E})
\end{tikzcd}.
\end{equation}
Then
we have 
\begin{align*}
  \mathcal S_E ^{(k)} (\nu_k^{\omega}({\rm tr}_{\lambda}(\beta)))=
  \nu_k^{\omega}(\mathcal S_E({\rm tr}_{\lambda}(\beta)))
  = \nu_k^{\omega}({\rm tr}_{\lambda_E}(\mathbb S_E(\beta))).
\end{align*}
Lemmas \ref{lem-key-mutation-polynomial-P3} and \ref{lem-key-mutation-polynomial-P4} show that $\mathcal S_E ^{(k)} (\nu_k^{\omega}({\rm tr}_{\lambda}(\beta)))
  = \nu_k^{\omega}({\rm tr}_{\lambda_E}(\mathbb S_E(\beta)))\in \mathcal{Z}^{\rm bl}_\omega(\mathcal{D}_{\lambda_E}^{(k)}).$

The definition of $\mathcal S_E^{(k)}$ (also \eqref{cutting_homomorphism_for_Z_omega}) implies that 
$\mathcal S_E^{(k)}\big(\mathcal{Z}_\omega(\mathcal D_{\lambda}^{(k)})\big)$ 
is a group subalgebra of the quantum torus 
$\mathcal{Z}_\omega(\mathcal D_{\lambda_E}^{(k)})$ (see \eqref{submonoid}).  
Together with 
$\mathcal S_E^{(k)}\big(\nu_k^{\omega}({\rm tr}_{\lambda}(\beta))\big)
\in \mathcal{Z}^{\rm bl}_\omega(\mathcal{D}_{\lambda_E}^{(k)})$,  
Lemma~\ref{lem-frac-subalgebra} then implies that 
\[
\nu_k^{\omega}({\rm tr}_{\lambda}(\beta)) \in 
\mathcal{Z}^{\rm bl}_\omega(\mathcal{D}_{\lambda}^{(k)}).
\]

\end{proof}

\begin{remark}
    Let $\fS$ be a triangulable pb surface with two triangulations $\lambda$ and $\lambda'$. Suppose $e\in\lambda$ is not a boundary edge, and 
 that $\lambda$ and $\lambda'$ are obtained from each other by a flip on the ideal arc $e$.
Recall that there is a sequence of mutations 
$\mu_{v_1},\cdots,\mu_{v_r}$ (see Figure \ref{Fig;mutation_sequence_for_flip}) such that 
$$\mathcal D_{\lambda'} = \mu_{v_r}\cdots \mu_{v_1}(\mathcal D_\lambda).$$
Define $\mathbbm{v}:=(v_1,\cdots,v_i)$.
Let $\mathbb P_4$ be the quadrilateral containing $e$ as a diagonal ideal arc.
Using the same technique, we can easily generalize Lemma \ref{lem-key-mutation-polynomial-P4} to the following:

For any element $\beta\in \overline{\cS}_\omega(\mathbb P_4)$, we have 
\begin{align}\label{eq-P4-sequence}
    \nu_{v_i}^\omega\cdots \nu_{v_1}^{\omega}({\rm tr}_{\lambda_4}(\beta))\in 
       \mathcal{Z}^{\rm mbl}_\omega(\mathcal{D}_{\lambda_4}^{\mathbbm{v}}).
\end{align}
    
Using \eqref{eq-P4-sequence} and the method of Proposition~\ref{prop-key-bal-polynomial} (with every occurrence of $\mathcal Z^{\rm bl}$ replaced by $\mathcal Z^{\rm mbl}$), we obtain the following:

 For any element $\beta\in \overline{\cS}_\omega(\fS)$, we have 
\begin{align}\label{eq-fS-sequence}
    \nu_{v_i}^\omega\cdots \nu_{v_1}^{\omega}({\rm tr}_{\lambda}(\beta))\in 
       \mathcal{Z}^{\rm mbl}_\omega(\mathcal{D}_{\lambda}^{\mathbbm{v}}).
\end{align}
Note that there is no restriction that $\fS$ has no interior punctures for \eqref{eq-fS-sequence}.

\end{remark}

\subsection{Proof of Theorem \ref{thm-main-3}}
\begin{proof}[Proof of Theorem \ref{thm-main-3}]
Let $\lambda$ be a triangulation of $\fS$. By Theorem~\ref{thm-upper-w-upper}, it suffices to show that 
$\dS \subset \mathbb T(w_\lambda)$ and 
$\dS \subset \mathbb T(w_\lambda^{(k)})$ for any mutable vertex $k \in V_\lambda$, where $w_\lambda$ is defined in \eqref{def-qum-seed-w}, $w_\lambda^{(k)}$ is defined in \eqref{eq-def-wD}, and $\mathbb T(-)$ is defined in \eqref{eq-T-w}.  
Since $\mathbb T(w_\lambda) = \A$, Lemma~\ref{lem-basic-lem}(a) implies that 
$\dS \subset \mathbb T(w_\lambda)$.  
Moreover, the commutative diagram~\eqref{eq-compability-tr-A-X-diag}, together with Propositions~\ref{prop-comp} and \ref{prop-key-bal-polynomial}, shows that 
$\dS \subset \mathbb T(w_\lambda^{(k)})$ for every mutable vertex $v \in V_\lambda$.
\end{proof}

\section{Splitting homomorphisms for (upper) quantum cluster algebras}\label{sec-splitting-upper}
In this section, we first construct a splitting homomorphism for quantum upper cluster algebras in general (Proposition~\ref{prop:split}). 
We then specialize to the ${\rm SL}_n$ quantum upper cluster algebra and show that this splitting homomorphism is compatible with the splitting homomorphism for the projected ${\rm SL}_n$-skein algebra via the inclusion in Theorem~\ref{thm-main-3} (Theorem~\ref{thm:splitU}). 
Theorem~\ref{thm:splitU} will be used in \S\ref{sec-skein-inclusion-cluster} to prove that the projected ${\rm SL}_n$-skein algebra is contained in the ${\rm SL}_n$ quantum cluster algebra. 
Moreover, it may be employed in the future to approach Conjecture~\ref{con-equality-Skein-A-U} by cutting the pb surface $\fS$ into smaller pieces.

\subsection{Construction for the general case}
Let $\mathscr U$ be an upper quantum cluster algebra. For two elements $x,x'\in \mathscr U$, we say that $x$ and $x'$ are \textbf{proportional}, denoted by $x\asymp x'$ if $x^{-1}x'\in R\langle A_v^{\pm 1}\mid v\in \mathcal V\setminus \mathcal{V}_{\rm mut}\rangle$.

\begin{definition}
Let $\mathscr R, \mathscr R'$ be (upper) quantum cluster algebras with initial seeds $s$ and $s'$, respectively. For a subset $I\subset \mathcal{V}_{\rm mut}$, we say that an $R$-algebra homomorphism $f:\mathscr R\to \mathscr R'$ a \emph{splitting map} on $I$ if the following hold:
    \begin{itemize}
        \item for any $v\in \mathcal{V}_{\rm mut}\setminus I$, $f(A_v)\asymp A_{v'}$ and $f(\mu_{v}(A_v))\asymp \mu_{v'}(A_{v'})$ for some $v'\in \mathcal V_{\rm mut}'$;
        \item for any $v\in I$, $f(A_v)\asymp A_{v'_1}A_{v'_2}$ for some $v'_1,v'_2\in \mathcal V'\setminus \mathcal{V'}_{\rm mut}$; 
        \item for any $v\in \mathcal V\setminus \mathcal{V}_{\rm mut}$, $f(A_v)\asymp 1$.
    \end{itemize}
\end{definition}

\begin{definition}\label{def-splitting-matrix}
Let $Q$ be the exchange matrix with vertex set $\mathcal V$. For a subset $\mathcal{V}_{\rm split}\subset \mathcal{V}_{\rm mut}$,
denote $\mathcal V'_{\rm split}=\{k'\mid k\in \mathcal{V}_{\rm split}\}$ and $\mathcal V''_{\rm split}=\{k''\mid k\in \mathcal{V}_{\rm split}\}$ as two copies of $\mathcal{V}_{\rm split}$, 

$(1)$ we say that an exchange matrix $Q_{\rm split}=(Q_{\rm split}(u,v))$ with unfrozen vertex set $\mathcal{V}_{\rm mut}\setminus \mathcal{V}_{\rm split}$ and frozen vertex set $\mathcal{V}_{\rm frozen}\cup \mathcal V'_{\rm split}\cup \mathcal V''_{\rm split}$ is a \emph{splitting} of $Q$ on vertices $\mathcal{V}_{\rm split}$ if 
    \begin{itemize}
        \item for any $u\in \mathcal V\setminus \mathcal{V}_{\rm split}$ and $v\in \mathcal{V}_{\rm split}$, we have $Q_{\rm split}(u,v')Q_{\rm split}(u,v'')\geq 0$ and $Q(u,v)=Q_{\rm split}(u,v')+Q_{\rm split}(u,v'')$,
        \item for any $u,v\in \mathcal V\setminus \mathcal{V}_{\rm split}$, we have $Q_{\rm split}(u,v)=Q(u,v)$.
    \end{itemize}

$(2)$ We say that a quantum seed $\mathbf s_{split}=(Q_{split}, \Pi_{split}, M_{split})$ is a \emph{splitting} of seed $\mathbf s=(Q,\Pi,M)$ on vertices $\mathcal{V}_{\rm split}$ if $Q_{\rm split}$ is a splitting of $Q$ on vertices $\mathcal{V}_{\rm split}$.
\end{definition}



In particular, for any triangulation $\lambda$ of a triangulable pb surface $\fS$ and edge $e\in \lambda$, the seed $\mathsf s(\lambda_e)$ is a splitting of $\mathsf s(\lambda)$ on the vertices lying on $e$.

For an $R$-algebra $\mathscr R$, we say that two elements $x,y\in \mathscr R$ are \emph{quasi-commutative} if $xy=\xi^myx$ ($\xi=\omega^n$) for some $m\in \mathbb Z$, and we denote $\Lambda(x,y)=m$. 

\begin{lemma}\label{lem:propotion}
Let $A,A',B_1,B_2,C_1,C_2,C_1',C_2'$ be pairwise quasi-commutative invertible elements. Let $X=[AB_1C_1]+[AB_2C_2]$ and $Y=[AA'B_1C_1']+[AA'B_2C_2']$. Assume that $[C_1^{-1}C'_1]=[C_2^{-1}C'_2]$ and $\Lambda([B_1C_1],[C_1^{-1}C'_1A'])=\Lambda([B_2C_2],[C_1^{-1}C'_1A'])$, then $X$ and $C_1^{-1}C'_1A'$ are quasi-commutative and $Y=[XC_1^{-1}C'_1A']$.
\end{lemma}

\begin{proof}
We have
    \begin{align*}
        \Lambda([AB_1C_1],[C_1^{-1}C'_1A'])=\Lambda(A,[C_1^{-1}C'_1A'])+\Lambda([B_1C_1],[C_1^{-1}C'_1A']),\\
        \Lambda([AB_2C_2],[C_1^{-1}C'_1A'])=\Lambda(A,[C_1^{-1}C'_1A'])+\Lambda([B_2C_2],[C_1^{-1}C'_1A']).
    \end{align*}

Thus, $\Lambda([AB_1C_1],[C_1^{-1}C'_1A'])=\Lambda([AB_2C_2],[C_1^{-1}C'_1A'])$, and $X, C_1^{-1}C'_1A'$ are quasi-commutative.

Therefore,
\begin{align*}
  [XC_1^{-1}C_1'A']  &= [X\cdot [C_1^{-1}C_1'A']]  \\
  & = [([AB_1C_1]+[AB_2C_2])\cdot [C_1^{-1}C_1'A']]\\
  & = [AB_1C_1C_1^{-1}C_1'A']+[AB_2C_2C_1^{-1}C_1'A']\\
  & = [AB_1C_1C_1^{-1}C_1'A']+[AB_2C_2C_2^{-1}C_2'A']\\
  & = Y.
\end{align*}
The proof is complete.
\end{proof}


Let $v\in\mathcal V_{\rm mut}$. The symbol 
``$\prod_{u\to v\in Q}$'' (resp. ``$\prod_{u\leftarrow v\in Q}$'') denotes the product over the multiset 
$\{\,u\in\mathcal V \mid Q(u,v)>0\,\}$
(resp. $\{\,u\in\mathcal V \mid Q(u,v)<0\,\}$), where each element $u$ occurs with multiplicity $|Q(u,v)|$ 
(for example, $\prod_{u\to v\in Q} A_u = A_{u'}A_{u''}^2$ if $u',u''\in\mathcal V$ are the only vertices with $Q(u,v)>0$, with $Q(u',v)=1$ and $Q(u'',v)=2$).

Define 
$$\text{$Q(-,v)=\{w\in \mathcal V\mid Q(w,v)>0\} \text{ and } Q(v,-)=\{w\in \mathcal V\mid Q(v,w)>0\}$.}$$

Let \(S_1\) be a set and \(S_2\) a multiset.  
Denote by \(F(S_2)\) the underlying set of \(S_2\), obtained by forgetting multiplicities.  
We define the intersection \(S_1 \cap S_2\) to be the multiset whose underlying set is
\(S_1 \cap F(S_2)\), and in which each element \(s \in S_1 \cap F(S_2)\) has the same
multiplicity as \(s\) does in \(S_2\).

\begin{proposition}[Splitting homomorphism for (upper) quantum cluster algebras]\label{prop:split} 
Let $\mathbf s$ be a quantum seed and $\mathbf s_{\rm split}$ be a splitting of $\mathbf s$ on vertices $\mathcal{V}_{\rm split}$. Suppose that 
we associate each $u\in \mathcal{V}_{\rm split}$ with two muti-subsets $\mathcal{V}'_u, \mathcal{V}''_u\subset \mathcal V\setminus \mathcal{V}_{\rm split}$ (may intersect) such that

$(a)$ for any $v_1,v_2\in \mathcal V\setminus \mathcal{V}_{\rm split}$,
$$\Lambda(A_{v_1},A_{v_2})=\Lambda([A_{v_1}\prod\limits_{
u\in \mathcal{V}^1_{v_1}} A_{u'}\prod\limits_{
u\in \mathcal{V}^2_{v_1}}A_{u''}],[A_{v_2}\prod\limits_{u\in \mathcal{V}^1_{v_2}} A_{u'}\prod\limits_{u\in \mathcal{V}_{v_2}^2}A_{u''}]),$$ 

$(b)$ for any $v_1\in V\setminus \mathcal{V}_{\rm split}$ and $u\in \mathcal{V}_{\rm split}$,
$$\Lambda(A_{v_1},A_{u})=\Lambda([A_{v_1}\prod\limits_{
v\in \mathcal{V}^1_{v_1}} A_{v'}\prod\limits_{
v\in \mathcal{V}^2_{v_1}}A_{v''}],[A_{u'}A_{u''}]),$$ 

$(c)$ for any $u_1,u_2\in \mathcal{V}_{\rm split}$, $$\Lambda(A_{u_1},A_{u_s})=\Lambda([A_{u'_1}A_{u''_1}],[A_{u'_2}A_{u''_2}]),$$

$(d)$ for any $v\in \mathcal{V}_{\rm mut}\setminus \mathcal{V}_{\rm split}$ and  $u\in \mathcal{V}_{\rm split}$, 
$$[Q_{\rm split}(u'',v)]_+\,+ \sum\limits_{v_1\in Q(-,v)\cap \mathcal{V}'_{u}}Q(v_1,v)=
[Q_{\rm split}(v,u'')]_+\,+ \sum\limits_{v_1\in Q(v,-)\cap \mathcal{V}'_{u}}Q(v,v_1)$$ 

$$[Q_{\rm split}(u',v)]_+\,+ \sum\limits_{v_1\in Q(-,v)\cap \mathcal{V}''_{u}}Q(v_1,v)=
[Q_{\rm split}(v,u')]_+\,+ \sum\limits_{v_1\in Q(v,-)\cap \mathcal{V}''_{u}}Q(v,v_1),$$

where
$$\mathcal V^1_w:=\{u\in \mathcal{V}_{\rm split}\mid
w\in \mathcal V_u'\}\text{ and }
\mathcal V^2_w:=\{u\in \mathcal{V}_{\rm split}\mid
w\in \mathcal V_u''\},
$$ for each $w\in \mathcal V$. In particular,  
$\mathcal V^1_u=\mathcal V^2_u=\emptyset$ for each $u\in \mathcal{V}_{\rm split}$.
Note that both $\mathcal V^1_w$ and $\mathcal V^2_w$ are multisets, and the multiplicity of $u\in \mathcal V^1_w$ (resp. $u\in \mathcal V^2_w$) is the multiplicity of $w$ in $\mathcal V_u'$ (resp. $\mathcal V_u''$).

Then the assignments 
    $$A_v\mapsto 
    \begin{cases}
    [A_v\prod\limits_{u\in  \mathcal{V}^1_{v}} A_{u'}\prod\limits_{u\in  \mathcal{V}^2_{v}}A_{u''}]
    & \text{ if } v\in \mathcal V\setminus \mathcal{V}_{\rm split},\vspace{2mm}\\
    [A_{v'}A_{v''}] & \text{ if } v\in \mathcal{V}_{\rm split},
    \end{cases}$$ 
    give an injective reflection-invariant (see \eqref{def-eq-ref-torus} and \S\ref{sub-sec-invariant}) splitting homomorphism $\iota:\mathscr{U}_s\to \mathscr{U}_{s_{\rm split}}$ on $\mathcal{V}_{\rm split}$. 
\end{proposition} 

\begin{proof}

The assignments give an injective $R$-algebra homomorphism $\iota:\mathbb T(\mathbf s)\hookrightarrow \mathbb T(\mathbf s_{split})$ by $(a),(b)$ and $(c)$. 
It follows from the definition of $\iota$ that it is reflection-invariant.

For any $v\in \mathcal{V}_{\rm mut}\setminus \mathcal{V}_{\rm split}$, in $\mathscr U_{\mathbf s}$, we have 
\begin{align*}
 \iota(\mu^{\mathbf s}_v(A_v)) 
 &= \iota([A_v^{-1}\prod_{v_1\to v\in Q}A_{v_1}]+[A_v^{-1}\prod_{v_2\leftarrow v\in Q}A_{v_2}]) \\  
 &= \Bigg [A^{-1}_v\prod\limits_{u\in \mathcal{V}^1_{v}} A^{-1}_{u'}
 \prod\limits_{u\in \mathcal{V}^2_{v}} A^{-1}_{u''}\\
 & \cdot
 \prod_{v_1\to v\in Q, v_1\notin \mathcal{V}_{\rm split}}\left(A_{v_1}
 \prod\limits_{u\in \mathcal{V}^1_{v_1}} A_{u'}
 \prod\limits_{u\in \mathcal{V}^2_{v_1}}A_{u''}\right) \cdot \prod_{u\to v\in Q, u\in \mathcal{V}_{\rm split}} A_{u'}A_{u''}\Bigg ]  \\
 &+ \Bigg [A^{-1}_v\prod\limits_{u\in \mathcal{V}^1_{v}} A^{-1}_{u'}
 \prod\limits_{u\in \mathcal{V}^2_{v}} A^{-1}_{u''}\\
 & \cdot
 \prod_{v_2\leftarrow v\in Q, v_2\notin \mathcal{V}_{\rm split}}\left(A_{v_2}
 \prod\limits_{u\in \mathcal{V}^1_{v_2}} A_{u'}
 \prod\limits_{u\in \mathcal{V}^2_{v_2}}A_{u''}\right) \cdot \prod_{u\leftarrow v\in Q, u\in \mathcal{V}_{\rm split}} A_{u'}x_{u''}\Bigg ]  \\
&=\Bigg [A^{-1}_v \cdot \prod\limits_{u\in \mathcal{V}^1_v} A^{-1}_{u'}
 \prod\limits_{u\in \mathcal{V}^2_v} A^{-1}_{u''}\cdot \prod_{v_1\to v\in Q, v_1\notin \mathcal{V}_{\rm split}} A_{v_1}
\\
 &\cdot  \prod_{u\to v\in Q, u\in \mathcal{V}_{\rm split}}A_{u'}A_{u''}
 \prod_{v_1\to v\in Q, v_1\notin \mathcal{V}_{\rm split}}\left(
 \prod\limits_{u\in \mathcal{V}^1_{v_1}} A_{u'}
 \prod\limits_{u\in \mathcal{V}^2_{v_1}}A_{u''}\right) \Bigg ]  \\
 &+\Bigg [A^{-1}_v \cdot\prod\limits_{u\in \mathcal{V}^1_v} A^{-1}_{u'}
 \prod\limits_{u\in \mathcal{V}^2_v} A^{-1}_{u''}\cdot 
 \prod_{v_2\leftarrow v\in Q, v_2\notin \mathcal{V}_{\rm split}} A_{v_2}
\\
 &\cdot \prod_{v\to u\in Q, u\in \mathcal{V}_{\rm split}}A_{u'}A_{u''}
 \prod_{v_2\leftarrow v\in Q, v_2\notin \mathcal{V}_{\rm split}}\left(
 \prod\limits_{u\in \mathcal{V}^1_{v_2}} A_{u'}
 \prod\limits_{u\in \mathcal{V}^2_{v_2}}A_{u''}\right) \Bigg ].
\end{align*}

In $\mathscr U_{\mathbf s_{\rm split}}$, we have 
\begin{align*}
  \mu_v^{\mathbf s_{split}}(A_v) &=  
  [A_v^{-1}\prod_{v_1\to v\in Q,v_1\notin \mathcal{V}_{\rm split}}A_{v_1}\prod_{u\in \mathcal{V}_{\rm split}, u'\to v\in Q_{\rm split}}A_{u'}\prod_{u\in \mathcal{V}_{\rm split}, u''\to v\in Q_{\rm split}}A_{u''}]\\
  &+[A_v^{-1}\prod_{v_2\leftarrow v\in Q,v_2\notin \mathcal{V}_{\rm split}}A_{v_2}\prod_{u\in \mathcal{V}_{\rm split}, u'\leftarrow v\in Q_{\rm split}}A_{u'}\prod_{u\in \mathcal{V}_{\rm split}, u''\leftarrow v\in Q_{\rm split}}A_{u''}].
\end{align*}

Denote 
$$A=A_v^{-1}, A'= \prod\limits_{u\in \mathcal{V}^1_v} A^{-1}_{u'}
 \prod\limits_{u\in \mathcal{V}^2_v} A^{-1}_{u''},$$ 
 
$$B_1=\prod_{v_1\to v\in Q,v_1\notin \mathcal{V}_{\rm split}}A_{v_1}, \qquad B_2=\prod_{v_2\leftarrow v\in Q,v_2\notin \mathcal{V}_{\rm split}}A_{v_2},$$
 
$$C_1=\prod_{u\in \mathcal{V}_{\rm split}, u'\to v\in Q_{\rm split}}A_{u'}\prod_{u\in \mathcal{V}_{\rm split}, u''\to v\in Q_{\rm split}}A_{u''},$$ 

$$C_2=\prod_{u\in \mathcal{V}_{\rm split}, u'\leftarrow v\in Q_{\rm split}}A_{u'}\prod_{u\in \mathcal{V}_{\rm split}, u''\leftarrow v\in Q_{\rm split}}A_{u''},$$ 

$$C'_1=\prod_{u\to v\in Q, u\in \mathcal{V}_{\rm split}}A_{u'}A_{u''}
 \prod_{v_1\to v\in Q, v_1\notin \mathcal{V}_{\rm split}}\left(
 \prod\limits_{u\in \mathcal{V}^1_{v_1}} A_{u'}
 \prod\limits_{u\in \mathcal{V}^2_{v_1}}A_{u''}\right),$$

$$C'_2=\prod_{u\leftarrow v\in Q, u\in \mathcal{V}_{\rm split}}A_{u'}A_{u''}
 \prod_{v_2\leftarrow v\in Q, v_2\notin \mathcal{V}_{\rm split}}\left(
 \prod\limits_{u\in \mathcal{V}^1_{v_2}} A_{u'}
 \prod\limits_{u\in \mathcal{V}^2_{v_2}}A_{u''}\right).$$

Then we have 
\begin{equation*}
   \iota(\mu^{\mathbf s}_v(A_v))=[AA'B_1C'_1]+[AA'B_2C'_2], \hspace{5mm} \mu_v^{\mathbf s_{split}}(A_v)=[AB_1C_1]+[AB_2C_2], 
\end{equation*}

$$[C_1^{-1}C_1']=
\prod_{u\in\mathcal V_{\rm split}}\left[A_{u'}^{[Q_{\rm split}(u'',v)]_+\,+ \sum\limits_{v_1\in Q(-,v)\cap \mathcal{V}'_{u}}Q(v_1,v)}
A_{u''}^{[Q_{\rm split}(u',v)]_+\, + \sum\limits_{v_1\in Q(-,v)\cap \mathcal{V}''_{u}}Q(v_1,v)}
\right],$$

$$[C_2^{-1}C_2']=
\prod_{u\in\mathcal V_{\rm split}}\left[A_{u'}^{[Q_{\rm split}(v,u'')]_{+}\, +\sum\limits_{v_1\in Q(v,-)\cap \mathcal{V}'_{u}}Q(v,v_1)}
A_{u''}^{[Q_{\rm split}(v,u')]_{+}\,+ \sum\limits_{v_1\in Q(v,-)\cap \mathcal{V}''_{u}}Q(v,v_1)}
\right].$$

Thus, we obtain $[C_1^{-1}C'_1]=[C_2^{-1}C'_2]$ by $(d)$.

As $\mathbf s_{\rm split}$ is a quantum seed and $\mu_v^{\mathbf s_{split}}(A_v)=[AB_1C_1]+[AB_2C_2]$, for any $u\in \mathcal{V}_{\rm split}$ we have $\Lambda(B_1C_1,A_{u'})=\Lambda(B_2C_2,A_{u'})$ and $\Lambda(B_1C_1,A_{u''})=\Lambda(B_2C_2,A_{u''})$. It follows that $$\Lambda(B_1C_1,C_1^{-1}C'_1A')=\Lambda(B_2C_2,C_1^{-1}C'_1A').$$

Therefore, by Lemma \ref{lem:propotion}, we have $\iota(\mu^{\mathbf s}_v(A_v))\asymp \mu_v^{\mathbf s_{split}}(A_v)$.
Note that $\iota\colon \mathbb T(\mathbf s)\to \mathbb T(\mathbf s_{\rm split})$ induces a skew-field embedding $\iota\colon \Fr(\mathbb T(\mathbf s))\rightarrow \Fr(\mathbb T(\mathbf s_{\rm split}))$. The fact that $\iota(\mu^{\mathbf s}_v(A_v))\asymp \mu_v^{\mathbf s_{split}}(A_v)$ implies that $\iota$ induces an $R$-algebra embedding $\iota: \mathbb T(\mu_v(\mathbf s))\to \mathbb T(\mu_v(\mathbf s_{\rm split}))$ for any $v\in \mathcal{V}_{\rm mut}\setminus \mathcal{V}_{\rm split}$. We thus obtain an $R$-algebra embedding
$$\iota: \mathbb T(\mathbf s)\cap \bigcap\limits_{v\in \mathcal{V}_{\rm mut}\setminus \mathcal{V}_{\rm split}} \mathbb T(\mu_v(\mathbf s))\to \mathbb T(\mathbf s_{\rm split})\cap \bigcap\limits_{v\in \mathcal{V}_{\rm mut}\setminus \mathcal{V}_{\rm split}} \mathbb T(\mu_v(\mathbf s_{\rm split}))=\mathscr U(\mathbf s_{\rm split}).$$

Combing the natural embedding $\mathscr U(\mathbf s)\hookrightarrow \mathbb T(\mathbf s)\cap \bigcap\limits_{v\in \mathcal{V}_{\rm mut}\setminus \mathcal{V}_{\rm split}} \mathbb T(\mu_v(\mathbf s))$, we obtain an injective splitting homomorphism $\iota: \mathscr U(\mathbf s)\to \mathscr U(\mathbf s_{\rm split})$.

The proof is complete.
\end{proof}


The following is an immediate corollary of the proof of Proposition \ref{prop:split} and Lemma \ref{lem:propotion}.

\begin{corollary}\label{cor:propotion}
    For any $v\in \mathcal{V}_{\rm mut}\setminus \mathcal{V}_{\rm split}$, we have
   \begin{align*}
     \iota(\mu^{\mathbf s}_v(A_v))   &= \Bigg[\mu_v^{\mathbf s_{split}}(A_v)\cdot \prod\limits_{u\in \mathcal{V}^1_v} A^{-1}_{u'}
 \prod\limits_{u\in \mathcal{V}^2_v} A^{-1}_{u''} \\
        & \cdot \prod\limits_{u\in \mathcal{V}_{\rm split}}(A^{-[Q_{\rm split}(u',v)]_+}_{u'}\cdot \prod\limits_{v_1\to v\in Q, u\in \mathcal{V}^1_{v_1}} A_{u'}\cdot A^{-[Q_{\rm split}(u'',v)]_+}_{u''}\cdot \prod\limits_{v_1\to v\in Q, u\in \mathcal{V}^2_{v_1}} A_{u''})\Bigg].
   \end{align*}
\end{corollary}

\subsection{The splitting homomorphism for $\mathscr U_\omega(\fS)$}
For any triangulation $\lambda$ of a triangulable pb surface $\fS$ without interior punctures and any edge $e \in \lambda$, assume that $e$ is the common side of two ideal triangles $\Delta_1$ and $\Delta_2$, as shown in Figure~\ref{Fig;cutting-e}. 
Suppose that the punctures adjacent to $e$ are labeled $p_1$ and $p_2$ (where it is possible that $p_1 = p_2$), as illustrated in Figure~\ref{Fig;cutting-e}.
The small vertices on $e$ are labeled $s_1, \dots, s_{n-1}$, also indicated in Figure~\ref{Fig;cutting-e}. 
In the figure, there are two oriented curves, one drawn in blue and the other in red. 
Suppose that, as we follow the red (resp.\ blue) curve from its starting point to its endpoint, we encounter the sequence of triangles \( \tau_1', \dots, \tau_k' \) (resp.\ \( \tau_1'', \dots, \tau_m'' \)) consecutively, as labeled in Figure~\ref{Fig;cutting-e}. 
Note that $\tau_1''=\Delta_1$ and $\tau_1'=\Delta_2$.



We use $\mathcal V$ to denote $V_\lambda$. Then $\mathcal V_{\rm mut}$ consists of all small vertices contained in the interior of $\fS$.
For each $1\leq i\leq n-1$,
we construct two muti-subsets $\mathcal{V}'_{s_i}, \mathcal{V}''_{s_i}\subset \mathcal V$ as follows:
\begin{enumerate}
    \item Each vertex in $\mathcal{V}'_{s_i}$ (resp. $\mathcal{V}_{s_i}''$) is contained in one of the triangles $\tau_1',\cdots,\tau_k'$ (resp. $\tau_1'',\cdots,\tau_m''$).
    \item For each $1\leq j\leq k$ (resp. $1\leq j\leq m$), we choose a barycentric coordinate for $\tau_j'$ (resp. $\tau_j''$) such that $p_1=(n,0,0)$ (resp. $p_2=(n,0,0)$).
    Then $\mathcal{V}'_{s_i}\cap \tau_j'$ consists of small vertices, other than $s_i$, whose first entry is $i$ (resp. $\mathcal{V}''_{s_i}\cap \tau_j''$ consists of small vertices, other than $s_i$, whose first entry is $n-i$).
\end{enumerate}

We provide a geometric interpretation of $\mathcal{V}'_{s_i}$ and $\mathcal{V}_{s_i}''$.
All the vertices in $\mathcal{V}'_{s_i}$ (resp. $\mathcal{V}_{s_i}''$) are contained in the curve $c_i'$ (resp. $c_i''$), see the right picture in Figure~\ref{Fig;cutting-e}. 

\begin{figure}[h]
    \centering
    \includegraphics[width=150mm]{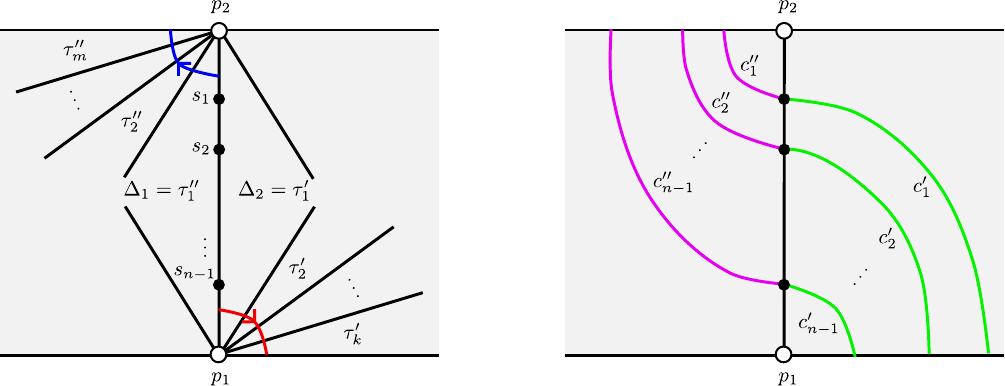}
    \caption{Left: the labelings of the triangles $\tau_1', \dots, \tau_k'$ and $\tau_1'', \dots, \tau_m''$. \\
    Right: schematic illustrations of the curves $c_1', \dots, c_{n-1}'$ (in green) and $c_1'', \dots, c_{n-1}''$ (in purple).}\label{Fig;cutting-e}
\end{figure}

\def\eqq{\overset{\omega}{=}}

Define 
\begin{equation}
\label{def-V-split-skein}
    \begin{split}
    &\text{set $\mathcal{V}_{\rm split}=\{s_1,\cdots,s_{n-1}\}$,}\\
    &\text{multiset $\mathcal{V}^1_v=\{u\in \mathcal{V}_{\rm split}\mid v\in \mathcal{V}'_u\}$ for any $v\in \mathcal V\setminus \mathcal{V}_{\rm split}$,}\\
    &\text{and multiset $\mathcal{V}^2_v=\{u\in \mathcal{V}_{\rm split}\mid v\in \mathcal{V}''_u\}$ for any $v\in \mathcal V\setminus \mathcal{V}_{\rm split}$.}
    \end{split}
\end{equation}
The multiplicity of $u\in \mathcal V^1_v$ (resp. $u\in \mathcal V^2_v$) is the multiplicity of $v$ in $\mathcal V_u'$ (resp. $\mathcal V_u''$).
After we cut $\fS$ along $e$, the small vertices $s_1',\cdots,s_{n-1}'$ (resp. $s_1'',\cdots,s_{n-1}''$) are contained in the triangle $\Delta_1$ (resp. $\Delta_2$).

For two elements $x,y$ in an $R$-algebra $\mathscr R$, we denote $x\eqq y$ if $x=\omega^{\frac{k}{2}}y\in \mathscr R$ for an integer $k$.

\begin{proposition}\label{prop:image}
Let $\fS$ be a triangulable pb surface with a triangulation $\lambda$, and assume that $\fS$ has no interior punctures.  
Then we have the following:
\begin{enumerate}[label={\rm (\alph*)}]\itemsep0,3em

\item For each $v \in \mathcal V=V_\lambda$, in $\widetilde{\mathscr{S}}_\omega(\cut_e (\fS))$ (and likewise in $\overline{\mathscr{S}}_\omega(\cut_e (\fS))$) we have  
\begin{align}\label{eq-splitting-gv}
   \mathbb S_e(\gaa_v)=
   \begin{cases}
   \big[\gaa_v \prod\limits_{u \in \mathcal{V}^1_v} \gaa_{u'} \prod\limits_{u \in \mathcal{V}^2_v} \gaa_{u''}\big] & \text{if } v \in \mathcal V \setminus \mathcal{V}_{\rm split},\vspace{2mm}\\
   [\gaa_{v'} \gaa_{v''}] & \text{if } v \in \mathcal{V}_{\rm split},
   \end{cases}
\end{align}
where $\cut_e (\fS)$ and $\mathbb S_e$ are defined in \S\ref{sub-splitting}.  

\item In particular, the algebra embedding  
\[
\mathbb{S}_e\colon \widetilde{\cS}_{\omega}(\fS) \longrightarrow \widetilde{\mathscr{S}}_\omega(\cut_e (\fS))
\]
extends uniquely to an algebra embedding  
\[
\mathbb{S}_e\colon \mathcal{A}_{\omega}(\fS, \lambda) \longrightarrow \mathcal{A}_{\omega}(\cut_e (\fS), \lambda_e),
\]
defined as in \eqref{eq-splitting-gv}, with $\gaa_v$ replaced by $A_v$, where $\lambda_e$ is defined in \S\ref{sub-sec-Kernel-trace}.
\end{enumerate}
\end{proposition}

\begin{proof}
(a) For any $1\leq i\leq n-1$, we have that
\begin{align*}
\mathbb S_e(\gaa_{s_i})& \eqq (-1)^{\binom{n}{2}}
    \mathbb S_e\left(
\begin{array}{c}\includegraphics[scale=1]{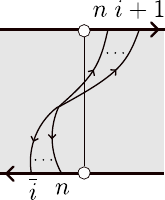}\end{array}\right)
\\&\eqq (-1)^{\binom{n}{2}}\begin{array}{c}\includegraphics[scale=1]{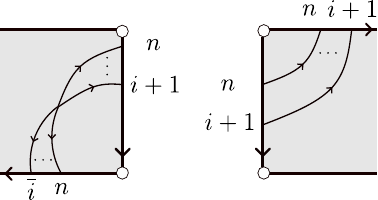}\end{array}
\\&\eqq\begin{array}{c}\includegraphics[scale=1]{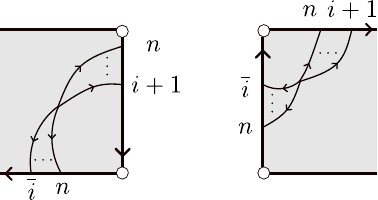}\end{array}
\\&\eqq[\gaa_{s_i'} \gaa_{s_i''}]
\end{align*}    
Lemma \ref{lem-reflection} implies that
$\mathbb S_e(\gaa_{s_i})=[\gaa_{s_i'} \gaa_{s_i''}].$

Let $v\in \mathcal V\setminus \mathcal{V}_{\rm split}$. 
Note that $|\mathcal V_v^1|, |\mathcal V_v^2|=0,1,2,3$, and $|\mathcal V_v^1|+|\mathcal V_v^2|\leq 3$.
We prove only the case where $|\mathcal V_v^1| = |\mathcal V_v^2| = 1$, since the same argument applies to the other cases as well.
Suppose that $\mathcal V_v^1=\{s_i\}$
and $\mathcal V_v^2=\{s_j\}$.
Then we have that
\begin{align*}
    \mathbb S_e(\gaa_v)
    &\eqq \gaa_v \left( 
    \begin{array}{c}\includegraphics[scale=1]{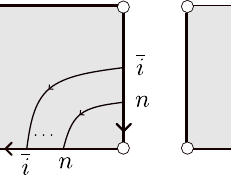}\end{array}\right)
    \left(\begin{array}{c}\includegraphics[scale=1]{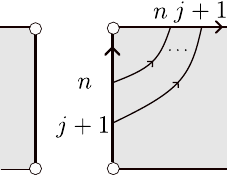}\end{array}\right)
    \\&\eqq
    \gaa_v \left( 
    \begin{array}{c}\includegraphics[scale=1]{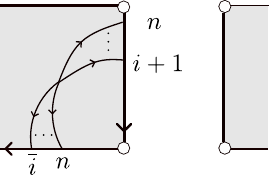}\end{array}\right)
    \left(\begin{array}{c}\includegraphics[scale=1]{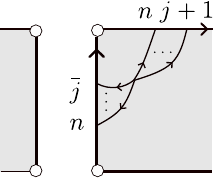}\end{array}\right)
    \\&\eqq [\gaa_v\gaa_{s_i'} \gaa_{s_j''}].
\end{align*}
Lemma \ref{lem-reflection} implies that
$\mathbb S_e(\gaa_{v})=[\gaa_v\gaa_{s_i'} \gaa_{s_j''}].$

(b) It follows from (a) and \eqref{sandwitch-projected}.
\end{proof}

Under the same assumption as Proposition \ref{prop:image},
we have the following immediate corollary.

\begin{corollary} \label{cor:quasi-com}
\begin{enumerate}[label={\rm (\alph*)}]\itemsep0,3em

\item For any $v_1,v_2\in \mathcal V\setminus \mathcal{V}_{\rm split}$,
$$\Lambda(A_{v_1},A_{v_2})=\Lambda([A_{v_1}\prod\limits_{
u\in \mathcal{V}^1_{v_1}} A_{u'}\prod\limits_{
u\in \mathcal{V}^2_{v_1}}A_{u''}],[A_{v_2}\prod\limits_{u\in \mathcal{V}^1_{v_2}} A_{u'}\prod\limits_{u\in \mathcal{V}_{v_2}^2}A_{u''}]).$$ 

\item For any $v_1\in \mathcal V\setminus \mathcal{V}_{\rm split}$ and $u\in \mathcal{V}_{\rm split}$,
$$\Lambda(A_{v_1},A_{u})=\Lambda([A_{v_1}\prod\limits_{
v\in \mathcal{V}^1_{v_1}} A_{v'}\prod\limits_{
v\in \mathcal{V}^2_{v_1}}A_{v''}],[A_{u'}A_{u''}]).$$ 

\item For any $u_1,u_2\in \mathcal{V}_{\rm split}$, $\Lambda(A_{u_1},A_{u_s})=\Lambda([A_{u'_1}A_{u''_1}],[A_{u'_2}A_{u''_2}]).$

\end{enumerate}
\end{corollary}

We use $Q$ (resp.\ $Q_{\rm split}$) to denote $Q_\lambda$ (resp.\ $Q_{\lambda_e}$).  
The following lemma is immediate.

\begin{lemma}\label{lem-splitting-Q}
$Q_{\rm split}$ is a splitting of $Q$ along the vertices $\mathcal V_{\rm split}$ (Definition~\ref{def-splitting-matrix}).
\end{lemma}

\begin{proof}
It is clear that  
\[
   Q_{\rm split}(u,v) = Q(u,v) \qquad 
   \text{for all } u,v \in \mathcal V \setminus \mathcal V_{\rm split}.
\]

Let $w \in \mathcal V \setminus \mathcal V_{\rm split}$ and $s \in \mathcal V_{\rm split}$.  
Trivially,
\[
   Q(w,s) = Q_{\rm split}(w,s') + Q_{\rm split}(w,s'').
\]
If $n \ge 3$, then $Q_{\rm split}(w,s') \, Q_{\rm split}(w,s'') = 0$.

When $n = 2$, the only nontrivial case occurs when $w$ and $s'$ lie in the same triangle
and $w$ and $s''$ also lie in the same triangle.  
In this situation it is straightforward to check that
\[
   Q_{\rm split}(u,s') \, Q_{\rm split}(u,s'') > 0.
\]

The proof is compelete.
\end{proof}

\begin{lemma}\label{lem:mut}
    For any $v\in \mathcal{V}_{\rm mut}\setminus \mathcal{V}_{\rm split}$ and  $u\in \mathcal{V}_{\rm split}$, 
    we have
\begin{align}
\label{lem-eq-Q-split-1}
    [Q_{\rm split}(u'',v)]_+\,+ \sum\limits_{v_1\in Q(-,v)\cap \mathcal{V}'_{u}}Q(v_1,v)=
[Q_{\rm split}(v,u'')]_+\,+ \sum\limits_{v_1\in Q(v,-)\cap \mathcal{V}'_{u}}Q(v,v_1),\\
\label{lem-eq-Q-split-2}
[Q_{\rm split}(u',v)]_+\,+ \sum\limits_{v_1\in Q(-,v)\cap \mathcal{V}''_{u}}Q(v_1,v)=
[Q_{\rm split}(v,u')]_+\,+ \sum\limits_{v_1\in Q(v,-)\cap \mathcal{V}''_{u}}Q(v,v_1).
\end{align}
\end{lemma}

\begin{proof}
First consider the case \(Q(u,v)=0\).  
By the construction of \(\mathcal V'_u\), \(\mathcal V''_u\), and \(Q\), both
\eqref{lem-eq-Q-split-1} and \eqref{lem-eq-Q-split-2} follow immediately.

It remains to treat the case \(Q(u,v)>0\) (the same argument works when \(Q(u,v)<0\)).
Let \(\tau_1,\tau_2\) be the triangles of \(\lambda\) containing \(v\)
(with \(\tau_1=\tau_2\) if \(v\) lies in only one triangle).
Observe that
\(Q_{\mathrm{split}}(u'',v)\,Q_{\mathrm{split}}(u',v)=0\)
except when \(n=2\), \(\tau_1\neq\tau_2\), and the subsurface
\(\tau_1\cup\tau_2\) is a cylinder.
In this exceptional situation a direct computation yields
\begin{align*}
[Q_{\mathrm{split}}(u'',v)]_+
  &=\sum_{v_1\in Q(v,-)\cap \mathcal V'_u} Q(v,v_1)
   = [Q_{\mathrm{split}}(u',v)]_+
   =\sum_{v_1\in Q(v,-)\cap \mathcal V''_u} Q(v,v_1) = 1,\\
\sum_{v_1\in Q(-,v)\cap \mathcal V'_u} Q(v_1,v)
  &= [Q_{\mathrm{split}}(v,u'')]_+
   =\sum_{v_1\in Q(-,v)\cap \mathcal V''_u} Q(v_1,v)
   = [Q_{\mathrm{split}}(v,u')]_+ = 0.
\end{align*}
Hence \eqref{lem-eq-Q-split-1} and \eqref{lem-eq-Q-split-2} hold in this case as well.

Next suppose \(Q_{\mathrm{split}}(u'',v)\,Q_{\mathrm{split}}(u',v)=0\).
If \(Q_{\mathrm{split}}(u'',v)=0\), then
\(Q_{\mathrm{split}}(u',v)=Q(u,v)>0\) (Lemma~\ref{lem-splitting-Q}).
From the definitions of \(\mathcal V'_u\), \(\mathcal V''_u\), and \(Q\),
we obtain
\begin{align}
\label{lem-eq-Q-split-3}
\sum_{v_1\in Q(-,v)\cap \mathcal V'_u} Q(v_1,v)
 &= \sum_{v_1\in Q(v,-)\cap \mathcal V'_u} Q(v,v_1),\\[4pt]
\label{lem-eq-Q-split-4}
Q(u,v)+\sum_{v_1\in Q(-,v)\cap \mathcal V''_u} Q(v_1,v)
 &= \sum_{v_1\in Q(v,-)\cap \mathcal V''_u} Q(v,v_1).
\end{align}
Equation \eqref{lem-eq-Q-split-1} follows from
\eqref{lem-eq-Q-split-3} together with \(Q_{\mathrm{split}}(u'',v)=0\),
and \eqref{lem-eq-Q-split-2} follows from
\eqref{lem-eq-Q-split-4} and the identity
\(Q_{\mathrm{split}}(u',v)=Q(u,v)\).
The case \(Q_{\mathrm{split}}(u',v)=0\) is analogous.

The proof is complete.
\end{proof}

The following theorem—the main result of this section—establishes a splitting homomorphism of the ${\rm SL}_n$ quantum upper cluster algebra, compatible with the splitting homomorphism for the projected stated ${\rm SL}_n$-skein algebra.
In \S\ref{sec-skein-inclusion-cluster}, we will use this map to prove the inclusion $\dS \subset \mathscr A_{\omega}(\fS)$ by decomposing $\fS$ into triangles and quadrilaterals.
Beyond this application, the theorem also offers a general tool for reducing other problems—such as Conjecture~\ref{con-equality-Skein-A-U}—to the case of triangles and quadrilaterals.

\begin{theorem}\label{thm:splitU}
Let $\lambda$ be a triangulation of a triangulable pb surface $\fS$ and edge $e$ be an ideal arc in $\lambda$. 
Assume that $\fS$ contains no interior punctures.
The assignments 
    $$A_v\mapsto 
    \begin{cases}
    [A_v\prod\limits_{u\in\mathcal V_v^1} A_{u'}\prod\limits_{u\in\mathcal V_v^2}A_{u''}] & \text{ if } v\in \mathcal V\setminus \mathcal{V}_{\rm split},\vspace{2mm}\\
    [A_{u'}A_{u''}] & \text{ if } u\in \mathcal{V}_{\rm split},
    \end{cases}
    $$ define an injective reflection-invariant (see \eqref{def-eq-ref-torus} and \S\ref{sub-sec-invariant}) splitting homomorphism $\mathbb S_e^U:\mathscr{U}_{\omega}(\fS)\to \mathscr{U}_{\omega}(\cut_e(\fS))$ on $\mathcal{V}_{\rm split}$, moreover, which is compatible with the splitting homomorphism for the projected stated ${\rm SL}_n$-skein algebras in the sense that the following diagram commutes:
    $$\centerline{\xymatrix{
  &\widetilde{\mathscr S}_\omega(\fS) \ar@{^{(}->}[rrr]^{\mathbb S_e}\ar@{^{(}->}[d]  &&& \widetilde{\mathscr S}_\omega(\cut_e(\fS)) \ar@{^{(}->}[d]          \\
  &  \mathscr U_{\omega}(\fS) \ar@{^{(}->}[rrr]^{\mathbb S_e^U}  &&& \mathscr U_{\omega}(\cut_e(\fS)).}}$$ 
  Here $\mathcal V=V_\lambda$, and $\mathcal V_{\rm split}$,  $\mathcal V_{v}^1$, and  $\mathcal V_{v}^2$ are defined in \eqref{def-V-split-skein}
\end{theorem}

\begin{proof}

The splitting homomorphism $\mathbb S_e^U:\mathscr{U}_{\omega}(\fS)\to \mathscr{U}_{\omega}(\cut_e(\fS))$ at $\mathcal{V}_{\rm split}$ is followed by Proposition~\ref{prop:split}, Corollary~\ref{cor:quasi-com} and Lemma~\ref{lem:mut}.
By Proposition~\ref{prop:image}, the diagram commutes.

The proof is complete.
\end{proof}



    





\section{The inclusion of skein algebras into quantum cluster algebras}
\label{sec-skein-inclusion-cluster}

In this section we prove the main result of the paper, Theorem~\ref{thm-skein-inclusion-A}, which shows that  
\[
\dS \subset \mathscr A_{\omega}(\fS)
\]
whenever every connected component of $\fS$ has at least two boundary components.  
A properly embedded oriented arc in $\fS$ is called an \emph{essential arc}  
if its endpoints lie on two distinct components of $\partial \fS$.  
Lemma~\ref{lem-essential-arc} asserts that, as an $R$-algebra,  
$\dS$ is generated by stated essential arcs.  
 Theorem~\ref{thm-skein-inclusion-A} provides an explicit expression for each stated essential arc in $\fS$ as the Weyl-ordered product (see \eqref{Weyl-A}) of an exchangeable cluster variable (Definition~\ref{def-quan-cluster-algebra}) and a monomial in frozen cluster variables. 
 
 If one can further show that  
$\mathscr U_{\omega}(\fS)$ is generated by these exchangeable cluster variables as an $R\langle A^{\pm1}_v\mid v\in \mathcal V\setminus \mathcal V_{\mathrm{mut}}\rangle$-algebra, it would follow that  
\[
\dS = \mathscr A_{\omega}(\fS) = \mathscr U_{\omega}(\fS),
\]
offering a promising path toward establishing this equality.

The main result of this section is the following.

\begin{theorem}\label{thm-skein-inclusion-A}
Let $\fS$ be a pb surface without interior punctures. We require that every component of $\fS$ contains at least two punctures.
    Then we have $$\dS\subset \mathscr A_{\omega}(\fS).$$
 Moreover, each stated essential arc is 
 the Weyl-ordered product (see \eqref{Weyl-A}) of an exchangeable cluster variable and a Laurent monomial in frozen cluster variables as stated in Theorem~\ref{intro-thm-skein-inclusion-A}(a), (b), and (c).
\end{theorem}

We will prove Theorem~\ref{thm-skein-inclusion-A} in three steps: 
\begin{enumerate}
    \item  Compute explicit expressions for the corner arcs of $\mathbb P_3$ in terms of the cluster variables in $\mathscr{A}_\omega(\mathbb P_3)$  
   (Lemma \ref{lem-Cij-bar-Cij}, Propositions~\ref{prop-Cij-bar-Cij} and \ref{prop-bar-Cij}).  

   \item Determine the cluster-variable expression for a stated arc connecting two opposite edges of $\mathbb P_4$  
   (Theorem~\ref{thm-P4-Dij}). 

  \item 
To prove Theorem~\ref{thm-skein-inclusion-A}, cut a copy of $\mathbb P_3$ or $\mathbb P_4$ out of $\fS$. Then apply the formulas established in Steps (1) and (2), together with the splitting homomorphisms for the projected ${\rm SL}_n$ skein algebra and upper cluster algebras, and their compatibility as stated in Theorem~\ref{thm:splitU}.
\end{enumerate}

\subsection{Triangle case}\label{sub-triangle-case}
Recall that we defined the barycentric coordinates for the (ideal) triangle 
\(\mathbb{P}_3\), whose three vertices are labeled \(v_1, v_2, v_3\) (see Figure~\ref{Fig;coord_ijk}), in 
\S\ref{sec-traceX}. 
Here, we introduce two different labelings of the vertices 
in $\{(i,j,k)\mid i,j,k\in \mathbb Z_{\geq 0}, i+j+k=n\}$, which apply to two different corner arcs cases (Lemma~\ref{lem-Cij-bar-Cij}, Propositions~\ref{prop-Cij-bar-Cij} and \ref{prop-bar-Cij}).
For any $i,j$ with $0\leq j\leq i\leq n$, we abbreviate $v_{ij}=(n-i,j,i-j)$ and $\overline v_{ji}=(j,n-i,i-j)$ and denote $A_{ij}$ and $\overline A_{ij}$ the quantum cluster variable associated to $v_{ij}$ and $\overline v_{ij}$, respectively. In Particular, we set $A_{00}=A_{n0}=A_{nn}=\overline A_{00}=\overline A_{0n}=\overline A_{nn}=1$.


For $1\leq j\leq i\leq n$, we use $C_{ij}\in\widetilde{\cS}_\omega(\mathbb P_3)\subset \mathscr U_{\omega}(\mathbb P_3)$ to denote the stated arc in Figure~\ref{Cij-P3}. Define $\overline{C}_{ij}$ as the stated arc in $\widetilde{\cS}$ obtained from $C_{ij}$ by reversing its orientation.

\begin{figure}[h]
    \centering
    \includegraphics[width=75pt]{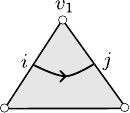}
    \caption{The picture for $C_{ij}$}\label{Cij-P3}
\end{figure}

Using Theorem~\ref{thm-trace-A}(d) and Theorem~\ref{lem-trace-image-cornerarc-P3}, the following is immediate by calculation.

\begin{lemma}\label{lem-Cij-bar-Cij}
\begin{enumerate}[label={\rm (\alph*)}]\itemsep0,3em
For any $i$, we have

\item  $C_{i1}=[A_{i1}\cdot A_{i0}^{-1}\cdot A_{11}^{-1}]\in \mathscr A_{\omega}(\mathbb P_3)$.

\item $C_{ii}=[A_{i0}^{-1}\cdot A_{i-1,0}] \in \mathscr A_{\omega}(\mathbb P_3)$.

\item  $\overline{C}_{ni}=
    [\overline A_{ii}^{-1}\cdot \overline A_{i-1,i}] \in \mathscr A_{\omega}(\mathbb P_3).$

\item $\overline{C}_{ii}=
     [\overline A_{in}^{-1}\cdot \overline A_{i-1,n}] \in \mathscr A_{\omega}(\mathbb P_3).$

\end{enumerate} 
\end{lemma}

For any $i,j$, we abbreviate $\mu_{v_{ij}}=\mu_{ij}$. For any $j,k$ with $1\leq j<k$, we denote 
\begin{align}\label{def-eq-mu-k;j}
\mu_{(k;j)}=\mu_{kj}\cdots\mu_{k2}\mu_{k1}.
\end{align}

The following is the main result of this subsection.

\begin{proposition}\label{prop-Cij-bar-Cij}
For any $1<j<i$, we have 
$$C_{ij}=[\mu_{(j;j-1)}\cdots \mu_{(i-2;j-1)} \mu_{(i-1;j-1)} (A_{j,j-1})\cdot  A_{i0}^{-1}\cdot A_{jj}^{-1}] \in \mathscr A_{\omega}(\mathbb P_3).$$
\end{proposition}

The rest of this section is devoted to prove Proposition  \ref{prop-Cij-bar-Cij}.

Recall that, for $1\leq j\leq i\leq n$, we defined the path set $\mathsf P(\mathbb P_3,v_1,i,j)$ in \eqref{def-path-P3-v1}. For each $p\in \mathsf P(\mathbb P_3,v_1,i,j)$, define
\begin{align}\label{def-path-A-monomial}
    A_p= \psi^{-1}_{\mathbb P_3}(Z^{n{\bf k}^p - {\bf k}_1})\in\mathcal A_\omega(\mathbb P_3),
\end{align}
where $\psi^{-1}_{\mathbb P_3}$ is defined in \eqref{eq-iso-A-balanced-psi-inv}, ${\bf k}_1$ is defined in \eqref{def:proj}, and ${\bf k}^p$ is defined in \eqref{eq-def-k-gamma}.
It follows from the definition of $\psi^{-1}_{\mathbb P_3}$ that
$A_p$ is a monomial.

Theorem~\ref{thm-trace-A}(d) and Theorem~\ref{lem-trace-image-cornerarc-P3} imply the following 
\begin{lemma}\label{lem-Cij-path-sum-A}
    For any $1\leq j\leq i$, we have 
    $$C_{ij} = \sum_{p\in\mathsf{P}(\mathbb P_3,v_1,i,j)} A_p \in \mathcal A_{\omega}(\mathbb P_3).$$
\end{lemma}

For any $n\geq 3$, consider the  vertices sets $V_{\mathbb P_3}^n=\{(i,j,k)\mid i,j,k\in \mathbb Z_{\geq 0}, i+j+k=n\}=V_{\mathbb P_3}\sqcup\{v_1,v_2,v_3\}$ and $V_{\mathbb P_3}^{n-1}=\{(i,j,k)\mid i,j,k\in \mathbb Z_{\geq 0}, i+j+k=n-1\}$ associated with ${\rm SL}_n$ and ${\rm SL}_{n-1}$, respectively. 
The map defined by $(i,j,k)\mapsto (i,j,k+1)$ gives a natural embedding $\iota: V_{\mathbb P_3}^{n-1}\hookrightarrow V_{\mathbb P_3}^{n}$. This embedding naturally induces an $R$-module inclusion of quantum tori (see \eqref{def-iota-from-n-1-to-n})
$$\iota: \mathcal{A}_{\omega}(\mathbb P_3,n-1)\hookrightarrow \mathcal{A}_{\omega}(\mathbb P_3,n),$$ 
where $\mathcal{A}_{\omega}(\mathbb P_3,n-1)$ and $\mathcal{A}_{\omega}(\mathbb P_3,n)$ denote the quantum tori associated with the triangle $\mathbb P_3$ for ${\rm SL}_{n-1}$ and ${\rm SL}_n$, respectively. 
Note that $\iota$ is not an algebra embedding.

For any $i,j$ with $1\leq j\leq i\leq n-1$, denote by $A^{(n-1)}_{n-1-i,j}$ the cluster variable in $\mathcal{A}_{\omega}(\mathbb P_3,n-1)$ associated with the vertex $(i,j,n-1-i-j)\in V_{\mathbb P_3}^{n-1}$ with the convention that $A^{(n-1)}_{0,0}=
A^{(n-1)}_{n-1,0}=A^{(n-1)}_{0,n-1}=1$.
For any sequence of integers $(k_{ij}),0\leq j\leq i\leq n-1$, then 
\begin{align}\label{def-iota-from-n-1-to-n}
    \iota\left(\Big[ \prod_{0\leq j\leq i\leq n-1} \big(A_{ij}^{(n-1)}\big)^{k_{ij}}\Big] \right) = \Big[ \prod_{0\leq j\leq i\leq n-1} (A_{i+1,j})^{k_{ij}}\Big].
\end{align}

Denote
\(
v_1^{(n-1)} = (n-1,0,0) \in V_{\mathbb P_3}^{\,n-1},
\)
regarded as the vertex \(v_1\) of \(\mathbb P_3\) for \({\rm SL}_{n-1}\).
For any \(1 \le j \le i \le n-1\), let
\(
\mathsf P^{\,n-1}\!\bigl(\mathbb P_3, v_1^{(n-1)}, i, j\bigr)
\)
denote the associated path set for \({\rm SL}_{n-1}\).
Moreover, for any \(1 \le j \le i \le n-1\) and any path
\(p \in \mathsf P^{\,n-1}\!\bigl(\mathbb P_3, v_1^{(n-1)}, i, j\bigr)\),
write
\(
A_p^{(n-1)}
\)
for the monomial in
\(\mathcal A_\omega(\mathbb P_3, n-1)\)
defined via \eqref{def-path-A-monomial} for \({\rm SL}_{n-1}\).

Notation defaults to ${\rm SL}_n$ unless otherwise specified for ${\rm SL}_{n-1}$.


\begin{lemma}\label{lem:divide2}
Under the embedding $\iota: V_{\mathbb P_3}^{n-1}\hookrightarrow V_{\mathbb P_3}^{n}$, for any $i,j$ with $2\leq j<i\leq n$, there is a one-to-one correspondence 
$$\mathsf P^{n-1}(\mathbb P_3,v_1^{(n-1)},i-1,j-1)\sqcup \mathsf P^{n-1}(\mathbb P_3,v_1^{(n-1)},i-1,j)\xrightarrow[1:1]{\iota} \mathsf P(\mathbb P_3,v_1,i,j)
.$$
Moreover, under this bijection,

\begin{enumerate}[label={\rm (\alph*)}]\itemsep0,3em

\item for any $p\in \mathsf P^{n-1}(\mathbb P_3,v^{(n-1)}_1,i-1,j-1)$, we have 
$$A_{\iota(p)}=\begin{cases}
\iota(A_{p}^{(n-1)}) & \mbox{ if $j\geq 3$,}\\
[A_{10}\cdot \iota(A_{p}^{(n-1)})] & \mbox{ if $j=2$.}
\end{cases}$$
 \item for any $p\in \mathsf P^{n-1}(\mathbb P_3,v'_1,i-1,j)$, we have 
 $$A_{\iota(p)}=[A_{j+1,j} \cdot A_{jj}^{-1} \cdot A_{j,j-1}^{-1}\cdot A_{j-1,j-1}\cdot \iota(A_{p}^{(n-1)})].$$
 \end{enumerate}
\end{lemma}

\begin{proof}
    The one-to-one correspondence is immediate. (a) and (b) follow by direct calculation.
\end{proof}

We denote by $Q_{\mathbb P_3}^{(n-1)}$ the corresponding anti-symmetric matrix for ${\rm SL}_{n-1}$. Let $k_1,\dots,k_m$ (with $m\ge 0$) be mutable vertices in $V_{{\mathbb P}_3}^{(n-1)}$, and set
\[
Q^{(n-1)'} = \mu_{k_1}\cdots \mu_{k_m}\bigl(Q_{\mathbb P_3}^{(n-1)}\bigr), 
\qquad
Q' = \mu_{\iota(k_1)}\cdots \mu_{\iota(k_m)}\bigl(Q_{\mathbb P_3}\bigr).
\]
Then, for any small vertices $u,v \in V_{{\mathbb P}_3}^{(n-1)}$ with 
$\{u,v\} \not\subset \{(i,j,k)\in V_{{\mathbb P}_3}^{(n-1)} \mid k=0\}$, we have
\(
Q^{(n-1)'}(u,v) = Q'(\iota(u),\iota(v)).
\)
Moreover, for any 
$w \in \{v_{kk} \mid k=1,\dots,n-1\}$ and 
$w' \in \{v_{st} \mid 3 \le s \le n,\, 1 \le t \le s-2\}$, we have
\(
Q'(w,w') = 0.
\)

From these observations, we obtain the following lemma.

\begin{lemma}\label{lem:emb1}
Under the embedding $\iota: \mathcal{A}_{\omega}(\mathbb P_3,n-1)\hookrightarrow \mathcal{A}_{\omega}(\mathbb P_3,n)$, for any $i,j$ with $3\leq j< i\leq n$, we have
\begin{align*}
    &\iota\Big(\Big[\mu_{(j-1;j-2)}\cdots \mu_{(i-3;j-2)} \mu_{(i-2;j-2)} (A^{(n-1)}_{j-1,j-2})\cdot (A_{i-1,0}^{(n-1)})^{-1}\cdot (A_{j-1,j-1}^{(n-1)})^{-1}\Big] \Big)\\
    =&\Big[\mu_{(j;j-2)}\cdots \mu_{(i-2;j-2)} \mu_{(i-1;j-2)} (A_{j,j-2})\cdot
    A_{i,0}^{-1}\cdot A_{j,j-1}^{-1}\Big], \text{ and}\\
    &\iota\Big(\Big[\mu_{(j;j-1)}\cdots \mu_{(i-3;j-1)} \mu_{(i-2;j-1)} (A^{(n-1)}_{j,j-1})\cdot  (A_{i-1,0}^{(n-1)})^{-1}\cdot (A_{j,j}^{(n-1)})^{-1}\Big]\Big)\\
    =&\Big[\mu_{(j+1;j-1)}\cdots \mu_{(i-2;j-1)} \mu_{(i-1;j-1)} (A_{j+1,j-1})
    \cdot A_{i,0}^{-1}\cdot A_{j+1,j}^{-1}\Big]
    \quad \text{when $i> j+1$}.
\end{align*}
   
\end{lemma}

 Lemmas \ref{lem:quiver0} and \ref{lem:quiver3} follow directly from a computation of the quiver mutation.

\begin{lemma}\label{lem:quiver0}
For $i$ with $i>3$, in the quiver $Q'=\mu_{31}\cdots \mu_{i-2,1} \mu_{i-1,1}(Q_{\mathbb P_3})$, we have $$Q'(v_{21},v)=\begin{cases}
        -1 & \mbox{ if $v=v_{10},v_{22}$ or $v_{i1}$,}\\
        1 & \mbox{ if $v=v_{11}$ or $v_{31}$,}\\
        0 & \mbox{ otherwise.}
    \end{cases}$$
\end{lemma}

\begin{lemma}\label{lem:quiver3}
  For any $i,j$ with $i=j+1>3$, in the quiver $Q'=\mu_{(j;j-1)}(Q_{\mathbb P_3})$, we have $$Q'(v_{j,j-1},v)=\begin{cases}
        1 & \mbox{ if $v=v_{j,j}$ or $v_{j,j-2}$,}\\
        -1 & \mbox{ if $v=v_{j0},v_{j-1,j-1}$ or $v_{j+1,j}$,}\\
        0 & \mbox{ otherwise.}
    \end{cases}$$
\end{lemma}

\begin{lemma}\label{lem:quiver1}
  For any $i,j$ with $i>j+1>3$, in the quiver $Q'=\mu_{(j;j-1)}\cdots \mu_{(i-2;j-1)} \mu_{(i-1;j-1)}(Q_{\mathbb P_3})$, we have $$Q'(v_{j,j-1},v)=\begin{cases}
        -1 & \mbox{ if $v=v_{j+1,j-1}$ or $v_{j-1,j-1}$,}\\
        1 & \mbox{ if $v=v_{j,j-2}$ or $v_{j,j}$,}\\
        0 & \mbox{ otherwise.}
    \end{cases}$$
\end{lemma}

\begin{lemma}\label{lem:mut2}
For any $n$ and $i,j$ with $3\leq j< i\leq n$, we have
   
\begin{enumerate}[label={\rm (\alph*)}]\itemsep0,3em

\item $\mu_{(j;j-1)}\cdots \mu_{(i-2;j-1)} \mu_{(i-1;j-1)} (A_{j,j-2})=\mu_{(j;j-2)}\cdots \mu_{(i-2;j-2)} \mu_{(i-1;j-2)} (A_{j,j-2})$, i.e., the quantum cluster variables associated to the vertex $v_{j,j-2}$ are the same in the quantum seeds $\mu_{(j;j-1)}\cdots \mu_{(i-2;j-1)} \mu_{(i-1;j-1)}(w_{\mathbb P_3})$ and $\mu_{(j;j-2)}\cdots \mu_{(i-2;j-2)} \mu_{(i-1;j-2)}(w_{\mathbb P_3})$ (see \eqref{def-qum-seed-w}).

\item $\mu_{(j;j-1)}\cdots \mu_{(i-2;j-1)} \mu_{(i-1;j-1)} (A_{j+1,j-1})=\mu_{(j+1;j-1)}\cdots \mu_{(i-2;j-1)} \mu_{(i-1;j-1)} (A_{j+1,j-1})$, i.e., the quantum cluster variables associated to the vertex $v_{j+1,j-1}$ are the same in the quantum seeds $\mu_{(j;j-1)}\cdots \mu_{(i-2;j-1)} \mu_{(i-1;j-1)}(w_{\mathbb P_3})$ and $\mu_{(j+1;j-1)}\cdots \mu_{(i-2;j-1)} \mu_{(i-1;j-1)}(w_{\mathbb P_3})$ (see \eqref{def-qum-seed-w}).  
\end{enumerate}
\end{lemma}

We left the proofs of Lemmas~\ref{lem:quiver1} and \ref{lem:mut2} to Appendix~\ref{sec:proof of Lemmas}. 

\medskip

\begin{proof}[Proof of Proposition \ref{prop-Cij-bar-Cij}]:

For $j=2$, if $i=3$, then
$$\mu_{21}(A_{21})=[A_{10}\cdot A_{22}\cdot A_{31}\cdot A_{21}^{-1}]+[A_{11}\cdot A_{20}\cdot A_{32}\cdot A_{21}^{-1}].$$

By Lemmas~\ref{lem-Cij-bar-Cij}, \ref{lem-Cij-path-sum-A}, and \ref{lem:divide2}, we have $$C_{32}=[A_{10}\cdot A_{31}\cdot A_{30}^{-1}\cdot A_{21}^{-1}]+[A_{32}\cdot A_{22}^{-1}\cdot A_{21}^{-1}\cdot A_{11}\cdot A_{30}^{-1}\cdot A_{20}].$$

Thus we have $C_{32}=[\mu_{21}(A_{21})\cdot A_{30}^{-1}\cdot A_{22}^{-1}]$ and Proposition \ref{prop-Cij-bar-Cij} holds for the case that $j=2,i=3$.

We then consider the case that $j=2$ and $i>3$. By Lemma~\ref{lem:quiver0}, we have 
 \begin{equation*}
    \begin{array}{rcl}
& & \mu_{21}\cdots \mu_{i-2,1} \mu_{i-1,1} (A_{21})\\
&=&[\mu_{31}\cdots \mu_{i-2,1} \mu_{i-1,1} (A_{31}) \cdot A_{11}\cdot A_{21}^{-1}]\\
&+& [A_{10}\cdot A_{22}\cdot A_{i1}\cdot A_{21}^{-1}].\\
&:=& I+[A_{10}\cdot A_{22}\cdot A_{i1}\cdot A_{21}^{-1}].
\end{array}
\end{equation*}
By induction on $i$, and Lemmas~\ref{lem-Cij-bar-Cij}(a),
\ref{lem-Cij-path-sum-A}, \ref{lem:divide2}, and \ref{lem:emb1}, we have
$$\begin{cases}
I=[\sum_{p\in \mathsf P^{n-1}(\mathbb P_3,v_1^{(n-1)},i-1,2)} \iota(A_{p}^{(n-1)})\cdot A_{i0}\cdot A_{32}\cdot A_{11}\cdot A_{21}^{-1}]=[\sum_{p\in \mathsf P^{n-1}(\mathbb P_3,v_1^{(n-1)},i-1,j-1)} A_{\iota(p)}\cdot A_{i0}\cdot A_{22}],\\
C_{i2}=\sum_{p\in \mathsf P^{n-1}(\mathbb P_3,v_1^{(n-1)},i-1,2)} A_{\iota(p)}+\sum_{p\in \mathsf P^{n-1}(\mathbb P_3,v_1^{(n-1)},i-1,1)} A_{\iota(p)},\\
\sum_{p\in \mathsf P^{n-1}(\mathbb P_3,v_1^{(n-1)},i-1,1)} A_{\iota(p)}=[A_{10}\cdot A_{i0}^{-1}\cdot A_{i1}\cdot A_{21}^{-1}].
\end{cases}$$

The induction hypothesis together with Lemma~\ref{lem:emb1} yields
\[
\Big(
  \sum_{p \in \mathsf P^{n-1}\!\bigl(\mathbb P_3,
      v_1^{(n-1)}, i-1, 2\bigr)}
      \iota(A_{p}^{(n-1)})\Big)
A_{i0}\,A_{32}\,A_{11}\,A_{21}^{-1}
=
\xi^{m}\,
A_{i0}\,A_{32}\,A_{11}\,A_{21}^{-1}
  \sum_{p \in \mathsf P^{n-1}\!\bigl(\mathbb P_3,
      v_1^{(n-1)}, i-1, 2\bigr)}
      \iota(A_{p}^{(n-1)})
\]
for some integer \(m\), where \(\xi = \omega^{n}\).
Hence the Weyl normalization notation \([\,\sim\,]\) used in the above identities are well defined.  
Moreover, every occurrence of \([\,\sim\,]\) in the subsequent proof is justified for the similar reason.

Thus, we obtain $C_{i2}=[\mu_{i-1,1}\cdots \mu_{31} \mu_{21}(A_{21})\cdot A_{i0}^{-1}\cdot A_{22}^{-1}]$. It follows that Proposition \ref{prop-Cij-bar-Cij} holds for $j=2$.

For $j\geq 3$, if $i=j+1$, the proof is similar as above by using Lemmas~\ref{lem-Cij-bar-Cij}(b) and \ref{lem:quiver3}.

For $j\geq 3$, if $i>j+1$, then by Lemmas \ref{lem:quiver1} and \ref{lem:mut2}, we have 

 \begin{equation*}
    \begin{array}{rcl}
& & \mu_{(j;j-1)}\cdots \mu_{(i-2;j-1)} \mu_{(i-1;j-1)} (A_{j,j-1})\\
&=&[\mu_{(j;j-2)}\cdots \mu_{(i-2;j-2)} \mu_{(i-1;j-2)} (A_{j,j-2}) \cdot A_{jj}\cdot A_{j,j-1}^{-1}]\\
&+& [\mu_{(j+1;j-1)}\cdots \mu_{(i-2;j-1)} \mu_{(i-1;j-1)} (A_{j+1,j-1})\cdot A_{j-1,j-1}\cdot A_{j,j-1}^{-1}]\\
&:=& I_1+I_2.
\end{array}
\end{equation*}
By induction on $i+j$, and Lemmas~\ref{lem-Cij-path-sum-A}, \ref{lem:divide2}, and \ref{lem:emb1}, we have 
$$\begin{cases}
I_1=[\sum_{p\in \mathsf P^{n-1}(\mathbb P_3,v_1^{(n-1)},i-1,j-1)} \iota(A_{p}^{(n-1)})\cdot A_{i0}\cdot A_{jj}]=[\sum_{p\in \mathsf P^{n-1}(\mathbb P_3,v_1^{(n-1)},i-1,j-1)} A_{\iota(p)}\cdot A_{i0}\cdot A_{jj}],\\
I_2=[\sum_{p\in \mathsf P^{n-1}(\mathbb P_3,v_1^{(n-1)},i-1,j)} \iota(A_{p}^{(n-1)})\cdot A_{i0}\cdot A_{j+1,j}\cdot  A_{j-1,j-1}\cdot A_{j,j-1}^{-1}]\\
\hspace{0.41cm}=[\sum_{p\in \mathsf P^{n-1}(\mathbb P_3,v_1^{(n-1)},i-1,j)} A_{\iota(p)}\cdot A_{i0}\cdot A_{jj}],\\
C_{ij}=\sum_{p\in \mathsf P^{n-1}(\mathbb P_3,v_1^{(n-1)},i-1,j)} A_{\iota(p)}+\sum_{p\in \mathsf P^{n-1}(\mathbb P_3,v_1^{(n-1)},i-1,j-1)} A_{\iota(p)}.
\end{cases}$$

Therefore, we obtain $$C_{ij}=[\mu_{(j;j-1)}\cdots \mu_{(i-2;j-1)} \mu_{(i-1;j-1)} (A_{j,j-1})\cdot A_{i0}^{-1}\cdot A_{jj}^{-1}].$$

The proof is complete.
\end{proof}

We abbreviate $\mu_{\overline v_{ij}}=\overline\mu_{ij}$. 
For any $k,t$ such that $k+t<n$, we denote 
\begin{align}\label{eq-def-bar-mu-k;t}
  \overline\mu_{(k;t)}=\overline \mu_{k,k+1} \overline \mu_{k,k+2}\cdots \overline \mu_{k,k+t}.
\end{align}

Recall that $\overline{C}_{ij}$ is the stated arc in $\mathbb P_3$ obtained from $C_{ij}$ by reversing its orientation.
A parallel statement to Theorem~\ref{lem-trace-image-cornerarc-P3} also holds for $\overline{C}_{ij}$ (see \cite[Theorem~10.5 and Figure~13(B)]{LY23}).
The following proposition is similar to Proposition \ref{prop-Cij-bar-Cij}, we omit its proof.

\begin{proposition}\label{prop-bar-Cij} For any $i,j$ with $1\leq j<i<n$, we have 
$$\overline C_{ij}=[\overline \mu_{(j;n-i)}\cdots \overline \mu_{(i-2;n-i)} \overline \mu_{(i-1;n-i)} (\overline A_{j,j+1})\cdot \overline A_{jj}^{-1}\cdot \overline A_{in}^{-1}]\in \mathscr A_{\omega}(\mathbb P_3).$$
\end{proposition}

\subsection{Quadrilateral case}
\label{Quadrilateral-case}

\begin{figure}[h]
    \centering
    \includegraphics[width=75pt]{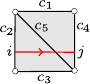}
    \caption{The picture for $D_{ij}$}\label{Fig:P4}
\end{figure}
Choose a triangulation $\lambda$ of the quadrilateral $\mathbb P_4$, we label the ideal arcs in $\lambda$ by $c_1,c_2,c_3,c_4$ and $c_5$, see Figure \ref{Fig:P4}. Denote by $\mathbb P_3^1$ and $\mathbb P_3^2$ the triangles formed by $c_2,c_3,c_5$ and $c_1,c_4,c_5$, respectively. Denote by $v_1$ and $\overline v_1$ the common vertices of $c_2, c_5$ and $c_4,c_5$, respectively. We label the vertices in $V_{\mathbb P_3^1}^n$ by $\{v_{ij}\mid 0\leq j\leq i\leq n\}$ and the cluster variables by $A_{ij}$ (i.e., $A_{ij}=A_{v_{ij}}\in\mathcal A_\omega(\mathbb P_4,\lambda)$ with the convention that $A_{00}= A_{n0}=A_{nn}=1$) with respect to $v_1$ (see \S\ref{sub-triangle-case}). We label the vertices in $V_{\mathbb P_3^2}^n$ by $\{\overline v_{ji}\mid 0\leq j\leq i\leq n\}$ and the cluster variables by $\overline A_{ij}$ (i.e., $\overline{A}_{ij}=A_{\overline v_{ij}}\in\mathcal A_\omega(\mathbb P_4,\lambda)$ with the convention that $\overline A_{00}= \overline A_{0n}=\overline A_{nn}=1$) with respect to $\overline v_1$ (see \S\ref{sub-triangle-case}). In particular, we have $v_{ii}=\overline v_{ii}$ and $A_{ii}=\overline A_{ii}$ for any $i$.

For any $j>1$, denote 
\begin{align}\label{sec7-def-mu-square-j}
    \overline \mu^{\diamondsuit}_j=\bar \mu_{(1;n-j)}\cdots \bar \mu_{(j-2;n-j)}\bar \mu_{(j-1;n-j)},
\end{align}
where $\overline{\mu}_{(k;t)}$ is defined as in \eqref{eq-def-bar-mu-k;t}.
Note that these mutations are applied to the triangulation of $\mathbb P_4$.

For any $i,j$ with $i\geq j>1$, denote 
\begin{align}\label{sec7-def-mu-square-ij}
    \mu^{\diamondsuit}_{(i;j-1)}=\left(\mu_{(2;1)}\mu_{(3;2)}\cdots\mu_{(j-1;j-2)}\right)\circ \left(\mu_{(j;j-1)}\mu_{(j+1;j-1)}\cdots \mu_{(i-1;j-1)}\right),
\end{align}
where $\mu_{(k;j)}$ is defined as in
\eqref{def-eq-mu-k;j}.
Note that these mutations are applied to the triangulation of $\mathbb P_4$.


For any $1\leq i,j\leq n$, we use $D_{ij}\in\widetilde{\cS}_\omega(\mathbb P_4)\subset \mathscr U_{\omega}(\mathbb P_4)$ to denote the stated arc in Figure~\ref{Fig:P4}.
The following is the main result of this subsection.

\begin{theorem}\label{thm-P4-Dij}
    In $\mathscr A_{\omega}(\mathbb P_4)$, for any $1\leq i,j\leq n$, we have 
    $$D_{ij}=
    \begin{cases}
        [A_{i1}\cdot A_{i0}^{-1}\cdot \overline A_{1n}^{-1}]
        & \mbox{ if $i\geq j=1$,}\\
        [\overline \mu^{\diamondsuit}_{j}(\overline A_{12})\cdot A_{10}^{-1}\cdot \overline A_{jn}^{-1} ] & \mbox{ if $j>i=1$,}\vspace{1.5mm}\\
        [\left(\mu_{j-1,j-1}\cdots\mu_{22}\mu_{11}\right)\circ \overline \mu^{\diamondsuit}_j\circ  \mu^{\diamondsuit}_{(i;j-1)}(A_{j-1,j-1}) \cdot A_{i0}^{-1}\cdot\overline A_{jn}^{-1}] & \mbox{ if $i\geq j>1$,}\vspace{1.5mm}\\
        [\left(\mu_{i-1,i-1}\cdots\mu_{22}\mu_{11}\right)\circ \overline \mu^{\diamondsuit}_j\circ \mu^{\diamondsuit}_{(i;i-1)}(A_{i-1,i-1}) \cdot A_{i0}^{-1}\cdot\overline A_{jn}^{-1}] & \mbox{ if $j>i>1$},
    \end{cases}$$
    where $\mu_{k,k}=\mu_{v_{kk}}$
    for $1\leq k\leq n-1$.
\end{theorem}

The rest of this section is devoted to prove Theorem \ref{thm-P4-Dij}.

Recall that we defined the splitting homomorphism 
$$\mathbb S_{c_5}^U\colon \mathscr{U}_\omega(\mathbb P_4)\rightarrow \mathscr{U}_\omega(\mathbb P_3^1)\otimes \mathscr{U}_\omega(\mathbb P_3^2)\quad \text{(Theorem~\ref{thm:splitU}).}$$
For each pair $i,j$ with $0\leq j\leq i\leq n$, we can naturally regard $A_{ij}$ (resp. $\overline{A}_{ji}$) as an element in $\mathscr{U}_\omega(\mathbb P_3^1)$ (resp. $\mathscr{U}_\omega(\mathbb P_3^2)$). Then we have the following.

\begin{lemma}\label{lem:splitk}
  For any $i,j$ with $j\leq i$ and any $k$ with $2\leq k\leq j$, in $\mathscr{U}_\omega(\mathbb P_3^1)\otimes \mathscr{U}_\omega(\mathbb P_3^2)$ we have 

\begin{enumerate}[label={\rm (\alph*)}]\itemsep0,3em

\item
 \begin{equation*}
    \begin{array}{rcl}
&&\mathbb S_{c_5}^U(\mu_{(k;k-1)}\cdots \mu_{(i-2;k-1)} \mu_{(i-1;k-1)} (A_{k,k-1}))\vspace{1mm}\\& =& 
[\mu_{(k;k-1)}\cdots \mu_{(i-2;k-1)} \mu_{(i-1;k-1)} (A_{k,k-1})\otimes \overline A_{k-1,k-1}\cdot \overline A_{ii}],
\end{array}
\end{equation*} 

\item  
 \begin{equation*}
    \begin{array}{rcl}
&&\mathbb S_{c_5}^U(\overline \mu_{(k;n-j)}\cdots \overline \mu_{(j-2;n-j)} \overline \mu_{(j-1;n-j)} (\overline A_{k,k+1}))\vspace{1mm}\\& =& 
[ A_{k-1,k-1}\cdot A_{jj}\otimes \overline \mu_{(k;n-j)}\cdots \overline \mu_{(j-2;n-j)} \overline \mu_{(j-1;n-j)} (\overline A_{k,k+1})].
\end{array}
\end{equation*} 
\end{enumerate}  
\end{lemma}

\begin{lemma}\label{lem:same}
For any $i,j$ with $j\leq i$ and any $k$ with $2\leq k\leq j$, in $\mathscr{U}_\omega(\mathbb P_4)$ we have 

\begin{enumerate}[label={\rm (\alph*)}]\itemsep0,3em

\item
 $\overline \mu^{\diamondsuit}_j\circ  \mu^{\diamondsuit}_{(i;j-1)}(A_{k,k-1})=\mu_{(k;k-1)}\cdots \mu_{(i-2;k-1)} \mu_{(i-1;k-1)} (A_{k,k-1}),$

\item  $\overline \mu^{\diamondsuit}_j\circ  \mu^{\diamondsuit}_{(i;j-1)}(\overline A_{k,k+1})=
\overline \mu_{(k;n-j)}\cdots \overline \mu_{(j-2;n-j)} \overline \mu_{(j-1;n-j)} (\overline A_{k,k+1})$.
\end{enumerate}
\end{lemma}

\begin{lemma}\label{lem:quiver2}
For any $i\geq j> 1$,
    in the quiver $Q'=\overline \mu^{\diamondsuit}_j\circ  \mu^{\diamondsuit}_{(i;j-1)}(Q_\lambda)$, we have the following: 
\begin{enumerate}[label={\rm (\alph*)}]\itemsep0,3em

\item The subquiver formed by the vertices $v_{j-1,j-1},\cdots, v_{22}, v_{11}$ is the type $A_{j-1}$ quiver with linear orientation.

\item For any $v\neq v_{11},v_{22},\cdots,v_{j-1,j-1}$, we have
$$Q'(v_{11},v)=\begin{cases}
    1 &  \mbox{ if $v=v_{21}, \bar v_{23}$,}\\
    -1 & \mbox{ if $v=v_{i1},\bar v_{12}$,}\\
    0 & \mbox{ otherwise,}
\end{cases}\hspace{5mm}
Q'(v_{j-1,j-1},v)=\begin{cases}
    1 &  \mbox{ if $v=v_{j,j-1}, \bar v_{j-1,n}$,}\\
    -1 & \mbox{ if $v=v_{j-1,j-2},\bar v_{j-1,j+1}$,}\\
    0 & \mbox{ otherwise}.
\end{cases}$$
$$Q'(v_{kk},v)=\begin{cases}
    1 &  \mbox{ if $v=v_{k+1,k}, \bar v_{k+1,k+2}$,}\\
    -1 & \mbox{ if $v=v_{k,k-1},\bar v_{k,k+1}$,}\\
    0 & \mbox{ otherwise,}
\end{cases}$$
for $1<k<j-1$.

\end{enumerate}
\end{lemma}

We will prove Lemmas \ref{lem:splitk},  \ref{lem:same}, and  \ref{lem:quiver2} in Appendix~\ref{sec:proof of Lemmas}.

The following is an immediate corollary of Lemma \ref{lem:quiver2} and the expansion formula for quantum cluster algebras of type $A$, as seen, for example, in \cite{R,CL,H1,H2,H3}.

\begin{corollary}\label{cor:cv}
If $i\geq j>1$, then, in $\mathscr{U}_\omega(\mathbb P_4)$, we have 
 \begin{equation*}
    \begin{array}{rcl}
&&\left(\mu_{j-1,j-1}\cdots\mu_{22}\mu_{11}\right)\circ \overline \mu^{\diamondsuit}_j\circ  \mu^{\diamondsuit}_{(i;j-1)}(A_{j-1,j-1})\vspace{1mm}\\& =& [A_{11}^{-1}\cdot A_{i1}\cdot \overline \mu^{\diamondsuit}_j\circ  \mu^{\diamondsuit}_{(i;j-1)}(\overline A_{12})]+[A_{j-1,j-1}^{-1}\cdot  \overline \mu^{\diamondsuit}_j\circ  \mu^{\diamondsuit}_{(i;j-1)}(A_{j,j-1})\cdot \overline A_{j-1,n}] \vspace{1mm}\\
&+&\sum_{k=2}^{j-1}[A_{k-1,k-1}^{-1}\cdot A_{kk}^{-1}\cdot \overline \mu^{\diamondsuit}_j\circ  \mu^{\diamondsuit}_{(i;j-1)}(A_{k,k-1})\cdot \overline \mu^{\diamondsuit}_j\circ  \mu^{\diamondsuit}_{(i;j-1)}(\overline A_{k,k+1})].
\end{array}
\end{equation*}

\end{corollary}



\begin{proof}[Proof of Theorem \ref{thm-P4-Dij}:]

{\bf Case 1}: $i\geq j=1$.
We have $$\mathbb S_{c_5}^U([A_{i1}\cdot A_{i0}^{-1}\cdot \overline A_{1n}^{-1}])=[A_{i1}\cdot A_{i0}^{-1}\cdot A_{11}^{-1}\otimes \overline A_{1n}^{-1}]$$ and by Theorem~\ref{thm:splitU} and Lemma \ref{lem-Cij-bar-Cij}, we have
$$
\mathbb S_{c_5}^U(D_{i1})=\mathbb S_{c_5}(D_{i1})=C_{i1}\otimes \overline{C}_{11}=[A_{i1}\cdot A_{i0}^{-1}\cdot A_{11}^{-1}\otimes \overline A_{1n}^{-1}].$$

As $\mathbb S_e^U$ is injective, we obtain $D_{i1}=[A_{i1}\cdot A_{i0}^{-1}\cdot \overline A_{1n}^{-1}]$.

{\bf Case 2}: $j>i=1$.
By Lemma \ref{lem:splitk}, we have
\begin{equation*}
    \begin{array}{rcl}
& &\mathbb S_{c_5}^U\left([\overline \mu_{(1;n-j)}\cdots \overline \mu_{(j-2;n-j)} \overline \mu_{(j-1;n-j)} (\overline A_{12})\cdot A_{10}^{-1}\cdot \overline A_{jn}^{-1} ]\right)\vspace{1mm}\\& =& 
[A_{10}^{-1}\otimes \overline \mu_{(1;n-j)}\cdots \overline \mu_{(j-2;n-j)} \overline \mu_{(j-1;n-j)} (\overline A_{12})\cdot \overline A_{11}^{-1}\cdot \overline A_{jn}^{-1}].
\end{array}
\end{equation*}

By Theorem~\ref{thm:splitU}, Lemma \ref{lem-Cij-bar-Cij}, and Proposition \ref{prop-bar-Cij}, we have
$$S_{c_5}^U(D_{ij})=\mathbb S_{c_5}(D_{1j})=C_{11}\otimes \overline{C}_{j1}=[A_{10}^{-1}\otimes \overline \mu_{(1;n-j)}\cdots \overline \mu_{(j-2;n-j)} \overline \mu_{(j-1;n-j)} (\overline A_{12})\cdot \overline A_{11}^{-1}\cdot \overline A_{jn}^{-1}].$$

As $\mathbb S_{c_5}^U$ is injective, we obtain $D_{1j}=[\overline \mu_{(1;n-j)}\cdots \overline \mu_{(j-2;n-j)} \overline \mu_{(j-1;n-j)} (\overline A_{12})\cdot A_{10}^{-1}\cdot \overline A_{jn}^{-1} ]$.

{\bf Case 3}: $i\geq j>1$. 
By Corollary \ref{cor:cv}, Lemmas \ref{lem:same}, and \ref{lem:splitk}, we have 
 \begin{equation}\label{eq:su1}
    \begin{array}{rcl}
& &\mathbb S_{c_5}^U\left(\left(\mu_{j-1,j-1}\cdots\mu_{22}\mu_{11}\right)\circ \overline \mu^{\diamondsuit}_j\circ  \mu^{\diamondsuit}_{(i;j-1)}(A_{j-1,j-1})\right)\vspace{1mm}\\
& =& [A_{i1}\cdot A_{jj}\cdot A_{11}^{-1}\otimes \overline A_{ii}\cdot \overline A_{11}^{-1}\cdot \overline \mu_{(1;n-j)}\cdots \overline \mu_{(j-2;n-j)} \overline \mu_{(j-1;n-j)} (\overline A_{12})]\vspace{1mm}\\& +&
[\mu_{(j;j-1)}\cdots \mu_{(i-2;j-1)} \mu_{(i-1;j-1)} (A_{j,j-1})\otimes \overline A_{ii}\cdot \overline A_{j-1,n}]\vspace{1mm} \\
&+&\sum_{k=2}^{j-1}[A_{jj}\cdot A_{kk}^{-1}\cdot \mu_{(k;k-1)}\cdots \mu_{(i-2;k-1)} \mu_{(i-1;k-1)} (A_{k,k-1})\vspace{1mm}\\&\otimes& \overline A_{ii} \cdot \overline A_{kk}^{-1}\cdot\overline \mu_{(k;n-j)}\cdots \overline \mu_{(j-2;n-j)} \overline \mu_{(j-1;n-j)} (\overline A_{k,k+1})]\vspace{1mm}
\end{array}
\end{equation}

On the other hand, by Theorem~\ref{thm:splitU}, Lemmas \ref{lem:same},  \ref{lem-Cij-bar-Cij}, and Propositions \ref{prop-Cij-bar-Cij}, \ref{prop-bar-Cij}, we have 
 \begin{equation}\label{eq:su2}
    \begin{array}{rcl}
& &
S_{c_5}^U(D_{ij})=
\mathbb S_e(D_{ij})=\sum_{k=1}^j C_{ik}\otimes \overline C_{jk}\vspace{1mm}\\& =& [A_{i1}\cdot A_{i0}^{-1}\cdot A_{11}^{-1}\otimes \overline A_{jn}^{-1}\cdot \overline A_{11}^{-1}\cdot \overline \mu_{(1;n-j)}\cdots \overline \mu_{(j-2;n-j)} \overline \mu_{(j-1;n-j)} (\overline A_{12})]\vspace{1mm}\\& +&
[A_{i0}^{-1}\cdot A_{jj}^{-1}\cdot\mu_{(j;j-1)}\cdots \mu_{(i-2;j-1)} \mu_{(i-1;j-1)} (A_{j,j-1})\otimes \overline A_{jn}^{-1}\cdot \overline A_{j-1,n}]\vspace{1mm} \\
&+&\sum_{k=2}^{j-1}[A_{i0}^{-1}\cdot A_{kk}^{-1}\cdot \mu_{(k;k-1)}\cdots \mu_{(i-2;k-1)} \mu_{(i-1;k-1)} (A_{k,k-1})]\vspace{1mm}\\&\otimes& [\overline A_{jn}^{-1} \cdot \overline A_{kk}^{-1}\cdot\overline \mu_{(k;n-j)}\cdots \overline \mu_{(j-2;n-j)} \overline \mu_{(j-1;n-j)} (\overline A_{k,k+1})].
\end{array}
\end{equation}

It follows by (\ref{eq:su1}) and (\ref{eq:su2}) that 
$$\mathbb S_{c_5}^U(D_{ij})=\mathbb S_{c_5}^U\left([\left(\mu_{j-1,j-1}\cdots\mu_{22}\mu_{11}\right)\circ \overline \mu^{\diamondsuit}_j\circ  \mu^{\diamondsuit}_{(i;j-1)}(A_{j-1,j-1})\cdot A_{i0}^{-1}\cdot \overline A_{jn}^{-1}]\right).$$

As $\mathbb S_e^U$ is injective, we obtain 
$$D_{ij}=[\left(\mu_{j-1,j-1}\cdots\mu_{22}\mu_{11}\right)\circ \overline \mu^{\diamondsuit}_j\circ  \mu^{\diamondsuit}_{(i;j-1)}(A_{j-1,j-1})\cdot A_{i0}^{-1}\cdot \overline A_{jn}^{-1}]$$

The case that $j>i>1$ is similar to the case that $i\geq j>1$.

The proof is complete.
\end{proof}

\subsection{General case}
Let $\fS$ be a pb surface such that each connected component has non-empty boundary. 
 Let $B$ be a collection of properly embedded, disjoint, oriented arcs in $\fS$. 
We say that $B$ is a \emph{saturated system} if the following hold:
\begin{enumerate}
    \item After cutting $\fS$ along $B$, every component of the resulting surface contains exactly one ideal point.
    \item $B$ is maximal with respect to condition (1).
\end{enumerate}
For any $b\in B$, and $1\leq i,j\leq n$, define $b_{ij}$ to be the stated arc such that the state of the starting (resp. ending) point of $b$ is equipped with the state $i$ (resp. $j$).

\begin{lemma}\cite[Theorem 6.2(3)]{LY23}\label{lem-saturated}
Let $\fS$ be a pb surface such that each connected component has non-empty boundary, and
 let $B$ be a saturated system of $\fS$.
    The algebra $\cS_{\omega}(\fS)$ is generated by
$$\{b_{ij}\mid b\in B, 1\leq i,j\leq n\}.$$
\end{lemma}

Recall that a properly embedded oriented arc in $\fS$ is called an essential arc  
if its endpoints lie on two distinct components of $\partial \fS$.

\begin{lemma}\label{lem-essential-arc}
Let $\fS$ be a triangulable pb surface without interior punctures. 
Suppose that each connected component of $\fS$ contains at least two punctures. 
Then the algebra $\cS_{\omega}(\fS)$ is generated by a finite family of stated essential arcs. 
\end{lemma}

\begin{proof}
By Lemma~\ref{lem-saturated}, the algebra $\cS_{\omega}(\fS)$ is generated by a finite family of stated arcs
$\{C_1,\dots,C_r\}$, where each $C_i$ is a properly embedded oriented arc in $\fS$.
If every $C_i$ is essential, the proof is complete.  

Suppose instead that for some $j$, the endpoints of $C_j$ both lie on the same boundary component $c\subset \partial \fS$.
Since each connected component of $\fS$ has at least two boundary components, there exists another boundary component $c'$ in the same connected component of $\fS$.  

We claim that there is an embedded arc $\gamma\colon[0,1]\to \fS$ such that 
$\gamma(0)\in C_j$, $\gamma(1)\in c'$, and $\gamma$ intersects $C_j$ only at $\gamma(0)$.
To construct $\gamma$, let $\Sigma$ be the surface obtained by cutting $\fS$ along $C_j$.
Then $\partial \Sigma$ contains two copies of $C_j$, denoted $C_j'$ and $C_j''$.
Since $C_j$ and $c'$ lie in the same connected component of $\fS$, 
at least one of $C_j'$ or $C_j''$ lies in the same connected component of $\Sigma$ as $c'$.
Without loss of generality, assume this is $C_j'$.
Then there exists a properly embedded arc $\gamma'\colon[0,1]\to \Sigma$ with $\gamma'(0)\in C_j'$ and $\gamma'(1)\in c'$.
Projecting back via the natural map ${\bf P}\colon \Sigma\to \fS$, we obtain the desired arc $\gamma={\bf P}(\gamma')$.

Using $\gamma$, we can isotope $C_j$ so that near $c'$, the local configuration of $C_j\cup c'$ matches the left-hand side of relation \eqref{wzh.seven}.
Applying relation \eqref{wzh.seven} (or equivalently \cite[Equation~(60)]{LS21}, depending on the relative height order of the endpoints of $C_j$), we may express
\[
C_j = \sum_{1\leq t\leq n} s_t W_t,
\]
where each $W_t$ is a product of two stated essential arcs.  

Thus, all generators can be expressed in terms of finitely many stated essential arcs, which proves the lemma.
\end{proof}

\begin{remark}
Under the assumptions of Lemma~\ref{lem-essential-arc}, there exists a saturated system $B$ of $\fS$ consisting entirely of essential arcs. 
Lemma~\ref{lem-saturated} then implies that the algebra $\cS_{\omega}(\fS)$ is generated by
\[
\{\, b_{ij}\mid b\in B,\; 1\leq i,j\leq n \,\},
\]
where each $b_{ij}$ is a stated essential arc.  
For the purposes of this paper, however, we do not require such a strong result: Lemma~\ref{lem-saturated} is sufficient, and it is both easier to understand and easier to prove.
\end{remark}


Under the assumptions of Lemma~\ref{lem-essential-arc}, let $\lambda$ be a triangulation of $\fS$ containing the ideal arcs $e_1,e_2,e_3$ as in Figure~\ref{Fig;tau-v1}(A), where $e_1\neq e_2$. 
Let $\tau$ be the triangle whose edges are $e_1,e_2,e_3$, and let 
\[
f_\tau\colon \mathbb P_3 \longrightarrow \fS \qquad \text{(see \eqref{eq-character-map})}
\]
be the characteristic map of $\tau$, sending each edge $e_i$ of $\mathbb P_3$ (Figure~\ref{Fig;coord_ijk}) to the corresponding $e_i$ in Figure~\ref{Fig;tau-v1}. 
In Lemma~\ref{lem-Cij-bar-Cij}, Propositions~\ref{prop-Cij-bar-Cij} and \ref{prop-bar-Cij}, we proved
that
$$C_{ij}, \overline C_{ij}\in \mathscr A_{\omega}(\mathbb P_3) \qquad \text{(see Figure~\ref{Cij-P3})}.$$
 
We regard \(C_{ij}\) and \(\overline{C}_{ij}\) as elements of \(\mathscr A_{\omega}(\fS)\) via the map \(f_\tau\),  
since \(f_\tau\) injectively maps all vertices of \(\mathbb P_3\) that realize \(C_{ij}\) and \(\overline{C}_{ij}\) as cluster variables.  
To distinguish them from the  elements \(C_{ij},\overline{C}_{ij}\in \mathscr{A}_\omega(\mathbb P_3)\),  
we denote them by \(C_{ij}^\fS\) and \(\overline{C}_{ij}^\fS\), respectively.
In other words, \(C_{ij}^\fS, \overline{C}_{ij}^\fS \in \mathscr A_{\omega}(\fS)\) are defined through the map \(f_\tau\) and the cluster–variable formulas given in  
Lemma~\ref{lem-Cij-bar-Cij} and Propositions~\ref{prop-Cij-bar-Cij} and \ref{prop-bar-Cij}.

Recall that we labeled small vertices in $V_{\mathbb P_3}$ by $v_{ij}$ for $(i,j)\in\{0\leq j\leq i\leq n\}\setminus\{(0,0),(n,0),(n,n)\}$.
By a slight abuse of notation, we continue to write $A_{ij}$ for 
$A_{f_\tau(v_{ij})}\in \mathcal A_\omega(\fS,\lambda)$.
Then $A_{ij}\in \mathcal A_\omega(\fS,\lambda)$ is the variable associated to the small vertex in $V_\lambda$ labeled by $ij$ in Figure~\ref{Fig;tau-v1}(B).

Let $\mathbb P_3 \sqcup \fS'$ be the pb surface obtained from $\fS$ by cutting along $e_2$, and let $\lambda'$ be the triangulation of $\fS'$ induced by $\lambda$ when $\fS'\neq\emptyset$.  
For each $1\le j\le n-1$, there is a small vertex 
$v_{nj}'' \in V_{\lambda'}$ that is identified with the small vertex 
$v_{nj} \in V_{\mathbb P_3}$ when we glue $\fS'$ and $\mathbb P_3$ back together to recover $\fS$.  
We denote by $A_{nj}''$ the element $A_{v_{nj}''}\in \mathcal A_\omega(\fS',\lambda')$.
We use the convention that $A_{n0}''=
A_{nn}''=1$.

We set $A_{nj}''=1$, for $0\leq j\leq n$, when $\fS'=\emptyset$, i.e., $\fS=\mathbb P_3$, and use the convention that $\widetilde{\cS}_\omega(\emptyset)=R$.



\begin{lemma}\label{lem:split_e2}
In $\mathscr{A}_\omega(\mathbb P_3)\otimes \mathscr{A}_\omega(\fS'
)$, we have the following:

\begin{enumerate}[label={\rm (\alph*)}]\itemsep0,3em

\item $\mathbb S^{U}_{e_2}(A_{ij})=
    [A_{ij}\otimes A''_{nj}]$ for $(i,j)\in\{0\leq j\leq i\leq n\}\setminus\{(0,0),(n,0),(n,n)\}$.

\item $\mathbb S^{U}_{e_2}(\mu_{(j;j-1)}\cdots \mu_{(i-2;j-1)} \mu_{(i-1;j-1)} (A_{j,j-1}))=[\mu_{(j;j-1)}\cdots \mu_{(i-2;j-1)} \mu_{(i-1;j-1)} (A_{j,j-1})\otimes A''_{nj}]$ for any $i,j$ with $1<j<i$.

\end{enumerate}
\end{lemma}

We will prove Lemma \ref{lem:split_e2} in Appendix \ref{sec:proof of Lemmas}.

\begin{lemma}\label{lem-splitting-corner-arc}

We have
\[
\mathbb S^{U}_{e_2}\!\bigl(C_{ij}^\fS\bigr)
= C_{ij}\otimes 1
\;\in\;
\mathscr A_{\omega}(\mathbb P_3)\otimes \mathscr A_{\omega}(\fS').
\]
\end{lemma}

\begin{proof}
    If $j=1$, then by Lemma \ref{lem-Cij-bar-Cij} and Lemma \ref{lem:split_e2}, 
    $$\mathbb S^{U}_{e_2}\!\bigl(C_{i1}^\fS\bigr)= \mathbb S^{U}_{e_2}\!\bigl([A_{i1}\cdot A_{i0}^{-1}\cdot A_{11}^{-1}]\bigr)=[A_{i1}\cdot A_{i0}^{-1}\cdot A_{11}^{-1}]\otimes 1=C_{i1}\otimes 1.$$

    If $j=i$, then by Lemma \ref{lem-Cij-bar-Cij} and Lemma \ref{lem:split_e2}, 
    $$\mathbb S^{U}_{e_2}\!\bigl(C_{ii}^\fS\bigr)= \mathbb S^{U}_{e_2}\!\bigl([A_{i0}^{-1}\cdot A_{i-1,0}]\bigr)=[A_{i0}^{-1}\cdot A_{i-1,0}]\otimes 1=C_{ii}\otimes 1.$$

    If $1<j<i$, then by Proposition \ref{prop-Cij-bar-Cij} and Lemma \ref{lem:split_e2},
 \begin{equation*}
    \begin{array}{rcl}
\mathbb S^{U}_{e_2}\!\bigl(C_{ij}^\fS\bigr)&=& \mathbb S^{U}_{e_2}\!\bigl([\mu_{(j;j-1)}\cdots \mu_{(i-2;j-1)} \mu_{(i-1;j-1)}(A_{j,j-1})\cdot A_{i0}^{-1}\cdot A_{jj}^{-1}]\bigr)\vspace{1mm}\\& =& 
[\mu_{(j;j-1)}\cdots \mu_{(i-2;j-1)} \mu_{(i-1;j-1)} (A_{j,j-1})\cdot A_{i0}^{-1}\cdot A_{jj}^{-1}]\otimes 1\vspace{1mm}\\& =& 
C_{ij}\otimes 1.
\end{array}
\end{equation*}

    The proof is complete.
\end{proof}

We can prove the following using the same technique as Lemma~\ref{lem-splitting-corner-arc}.

\begin{lemma}\label{lem-splitting-corner-bar-arc}

We have
\[
\mathbb S^{U}_{e_2}\!\bigl(\overline C_{ij}^\fS\bigr)
=\overline C_{ij}\otimes 1
\;\in\;
\mathscr A_{\omega}(\mathbb P_3)\otimes \mathscr A_{\omega}(\fS').
\]
\end{lemma}

Under the assumptions of Lemma~\ref{lem-essential-arc},
let $\gamma$ be an essential arc that is not a corner arc (see the red arc in Figure~\ref{Fig;essential-P4}(A)).  
There are five ideal arcs $c_1.c_2,c_3,c_4,c_5$ nearby $\gamma$ as depicted in Figure~\ref{Fig;essential-P4}(A).
Let $\lambda$ be a triangulation of $\fS$ containing the ideal arcs $c_1,c_2,c_3,c_4,c_5$.
(note that it is possible that $c_1=c_3$).  
Define 
\[
f_\gamma \colon \mathbb P_4 \longrightarrow \fS
\]
to be a homeomorphism such that $f_\gamma$ maps each $c_i$ in Figure~\ref{Fig:P4} to the corresponding $c_i$ in Figure~\ref{Fig;essential-P4}, and $f_\gamma$ is injective on $\mathbb P_4 \setminus (c_1\cup c_3)$.  

For $1\leq i,j\leq n$, Theorem~\ref{thm-P4-Dij} shows that
\[
D_{ij}\in \mathscr A_{\omega}(\mathbb P_4)\qquad\text{(see Figure~\ref{Fig:P4})}.
\]
We may regard $D_{ij}$ as an element of $\mathscr A_{\omega}(\fS)$ via the map $f_\gamma$, since $f_\gamma$ injectively maps all vertices of $\mathbb P_4$ that define $D_{ij}$.
We denote this element by $D_{ij}^\fS\in \mathscr A_{\omega}(\fS)$ to distinguish it from $D_{ij}\in \mathscr A_{\omega}(\mathbb P_4)$.

Let $\mathbb P_4 \sqcup \fS''$ be the pb surface obtained from $\fS$ by cutting along $c_1$ and $c_3$.  
Note that $\fS''=\emptyset$ when $c_1=c_3$.
Let $\lambda''$ be the triangulation of $\fS''$ induced by $\lambda$ when $\fS''\neq\emptyset$, and let $\lambda_4=\{c_1,c_2,c_3,c_4,c_5\}$ denote the triangulation of $\mathbb P_4$.  
Recall the labeling of $A$–variables in $\mathcal A_\omega(\mathbb P_4,\lambda_4)$ introduced in \S\ref{Quadrilateral-case}.  
The map $f_\gamma$ induces the same labeling for the $A$–variables in  
$\mathcal A_\omega(\fS,\lambda)$ associated with the small vertices contained in  
$f_\gamma\bigl(\mathbb P_4 \setminus (c_1 \cup c_3)\bigr)$.  
Similarly, all mutations appearing in Theorem~\ref{thm-P4-Dij} are well defined in  
$\mathscr A_\omega(\fS)$ via $f_\gamma$.

For each $1 \le j \le n-1$, there is a small vertex  
$v_{nj}'' \in V_{\lambda''}$  
(resp.\ $\overline{v}_{0j}'' \in V_{\lambda''}$) that is identified with the small vertex  
$v_{nj} \in V_{\mathbb P_4}$  
(resp.\ $\overline v_{0j} \in V_{\mathbb P_4}$) when we glue $\fS''$ and $\mathbb P_4$ back together to recover $\fS$.  
We denote by $A_{nj}''$ (resp.\ $\overline A_{0j}''$) the element  
$A_{v_{nj}''} \in \mathcal A_\omega(\fS'',\lambda'')$  
(resp.\ $A_{\overline v_{0j}''} \in \mathcal A_\omega(\fS'',\lambda'')$).
Note that, when $\fS''=\emptyset$, we have 
$A_{nj}''=\overline A_{0j},
\overline A_{0j}''=A_{nj}
\in \mathcal A_\omega(\mathbb P_4,\lambda_4)$.

We use the convention that 
$A_{n0}''=A_{nn}''=\overline{A}_{00}''=
\overline{A}_{0n}''=1.$

\begin{lemma}\label{lem:spitP_4}
Recall the convention that $\widetilde{\cS}_\omega(\emptyset)=R$.
In $\mathscr{A}_\omega(\mathbb P_4)\otimes \mathscr{A}_\omega(\fS''
)$, we have the following (we regard ``$\otimes$" as the product in $\mathscr{A}_\omega(\mathbb P_4)$ when $\fS''=\emptyset$):

\begin{enumerate}[label={\rm (\alph*)}]\itemsep0,3em

\item
For $(i,j)\in\{0\leq i,j\leq n\}\setminus\{(0,0), (0,n), (n,0), (n,n)\}$, we have
$$\text{$\mathbb S^{U}_{c_1}\!\bigl(\mathbb S^{U}_{c_3}(A_{ij})\bigr)=[A_{ij}\otimes A''_{nj}\cdot \overline A''_{0i}]$ and $\mathbb S^{U}_{c_1}\!\bigl(\mathbb S^{U}_{c_3}(\overline A_{ij})\bigr)=[\overline A_{ij}\otimes A''_{ni}\cdot \overline A''_{0j}]$.}$$

\item For $j>1$, we have
 \begin{equation*}
    \begin{array}{rcl}
&&\mathbb S^{U}_{c_1}\!\bigl(\mathbb S^{U}_{c_3}(\overline \mu_{(1;n-j)}\cdots \overline \mu_{(j-2;n-j)} \overline \mu_{(j-1;n-j)} (\overline A_{12}))\bigr)\vspace{1mm}\\& =& 
[\overline \mu_{(1;n-j)}\cdots \overline \mu_{(j-2;n-j)} \overline \mu_{(j-1;n-j)} (\overline A_{12})\otimes A''_{nj}\cdot \overline A''_{01}].
\end{array}
\end{equation*}

\item For $1<j<i$, we have
\begin{equation*}
    \begin{array}{rcl}
&&\mathbb S^{U}_{c_1}\!\bigl(\mathbb S^{U}_{c_3}(\left(\mu_{j-1,j-1}\cdots\mu_{22}\mu_{11}\right)\circ \overline \mu^{\diamondsuit}_j\circ  \mu^{\diamondsuit}_{(i;j-1)}(A_{j-1,j-1})))\bigr)\vspace{1mm}\\& =& 
[\left(\mu_{j-1,j-1}\cdots\mu_{22}\mu_{11}\right)\circ \overline \mu^{\diamondsuit}_j\circ  \mu^{\diamondsuit}_{(i;j-1)}(A_{j-1,j-1})\otimes A''_{nj}\cdot \overline A''_{0i}].
\end{array}
\end{equation*}

\end{enumerate}
\end{lemma}

We will prove Lemma \ref{lem:spitP_4} in Appendix \ref{sec:proof of Lemmas}.

\begin{lemma}\label{lem-splitting-essential-arc}
We have
\[
\mathbb S^{U}_{c_1}\!\bigl(\mathbb S^{U}_{c_3}(D_{ij}^\fS)\bigr)
= D_{ij}\otimes 1
\;\in\;
\mathscr A_{\omega}(\mathbb P_4)\otimes \mathscr A_{\omega}(\fS'').
\]
\end{lemma}

\begin{proof}
    If $i\geq j=1$, then by Theorem \ref{thm-P4-Dij} and Lemma \ref{lem:spitP_4}(a), 
    $$\mathbb S^{U}_{c_1}\!\bigl(\mathbb S^{U}_{c_3}(D_{ij}^\fS)\bigr)=\mathbb S^{U}_{c_1}\!\bigl(\mathbb S^{U}_{c_3}([A_{i1}\cdot A_{i0}^{-1}\cdot \overline A_{1n}^{-1}])\bigr)=[A_{i1}\cdot A_{i0}^{-1}\cdot\overline A_{1n}^{-1}]\otimes 1=D_{ij}\otimes 1.$$

   If $j>i=1$, then by Theorem \ref{thm-P4-Dij} and Lemma \ref{lem:spitP_4}(a), (b), we have
     \begin{equation*}
    \begin{array}{rcl}
\mathbb S^{U}_{c_1}\!\bigl(\mathbb S^{U}_{c_3}(D_{ij}^\fS)\bigr)&=&\mathbb S^{U}_{c_1}\!\bigl(\mathbb S^{U}_{c_3}([\overline \mu_{(1;n-j)}\cdots \overline \mu_{(j-2;n-j)} \overline \mu_{(j-1;n-j)} (\overline A_{12}) \cdot A_{10}^{-1}\cdot \overline A_{jn}^{-1}])\bigr)\vspace{1mm}\\& =& 
[\overline \mu_{(1;n-j)}\cdots \overline \mu_{(j-2;n-j)} \overline \mu_{(j-1;n-j)} (\overline A_{12})\cdot A_{10}^{-1}\cdot \overline A_{jn}^{-1}]\otimes 1\vspace{1mm}\\& =& 
D_{ij}\otimes 1.
\end{array}
\end{equation*}

 If $i\geq j>1$, then by Theorem \ref{thm-P4-Dij} and Lemma \ref{lem:spitP_4}(a), (c), we have
$$\begin{array}{rcl}
\mathbb S^{U}_{c_1}\!\bigl(\mathbb S^{U}_{c_3}(D_{ij}^\fS)\bigr)&=&\mathbb S^{U}_{c_1}\!\bigl(\mathbb S^{U}_{c_3}([\left(\mu_{j-1,j-1}\cdots\mu_{22}\mu_{11}\right)\circ \overline \mu^{\diamondsuit}_j\circ  \mu^{\diamondsuit}_{(i;j-1)}(A_{j-1,j-1}) \cdot A_{i0}^{-1}\cdot \overline A_{jn}^{-1}])\bigr)\vspace{1mm}\\& =&
[\left(\mu_{j-1,j-1}\cdots\mu_{22}\mu_{11}\right)\circ \overline \mu^{\diamondsuit}_j\circ  \mu^{\diamondsuit}_{(i;j-1)}(A_{j-1,j-1}) \cdot A_{i0}^{-1}\cdot \overline A_{jn}^{-1}]\otimes 1\vspace{1mm}\\& =&
D_{ij}\otimes 1.
\end{array}$$

The proof is complete. 
\end{proof}

\begin{proof}[Proof of Theorem \ref{thm-skein-inclusion-A}]
It follows from Lemma~\ref{lem-essential-arc} that it suffices to show Equations~\eqref{intro-eq-Cij}, \eqref{intro-eq-bar-Cij}, and\eqref{into-eq-Dij}.

Equations~\eqref{intro-eq-Cij} and \eqref{intro-eq-bar-Cij} follow immediately from Lemmas \ref{lem-Cij-bar-Cij}, \ref{lem-splitting-corner-arc}, \ref{lem-splitting-corner-bar-arc}, Propositions \ref{prop-Cij-bar-Cij}, \ref{prop-bar-Cij}, and the injectivity and compatibility of the splitting homomorphisms (Theorem \ref{thm:splitU}).

Equation~\eqref{into-eq-Dij} follows immediately from Theorem \ref{thm-P4-Dij}, Lemma \ref{lem-splitting-essential-arc}, and the injectivity and compatibility of the splitting homomorphisms (Theorem \ref{thm:splitU}).

The proof is complete.
\end{proof}

\begin{appendices}
\section{Proof of Lemmas \ref{lem:quiver1}, \ref{lem:mut2}, \ref{lem:splitk}, \ref{lem:same}, \ref{lem:quiver2}, \ref{lem:split_e2}, and \ref{lem:spitP_4}}\label{sec:proof of Lemmas}

We follow the notation of \S\ref{sec-skein-inclusion-cluster} and do not distinguish between a quiver and its signed adjacency matrix.

The following four Lemmas follow directly from a computation of the quiver mutation.

\begin{lemma}\label{lem:mutation1}
    Let $Q$ be the quiver shown below.
    
\begin{center}
\begin{tikzpicture}[->,>=stealth, node distance=1.5cm, auto]
\node (A) at (0,0) {$30$};
\node (B) at (2,0) {$31$};
\node (C) at (4,0) {$32$};
\node (D) at (6,0) {$\cdots$};
\node (E) at (8,0) {$3,t-1$};
\node (F) at (10,0) {$3t$};
\node (G) at (12,0) {$3,t+1$};
\draw (G) -- node {} (F);
\draw (B) -- node {} (A);
\draw (C) -- node {} (B);
\draw (F) -- node {} (E);
\draw (5.55,0) -- (4.35,0);
\draw (7.4,0) -- (6.35,0);
\node (A') at (0,-1) {$20$};
\node (B') at (2,-1) {$21$};
\node (C') at (4,-1) {$22$};
\node (D') at (6,-1) {$\cdots$};
\node (E') at (8,-1) {$2,t-1$};
\node (F') at (10,-1) {$2t$};
\node (G') at (12,-1) {$2,t+1$};
\draw (G') -- node {} (F');
\draw (B') -- node {} (A');
\draw (C') -- node {} (B');
\draw (F') -- node {} (E');
\draw (5.55,-1) -- (4.35,-1);
\draw (7.4,-1) -- (6.35,-1);
\node (A'') at (0,-2) {$10$};
\node (B'') at (2,-2) {$11$};
\node (C'') at (4,-2) {$12$};
\node (D'') at (6,-2) {$\cdots$};
\node (E'') at (8,-2) {$1,t-1$};
\node (F'') at (10,-2) {$1t$};
\node (G'') at (12,-2) {$1,t+1$};
\draw (G'') -- node {} (F'');
\draw (B'') -- node {} (A'');
\draw (C'') -- node {} (B'');
\draw (F'') -- node {} (E'');
\draw (5.55,-2) -- (4.35,-2);
\draw (7.4,-2) -- (6.35,-2);
\draw (A) -- node {} (B');
\draw (A') -- node {} (B'');
\draw (B) -- node {} (C');
\draw (B') -- node {} (C'');
\draw (E) -- node {} (F');
\draw (E') -- node {} (F'');
\draw (B'') -- node {} (B');
\draw (B') -- node {} (B);
\draw (C'') -- node {} (C');
\draw (C') -- node {} (C);
\draw (F) -- node {} (G');
\draw (F') -- node {} (G'');
\draw (4.35,-0.1) -- (5.55,-0.9);
\draw (4.35,-1.1) -- (5.55,-1.9);
\draw (6.35,-0.1) -- (7.4,-0.9);
\draw (6.35,-1.1) -- (7.4,-1.9);
\draw (5.6,-0.8) -- (5.6,-0.2);
\draw (5.6,-1.8) -- (5.6,-1.2);
\draw (6.3,-0.8) -- (6.3,-0.2);
\draw (6.3,-1.8) -- (6.3,-1.2);
\draw (E'') -- node {} (E');
\draw (E') -- node {} (E);
\draw (F'') -- node {} (F');
\draw (F') -- node {} (F);
\draw (G'') -- node {} (G');
\draw (G') -- node {} (G);
\end{tikzpicture}
\end{center}
    
Then the full subquiver of $\mu_{2t}\cdots \mu_{22}\mu_{21}(Q)$ consisting of all vertices incident to at least one vertex in $\{31,32,\cdots, 3t\}$ is isomorphic to the quiver shown below.

    \begin{center}
\begin{tikzpicture}[->,>=stealth, node distance=1.5cm, auto]
\node (B) at (2,0) {$31$};
\node (C) at (4,0) {$32$};
\node (D) at (6,0) {$\cdots$};
\node (E) at (8,0) {$3,t-1$};
\node (F) at (10,0) {$3t$};
\node (G) at (12,0) {$3,t+1$};
\draw (G) -- node {} (F);
\draw (C) -- node {} (B);
\draw (F) -- node {} (E);
\draw (F') -- node {} (G');
\draw (5.55,0) -- (4.35,0);
\draw (7.4,0) -- (6.35,0);
\node (A') at (0,-1) {$11$};
\node (B') at (2,-1) {$21$};
\node (C') at (4,-1) {$22$};
\node (D') at (6,-1) {$\cdots$};
\node (E') at (8,-1) {$2,t-1$};
\node (F') at (10,-1) {$2t$};
\node (G') at (12,-1) {$2,t+1$};
\draw (B') -- node {} (A');
\draw (C') -- node {} (B');
\draw (F') -- node {} (E');
\draw (5.55,-1) -- (4.35,-1);
\draw (7.4,-1) -- (6.35,-1);
\draw (B) -- node {} (B');
\draw (C) -- node {} (C');
\draw (A') -- node {} (B);
\draw (B') -- node {} (C);
\draw (E') -- node {} (F);
\draw (4.35,-0.9) -- (5.55,-0.1);
\draw (6.35,-0.9) -- (7.4,-0.1);
\draw (5.6,-0.2) -- (5.6,-0.8);
\draw (6.3,-0.2) -- (6.3,-0.8);
\draw (E) -- node {} (E');
\draw (F) -- node {} (F');
\draw (G') -- node {} (G);
\end{tikzpicture}
\end{center}
\end{lemma}

\begin{lemma}\label{lem:mutation2}
    Let $Q$ be the quiver shown below. 
    
     \begin{center}
\begin{tikzpicture}[->,>=stealth, node distance=1.5cm, auto]
\node (A'') at (0,1) {$30$};
\node (B'') at (2,1) {$31$};
\node (C'') at (4,1) {$32$};
\node (D'') at (6,1) {$\cdots$};
\node (E'') at (8,1) {$3,t-1$};
\node (F'') at (10,1) {$3t$};
\node (G'') at (12,1) {$3,t+1$};
\draw (G'') -- node {} (F'');
\draw (B'') -- node {} (A'');
\draw (C'') -- node {} (B'');
\draw (F'') -- node {} (E'');
\draw (5.55,1) -- (4.35,1);
\draw (7.4,1) -- (6.35,1);
\node (B) at (2,0) {$21$};
\node (C) at (4,0) {$22$};
\node (D) at (6,0) {$\cdots$};
\node (E) at (8,0) {$2,t-1$};
\node (F) at (10,0) {$2t$};
\node (G) at (12,0) {$2,t+1$};
\draw (G) -- node {} (F);
\draw (C) -- node {} (B);
\draw (F) -- node {} (E);
\draw (F') -- node {} (G');
\draw (5.55,0) -- (4.35,0);
\draw (7.4,0) -- (6.35,0);
\node (A') at (0,-1) {$10$};
\node (B') at (2,-1) {$11$};
\node (C') at (4,-1) {$12$};
\node (D') at (6,-1) {$\cdots$};
\node (E') at (8,-1) {$1,t-1$};
\node (F') at (10,-1) {$1t$};
\node (G') at (12,-1) {$1,t+1$};
\draw (B') -- node {} (A');
\draw (C') -- node {} (B');
\draw (F') -- node {} (E');
\draw (5.55,-1) -- (4.35,-1);
\draw (7.4,-1) -- (6.35,-1);
\draw (B) -- node {} (B');
\draw (C) -- node {} (C');
\draw (A') -- node {} (B);
\draw (B') -- node {} (C);
\draw (E') -- node {} (F);
\draw (4.35,-0.9) -- (5.55,-0.1);
\draw (6.35,-0.9) -- (7.4,-0.1);
\draw (5.6,-0.2) -- (5.6,-0.8);
\draw (6.3,-0.2) -- (6.3,-0.8);
\draw (E) -- node {} (E');
\draw (F) -- node {} (F');
\draw (G') -- node {} (G);
\draw (B) -- node {} (B'');
\draw (C) -- node {} (C'');
\draw (E) -- node {} (E'');
\draw (F) -- node {} (F'');
\draw (G) -- node {} (G'');
\draw (A'') -- node {} (B);
\draw (B'') -- node {} (C);
\draw (E'') -- node {} (F);
\draw (F'') -- node {} (G);
\draw (4.35,0.9) -- (5.55,0.1);
\draw (6.35,0.9) -- (7.4,0.1);
\draw (5.6,0.2) -- (5.6,0.8);
\draw (6.3,0.2) -- (6.3,0.8);
\end{tikzpicture}
\end{center}   
    
Then the full subquiver of $\mu_{2t}\cdots \mu_{22}\mu_{21}(Q)$ consisting of all vertices incident to at least one vertex in $\{31,32,\cdots, 3t\}$ is isomorphic to the quiver shown below.

    \begin{center}
\begin{tikzpicture}[->,>=stealth, node distance=1.5cm, auto]
\node (B) at (2,0) {$31$};
\node (C) at (4,0) {$32$};
\node (D) at (6,0) {$\cdots$};
\node (E) at (8,0) {$3,t-1$};
\node (F) at (10,0) {$3t$};
\node (G) at (12,0) {$3,t+1$};
\draw (G) -- node {} (F);
\draw (C) -- node {} (B);
\draw (F) -- node {} (E);
\draw (F') -- node {} (G');
\draw (5.55,0) -- (4.35,0);
\draw (7.4,0) -- (6.35,0);
\node (A') at (0,-1) {$10$};
\node (B') at (2,-1) {$21$};
\node (C') at (4,-1) {$22$};
\node (D') at (6,-1) {$\cdots$};
\node (E') at (8,-1) {$2,t-1$};
\node (F') at (10,-1) {$2t$};
\node (G') at (12,-1) {$2,t+1$};
\draw (B') -- node {} (A');
\draw (C') -- node {} (B');
\draw (F') -- node {} (E');
\draw (5.55,-1) -- (4.35,-1);
\draw (7.4,-1) -- (6.35,-1);
\draw (B) -- node {} (B');
\draw (C) -- node {} (C');
\draw (A') -- node {} (B);
\draw (B') -- node {} (C);
\draw (E') -- node {} (F);
\draw (4.35,-0.9) -- (5.55,-0.1);
\draw (6.35,-0.9) -- (7.4,-0.1);
\draw (5.6,-0.2) -- (5.6,-0.8);
\draw (6.3,-0.2) -- (6.3,-0.8);
\draw (E) -- node {} (E');
\draw (F) -- node {} (F');
\draw (G') -- node {} (G);
\end{tikzpicture}
\end{center}
\end{lemma}

\begin{lemma}\label{lem:mutation4}
\begin{enumerate}[label={\rm (\alph*)}]\itemsep0,3em

\item Let $Q$ be the quiver shown below.
    
    \begin{center}
\begin{tikzpicture}[->,>=stealth, node distance=1.5cm, auto]
\node (A) at (0,0) {$30$};
\node (B) at (2,0) {$31$};
\node (C) at (4,0) {$32$};
\node (D) at (6,0) {$\cdots$};
\node (E) at (8,0) {$3,t-1$};
\node (F) at (10,0) {$3t$};
\draw (B) -- node {} (A);
\draw (C) -- node {} (B);
\draw (F) -- node {} (E);
\draw (5.55,0) -- (4.35,0);
\draw (7.4,0) -- (6.35,0);
\node (A') at (0,-1) {$20$};
\node (B') at (2,-1) {$21$};
\node (C') at (4,-1) {$22$};
\node (D') at (6,-1) {$\cdots$};
\node (E') at (8,-1) {$2,t-1$};
\node (F') at (10,-1) {$2t$};
\node (G') at (12,-1) {$2,t+1$};
\draw (G') -- node {} (F');
\draw (B') -- node {} (A');
\draw (C') -- node {} (B');
\draw (F') -- node {} (E');
\draw (5.55,-1) -- (4.35,-1);
\draw (7.4,-1) -- (6.35,-1);
\node (A'') at (0,-2) {$10$};
\node (B'') at (2,-2) {$11$};
\node (C'') at (4,-2) {$12$};
\node (D'') at (6,-2) {$\cdots$};
\node (E'') at (8,-2) {$1,t-1$};
\node (F'') at (10,-2) {$1t$};
\node (G'') at (12,-2) {$1,t+1$};
\draw (G'') -- node {} (F'');
\draw (B'') -- node {} (A'');
\draw (C'') -- node {} (B'');
\draw (F'') -- node {} (E'');
\draw (5.55,-2) -- (4.35,-2);
\draw (7.4,-2) -- (6.35,-2);
\draw (A) -- node {} (B');
\draw (A') -- node {} (B'');
\draw (B) -- node {} (C');
\draw (B') -- node {} (C'');
\draw (E) -- node {} (F');
\draw (E') -- node {} (F'');
\draw (B'') -- node {} (B');
\draw (B') -- node {} (B);
\draw (C'') -- node {} (C');
\draw (C') -- node {} (C);
\draw (F) -- node {} (G');
\draw (F') -- node {} (G'');
\draw (4.35,-0.1) -- (5.55,-0.9);
\draw (4.35,-1.1) -- (5.55,-1.9);
\draw (6.35,-0.1) -- (7.4,-0.9);
\draw (6.35,-1.1) -- (7.4,-1.9);
\draw (5.6,-0.8) -- (5.6,-0.2);
\draw (5.6,-1.8) -- (5.6,-1.2);
\draw (6.3,-0.8) -- (6.3,-0.2);
\draw (6.3,-1.8) -- (6.3,-1.2);
\draw (E'') -- node {} (E');
\draw (E') -- node {} (E);
\draw (F'') -- node {} (F');
\draw (F') -- node {} (F);
\draw (G'') -- node {} (G');
\end{tikzpicture}
\end{center}
    
Then the full subquiver of $\mu_{2t}\cdots \mu_{22}\mu_{21}(Q)$ consisting of all vertices incident to at least one vertex in $\{31,32,\cdots, (3,t-1)\}$ is isomorphic to the quiver shown below.

   \begin{center}
\begin{tikzpicture}[->,>=stealth, node distance=1.5cm, auto]
\node (B) at (2,0) {$31$};
\node (C) at (4,0) {$32$};
\node (D) at (6,0) {$\cdots$};
\node (E) at (8,0) {$3,t-1$};
\node (F) at (10,0) {$3t$};
\draw (C) -- node {} (B);
\draw (F) -- node {} (E);
\draw (5.55,0) -- (4.35,0);
\draw (7.4,0) -- (6.35,0);
\node (A') at (0,-1) {$11$};
\node (B') at (2,-1) {$21$};
\node (C') at (4,-1) {$22$};
\node (D') at (6,-1) {$\cdots$};
\node (E') at (8,-1) {$2,t-1$};
\draw (B') -- node {} (A');
\draw (C') -- node {} (B');
\draw (5.55,-1) -- (4.35,-1);
\draw (7.4,-1) -- (6.35,-1);
\draw (B) -- node {} (B');
\draw (C) -- node {} (C');
\draw (A') -- node {} (B);
\draw (B') -- node {} (C);
\draw (E') -- node {} (F);
\draw (4.35,-0.9) -- (5.55,-0.1);
\draw (6.35,-0.9) -- (7.4,-0.1);
\draw (5.6,-0.2) -- (5.6,-0.8);
\draw (6.3,-0.2) -- (6.3,-0.8);
\draw (E) -- node {} (E');
\end{tikzpicture}
\end{center}

\item   Let $Q$ be the quiver shown below. 
    
     \begin{center}
\begin{tikzpicture}[->,>=stealth, node distance=1.5cm, auto]
\node (A'') at (0,1) {$30$};
\node (B'') at (2,1) {$31$};
\node (C'') at (4,1) {$32$};
\node (D'') at (6,1) {$\cdots$};
\node (E'') at (8,1) {$3,t-1$};
\node (F'') at (10,1) {$3t$};
\draw (B'') -- node {} (A'');
\draw (C'') -- node {} (B'');
\draw (F'') -- node {} (E'');
\draw (5.55,1) -- (4.35,1);
\draw (7.4,1) -- (6.35,1);
\node (B) at (2,0) {$21$};
\node (C) at (4,0) {$22$};
\node (D) at (6,0) {$\cdots$};
\node (E) at (8,0) {$2,t-1$};
\node (F) at (10,0) {$2t$};
\node (G) at (12,0) {$2,t+1$};
\draw (G) -- node {} (F);
\draw (C) -- node {} (B);
\draw (F) -- node {} (E);
\draw (5.55,0) -- (4.35,0);
\draw (7.4,0) -- (6.35,0);
\node (A') at (0,-1) {$10$};
\node (B') at (2,-1) {$11$};
\node (C') at (4,-1) {$12$};
\node (D') at (6,-1) {$\cdots$};
\node (E') at (8,-1) {$1,t-1$};
\node (F') at (10,-1) {$1t$};
\draw (B') -- node {} (A');
\draw (C') -- node {} (B');
\draw (F') -- node {} (E');
\draw (5.55,-1) -- (4.35,-1);
\draw (7.4,-1) -- (6.35,-1);
\draw (B) -- node {} (B');
\draw (C) -- node {} (C');
\draw (A') -- node {} (B);
\draw (B') -- node {} (C);
\draw (E') -- node {} (F);
\draw (4.35,-0.9) -- (5.55,-0.1);
\draw (6.35,-0.9) -- (7.4,-0.1);
\draw (5.6,-0.2) -- (5.6,-0.8);
\draw (6.3,-0.2) -- (6.3,-0.8);
\draw (E) -- node {} (E');
\draw (F) -- node {} (F');
\draw (F') -- node {} (G);
\draw (B) -- node {} (B'');
\draw (C) -- node {} (C'');
\draw (E) -- node {} (E'');
\draw (F) -- node {} (F'');
\draw (A'') -- node {} (B);
\draw (B'') -- node {} (C);
\draw (E'') -- node {} (F);
\draw (F'') -- node {} (G);
\draw (4.35,0.9) -- (5.55,0.1);
\draw (6.35,0.9) -- (7.4,0.1);
\draw (5.6,0.2) -- (5.6,0.8);
\draw (6.3,0.2) -- (6.3,0.8);
\end{tikzpicture}
\end{center}   
    
Then the full subquiver of $\mu_{2t}\cdots \mu_{22}\mu_{21}(Q)$ consisting of all vertices incident to at least one vertex in $\{31,32,\cdots, (3,t-1)\}$ is isomorphic to the quiver shown below.

 \begin{center}
\begin{tikzpicture}[->,>=stealth, node distance=1.5cm, auto]
\node (B) at (2,0) {$31$};
\node (C) at (4,0) {$32$};
\node (D) at (6,0) {$\cdots$};
\node (E) at (8,0) {$3,t-1$};
\node (F) at (10,0) {$3t$};
\draw (C) -- node {} (B);
\draw (F) -- node {} (E);
\draw (5.55,0) -- (4.35,0);
\draw (7.4,0) -- (6.35,0);
\node (A') at (0,-1) {$10$};
\node (B') at (2,-1) {$21$};
\node (C') at (4,-1) {$22$};
\node (D') at (6,-1) {$\cdots$};
\node (E') at (8,-1) {$2,t-1$};
\draw (B') -- node {} (A');
\draw (C') -- node {} (B');
\draw (5.55,-1) -- (4.35,-1);
\draw (7.4,-1) -- (6.35,-1);
\draw (B) -- node {} (B');
\draw (C) -- node {} (C');
\draw (A') -- node {} (B);
\draw (B') -- node {} (C);
\draw (E') -- node {} (F);
\draw (4.35,-0.9) -- (5.55,-0.1);
\draw (6.35,-0.9) -- (7.4,-0.1);
\draw (5.6,-0.2) -- (5.6,-0.8);
\draw (6.3,-0.2) -- (6.3,-0.8);
\draw (E) -- node {} (E');
\end{tikzpicture}
\end{center}

\end{enumerate}
\end{lemma}

\begin{lemma}\label{lem:mutation3}
Assume that $1<j<i$. 
\begin{enumerate}[label={\rm (\alph*)}]\itemsep0,3em

\item For any $k$ with $1<k\leq t$, in the quiver $Q'=\mu_{i-1,k-1}\cdots \mu_{i-1,2}\mu_{i-1,1}(Q_{\mathbb P_3})$, we have
    $$Q'(v_{i-1,k},v)=\begin{cases}
        1 &\mbox{ if } v=v_{i,k},v_{i-1,0},v_{i-2,k+1},\\
        -1 &\mbox{ if } v=v_{i-1,k-1},v_{i-1,k+1},\\
        0 &\mbox{ otherwise.}
    \end{cases}$$
    
\item For any $s$ with $j\leq s<i-1$, in the quiver $Q'=\mu_{(s+1;j-1)}\cdots \mu_{(i-2;j-1)} \mu_{(i-1;j-1)}(Q_{\mathbb P_3})$, we have
    $$Q'(v_{s1},v)=\begin{cases}
        1 &\mbox{ if } v=v_{s-1,1},v_{s+1,1},\\
        -1 &\mbox{ if } v=v_{i1},v_{s-1,0}, v_{s2},\\
        0 &\mbox{ otherwise.}
    \end{cases}$$

\item For any $(s,t)$ with $j\leq s<i-1$, $1<t\leq j-1$,
in the quiver $$Q'=(\mu_{s,t-1}\cdots \mu_{s2}\mu_{s1})\circ \mu_{(s+1;j-1)}\cdots \mu_{(i-2;j-1)} \mu_{(i-1;j-1)}(Q_{\mathbb P_3}),$$
we have
    $$Q'(v_{st},v)=\begin{cases}
        1 &\mbox{ if } v=v_{s-1,t},v_{s+1,t},\\
        -1 &\mbox{ if } v=v_{s,t-1},v_{s,t+1},\\
        0 &\mbox{ otherwise.}
    \end{cases}$$ 
\end{enumerate}
\end{lemma}

Note that Lemma~\ref{lem:mutation3} also applies to $Q_\lambda$, where $\lambda$ is the triangulation in Figure~\ref{Fig:P4} for $\mathbb P_4$, by identifying $\mathbb P_3$ with the triangle bounded by $c_2,c_3,c_5$.

\begin{remark}
Indeed, Lemma \ref{lem:mutation3} follows from Lemmas \ref{lem:mutation1} and \ref{lem:mutation2} via induction.
\end{remark}


\begin{corollary}\label{cor:quiver1}
Let $\lambda$ be the triangulation in Figure~\ref{Fig:P4} for $\mathbb P_4$.
 Let $Q$ be the full subquiver formed by the vertices $v_{ij}$ of the quiver $Q_\lambda$, for $(i,j)\in\{0\leq j\leq i\leq n\}\setminus\{(0,0),(n,0),(n,n)\}$. Then in the quiver $Q'=\mu^{\diamondsuit}_{(i;j-1)}(Q)$ (see \eqref{sec7-def-mu-square-ij}), we have 
\begin{enumerate}[label={\rm (\alph*)}]\itemsep0,3em

\item The subquiver formed by the vertices $v_{j-1,j-1},\cdots, v_{22}, v_{11}$ is the type $A_{j-1}$ quiver with linear orientation.

\item for any $v\neq v_{11},v_{22},\cdots,v_{j-1,j-1}$, we have
$$Q'(v_{11},v)=\begin{cases}
    1 &  \mbox{ if $v=v_{21}$,}\\
    -1 & \mbox{ if $v=v_{i1}$,}\\
    0 & \mbox{ otherwise},
\end{cases}\hspace{5mm}
Q'(v_{j-1,j-1},v)=\begin{cases}
    1 &  \mbox{ if $v=v_{j,j-1}$,}\\
    -1 & \mbox{ if $v=v_{j-1,j-2}$,}\\
    0 & \mbox{ otherwise}.
\end{cases}$$
$$Q'(v_{kk},v)=\begin{cases}
    1 &  \mbox{ if $v=v_{k+1,k}$,}\\
    -1 & \mbox{ if $v=v_{k,k-1}$,}\\
    0 & \mbox{ otherwise,}
\end{cases}$$
for $1<k<j-1$.

\end{enumerate}
\end{corollary}

\begin{proof}
    It follows from Lemmas~\ref{lem:mutation1},  \ref{lem:mutation2}, and \ref{lem:mutation4} via induction. 
\end{proof}

The following Lemma can be proved similarly.

\begin{lemma}\label{lem:quiver5}
We use the $\overline v_{ij}$ labeling for $V_{\mathbb P_3}$. Then for each $j>1$, we defined the mutation  $\overline\mu^{\diamondsuit}_j$ as in \eqref{sec7-def-mu-square-j}.
    In the quiver $Q'=\overline  \mu^{\diamondsuit}_j(Q_{\mathbb P_3})$, we have 
\begin{enumerate}[label={\rm (\alph*)}]\itemsep0,3em

\item The subquiver formed by the vertices $\overline v_{j-1,j-1},\cdots, \overline v_{22},\overline v_{11}$ is the type $A_{j-1}$ quiver with linear orientation.

\item For any $v\neq \overline v_{11},\overline v_{22},\cdots,\overline v_{j-1,j-1}$, we have
$$Q'(\overline v_{11},v)=\begin{cases}
    1 &  \mbox{ if $v=\bar v_{23}$,}\\
    -1 & \mbox{ if $v=\bar v_{12}$,}\\
    0 & \mbox{ otherwise}.
\end{cases}\hspace{5mm}
Q'(\overline v_{j-1,j-1},v)=\begin{cases}
    1 &  \mbox{ if $v=\bar v_{j-1,n}$,}\\
    -1 & \mbox{ if $v=\bar v_{j-1,j+1}$,}\\
    0 & \mbox{ otherwise}.
\end{cases}$$
$$Q'(\overline v_{kk},v)=\begin{cases}
    1 &  \mbox{ if $v=\bar v_{k+1,k+2}$,}\\
    -1 & \mbox{ if $v=\bar v_{k,k+1}$,}\\
    0 & \mbox{ otherwise,}
\end{cases}$$
for $1<k<j-1$.

\end{enumerate}
\end{lemma}

\begin{proposition}\label{lem:imageofsplit1}
Let $e=c_5$, the ideal arc in $\mathbb P_4$ depicted in Figure~\ref{Fig:P4}.
Assume that $1<j<i$. For any $s$ with $j\leq s\leq i-1$ and $t$ with $1\leq t\leq j-1$, we have 
\begin{equation*}
    \begin{array}{rcl}
&&\mathbb S_e^U\left((\mu_{st}\cdots \mu_{s2}\mu_{s1})\circ \mu_{(s+1;j-1)}\cdots \mu_{(i-2;j-1)} \mu_{(i-1;j-1)} (A_{st})\right)\vspace{1mm}\\& =& 
[(\mu_{st}\cdots \mu_{s2}\mu_{s1})\circ \mu_{(s+1;j-1)}\cdots \mu_{(i-2;j-1)} \mu_{(i-1;j-1)} (A_{st})\otimes \overline A_{s-1,s-1}\cdot \overline A_{ii}].
\end{array}
\end{equation*}   
\end{proposition}
\begin{proof}
We proceed by induction on $i-1-s$ and $t$.

First, consider the case $s=i-1$. 

If $t=1$, then 
$$\mathbb S^U_e(\mu_{i-1,1}(A_{i-1,1}))=\mathbb S^U_e([A^{-1}_{i-1,1}\cdot A_{i-1,0}\cdot A_{i-2,1}\cdot A_{i2}]+[A^{-1}_{i-1,1}\cdot A_{i-2,0}\cdot A_{i-1,2}\cdot A_{i1}])=[\mu_{i-1,1}(A_{i-1,1})\otimes \overline A_{i-2,i-2}\cdot \overline A_{ii}].$$

Now suppose $t>1$. By Lemma \ref{lem:mutation3}(a) and the induction hypothesis on $t$, we have 

\begin{equation*}
    \begin{array}{rcl}
&&\mathbb S_e^U\left(\mu_{i-1,t}\cdots \mu_{i-1,2}\mu_{i-1,1}(A_{i-1,t})\right)\vspace{1mm}\\& =& 
\mathbb S^U_e([A^{-1}_{i-1,t}\cdot A_{it}\cdot A_{i-1,0}\cdot A_{i-2,t+1}]+[A^{-1}_{i-1,t}\cdot \mu_{i-1,t-1}\cdots \mu_{i-1,2}\mu_{i-1,1}(A_{i-1,t-1})\cdot A_{i-1,t+1}])\vspace{1mm}\\& =& 
[\mu_{i-1,t}\cdots \mu_{i-1,2}\mu_{i-1,1}(A_{i-1,t})\otimes \overline A_{i-2,i-2}\cdot \overline A_{ii}].
\end{array}
\end{equation*} 

This establishes the result for $s=i-1$.

The case $t=1$ (for general $s$) follows similarly, using Lemma~\ref{lem:mutation3}(b).

Next, consider the case $s<i-1$ and $t>1$.  Applying Lemma~\ref{lem:mutation3}(c) and the induction hypothesis on $i-1-s+t$, we compute

\begin{equation*}
    \begin{array}{rcl}
&&\mathbb S_e^U\left(\mu_{st}\cdots \mu_{s2}\mu_{s1})\circ \mu_{(s+1;j-1)}\cdots \mu_{(i-2;j-1)} \mu_{(i-1;j-1)}(A_{s,t})\right)\vspace{1mm}\\& =& 
\mathbb S^U_e([A^{-1}_{st}\cdot A_{s-1,t}\cdot  (\mu_{s+1,t}\cdots \mu_{s+1,2}\mu_{s+1,1})\circ \mu_{(s+2;j-1)}\cdots \mu_{(i-2;j-1)} \mu_{(i-1;j-1)}(A_{s+1,t})]
\vspace{1mm}\\& +& 
\mathbb S^U_e([A^{-1}_{st}\cdot (\mu_{s,t-1}\cdots \mu_{s2}\mu_{s1})\circ \mu_{(s+1;j-1)}\cdots \mu_{(i-2;j-1)} \mu_{(i-1;j-1)}(A_{s,t-1})\cdot A_{s,t+1}])\vspace{1mm}\\& =& 
[(\mu_{st}\cdots \mu_{s2}\mu_{s1})\circ \mu_{(s+1;j-1)}\cdots \mu_{(i-2;j-1)} \mu_{(i-1;j-1)} (A_{st})\otimes \overline A_{s-1,s-1}\cdot \overline A_{ii}].
\end{array}
\end{equation*} 
This completes the proof.
\end{proof}

\begin{proof}[Proof of Lemma \ref{lem:quiver1}]
By induction, and by repeatedly applying Lemmas~\ref{lem:mutation1} and \ref{lem:mutation2}, the result follows immediately.  
\end{proof}

\begin{proof}[Proof of Lemma \ref{lem:mut2}]
(a) 
The result is clear if $i=j+1$.

For $i>j+1$, repeatedly applying Lemmas~\ref{lem:mutation1} and ~\ref{lem:mutation2}, we have $Q'(v_{j+1,j-1},v)=Q'(v,v_{j+1,j-1})=0$ for all $v\in \{v_{jt}\mid 1\leq t\leq j-2\}$ in the quiver $Q'=\mu_{(j+1;j-2)}\mu_{(j+2;j-1)}\cdots  \mu_{(i-1;j-1)}(Q_{\mathbb P_3})$. 

Therefore, the mutation at $v_{j+1,j-1}$ commutes with the mutations $\mu_{(j;j-2)}$ for the quiver $Q'=\mu_{(j+1;j-2)}\mu_{(j+2;j-1)}\cdots  \mu_{(i-1;j-1)}(Q_{\mathbb P_3})$, and we obtain
\begin{equation*}
    \begin{array}{rcl}
&&\mu_{(j;j-1)}\mu_{(j+1;j-1)}\mu_{(j+2;j-1)}\cdots  \mu_{(i-1;j-1)} (A_{j,j-2})\vspace{1mm}\\& =& 
\mu_{(j;j-2)}\left(\mu_{j+1,j-1}\mu_{(j+1;j-2)}\right) \mu_{(j+2;j-1)}\cdots \mu_{(i-1;j-1)} (A_{j,j-2})\vspace{1mm}\\& =& 
\mu_{j+1,j-1}\mu_{(j;j-2)}\mu_{(j+1;j-2)} \mu_{(j+2;j-1)}\cdots \mu_{(i-1;j-1)} (A_{j,j-2})\vspace{1mm}\\& =& 
\mu_{(j;j-2)}\mu_{(j+1;j-2)}\mu_{(j+2;j-1)} \cdots \mu_{(i-1;j-1)} (A_{j,j-2}).
\end{array}
\end{equation*}

Repeatedly applying Lemma~\ref{lem:mutation1} in a similar manner, we further deduce that
$$\mu_{(j;j-1)}\cdots \mu_{(i-2;j-1)} \mu_{(i-1;j-1)} (A_{j,j-2})=\mu_{(j;j-2)}\cdots \mu_{(i-2;j-2)} \mu_{(i-1;j-2)} (A_{j,j-2}).$$

(b) The proof is analogous to that of (a), and is therefore omitted.
\end{proof}

\begin{proof}[Proof of Lemma \ref{lem:splitk}]
Part (a) follows directly as a corollary of Proposition \ref{lem:imageofsplit1}.

Part (b) can be established using a similar argument to that of part (a).
\end{proof}

\begin{proof}[Proof of Lemma \ref{lem:same}]
The proof of Lemma \ref{lem:same} is analogous to that of Lemma \ref{lem:mut2}, and is therefore omitted.    
\end{proof}

\begin{proof}[Proof of Lemma \ref{lem:quiver2}]
It follows by Corollary \ref{cor:quiver1} and Lemma \ref{lem:quiver5}.
\end{proof}

\begin{proposition}\label{lem:imageofsplite2}
With the notation in Lemma \ref{lem:split_e2}.
Assume that $1<j<i$. For any $s$ with $j\leq s\leq i-1$ and $t$ with $1\leq t\leq j-1$, we have 

\begin{enumerate}[label={\rm (\alph*)}]\itemsep0,3em
\item $\mathbb S^{U}_{e_2}(A_{ij})=
    [A_{ij}\otimes A''_{nj}]$ for $(i,j)\in\{0\leq j\leq i\leq n\}\setminus\{(0,0),(n,0),(n,n)\}$.

\item 
\begin{equation*}
    \begin{array}{rcl}
&&\mathbb S_{e_2}^U\left((\mu_{st}\cdots \mu_{s2}\mu_{s1})\circ \mu_{(s+1;j-1)}\cdots \mu_{(i-2;j-1)} \mu_{(i-1;j-1)} (A_{st})\right)\vspace{1mm}\\& =& 
[(\mu_{st}\cdots \mu_{s2}\mu_{s1})\circ \mu_{(s+1;j-1)}\cdots \mu_{(i-2;j-1)} \mu_{(i-1;j-1)} (A_{st})\otimes A''_{n,t+1}].
\end{array}
\end{equation*}  
\end{enumerate}
\end{proposition}
\begin{proof}
    Using Lemma~\ref{lem:mutation3}, the proof of Proposition~\ref{lem:imageofsplit1} works here.
\end{proof}

\begin{proof}[Proof of Lemma \ref{lem:split_e2}]
 Lemma \ref{lem:split_e2} is an immediate corollary of Proposition \ref{lem:imageofsplite2}.
\end{proof}

The following lemma can be proved similarly as Lemma \ref{lem:split_e2}.

\begin{lemma}\label{lem:splitk1}
  For any $i,j$ with $j\leq i$ and any $k$ with $2\leq k\leq j$, we have 

\begin{enumerate}[label={\rm (\alph*)}]\itemsep0,3em

\item
 \begin{equation*}
    \begin{array}{rcl}
&&\mathbb S^{U}_{c_1}\!\bigl(\mathbb S^{U}_{c_3}(\mu_{(k;k-1)}\cdots \mu_{(i-2;k-1)} \mu_{(i-1;k-1)} (A_{k,k-1})))\vspace{1mm}\\& =& 
[\mu_{(k;k-1)}\cdots \mu_{(i-2;k-1)} \mu_{(i-1;k-1)} (A_{k,k-1})\otimes A''_{n,k}\cdot \overline A''_{0,k-1}\cdot \overline A''_{0,i}].
\end{array}
\end{equation*} 

\item  
 \begin{equation*}
    \begin{array}{rcl}
&&\mathbb S^{U}_{c_1}\!\bigl(\mathbb S^{U}_{c_3}(\overline \mu_{(k;n-j)}\cdots \overline \mu_{(j-2;n-j)} \overline \mu_{(j-1;n-j)} (\overline A_{k,k+1})))\vspace{1mm}\\& =& 
[\overline \mu_{(k;n-j)}\cdots \overline \mu_{(j-2;n-j)} \overline \mu_{(j-1;n-j)} (\overline A_{k,k+1})\otimes A''_{nj}\cdot A''_{n,k-1}\cdot \overline A''_{0k}].
\end{array}
\end{equation*} 
\end{enumerate}  
\end{lemma}

Note that Lemma \ref{lem:same} and Corollary \ref{cor:cv} also hold in $\mathscr U_\omega(\fS)$ for the same reason.

\begin{proof}[Proof of Lemma \ref{lem:spitP_4}]
Part (a) is immediate.

Part (b) is an immediate corollary of Lemma \ref{lem:splitk1}.

For part (c), by Lemma \ref{lem:splitk1}, Lemma \ref{lem:same} and Corollary \ref{cor:cv}, we have 
 \begin{equation*}
    \begin{array}{rcl}
&&\mathbb S^{U}_{c_1}\!\bigl(\mathbb S^{U}_{c_3}\left(\left(\mu_{j-1,j-1}\cdots\mu_{22}\mu_{11}\right)\circ \overline \mu^{\diamondsuit}_j\circ  \mu^{\diamondsuit}_{(i;j-1)}(A_{j-1,j-1})\right)\vspace{1mm}\\& =& \mathbb S^{U}_{c_1}\!\bigl(\mathbb S^{U}_{c_3}\left([A_{11}^{-1}\cdot A_{i1}\cdot \overline \mu^{\diamondsuit}_j\circ  \mu^{\diamondsuit}_{(i;j-1)}(\overline A_{12})]+[A_{j-1,j-1}^{-1}\cdot  \overline \mu^{\diamondsuit}_j\circ  \mu^{\diamondsuit}_{(i;j-1)}(A_{j,j-1})\cdot \overline A_{j-1,n}]\right) \vspace{1mm}\\
&+&\mathbb S^{U}_{c_1}\!\bigl(\mathbb S^{U}_{c_3}\left(\sum_{k=2}^{j-1}[A_{k-1,k-1}^{-1}\cdot A_{kk}^{-1}\cdot \overline \mu^{\diamondsuit}_j\circ  \mu^{\diamondsuit}_{(i;j-1)}(A_{k,k-1})\cdot \overline \mu^{\diamondsuit}_j\circ  \mu^{\diamondsuit}_{(i;j-1)}(\overline A_{k,k+1})]\right)\vspace{1mm}\\
&=& [\left(\mu_{j-1,j-1}\cdots\mu_{22}\mu_{11}\right)\circ \overline \mu^{\diamondsuit}_j\circ  \mu^{\diamondsuit}_{(i;j-1)}(A_{j-1,j-1})\otimes A''_{nj}\cdot \overline A''_{0i}].
\end{array}
\end{equation*}  
This completes the proof.
\end{proof}

\end{appendices}

\bibliography{ref.bib}

\end{document}